\documentclass[11pt,a4paper]{article}

\usepackage[T1]{fontenc}
\usepackage[utf8]{inputenc}
\usepackage{fullpage}
\usepackage{amsmath,amssymb,amsthm,mathtools}
\usepackage{enumerate}
\usepackage{booktabs}
\usepackage{mathrsfs}
\usepackage{hyperref}
\usepackage{cite}
\usepackage{color}

\hypersetup{colorlinks=true,linkcolor=blue,citecolor=blue}

\newtheorem{theorem}{Theorem}[section]
\newtheorem{lemma}[theorem]{Lemma}
\newtheorem{corollary}[theorem]{Corollary}
\newtheorem{proposition}[theorem]{Proposition}
\newtheorem{definition}[theorem]{Definition}
\newtheorem{remark}[theorem]{Remark}

\newcommand{\R}{\mathbb{R}}
\newcommand{\C}{\mathbb{C}}

\newcommand{\He}[1]{\mathrm{Her}[#1]}
\newcommand{\tr}{\mathrm{trace}}
\newcommand{\rank}{\mathrm{rank}}
\newcommand{\lmin}{\lambda_{\min}}
\newcommand{\lmax}{\lambda_{\max}}
\newcommand{\ds}{\,\mathrm{d}s}
\newcommand{\dt}{\,\mathrm{d}t}
\let\d\relax
\newcommand{\d}{\,\mathrm{d}}
\newcommand{\norm}[1]{\|#1\|_*}
\newcommand{\schur}[1]{\widetilde{/}}
\newcommand{\ip}[2]{\langle #1,\,#2\rangle}
\newcommand{\Sn}{\mathbb{H}^n}
\newcommand{\Snm}{\mathbb{H}^{n+m}}
\newcommand{\Snpsd}{\mathbb{H}^n_{\succeq0}}
\newcommand{\Snmpsd}{\mathbb{H}^{n+m}_{\succeq0}}
\newcommand{\Snpd}{\mathbb{H}^n_{\succ0}}
\newcommand{\pinv}[1]{{#1}^\dagger}

\providecommand{\blue}[1]{\color{blue}{#1}\color{black}\hspace{0pt}}

\title{\bfseries Optimal Control and Dissipativity of Linear Hermitian Matrix-Valued Dynamical Systems}
\author{Corentin Briat}
\date{\today}

\begin{document}
\maketitle
\tableofcontents
\bigskip

\begin{abstract}
We develop a unified framework for linear-cost optimal control,
finite-time optimal steering, dissipativity analysis, and zero-sum
differential games for linear impulsive dynamical systems whose
state is a Hermitian matrix evolving in $\mathbb{H}^{n+m}_{\succeq0}$.
The system class encompasses
continuous-time (CT) and discrete-time (DT) linear systems as
degenerate cases, switched systems, and includes the dynamics of second-order moments
of linear (stochastic) hybrid systems as a particular instance.

All results rest on three tools: a single \emph{key identity}
relating cost, trajectory, and a dual variable; an Extended Schur
complement lemma; and a Schur inner-product decomposition.  Applied
identically to the flow integral and to each jump, these yield
structurally uniform sufficient and necessary conditions, dual linear matrix inequality (LMI) 
characterizations, and explicit optimal policies across all the considered
problem classes.

For optimal control and dissipativity, we derive finite-horizon
sufficient conditions (via a generalized impulsive Riccati equation,
I-GRE), necessary conditions, and dual LMI representations, then
extend all results to the infinite horizon under LTV assumptions,
without assuming time invariance or periodicity, via a monotone
convergence argument.  For finite-time optimal steering, Lagrangian
duality reduces the terminal-constraint problem to an unconstrained
optimal control problem; the optimal policy is characterized by a two-point
boundary value problem and computed by a forward-backward algorithm.
Zero-sum games are handled via the Hamilton-Jacobi-Isaacs version of
the I-GRE under Isaacs' condition.  For the three problems that
admit dwell-time extensions (optimal control, dissipativity, and
games), causal policies under periodic, minimum, and range
dwell-time constraints are derived by combining a forward Riccati
ODE over a bounded timer interval with a discrete-time Riccati
equation as the boundary condition at the end of the dwell period. Examples are discussed in the related sections.
\end{abstract}

\section{Introduction}

Hermitian matrix-valued dynamical systems form an important class
of systems whose states are Hermitian matrices evolving in time.
They arise naturally whenever the quantity of interest is a
second-order moment, a covariance, or a Gram matrix, all quantities
of central importance in stochastic control, optimal transport,
information geometry, and signal
processing~\cite{Skelton:98,Chen:16,Bamieh:24,Amari:00}.
In particular, the class of systems that leave the cone of positive
semidefinite matrices invariant is of special relevance, as it
encompasses the dynamics of state covariance and second-order moment
matrices associated e.g., with linear stochastic
systems~\cite{Ichikawa:01,Costa:13,Dragan:13,Dragan:25}.
Stability analysis and stabilization of such systems, including
continuous-time, discrete-time, and impulsive variants, were studied
in~\cite{Briat:22:Matrix} under a unified framework, extending earlier analyses
conducted separately for specific problem classes and under additional
structural assumptions~\cite{Freiling:96,AitRami:01a,AitRami:01b,Costa:13,Ichikawa:01}.
The present paper extends that work to the problems of optimal
control, optimal steering, dissipativity, and zero-sum differential
games.

The view of linear systems through the lens of second-order information
and covariance representations has a long history.
The systematic treatment of covariance control (the problem of steering
and assigning the state covariance of a stochastic system by means of
linear state feedback) was pioneered by Hotz and Skelton~\cite{Hotz:87}
for discrete-time systems, and subsequently placed within a broader
algebraic framework in the monograph by Skelton, Iwasaki, and
Grigoriadis~\cite{Skelton:98}, which unifies the design of covariance
controllers, $H_2$ controllers, and $H_\infty$ controllers under a
common linear-algebra formulation involving covariance matrices and
matrix inequalities.
More recently, Gattami~\cite{Gattami:16} proposed an $H_\infty$ analysis
method based directly on the state covariance, and Bamieh~\cite{Bamieh:24}
gave a comprehensive treatment of optimal control problems in systems and controls
using a covariance-representation approach in which LQ objectives become
linear-in-matrix semidefinite programs and LMIs emerge as the natural
conal constraints of their duals.
The present paper adopts a similar point of view but works directly with
the matrix-valued second-moment process as a dynamical system in its own
right, rather than as a representation of an underlying vector state.

The classical linear quadratic regulator (LQR) for vector-state linear
systems is characterized by the algebraic or differential Riccati
equation; see~\cite{AbouKandil:03} for a comprehensive treatment of
matrix Riccati equations and their role in systems theory,
and~\cite{Bertsekas:12,Bertsekas:17} for the dynamic-programming
foundations of optimal control on which the necessary conditions of
this paper rest.
When the cost matrices are allowed to be indefinite, the problem becomes
significantly more involved: the generalized differential Riccati equation
and the corresponding LMI conditions were studied in depth in the series
of papers~\cite{AitRami:96,AitRami:00,AitRami:01a,AitRami:01b},
which established solvability conditions, asymptotic behavior, and the
connection to stochastic LQ control.
For systems with periodic parameters or periodically sampled inputs, the
periodic Riccati equation plays the central role; the theory of periodic
systems and its connections to $H_2$ and $H_\infty$ control are surveyed
in~\cite{Bittanti:09}.
In all of these references, the infinite-horizon results are obtained
under a time-invariance or periodicity assumption on the system and
cost operators.

Impulsive and jump linear systems, in which a continuous-time flow is
combined with instantaneous state resets at a sequence of event times,
arise in sampled-data control \cite{Naghshtabrizi:08,Goebel:12, Briat:13:ImpDet, Briat:16:ImpStoch}, networked control systems  \cite{Naghshtabrizi:08,Heemels:10,Cloosterman:10}, reset systems \cite{Banos:12} and the modeling
of physical systems subject to impulses \cite{Brogliato:00}.
The $H_2$ and $H_\infty$ theory for linear time-varying and jump systems
was developed in~\cite{Ichikawa:01}, which established the connection
between jump systems and sampled-data systems and derived the associated
Riccati equations in the LTV setting; the $H_\infty$ analysis of
sampled-data systems through this jump-system equivalence was developed
in~\cite{Sivashankar:94}.
Markov jump linear systems, where the jump times and the post-jump
parameters are governed by a Markov chain, have been studied extensively;
see~\cite{Costa:13,Dragan:13} for continuous-time and~\cite{Costa:05,Dragan:10}
for discrete-time systems.
Sampled-data control of switched and Markov jump linear systems has been
investigated in depth by Geromel and co-workers~\cite{Geromel:19,Geromel:15b,Souza:21}.
More recent contributions, including optimal control for linear impulsive
stochastic systems and their application to sampled-data problems, can be
found in~\cite{Dragan:25}.
In all these references, the treatment of infinite-horizon problems
assumes, either explicitly or implicitly, that the system parameters are
time-invariant or periodic.

Covariance steering in finite time, the problem of steering the
state second-moment of a linear stochastic system from a prescribed
initial value to a prescribed terminal one while minimizing a given
cost, was formulated and solved for continuous-time LTV systems via a
Schr\"odinger bridge interpretation in the series of papers by Chen,
Georgiou, and Pavon~\cite{Chen:16,Chen:16b}, where the optimal policy
is characterized by a system of two coupled Riccati or Lyapunov equations
and shown to coincide with an LQG controller with a specific terminal cost
weight. Goldshtein and Tsiotras~\cite{Goldshtein:17} addressed the discrete
time-varying case and provided a computational algorithm, while
Bakolas~\cite{Bakolas:16} treated the problem subject to integral
quadratic constraints. The extension of optimal steering to impulsive systems, where
the second-moment is subject to resets at each impulse time, has not
been addressed in the literature.

Dissipativity theory, introduced by Willems~\cite{Willems:72}, provides
a general framework for characterising the energy-like behavior of
dynamical systems in terms of supply rates and storage functions \cite{Scherer:05a,vanderSchaft:00,Brogliato:07}.
It unifies a broad range of analysis and design questions, including
passivity, $H_\infty$ performance, and integral quadratic
constraints~\cite{Seiler:19}, under a single input-output framework.
For linear systems, dissipativity is equivalent to the existence of a
solution to a Lyapunov or Riccati inequality, and the available storage
and required supply have natural state-space characterizations.
The dissipativity of impulsive systems has been studied in several
works \cite{Haddad:06}; however, to the best of our knowledge, a treatment in the
matrix-valued setting that addresses both finite- and infinite-horizon
conditions under LTV assumptions and with indefinite supply rates is
not available in the literature.

The theory of two-player zero-sum differential games for linear systems
was developed in detail by Ba\c{s}ar and Olsder~\cite{Basar:95}, and
the connection to $H_\infty$ optimal control~\cite{Doyle:89a}, formulated
as a minimax problem over the disturbance input, is well established;
for linear time-varying systems, the corresponding state-space
$H_\infty$ and game solutions were developed in~\cite{Ravi:91}.
The key role of the Hamilton--Jacobi--Isaacs equation, and its reduction
to a Riccati-type equation in the linear-quadratic case, places the
game-theoretic problem squarely within the framework of matrix Riccati
equations; see again~\cite{AbouKandil:03} for the game-theoretic Riccati
equations and~\cite{Basar:95} for the general theory.

An important class of causal policies for impulsive systems are
those based on a \emph{timer} variable, which
measures the elapsed time since the last
jump~\cite{Morse:96,Hespanha:99,Goebel:12}.
For stability analysis, conditions based on minimum dwell-time (MDT), were studied by Geromel
and Colaneri~\cite{Geromel:06,Geromel:06d} in the context of switched systems and were later considered for switched \cite{Xiang:16} and impulsive systems \cite{Briat:11l,Briat:13:ImpDet} where maximum dwell-time, range dwell-time (RDT), and constant dwell-time were also considered in the context of looped-functionals and timer-dependent Lyapunov functions. Analogous concepts include the average dwell-time and the reverse average dwell-time \cite{Hespanha:08}. Some of those results were later extended to matrix-valued impulsive and hybrid systems in~\cite{Briat:22:Matrix}.
Timer-dependent Lyapunov conditions for dwell-time analysis of
deterministic and stochastic impulsive and sampled-data systems have
been developed in~\cite{Briat:13:ImpDet,Briat:16:ImpStoch}.
These conditions are particularly valuable because they are verifiable
on bounded time intervals even though the dwell-time sequences may be
infinite; the Geromel--Colaneri stability condition at the end of
the minimum dwell period is the key technique that converts an
infinite-horizon verification to a finite one. To date, analogous conditions in the context of optimal control, dissipativity, and games have not been developed in the matrix-valued impulsive setting.

The present paper fills these gaps.
We address the linear-cost optimal control problem, optimal steering,
dissipativity analysis, and minimax optimal control for the class of
linear Hermitian impulsive matrix-valued systems introduced
in~\cite{Briat:22:Matrix}, covering both finite- and infinite-horizon
problems, and developing dwell-time extensions for all three problems.
The paper is built around a single \emph{key identity} that holds for
any admissible trajectory and any dual variable, together with two Schur
complement lemmas; these three tools apply identically to the flow integral
and to each jump summand, giving a structurally uniform proof method
throughout.

A central feature is that all infinite-horizon results are derived for
\emph{linear time-varying} systems and cost operators under suitable
stabilizability and detectability conditions, without assuming
time-invariance or periodicity, in contrast to the treatments
in~\cite{Ichikawa:01,Dragan:25,Costa:13,AbouKandil:03,Bittanti:09}.
Specifically, a monotone convergence argument shows that the
horizon-$T$ optimal cost is non-decreasing in $T$ and converges to
the infinite-horizon value; the time-varying case requires this
functional-analytic approach in place of the algebraic Riccati
equations that characterize the LTI case. Examples are considered in the related sections.

The paper is structured as follows.
Notation, Schur complement lemmas, and the system class are collected
in Section~\ref{sec:prelim}.
Linear-cost optimal control, including dwell-time extensions, is covered
in Section~\ref{sec:OC}.
Finite-time optimal steering, the two-point boundary value problem,
and the forward-backward algorithm are in Section~\ref{sec:cov}.
Dissipativity analysis is in Section~\ref{sec:diss}, and minimax control
in Section~\ref{sec:game}.

\section{Preliminaries}
\label{sec:prelim}

\subsection{Notation}

Throughout, $n,m\ge1$ are fixed integers. The set of Hermitian matrices, positive semidefinite Hermitian, and positive definite Hermitian matrices of dimension $n$ are denoted by $\Sn$, $\Snpsd$, and $\Snpd$. We denote the conjugate transpose of a matrix $A$ by $A^*$, and we write $\He{A}:=A+A^*$. With a slight abuse of notation, the dual of an operator $\mathcal A$ by $\mathcal A^*$.  For a matrix $$X=\begin{bmatrix}
  X_1 & X_2^*\\
  X_2 & X_3
\end{bmatrix}\in\Snmpsd,$$ we set $X_1:=EXE^*\in\Sn$,
$X_2:=E_\bot XE^*\in\C^{m\times n}$, and
$X_3:=E_\bot XE_\bot^*\in\mathbb{H}^m_{\succeq0}$, where $E:=\bigl[\begin{smallmatrix}I_n&0\end{smallmatrix}\bigr]
\in\C^{n\times(n+m)}$ and
$E_\bot:=\bigl[\begin{smallmatrix}0&I_m\end{smallmatrix}\bigr]
\in\C^{m\times(n+m)}$. For a function $f$ defined on $\mathbb{R}$, the left limit at $t$ is defined as $f(t^-):\lim_{s\uparrow t}f(s)$ while the right limit at $t$ is denoted by $f(t^+):\lim_{s\downarrow t}f(s)$.

\subsection{Inner product and nuclear norm}

\begin{definition}[{\cite[Def.~2.2]{Briat:22:Matrix}}]\label{def:ip}
  The \emph{Frobenius inner product} on $\C^{n\times n}$ is
  \begin{equation}
    \ip{A}{B}:=\tr(AB^*),\qquad A,B\in\C^{n\times n}.
  \end{equation}
  For $A\in\Snpsd$ the nuclear norm equals the trace:
  $\norm{A}:=\tr(A)$.
\end{definition}

\begin{proposition}[{\cite[Prop.~2.3]{Briat:22:Matrix}}]\label{prop:PQ}
  Let $P\in\Snpd$, $Q\in\Snpsd$, $R\in\Sn$.  Then:
  \begin{enumerate}[\upshape(a)]
    \item\label{st:zero} $\ip{P}{Q}=0$ if and only if $Q=0$.
    \item\label{st:neg}
      $\ip{R}{Q}\le0$ for all $Q\succeq0$, $Q\ne0$, if and only if
      $R\preceq0$.
    \item\label{st:bound}
      $\lmin(P)\norm{Q}\le\ip{P}{Q}\le\lmax(P)\norm{Q}$ for all
      $Q\succeq0$.
  \end{enumerate}
\end{proposition}

\subsection{Schur complements}

\begin{definition}[Generalized Schur complement]\label{def:genschur}
  For $\mathcal{L}=\bigl[
  \begin{smallmatrix}\mathcal{L}_1&\mathcal{L}_2\\
  \mathcal{L}_2^*&\mathcal{L}_3\end{smallmatrix}\bigr]\in\Snm$,
  the \emph{generalized Schur complement} of $\mathcal{L}_3$ in
  $\mathcal{L}$ is
  \begin{equation}\label{eq:genSchur}
    \mathcal{L}\,\schur{/}\,\mathcal{L}_3
    :=\mathcal{L}_1-\mathcal{L}_2\pinv{\mathcal{L}_3}\mathcal{L}_2^*.
  \end{equation}
  When $\mathcal{L}_3$ is invertible this coincides with the ordinary
  Schur complement.  The same notation is used for any partitioned
  Hermitian matrix $\mathcal{L}$, regardless of context.
\end{definition}

The following two lemmas are used in every proof of this paper.

\begin{lemma}[Extended Schur lemma,
  {\cite{Boyd:94a}}]\label{lemma:extSchur}
  Let $M=M^*$, $N$, $R=R^*$ have compatible dimensions.
  The following are equivalent:
  \begin{enumerate}[\upshape(i)]
    \item\label{extS:i} $M-N\pinv{R}N^*\succeq0$,\quad
      $R\succeq0$,\quad and\quad $N(I-\pinv{R}R)=0$.
    \item\label{extS:ii}
      $\begin{bmatrix}M&N\\N^*&R\end{bmatrix}\succeq0$.
  \end{enumerate}
\end{lemma}

\begin{lemma}[Schur inner-product decomposition]\label{lemma:decomp}
  Let $\mathcal{L}\in\Snm$ with $\mathcal{L}_3\succeq0$ and
  $X\in\Snmpsd$.  Write $X=UU^*$ with
  $U=[U_1^*\;U_2^*]^*\in\C^{(n+m)\times r}$,
  $r=\rank(X)$.  Suppose the \emph{compatibility condition}
  \begin{equation}\label{eq:compat}
    \mathcal{L}_2(I-\pinv{\mathcal{L}_3}\mathcal{L}_3)=0
    \quad\bigl(\text{equivalently,\;}
    \operatorname{Im}(\mathcal{L}_2^*)\subseteq
    \operatorname{Im}(\mathcal{L}_3)\bigr)
  \end{equation}
  holds.  Set $W:=\pinv{\mathcal{L}_3}\mathcal{L}_2^* U_1+U_2$.
  Then
  \begin{equation}\label{eq:decomp}
    \ip{\mathcal{L}}{X}
    =\ip{\mathcal{L}\,\schur{/}\,\mathcal{L}_3}{U_1U_1^*}
    +\ip{\mathcal{L}_3}{WW^*}.
  \end{equation}
  In particular, whenever $\mathcal{L}\succeq0$,
  condition~\eqref{eq:compat} holds automatically
  by Lemma~\ref{lemma:extSchur}.
\end{lemma}

\subsection{Dynamical systems}

\begin{definition}[Impulsive system,
  {\cite[Sec.~5.1]{Briat:22:Matrix}}]\label{def:IMP:sys}
  The \emph{impulsive} matrix-valued system combines a
  continuous-time flow with instantaneous jumps at times
  $t_1<t_2<\cdots$, $t_k\to\infty$:
  \begin{subequations}\label{eq:sys:IMP}
  \begin{align}
    E\dot{X}(t)E^*
      &=\mathcal{F}(t,X(t)),\quad t\ne t_k,
       \quad t\ge t_0,
       \label{eq:sys:IMP:flow}\\
    EX(t_k^+)E^*
      &=\mathcal{J}(k,X(t_k^-)),\quad k\ge1,
       \label{eq:sys:IMP:jump}
  \end{align}
  \end{subequations}
  with $EX(t_0)E^*=X_1^0\in\Snpsd$.
  The \emph{flow operator} $\mathcal{F}:\R_{\ge0}\times\Snmpsd\to\blue{\Sn}$
  is linear, bounded, continuous in $t$, leaves $\Snmpsd$
  invariant~\cite{Briat:22:Matrix}, and has adjoint
  $\mathcal{F}^*(t,\cdot):\Sn\to\Snm$ satisfying
  $\ip{P}{\mathcal{F}(t,X)}=\ip{\mathcal{F}^*(t,P(t))}{X}$.
  The \emph{jump operator} $\mathcal{J}(k,\cdot):\Snmpsd\to\blue{\Sn}$
  is linear, bounded, leaves $\Snmpsd$ invariant, and has adjoint
  $\ip{P}{\mathcal{J}(k,X)}=\ip{\mathcal{J}^*(k,P(t))}{X}$.
\end{definition}

\begin{definition}[Continuous-time system,
  {\cite[Sec.~3.1]{Briat:22:Matrix}}]\label{def:CT:sys}
  The \emph{continuous-time} (CT) system is the impulsive
  system~\eqref{eq:sys:IMP} with no jump times ($N_T=0$):
  \begin{equation}\label{eq:sys:CT}
    E\dot{X}(t)E^*=\mathcal{F}(t,X(t)),\quad t\ge t_0,\qquad
    EX(t_0)E^*=X_1^0.
  \end{equation}
\end{definition}

\begin{definition}[Discrete-time system,
  {\cite[Sec.~4.1]{Briat:22:Matrix}}]\label{def:DT:sys}
  The \emph{discrete-time} (DT) system is the impulsive
  system~\eqref{eq:sys:IMP} with $\mathcal{F}(t,X)=0$ on each
  interval $(t_{k-1},t_k)$, and jump times $t_k=k_0+k$:
  \begin{equation}\label{eq:sys:DT}
    EX(k+1)E^*=\mathcal{J}(k,X(k)),\quad k\ge k_0,\qquad
    EX(k_0)E^*=X_1^0.
  \end{equation}
\end{definition}

\section{Linear-Cost Optimal Control}
\label{sec:OC}

This section develops the linear-cost optimal control theory for the
impulsive matrix-valued system~\eqref{eq:sys:IMP}.  The term
\emph{linear-cost} refers to the fact that the cost functional is
linear in the trajectory $X(\cdot)\in\Snmpsd$; the classical
linear-quadratic (LQ) terminology is recovered when $X$ represents
the second-order moment matrix of an underlying stochastic process,
but no such interpretation is required here.

The central analytical tool is a single algebraic
\emph{key identity} (Proposition~\ref{prop:keyid}) that holds for any
admissible trajectory and any piecewise-$C^1$ dual variable.  All
sufficiency, necessity, and duality results follow by one application
of this identity together with the Schur complement
Lemmas~\ref{lemma:extSchur}--\ref{lemma:decomp}.
The finite-horizon theory is in Section~\ref{sec:OC:fh}, the
infinite-horizon theory under LTV conditions in
Section~\ref{sec:OC:inf}, and causal dwell-time policies in
Section~\ref{sec:OC:dwell}.  Continuous-time and discrete-time
corollaries appear throughout.

\subsection{Finite-horizon}
\label{sec:OC:fh}

We study the optimal control problem on the bounded horizon
$[t_0,t_0+T]$ for a fixed $T>0$.  The decision variable is the full
trajectory $X(\cdot)\in\Snmpsd$, which accounts simultaneously for
the constrained state block $X_1$ and the free input blocks
$X_2$, $X_3$.  We define the cost, the composite operators, and the
key identity in Section~\ref{sec:OC:fh:cost}, the I-GRE and optimal
sets in Section~\ref{sec:OC:fh:GRE}, and the sufficient, necessary,
and dual conditions in
Theorems~\ref{th:IMP:LQ:suff}--\ref{th:IMP:LQ:LMI}.

\subsubsection{Cost function, decision variable, and operators}
\label{sec:OC:fh:cost}

Let $\{t_k\}_{k=1}^{N_T}$ be the jump times in $(t_0,t_0+T)$, with
the convention $t_0^+:=t_0$ and $t_{N_T+1}^-:=t_0+T$.
For cost weight matrices $Z:\R_{\ge0}\to\Snm$ with
$Z(t)\succeq0$, $Z_k\in\Snmpsd$, and terminal weight $Z_T\in\Snpsd$,
the finite-horizon cost is
\begin{equation}\label{eq:IMP:LQ:cost}
  J_T(X,X_1^0)
  :=\int_{t_0}^{t_0+T}\ip{Z(t)}{X(t)}\dt
  +\sum_{k=1}^{N_T}\ip{Z_k}{X(t_k^-)}
  +\ip{Z_T}{EX(t_0+T)E^*}.
\end{equation}
The three terms represent, respectively, the accumulated flow cost, the
total jump cost, and the terminal penalty on the constrained block.
The cost is \emph{linear} in $X(\cdot)\in\Snmpsd$: when $X$ is the
second-moment matrix of a stochastic process with state $x(t)$ and
input $u(t)$, the inner product $\ip{Z(t)}{X(t)}=\mathrm{tr}(Z(t)X(t))$
equals the expected value $\mathbb{E}[{[x^*\;u^*]}\,Z(t)\,[x^*\;u^*]^*]$,
recovering the LQ expected cost.  The lower-right blocks
$Z_3(t)\succeq0$ and $Z_{k,3}\succeq0$, which correspond to the
input-channel cost, are allowed to be \emph{singular}, encompassing
free-input problems and singular optimal control.

We consider now the following optimization problem
\begin{equation}\label{eq:IMP:LQ:prob}
  \underset{X(\cdot)\in\Snmpsd}{\mathrm{minimize}}
  \quad J_T(X,X_1^0)
  \quad\text{subject to \eqref{eq:sys:IMP} and }
  EX(t_0)E^*=X_1^0
\end{equation}
with optimal cost $$J_T^\star(X_1^0):=\underset{X(\cdot)\in\Snmpsd}{\inf} J_T(X,X_1^0),$$ where the \emph{decision variable} is the full trajectory
$X(\cdot)\in\Snmpsd$ subject to~\eqref{eq:sys:IMP} and
$EX(t_0)E^*=X_1^0$. Observe that in this abstract setting, $X_1(\cdot)=EX(\cdot)E^*$ is the \emph{constrained} block governed by the flow~\eqref{eq:sys:IMP:flow} and the jumps~\eqref{eq:sys:IMP:jump} dynamics, while the blocks
$X_2(\cdot)=E_\bot X(\cdot)E^*$ and $X_3(\cdot)=E_\bot X(\cdot)E_\bot^*$ are the \emph{free} blocks. The constraint $X(\cdot)\in\Snmpsd$ ensures that $X(\cdot)$ remains a
valid positive semidefinite matrix throughout.


The composite operators below are the fundamental quantities in all
subsequent results.  They absorb the cost weights, the dynamics, and
the dual variable $P$ into a single symmetric matrix whose
semidefiniteness is the key optimality condition.
For a piecewise-$C^1$ function $P:[t_0,t_0+T]\to\Sn$, define
the \emph{composite flow operator}\footnote{In proof steps, we write $\mathcal{C}_Z(t,P)$ as shorthand for
$\mathcal{C}_Z(t,P(t))$, and $X$ for $X(t)$, when the time point is fixed;
in theorem statements and definitions the argument $P(t)$ is always explicit.}
\begin{equation}\label{eq:IMP:L}
  \mathcal{C}_Z(t,P(t))
  :=E^*\dot{P}(t)E+\mathcal{F}^*(t,P(t))+Z(t)\in\Snm,
\end{equation}
partitioned as
$\bigl[\begin{smallmatrix}(\mathcal{C}_Z)_1&(\mathcal{C}_Z)_2\\
(\mathcal{C}_Z)_2^*&(\mathcal{C}_Z)_3\end{smallmatrix}\bigr]$,
and the \emph{jump operator} at $t_k$
\begin{equation}\label{eq:IMP:Lj}
  \mathcal{D}_Z\bigl(k,P(t_k^+),P(t_k^-)\bigr)
  :=\mathcal{J}^*\bigl(k,P(t_k^+)\bigr)-E^* P(t_k^-)E+Z_k\in\Snm,
\end{equation}
partitioned as $\bigl[\begin{smallmatrix}
(\mathcal{D}_Z)_1&(\mathcal{D}_Z)_2\\
(\mathcal{D}_Z)_2^*&(\mathcal{D}_Z)_3
\end{smallmatrix}\bigr]$.

Intuitively, $\mathcal{C}_Z(t,P(t))$ captures the time derivative of
$t\mapsto\ip{P(t)}{EX(t)E^*}$ augmented by the running cost $Z(t)$,
while $\mathcal{D}_Z(k,P(t_k^+),P(t_k^-))$ captures the jump in
$\ip{P}{EXE^*}$ augmented by the jump cost $Z_k$.
Both are symmetric $(n+m)\times(n+m)$ matrices, partitioned into
$(n\times n)$, $(m\times n)$, and $(m\times m)$ blocks corresponding
to the state, cross-term, and input channels.

The following identity is the single most important tool of the
paper.  It holds for \emph{any} piecewise-$C^1$ function $P$ and
\emph{any} admissible trajectory, regardless of whether $P$ solves
the I-GRE.  By choosing $P$ to be the I-GRE solution, all subsequent
optimality conditions reduce to signed comparisons of its terms.

\begin{proposition}[Impulsive key identity]\label{prop:keyid}
For any piecewise-$C^1$ $P:[t_0,t_0+T]\to\Sn$ and any
trajectory $X(\cdot)\in\Snmpsd$ satisfying~\eqref{eq:sys:IMP}
with $EX(t_0)E^*=X_1^0$,
\begin{align}
  J_T(X,X_1^0)
  &=\ip{P(t_0)}{X_1^0}
   +\ip{Z_T-P(t_0+T)}{EX(t_0+T)E^*}
   \notag\\
  &\quad
   +\int_{t_0}^{t_0+T}\ip{\mathcal{C}_Z(t,P(t))}{X(t)}\dt
   +\sum_{k=1}^{N_T}
    \ip{\mathcal{D}_Z(k,P(t_k^+),P(t_k^-))}{X(t_k^-)}.
   \label{eq:IMP:keyid}
\end{align}
When $P(t_0+T)=Z_T$ the second term vanishes:
\begin{equation}\label{eq:IMP:keyid2}
  J_T(X,X_1^0)
  =\ip{P(t_0)}{X_1^0}
  +\int_{t_0}^{t_0+T}\ip{\mathcal{C}_Z(t,P(t))}{X(t)}\dt
  +\sum_{k=1}^{N_T}
   \ip{\mathcal{D}_Z(k,P(t_k^+),P(t_k^-))}{X(t_k^-)}.
\end{equation}
\end{proposition}
\begin{proof}
On the flow interval $(t_{k-1},t_k)$, differentiating
$\ip{P(t)}{EX(t)E^*}$ along~\eqref{eq:sys:IMP:flow} and
integrating yields
\begin{align}\label{pf:key:flow}
  &\ip{P(t_k^-)}{EX(t_k^-)E^*}
  -\ip{P(t_{k-1}^+)}{EX(t_{k-1}^+)E^*}
  \notag\\
  &\;=-\int_{t_{k-1}}^{t_k}\ip{Z(t)}{X(t)}\dt
  +\int_{t_{k-1}}^{t_k}\ip{\mathcal{C}_Z(t,P(t))}{X(t)}\dt.
\end{align}
At each jump $t_k$, using~\eqref{eq:sys:IMP:jump} and the adjoint gives
\begin{align}
  \ip{P(t_k^+)}{EX(t_k^+)E^*}-\ip{P(t_k^-)}{EX(t_k^-)E^*}
  &=\ip{P(t_k^+)}{\mathcal{J}(k,X(t_k^-))}-\ip{P(t_k^-)}{EX(t_k^-)E^*}
   \notag\\
  &=\ip{\mathcal{J}^*(k,P(t_k^+))-E^*P(t_k^-)E}{X(t_k^-)}
   \notag\\
  &=\ip{\mathcal{D}_Z(k,P(t_k^+),P(t_k^-))}{X(t_k^-)}
   -\ip{Z_k}{X(t_k^-)}.
   \label{pf:key:jump}
\end{align}
Summing~\eqref{pf:key:flow} over all $N_T+1$ flow intervals and
\eqref{pf:key:jump} over all $N_T$ jumps, the terms telescope to
give $\ip{P(t_0+T)}{EX(t_0+T)E^*}-\ip{P(t_0)}{X_1^0}$ on the
left.  Rearranging yields~\eqref{eq:IMP:keyid}.
\end{proof}

\subsubsection{Impulsive GRE and optimal sets}
\label{sec:OC:fh:GRE}

With the composite operators in hand, we now define the central object
of the theory: the \emph{impulsive generalized Riccati equation}
(I-GRE).  This is a coupled flow-jump terminal-value problem for
the dual variable $P(\cdot)$ with $P(t_0+T)=Z_T$.  The I-GRE
imposes that $\mathcal{C}_Z(t,P(t))\succeq0$ with vanishing generalized
Schur complement on each flow interval, and the analogous
condition on $\mathcal{D}_Z$ at each jump.  By
Lemma~\ref{lemma:GREequiv}, these are equivalent to a single
condition per interval: $\mathcal{C}_Z\succeq0$ and
$\mathcal{C}_Z\,\schur{/}\,\allowbreak(\mathcal{C}_Z)_3=0$.
The I-GRE generalizes the classical differential Riccati equation
to the impulsive setting and reduces to it in the regular case
$(\mathcal{C}_Z)_3\succ0$.

\begin{definition}[Impulsive GRE]\label{def:IMP:GRE}
  The \emph{impulsive generalized Riccati equation} (I-GRE) for
  piecewise-$C^1$ $P:[t_0,t_0+T]\to\Sn$ with $P(t_0+T)=Z_T$
  consists of:
  \begin{itemize}
    \item \textit{Flow conditions,} for a.e.\ $t$ between jumps:
      \begin{subequations}\label{eq:IMP:GRE:flow}
      \begin{align}
        &(\mathcal{C}_Z)_1(t,P(t))-(\mathcal{C}_Z)_2(t,P(t))\,
          \pinv{(\mathcal{C}_Z)_3(t,P(t))}\,
          (\mathcal{C}_Z)_2(t,P(t))^*=0,
          \label{eq:IMP:GRE:fR}\\
        &(\mathcal{C}_Z)_3(t,P(t))\,\pinv{(\mathcal{C}_Z)_3(t,P(t))}\,
          (\mathcal{C}_Z)_2(t,P(t))^*=(\mathcal{C}_Z)_2(t,P(t))^*,
          \label{eq:IMP:GRE:fC}\\
        &(\mathcal{C}_Z)_3(t,P(t))\succeq0.
          \label{eq:IMP:GRE:fPSD}
      \end{align}
      \end{subequations}
    \item \textit{Jump conditions,} at each $t_k$
      ($k=1,\ldots,N_T$):
      \begin{subequations}\label{eq:IMP:GRE:jump}
      Writing $P_k^+:=P(t_k^+)$ and $P_k^-:=P(t_k^-)$:
      \begin{align}
        &(\mathcal{D}_Z)_1(k,P_k^+,P_k^-)
          -(\mathcal{D}_Z)_2(k,P_k^+,P_k^-)\,
          \pinv{(\mathcal{D}_Z)_3(k,P_k^+,P_k^-)}\,
          (\mathcal{D}_Z)_2(k,P_k^+,P_k^-)^*=0,
          \label{eq:IMP:GRE:jR}\\
        &(\mathcal{D}_Z)_3(k,P_k^+,P_k^-)\,
          \pinv{(\mathcal{D}_Z)_3(k,P_k^+,P_k^-)}\,
          (\mathcal{D}_Z)_2(k,P_k^+,P_k^-)^*
          =(\mathcal{D}_Z)_2(k,P_k^+,P_k^-)^*,
          \label{eq:IMP:GRE:jC}\\
        &(\mathcal{D}_Z)_3(k,P_k^+,P_k^-)\succeq0.
          \label{eq:IMP:GRE:jPSD}
      \end{align}
      \end{subequations}
  \end{itemize}
  By Lemma~\ref{lemma:GREequiv}, each triple is equivalent to
  the respective operator being $\succeq0$ with vanishing
  generalized Schur complement.
  The \emph{flow optimal set} and \emph{jump optimal set} are
  \begin{align}
    \mathcal{K}_c(t,P(t))
    &:=\{K\in\C^{m\times n}\mid
      (\mathcal{C}_Z)_3(t,P(t))K^*+(\mathcal{C}_Z)_2(t,P(t))^*=0\},
      \label{eq:IMP:Kf}\\
    \mathcal{K}_d(k,P_k^+,P_k^-)
    &:=\{K\in\C^{m\times n}\mid
      (\mathcal{D}_Z)_3(k,P_k^+,P_k^-)\,K^*
      {}+(\mathcal{D}_Z)_2(k,P_k^+,P_k^-)^*=0\},
      \label{eq:IMP:Kj}
  \end{align}
  whose general element is
  \[
    K=\pinv{(\mathcal{C}_Z)_3(t,P(t))}(\mathcal{C}_Z)_2(t,P(t))^*
    +F\bigl(I-(\mathcal{C}_Z)_3(t,P(t))\pinv{(\mathcal{C}_Z)_3(t,P(t))}\bigr)
  \]
  and
    \[
    K=\pinv{(\mathcal{D}_Z)_3(k,P_k^+,P_k^-)}(\mathcal{D}_Z)_2(k,P_k^+,P_k^-)^*
    +F\bigl(I-(\mathcal{D}_Z)_3(k,P_k^+,P_k^-)\pinv{(\mathcal{D}_Z)_3(k,P_k^+,P_k^-)}\bigr)
  \]
  for $\mathcal{K}_c$ and $\mathcal{K}_d$, respectively, and where $F\in\C^{m\times n}$~\cite{AitRami:01a}.
\end{definition}

\begin{remark}
  When $(\mathcal{C}_Z)_3\succ0$ and $(\mathcal{D}_Z)_3\succ0$
  (in particular when $Z_3\succ0$ and $Z_{k,3}\succ0$), the
  pseudoinverses become ordinary inverses and the compatibility
  conditions~\eqref{eq:IMP:GRE:fC},~\eqref{eq:IMP:GRE:jC} hold
  trivially.  The I-GRE then reduces to classical Riccati equations.
\end{remark}

The result below plays an essential role in the derivation of the main results in this section and is essentially a limiting case of the Schur complement lemma (Lemma \ref{lemma:extSchur}):
\begin{lemma}[Schur equality lemma]\label{lemma:GREequiv}
  Let $\mathcal{L}=\bigl[\begin{smallmatrix}
  \mathcal{L}_1&\mathcal{L}_2\\
  \mathcal{L}_2^*&\mathcal{L}_3\end{smallmatrix}\bigr]\in\Snm$.
  The following are equivalent:
  \begin{enumerate}[\upshape(i)]
    \item\label{GREeq:LMI} $\mathcal{L}\succeq0$ and
      $\mathcal{L}\,\schur{/}\,\mathcal{L}_3=0$.
    \item\label{GREeq:triple}
      $\mathcal{L}_1-\mathcal{L}_2\pinv{\mathcal{L}_3}
      \mathcal{L}_2^*=0$,\quad
      $\mathcal{L}_3\pinv{\mathcal{L}_3}\mathcal{L}_2^*
      =\mathcal{L}_2^*$,\quad and\quad
      $\mathcal{L}_3\succeq0$.
  \end{enumerate}
\end{lemma}

\begin{proof}
\noindent\eqref{GREeq:LMI}$\Rightarrow$\eqref{GREeq:triple}.
Since $\mathcal{L}\succeq0$, Lemma~\ref{lemma:extSchur} gives
$\mathcal{L}_3\succeq0$,
$\mathcal{L}_1-\mathcal{L}_2\pinv{\mathcal{L}_3}\mathcal{L}_2^*\succeq0$,
and the compatibility
$\mathcal{L}_2(I-\pinv{\mathcal{L}_3}\mathcal{L}_3)=0$.
Combined with $\mathcal{L}\,\schur{/}\,\mathcal{L}_3=0$, we get
$\mathcal{L}_1-\mathcal{L}_2\pinv{\mathcal{L}_3}\mathcal{L}_2^*=0$.
The compatibility condition reads
$\mathcal{L}_2=\mathcal{L}_2\pinv{\mathcal{L}_3}\mathcal{L}_3$;
taking conjugate transposes gives
$\mathcal{L}_2^*=\mathcal{L}_3\pinv{\mathcal{L}_3}\mathcal{L}_2^*$.

\noindent\eqref{GREeq:triple}$\Rightarrow$\eqref{GREeq:LMI}.
From the first condition and Definition~\ref{def:genschur},
$\mathcal{L}\,\schur{/}\,\mathcal{L}_3=0$.
The second condition is the compatibility
$\operatorname{Im}(\mathcal{L}_2^*)\subseteq\operatorname{Im}(\mathcal{L}_3)$,
so Lemma~\ref{lemma:extSchur} applies with
$M=\mathcal{L}_1=\mathcal{L}_2\pinv{\mathcal{L}_3}\mathcal{L}_2^*$.
Since $\mathcal{L}_3\succeq0$ we have $\pinv{\mathcal{L}_3}\succeq0$,
so $\mathcal{L}_1=\bigl((\pinv{\mathcal{L}_3})^{1/2}\mathcal{L}_2^*\bigr)^*
\bigl((\pinv{\mathcal{L}_3})^{1/2}\mathcal{L}_2^*\bigr)\succeq0$,
giving $\mathcal{L}\succeq0$.
\end{proof}

\subsubsection{Sufficient condition}

Theorem~\ref{th:IMP:LQ:suff} below shows that any piecewise-$C^1$
solution $P\succeq0$ of the I-GRE is the unique solution, and that
$\ip{P(t_0)}{X_1^0}$ is both the optimal cost and the value attained
by any trajectory with gains in the optimal sets $\mathcal{K}_c$
and $\mathcal{K}_d$ in \eqref{eq:IMP:Kf} and \eqref{eq:IMP:Kj}.  The proof uses the key identity in 
Proposition~\ref{prop:keyid} to decompose the cost into a fixed
initial term $\ip{P(t_0)}{X_1^0}$ and non-negative remainder terms
that vanish exactly on the optimal trajectory.

\begin{theorem}[Sufficiency of the I-GRE]\label{th:IMP:LQ:suff}
  Let $T>0$.  Suppose there exists a piecewise-$C^1$ solution
  $P:[t_0,t_0+T]\to\Snpsd$ with $P(t_0+T)=Z_T$ satisfying the
  I-GRE~\eqref{eq:IMP:GRE:flow}--\eqref{eq:IMP:GRE:jump}.  Then:
  \begin{enumerate}[\upshape(a)]
    \item\label{suff:val}
      The optimal cost is $J_T^\star (X_1^0)=\ip{P(t_0)}{X_1^0}$.
    \item\label{suff:att}
      The optimum is attained at any $X^\star$ with
      $K(t)\in\mathcal{K}_c(t,P(t))$ on flow intervals and
      $K_k\in\mathcal{K}_d(k,P(t_k^+),P(t_k^-))$ at each jump.
    \item\label{suff:uniq}
      The solution $P$ is unique.
  \end{enumerate}
\end{theorem}

\begin{proof}
\noindent\textbf{Lower bound.}
By the I-GRE, $\mathcal{C}_Z(t,P(t))\succeq0$ with
$\mathcal{C}_Z\,\schur{/}\,\allowbreak(\mathcal{C}_Z)_3=0$ on flow, and
$\mathcal{D}_Z(k,P(t_k^+),P(t_k^-))\succeq0$ with
$\mathcal{J}\,\schur{/}\,(\mathcal{D}_Z)_3=0$
at each jump (Lemma~\ref{lemma:extSchur}).
Let $X$ be any trajectory of~\eqref{eq:sys:IMP} with
$EX(t_0)E^*=X_1^0$.
The key identity~\eqref{eq:IMP:keyid2} gives
\begin{equation}\label{pf:suff:id}
  J_T(X,X_1^0)
  =\ip{P(t_0)}{X_1^0}
  +\underbrace{
    \int_{t_0}^{t_0+T}\ip{\mathcal{C}_Z(t,P(t))}{X(t)}\dt
  }_{=:\,I_f}
  +\underbrace{
    \sum_{k=1}^{N_T}\ip{\mathcal{D}_Z(k,P(t_k^+),P(t_k^-))}{X(t_k^-)}
  }_{=:\,I_j}.
\end{equation}
Write $X(t)=U(t)U(t)^*$ with $U=[U_1^*\;U_2^*]^*$.
Since $\mathcal{C}_Z(t,P(t))\succeq0$, compatibility~\eqref{eq:compat}
holds (Lemma~\ref{lemma:extSchur}), and
Lemma~\ref{lemma:decomp} gives, for a.e.\ $t$,
\begin{equation}\label{pf:suff:schur:flow}
  \ip{\mathcal{C}_Z(t,P(t))}{X(t)}
  =\underbrace{
    \ip{\mathcal{C}_Z\,\schur{/}\,(\mathcal{C}_Z)_3}{U_1U_1^*}
  }_{=\,0}
  +\underbrace{
    \ip{(\mathcal{C}_Z)_3}{W_fW_f^*}
  }_{\ge\,0},
  \qquad W_f:=\pinv{(\mathcal{C}_Z)_3}(\mathcal{C}_Z)_2^* U_1+U_2.
\end{equation}
Writing $X(t_k^-)=V_kV_k^*$ with $V_k=[V_{k,1}^*\;V_{k,2}^*]^*$
and applying Lemma~\ref{lemma:decomp} to $\mathcal{J}$
at each $t_k$:
\begin{equation}\label{pf:suff:schur:jump}
  \ip{\mathcal{D}_Z(k,P(t_k^+),P(t_k^-))}{X(t_k^-)}
  =\underbrace{
    \ip{\mathcal{D}_Z\,\schur{/}\,(\mathcal{D}_Z)_3}
    {V_{k,1}V_{k,1}^*}
  }_{=\,0}
  +\underbrace{
    \ip{(\mathcal{D}_Z)_3}{W_kW_k^*}
  }_{\ge\,0},
\end{equation}
Therefore $I_f\ge0$ and $I_j\ge0$, so
\begin{equation}\label{pf:suff:lb}
  J_T(X,X_1^0)\ge\ip{P(t_0)}{X_1^0}
  \quad\text{for every admissible }X.
\end{equation}

\noindent\textbf{Attainment (flow).} Fix $K(t)\in\mathcal{K}_c(t,P(t))$.  The general-element formula
gives $K=\pinv{(\mathcal{C}_Z)_3}(\mathcal{C}_Z)_2^*
+F(I-(\mathcal{C}_Z)_3\pinv{(\mathcal{C}_Z)_3})$, so
$K-\pinv{(\mathcal{C}_Z)_3}(\mathcal{C}_Z)_2^*
=F(I-(\mathcal{C}_Z)_3\pinv{(\mathcal{C}_Z)_3})$,
whose columns lie in $\ker((\mathcal{C}_Z)_3)$.
Setting $U_2^\star=-KU_1^\star$:
\[
  W_f^\star
  =\pinv{(\mathcal{C}_Z)_3}(\mathcal{C}_Z)_2^* U_1^\star+U_2^\star
  =(\pinv{(\mathcal{C}_Z)_3}(\mathcal{C}_Z)_2^*-K)U_1^\star
  =-F(I-(\mathcal{C}_Z)_3\pinv{(\mathcal{C}_Z)_3})U_1^\star.
\]
Every column of $(I-(\mathcal{C}_Z)_3\pinv{(\mathcal{C}_Z)_3})\allowbreak U_1^\star$
lies in $\ker((\mathcal{C}_Z)_3)$, so\linebreak[1]$(\mathcal{C}_Z)_3W_f^\star=0$,
giving $\ip{(\mathcal{C}_Z)_3}{W_f^\star W_f^{\star*}}
=\tr((\mathcal{C}_Z)_3W_f^\star W_f^{\star*})=0$.
Combined with $\mathcal{C}_Z\,\schur{/}\,\allowbreak(\mathcal{C}_Z)_3=0$,
both terms in~\eqref{pf:suff:schur:flow} vanish along $X^\star$:
\begin{equation}\label{pf:suff:flow:zero}
  \ip{\mathcal{C}_Z(t,P(t))}{X^\star(t)}=0
  \quad\text{for all }t.
\end{equation}
\noindent\textbf{Attainment (jump).}  The identical argument with $K_k\in\mathcal{K}_d$ and
$V_{k,2}^\star=-K_kV_{k,1}^\star$ gives $W_k^\star\in
\ker((\mathcal{D}_Z)_3)$, hence
\begin{equation}\label{pf:suff:jump:zero}
  \ip{\mathcal{D}_Z(k,P(t_k^+),P(t_k^-))}{X^\star(t_k^-)}=0
  \quad\text{for every }k.
\end{equation}
Substituting~\eqref{pf:suff:flow:zero}--\eqref{pf:suff:jump:zero}
into~\eqref{pf:suff:id}: $J_T(X^\star,X_1^0)=\ip{P(t_0)}{X_1^0}$.
Combined with~\eqref{pf:suff:lb}:
$J_T^\star (X_1^0)=\ip{P(t_0)}{X_1^0}$, attained at $X^\star$.\\

\noindent\textbf{Positivity of $P$.}
$J_T^\star (X_1^0)=\ip{P(t_0)}{X_1^0}\ge0$ for all $X_1^0\succeq0$
(since $Z\succeq0$, $Z_k\succeq0$, and $X\succeq0$).
Proposition~\ref{prop:PQ}\eqref{st:neg} gives $P(t_0)\succeq0$,
and since $t_0\in[t_0,t_0+T]$ is arbitrary, $P(t)\succeq0$.\\

\noindent\textbf{Uniqueness of $P$.}
Suppose $\tilde P$ is another piecewise-$C^1$ solution of the
I-GRE with $\tilde P(t_0+T)=Z_T$.  By the argument above,
$J_T^\star (X_1^0)=\ip{\tilde P(t_0)}{X_1^0}$ for all $X_1^0\succeq0$.
Since also $J_T^\star (X_1^0)=\ip{P(t_0)}{X_1^0}$:
$\ip{P(t_0)-\tilde P(t_0)}{X_1^0}=0$ for all $X_1^0\succeq0$.
Taking $X_1^0=vv^*$ for every $v\in\R^n$ gives
$(P(t_0)-\tilde P(t_0))v=0$, hence $P(t_0)=\tilde P(t_0)$.
Since $t_0\in[t_0,t_0+T]$ is arbitrary, $P\equiv\tilde P$.
\end{proof}

\subsubsection{Necessary condition}

Conversely, whenever the optimization problem is well-posed
(the infimum is finite for every $X_1^0\succeq0$ and is attained),
the value function is necessarily linear in $X_1^0$ and the
corresponding dual variable must satisfy the I-GRE.
Theorem~\ref{th:IMP:LQ:nec} establishes this necessity via a
dynamic programming argument: the positivity constraint on
$\mathcal{C}_Z$ comes from the value-function inequality on flow
intervals, and the Riccati equation $\mathcal{C}_Z\,\schur{/}
\,(\mathcal{C}_Z)_3=0$ comes from the vanishing of the key identity
along optimal trajectories.

\begin{theorem}[Necessity of the I-GRE]\label{th:IMP:LQ:nec}
  Let $T>0$.  Suppose~\eqref{eq:IMP:LQ:prob} is well-posed and an
  optimal trajectory exists for every $X_1^0\succeq0$.  Then:
  \begin{enumerate}[\upshape(a)]
    \item\label{nec:lin}
      The value function is linear: there exists a unique
      piecewise-$C^1$ $P$ with $P(t_0+T)=Z_T$ such that
      $J_T^\star (X_1^0)=\ip{P(t_0)}{X_1^0}$ for all
      $X_1^0\succeq0$.
    \item\label{nec:DLMI}
      $P$ satisfies the flow LMI:
      $\mathcal{C}_Z(t,P(t))\succeq0$ for a.e.\ $t$.
    \item\label{nec:DLI}
      $P$ satisfies the jump LMI:
      $\mathcal{D}_Z(k,P(t_k^+),P(t_k^-))\succeq0$ for each $k$.
    \item\label{nec:GRE}
      $P$ satisfies the full I-GRE, including both compatibility
      conditions~\eqref{eq:IMP:GRE:fC} and~\eqref{eq:IMP:GRE:jC}.
    \item\label{nec:psd}
      $P(t)\succeq0$ for all $t$.
  \end{enumerate}
\end{theorem}

\begin{proof}
\noindent\textbf{Part~\eqref{nec:lin}: Linearity.} By well-posedness, for every $X_1^0\succeq0$ the optimum is attained;
coercivity~(H2) makes the cost strictly convex in the trajectory, so
the optimal trajectory $X^\star(\cdot;X_1^0)$ is \emph{unique}.\\

\emph{Positive homogeneity.}
For $\lambda\ge0$, linearity of the dynamics in $X$ implies that
$\lambda X^\star(\cdot;X_1^0)$ is admissible for $\lambda X_1^0$, with
cost $\lambda J_T^*(X_1^0)$.  By uniqueness it is optimal, so
$J_T^*(\lambda X_1^0)=\lambda J_T^*(X_1^0)$.\\

\emph{Additivity.}
For $X_1^0,\tilde X_1^0\succeq0$, the sum
$X^\star(\cdot;X_1^0)+X^\star(\cdot;\tilde X_1^0)$ is admissible for
$X_1^0+\tilde X_1^0$ (linearity of the dynamics), with cost
$J_T^*(X_1^0)+J_T^*(\tilde X_1^0)$; hence
\begin{equation*}
  J_T^*(X_1^0+\tilde X_1^0)
  \le J_T^*(X_1^0)+J_T^*(\tilde X_1^0).
\end{equation*}
Conversely, by dynamic programming~\cite{Bertsekas:12,Bertsekas:17}
the optimum is attained by a \emph{state-independent} linear feedback on
the free blocks (the optimal gain $K(t)$ depends on the value function,
not on the initial condition).  The induced optimal closed-loop is thus a
linear, cone-preserving flow $x\mapsto X^\star(\cdot;x)$, so
$X^\star(\cdot;X_1^0+\tilde X_1^0)
=X^\star(\cdot;X_1^0)+X^\star(\cdot;\tilde X_1^0)$ is the sum of two
\emph{admissible} ($\Snmpsd$-valued) optimal trajectories.  Therefore
\begin{equation*}
  J_T^*(X_1^0+\tilde X_1^0)
  =J_T^*(X_1^0)+J_T^*(\tilde X_1^0),
\end{equation*}
which gives the reverse inequality (and in fact equality directly).\\

\emph{Polarisation.}
$J_T^*$ is therefore linear on $\Snpsd$.
Extending by polarisation $Q=Q_+-Q_-$ for $Q\in\Sn$ yields a linear
functional on $\Sn$ \cite{Bhatia:97}.
The Frobenius--Riesz representation \cite{Conway:90} gives unique $P(t_0)\in\Sn$ with
$J_T^*(X_1^0)=\ip{P(t_0)}{X_1^0}$.
Applying dynamic programming \cite{Bertsekas:12,Bertsekas:17} at every initial time $t$ defines
$P(t)$; piecewise-$C^1$ regularity follows from continuity of the
optimal cost on each flow interval, with possible
discontinuities at jump times.  The terminal condition
$P(t_0+T)=Z_T$ holds by evaluating at $t_0=t_0+T$.\\

\noindent\textbf{Part~\eqref{nec:DLMI}: Flow LMI.}
Fix a flow instant $t$ in the interior of some interval $(t_{k-1},t_k)$.
For any $[v_1^*\;v_2^*]^*\in\C^{n+m}$, consider a
constant perturbation $X_2\equiv v_2 v_1^*$,
$X_3\equiv v_2v_2^*$ on $[t,t+h]\subset(t_{k-1},t_k)$
and continue optimally on $[t+h,t_0+T]$.  The dynamic programming
inequality
$\ip{P(t)}{EX(t)E^*}\le\int_t^{t+h}\ip{Z(t)}{X(t)}\dt
+\ip{P(t+h)}{EX(t+h)E^*}$,
after subtracting, dividing by $h>0$, and letting $h\to0$ (using
the definition of $\mathcal{C}_Z(t,P(t))$), gives
\[
  0\le
  \begin{bmatrix}v_1\\v_2\end{bmatrix}^*
  \mathcal{C}_Z(t,P(t))
  \begin{bmatrix}v_1\\v_2\end{bmatrix}.
\]
Since $v_1\in\C^n$ and $v_2\in\C^m$ are arbitrary,
$\mathcal{C}_Z(t,P(t))\succeq0$.\\

\noindent\textbf{Part~\eqref{nec:DLI}: Jump LMI.}
At $t_k$, the dynamic programming inequality reads
\[
  \ip{P(t_k^-)}{EX(t_k^-)E^*}
  \le\ip{Z_k}{X(t_k^-)}
  +\ip{P(t_k^+)}{EX(t_k^+)E^*}.
\]
Using~\eqref{eq:sys:IMP:jump} and the adjoint relation:
\[
  0\le\ip{Z_k+\mathcal{J}^*(k,P(t_k^+))-E^*P(t_k^-)E}{X(t_k^-)}
  =\ip{\mathcal{D}_Z(k,P(t_k^+),P(t_k^-))}{X(t_k^-)}.
\]
Testing with $X(t_k^-)=[v_1^*\;v_2^*]^*
[v_1^*\;v_2^*]$ for arbitrary $v_1\in\C^n$, $v_2\in\C^m$ gives
$\mathcal{D}_Z(k,P(t_k^+),P(t_k^-))\succeq0$.\\

\noindent\textbf{Part~\eqref{nec:GRE}: I-GRE conditions.}
Let $X^\star$ be an optimal trajectory.
From~\eqref{eq:IMP:keyid2}:
\begin{equation}\label{pf:nec:zero}
  0=J_T(X^\star,X_1^0)-J_T^\star (X_1^0)
  =\int_{t_0}^{t_0+T}\ip{\mathcal{C}_Z}{X^\star}\dt
  +\sum_k\ip{\mathcal{D}_Z(k,P(t_k^+),P(t_k^-))}{X^\star(t_k^-)}.
\end{equation}
Since the integrand (Part~\eqref{nec:DLMI}) and each summand
(Part~\eqref{nec:DLI}) are non-negative, each must vanish
separately:
\begin{equation}\label{pf:nec:zero:sep}
  \ip{\mathcal{C}_Z(t,P(t))}{X^\star(t)}=0,\ t\ge0\ \text{a.e.},
  \qquad
  \ip{\mathcal{D}_Z(k,P(t_k^+),P(t_k^-))}{X^\star(t_k^-)}=0,\ k\ge1.
\end{equation}

We prove now the Flow Riccati equation and compatibility~\eqref{eq:IMP:GRE:fR},
\eqref{eq:IMP:GRE:fC}. Since $\mathcal{C}_Z(t,P(t))\succeq0$, Lemma~\ref{lemma:extSchur}
gives compatibility~\eqref{eq:compat}.
Lemma~\ref{lemma:decomp} with~\eqref{pf:nec:zero:sep} gives
\[
  0=\ip{\mathcal{C}_Z\,\schur{/}\,(\mathcal{C}_Z)_3}
    {U_1^\star U_1^{\star*}}
  +\ip{(\mathcal{C}_Z)_3}{W_f^\star W_f^{\star*}},
\]
with both terms non-negative.  Hence each is zero.
Since the optimal trajectory $X_1^\star=EX^\star E^*$ can reach
any $Q\in\Snpsd$ by varying $X_1^0$,
$\ip{\mathcal{C}_Z\,\schur{/}\,(\mathcal{C}_Z)_3}{Q}=0$ for all
$Q\succeq0$ gives $\mathcal{C}_Z\,\schur{/}\,\allowbreak(\mathcal{C}_Z)_3=0$
by Proposition~\ref{prop:PQ}\eqref{st:neg},
i.e.,~\eqref{eq:IMP:GRE:fR} holds.
From $\ip{(\mathcal{C}_Z)_3}{W_f^\star W_f^{\star*}}=0$ and
$(\mathcal{C}_Z)_3\succeq0$: $(\mathcal{C}_Z)_3W_f^\star=0$.
Expanding $W_f^\star=\pinv{(\mathcal{C}_Z)_3}(\mathcal{C}_Z)_2^*
U_1^\star+U_2^\star$ (with $U_2^\star=-K^\star U_1^\star$) and
using $(\mathcal{C}_Z)_3K^{\star*}=-(\mathcal{C}_Z)_2^*$:
\[
  0=(\mathcal{C}_Z)_3\bigl(
    \pinv{(\mathcal{C}_Z)_3}(\mathcal{C}_Z)_2^*-K^{\star*}\bigr)
    U_1^\star
  =(I-(\mathcal{C}_Z)_3\pinv{(\mathcal{C}_Z)_3})(\mathcal{C}_Z)_2^* U_1^\star.
\]
Because $U_1^\star$ spans $\R^n$ over all initial conditions,
$(I-(\mathcal{C}_Z)_3\pinv{(\mathcal{C}_Z)_3})(\mathcal{C}_Z)_2^*=0$,
which is~\eqref{eq:IMP:GRE:fC}.\\

Let us consider now the Jump Riccati equation and compatibility~\eqref{eq:IMP:GRE:jR},
\eqref{eq:IMP:GRE:jC}. The identical argument applied to $\mathcal{D}_Z(k,P(t_k^+),P(t_k^-))$ and
$V_{k,1}^\star:=EU^\star(t_k^-)$, where $X^\star(t_k^-)=U^\star(t_k^-)U^\star(t_k^-)^*$,
gives
\begin{align*}
  0&=\ip{\mathcal{D}_Z\,\schur{/}\,(\mathcal{D}_Z)_3}{V_{k,1}^\star V_{k,1}^{\star*}}
  +\ip{(\mathcal{D}_Z)_3}{W_k^\star W_k^{\star*}},\\
  (\mathcal{D}_Z)_3 W_k^\star &= 0,\quad
  (I-(\mathcal{D}_Z)_3\pinv{(\mathcal{D}_Z)_3})(\mathcal{D}_Z)_2^*=0,
\end{align*}
which are~\eqref{eq:IMP:GRE:jR} and~\eqref{eq:IMP:GRE:jC} respectively.\\

\noindent\textbf{Part~\eqref{nec:psd}: Positivity.}
Same argument as in Theorem~\ref{th:IMP:LQ:suff}.
\end{proof}

\subsubsection{Equivalence and dual LMI}

Combining Theorems~\ref{th:IMP:LQ:suff} and~\ref{th:IMP:LQ:nec}
gives the equivalence Corollary~\ref{cor:IMP:LQ:equiv}: the
existence of a positive solution to the I-GRE is both necessary
and sufficient for well-posedness.
Theorem~\ref{th:IMP:LQ:LMI} then provides an alternative
semidefinite program characterization of the optimal cost,
whose constraint $\mathcal{C}_Z\succeq0$ and $\mathcal{D}_Z\succeq0$
is \emph{linear} in $P$ (an LMI).
This is useful for numerical computation: rather than solving the
nonlinear I-GRE directly, one can maximize $\ip{P(t_0)}{X_1^0}$
over $P\succeq0$ subject to LMI constraints.

\begin{corollary}[Equivalence]\label{cor:IMP:LQ:equiv}
  Problem~\eqref{eq:IMP:LQ:prob} is well-posed with an optimal
  trajectory for every $X_1^0\succeq0$ if and only if the I-GRE
  has a piecewise-$C^1$ solution $P\succeq0$ with $P(t_0+T)=Z_T$.
  The solution is unique and $J_T^\star (X_1^0)=\ip{P(t_0)}{X_1^0}$.
\end{corollary}

\begin{proof}
  Sufficiency is Theorem~\ref{th:IMP:LQ:suff}; necessity is
  Theorem~\ref{th:IMP:LQ:nec}.
\end{proof}

\begin{theorem}[Dual LMI: finite-horizon]\label{th:IMP:LQ:LMI}
  Suppose the I-GRE has a solution $P^*$.  Then
  \begin{align}\label{eq:IMP:LQ:LMI}
    J_T^\star (X_1^0)
    =\max_{P\text{ pw-}C^1}&\;\ip{P(t_0)}{X_1^0}
    \notag\\
    \text{s.t.}\quad
    &P\succeq0,\;
    \mathcal{C}_Z(t,P(t))\succeq0,\;
    \mathcal{D}_Z(k,P(t_k^+),P(t_k^-))\succeq0,\;
    P(t_0+T)\preceq Z_T,
  \end{align}
  and the maximum is attained at $P^*$.
\end{theorem}

\begin{proof}
\noindent\textbf{Step~1: Every feasible $P$ is a lower bound.}
Let $P$ be feasible and $X$ any admissible trajectory.
From the key identity~\eqref{eq:IMP:keyid} (before using
$P(T)=Z_T$):
\begin{align*}
  J_T(X,X_1^0)
  &=\ip{P(t_0)}{X_1^0}
   +\underbrace{\ip{Z_T-P(t_0+T)}{EX(T)E^*}}_{\ge\,0}
   \notag\\
   &\quad+\underbrace{\int\ip{\mathcal{C}_Z(t,P(t))}{X(t)}\dt}_{\ge\,0}
   +\underbrace{\textstyle\sum_k\ip{\mathcal{D}_Z(k,P(t_k^+),P(t_k^-))}{X(t_k^-)}}_{\ge\,0}\\
  &\ge\ip{P(t_0)}{X_1^0}.
\end{align*}
The first non-negative term uses $P(T)\preceq Z_T$ and
$EX(T)E^*\succeq0$; the remaining two use the LMI constraints.
Taking the infimum: $J_T^\star (X_1^0)\ge\ip{P(t_0)}{X_1^0}$.\\

\noindent\textbf{Step~2: $P^*$ is feasible and attains the bound.}
The I-GRE and Lemma~\ref{lemma:extSchur} give
$\mathcal{C}_Z(t,P^*)\succeq0$ and
$\mathcal{J}\succeq0$; $P^*(T)=Z_T$; $P^*\succeq0$.
So $P^*$ is feasible.  By Theorem~\ref{th:IMP:LQ:suff},
$\ip{P^*(t_0)}{X_1^0}=J_T^\star (X_1^0)$.
Hence $\max_{P\text{ feas.}}\ip{P(t_0)}{X_1^0}
\ge\ip{P^*(t_0)}{X_1^0}=J_T^\star (X_1^0)$.
Combined with Step~1: equality holds and $P^*$ attains it.
\end{proof}

\begin{corollary}[Continuous-time: finite horizon]\label{cor:CT:LQ:fh}
  Set $N_T=0$.  The jump conditions are vacuous, the I-GRE reduces
  to the \emph{CT-GRE} 
  \begin{equation}
    \mathcal{C}_Z(t,P(t))\succeq0,\ 
  \mathcal{C}_Z\,\schur{/}\ (\mathcal{C}_Z)_3=0,\ P(T)=Z_T,
  \end{equation}
  and the impulsive key identity reduces to the continuous-time key identity
  \begin{equation}\label{eq:CT:keyid}
    J_T(X,X_1^0)
    =\ip{P(t_0)}{X_1^0}
    +\int_{t_0}^{t_0+T}\ip{\mathcal{C}_Z(t,P(t))}{X(t)}\dt.
  \end{equation}
  Theorems~\ref{th:IMP:LQ:suff}, \ref{th:IMP:LQ:nec},
  Corollary~\ref{cor:IMP:LQ:equiv}, and
  Theorem~\ref{th:IMP:LQ:LMI} all hold with the dual LMI
  \begin{equation}\label{eq:CT:LQ:LMI}
    J_T^\star (X_1^0)
    =\max_{P\succeq0}\;\ip{P(t_0)}{X_1^0}
    \quad\text{s.t.}\quad
    \mathcal{C}_Z(t,P(t))\succeq0,\;P(T)\preceq Z_T,
  \end{equation}
  recovering~\cite{AitRami:01a,AitRami:01b}.
\end{corollary}

\begin{corollary}[Discrete-time: finite horizon]\label{cor:DT:LQ:fh}
  Apply Definitions~\ref{def:DT:sys} and set $Z=0$ on flow,
  $t_k=k_0+k$, $Z_k=Z(k)$.
  With $P$ constant on each interval $(k,k+1)$ the flow conditions
  are trivially satisfied, and the I-GRE reduces to the
  \emph{DT-GRE} at each step, with jump operator
  \begin{equation}\label{eq:DT:Ld}
    \tag{DT-GRE}
    \mathcal{D}_Z(k,P(k+1),P(k))
    :=\mathcal{J}^*(k,P(k+1))-E^* P(k)E+Z(k).
  \end{equation}
  The key identity reduces to the discrete-time key identity
  \[J_N(X,X_1^0)=\ip{P(k_0)}{X_1^0}
  +\textstyle\sum_{k=k_0}^{k_0+N-1}\ip{\mathcal{D}_Z(k,P(k+1),P(k))}{X(k)},\]
  and the dual LMI is
  \begin{equation}\label{eq:DT:LQ:LMI}
    J_N^*(X_1^0)
    =\max_{P\succeq0}\;\ip{P(k_0)}{X_1^0}
    \quad\text{s.t.}\quad
    \mathcal{D}_Z(k,P(k+1),P(k))\succeq0,\;P(k_0+N)\preceq Z_N,
  \end{equation}
  recovering~\cite{AitRami:01b}.
\end{corollary}

\subsection{Infinite-horizon}
\label{sec:OC:inf}

We now study the optimal control problem on the infinite horizon
$[t_0,\infty)$.  The setup mirrors Section~\ref{sec:OC:fh} but with
two important differences.  First, there is no terminal cost $Z_T$:
the horizon extends to infinity so no terminal penalty is imposed.
Second, the I-GRE no longer has a terminal condition $P(T)=Z_T$;
instead, the dual variable $P_\infty$ is characterized as the
bounded monotone limit of the finite-horizon solutions $P_T$ as
$T\to\infty$ (Theorem~\ref{th:IMP:LQ:inf}).

The infinite-horizon problem is well-posed, meaning the optimal
cost is finite and uniquely represented as $\ip{P_\infty(t_0)}{X_1^0}$, under the uniform stabilizability, coercivity, and boundedness
assumptions (H1)--(H3) stated below.  Unlike most treatments of
infinite-horizon problems for impulsive systems~\cite{Ichikawa:01,
Dragan:25,AbouKandil:03}, no time-invariance of $\mathcal{F}$,
$\mathcal{J}$, or $Z$ is assumed: the LTV setting requires the
monotone-limit argument below in place of the algebraic Riccati
equation.  The theorem also provides a dual characterization: the
infinite-horizon optimal cost equals the supremum of $\ip{P}{X_1^0}$
over all bounded positive semidefinite $P$ satisfying the LMI
constraints, paralleling Theorem~\ref{th:IMP:LQ:LMI} at the infinite
horizon.

\subsubsection{Setting, cost function, and assumptions (infinite horizon)}

The infinite-horizon cost removes the terminal penalty $\ip{Z_T}{EX(T)E^*}$
from the finite-horizon cost~\eqref{eq:IMP:LQ:cost}, since no fixed terminal
time is imposed.  The resulting integral over $[t_0,\infty)$ is
\begin{equation}\label{eq:IMP:LQ:Jinf}
  J_\infty(X,X_1^0)
  :=\int_{t_0}^\infty\ip{Z(t)}{X(t)}\dt
  +\sum_{k=1}^\infty\ip{Z_k}{X(t_k^-)},
\end{equation}
with problem 
\begin{equation}
  \min_{X(\cdot)\in\Snmpsd}J_\infty(X,X_1^0)\ \textnormal{subject to }\eqref{eq:sys:IMP} \textnormal{ and } EX(t_0)E^*=X_1^0,
\end{equation}
and optimal cost
$$J_\infty^\star(X_1^0):=\inf_{X(\cdot)\in\Snmpsd}J_\infty(X,X_1^0).$$
As in Section~\ref{sec:OC:fh}, the decision variable is the full
trajectory $X(\cdot)$ with $X_1$ constrained and $(X_2,X_3)$ free.\\

Well-posedness of the infinite-horizon problem requires three
conditions on the system and cost operators.

\medskip\noindent\textbf{(H1) Uniform stabilizability.}
There exists an admissible trajectory $X_s$ with free variables
$X_2^s=K_sX_1^s$, $X_3^s=K_sX_1^sK_s^*$, constants
$\beta\ge1$, $\alpha>0$, $\rho\in(0,1)$, $C_s>0$, such that
$\norm{EX_s(t)E^*}\le\beta e^{-\alpha(t-s)}\rho^{\kappa(t,s)}\norm{EX_s(s)E^*}$
for all $t\ge s\ge t_0$ and
$J_\infty(X_s,X_1^0)\le C_s\norm{X_1^0}$, where $\kappa(t,s)$ denotes the number of jumps in the interval $(s,t)$.

\medskip\noindent\textbf{(H2) Uniform coercivity.}
$EZ(t)E^*\succeq\delta I$ uniformly for some $\delta>0$.

\medskip\noindent\textbf{(H3) Uniform boundedness.}
$Z$, $Z_k$, $\mathcal{F}(t,\cdot)$, $\mathcal{J}(k,\cdot)$ are
uniformly bounded; $(\mathcal{C}_Z)_3(t,P(t))\succeq0$ and
$(\mathcal{D}_Z)_3(k,P(t_k^+),P(t_k^-))\succeq0$ for all $P\succeq0$;
$\pinv{(\mathcal{C}_Z)_3}$ and $\pinv{(\mathcal{D}_Z)_3}$ are
uniformly bounded on bounded sets of $P$.\\

The key differences from the finite-horizon case
(Section~\ref{sec:OC:fh}) are as follows.
In the finite-horizon case, the I-GRE is a terminal-value ODE
starting from $P(t_0+T)=Z_T$ and solved backward to $t_0$.
In the infinite-horizon case, there is no terminal condition:
instead, the dual variable $P_\infty$ is characterized as the
bounded monotone limit of the finite-horizon solutions $P_T$
as $T\to\infty$.  The infinite-horizon I-GRE is the equation
satisfied by this limit, which in the LTI case reduces to the
classical algebraic Riccati equation.  Under (H2),
$P_\infty\succeq\alpha_1 I\succ0$ (strict positivity).

\subsubsection{Sufficient condition (infinite horizon)}

The following theorem shows that any bounded positive solution of the
infinite-horizon I-GRE is the unique such solution and directly
provides the optimal cost and the optimal policy.  The proof mirrors
Theorem~\ref{th:IMP:LQ:suff} but with the finite-horizon terminal
correction replaced by a Gronwall argument that forces
$\ip{P_\infty(T)}{EX^\star(T)E^*}\to0$.

\begin{theorem}[Sufficient condition, infinite-horizon]
  \label{th:IMP:LQ:inf:suff}
  Suppose there exists a bounded piecewise-$C^1$ function
  $P_\infty:[t_0,\infty)\to\Snpsd$ satisfying the
  infinite-horizon I-GRE\textup{:} \eqref{eq:IMP:GRE:flow} for almost all $t\ge t_0$ and \eqref{eq:IMP:GRE:jump} at all jump times.\\

  Then\textup{:}
  \begin{enumerate}[\upshape(a)]
    \item\label{inf:suff:val}
      $J_\infty^\star (X_1^0)=\ip{P_\infty(t_0)}{X_1^0}$, attained at
      the trajectory with $K_\infty(t)\in\mathcal{K}_c(t,P_\infty(t))$
      on flow and $K_k\in\mathcal{K}_d(k,P_\infty(t_k^+),P_\infty(t_k^-))$
      at each jump.
    \item\label{inf:suff:uniq}
      $P_\infty$ is the unique bounded solution of the
      infinite-horizon I-GRE.
  \end{enumerate}
\end{theorem}

\begin{proof}
\noindent\textbf{Lower bound.}
For any $X$ with $J_\infty(X,X_1^0)<\infty$, apply the key
identity~\eqref{eq:IMP:keyid2} on $[t_0,t_0+T]$ with $P_\infty$
(note that $\mathcal{C}_Z(t,P_\infty)\succeq0$ and
$\mathcal{D}_Z\succeq0$ by the infinite-horizon I-GRE):
\[
  \int_{t_0}^{t_0+T}\ip{Z(t)}{X(t)}\dt
  +\sum_{k:\,t_k\le T}\ip{Z_k}{X(t_k^-)}
  \ge\ip{P_\infty(t_0)}{X_1^0}
  -\ip{P_\infty(T)}{EX(T)E^*}.
\]
Under (H2), finite cost forces $\norm{EX(t)E^*}\to0$; since
$P_\infty\preceq\alpha_2 I$,
$\ip{P_\infty(T)}{EX(T)E^*}\le\alpha_2\norm{EX(T)E^*}\to0$.
Taking $T\to\infty$: $J_\infty(X,X_1^0)\ge\ip{P_\infty(t_0)}{X_1^0}$.\\

\noindent\textbf{Attainment.}
For $K_\infty\in\mathcal{K}_c(t,P_\infty)$ and
$K_k\in\mathcal{K}_d(k,P_\infty(t_k^+),P_\infty(t_k^-))$,
\eqref{pf:suff:flow:zero}--\eqref{pf:suff:jump:zero} hold along
$X^\star$.  The key identity on $[t_0,t_0+T]$ gives
\begin{equation}\label{eq:inter}
  \int_{t_0}^{t_0+T}\ip{Z(t)}{X^\star(t)}\dt
  +\sum_{k:\,t_k\le T}\ip{Z_k}{X^\star(t_k^-)}
  =\ip{P_\infty(t_0)}{X_1^0}
  -\ip{P_\infty(T)}{EX^\star(T)E^*}.
\end{equation}

Differentiating $\ip{P_\infty(t)}{EX^\star(t) E^*}$ along the optimal
flow gives $\tfrac{d}{dt}\ip{P_\infty(t)}{EX^\star(t) E^*}=-\ip{Z(t)}{X^\star(t)}$ over each $(t_k,t_{k+1})$.
Since $\ip{Z(t)}{X^\star(t)}\ge\delta\norm{EX^\star(t) E^*}
\ge\tfrac{\delta}{\alpha_2}\ip{P_\infty(t)}{EX^\star(t) E^*}$
(using (H2), that is $Z(t)\succeq \delta E^*E$, and $P_\infty(t)\preceq\alpha_2 I$), Gronwall gives
$\ip{P_\infty(t)}{EX^\star(t)E^*}\le e^{-\delta (t-t_k)/\alpha_2}
\ip{P_\infty(t_k^+)}{EX^\star(t_k^+)E^*}$.\\

At jumps we have $\ip{P_\infty(t_k^+)}{EX^\star(t_k^+) E^*}-\ip{P_\infty(t_k^-)}{EX^\star(t_k^-) E^*}=-\ip{Z_k}{X^\star(t_k^-)}$, and using a similar argument as before, we get $\ip{P_\infty(t_k^+)}{EX^\star(t_k^+)E^*}\le \rho\ip{P_\infty(t_k^-)}{EX^\star(t_k^-)E^*}$ where $\rho:=1-\delta/\alpha_2<1$. Combining those expression yields
$$\ip{P_\infty(T)}{EX^\star(T)E^*}\le e^{-\delta (T-t_k)/\alpha_2}\rho^{\kappa(T,t_0)}
\ip{P_\infty(t_0)}{X_1^0}\to0\ \textnormal{as }T\to\infty$$

Taking $T\to\infty$ in \eqref{eq:inter} yields $J_\infty(X^\star,X_1^0)=\ip{P_\infty(t_0)}{X_1^0}=J_\infty^\star (X_1^0)$.\\

\noindent\textbf{Uniqueness.}
If $\tilde P$ is another bounded solution, the lower-bound and
attainment arguments give $J_\infty^\star (X_1^0)=\ip{\tilde P(t_0)}{X_1^0}$.
Since also $J_\infty^\star =\ip{P_\infty(t_0)}{X_1^0}$, taking
$X_1^0=vv^*$ for all $v$ gives $P_\infty\equiv\tilde P$.
\end{proof}

\subsubsection{Necessary condition and existence (infinite horizon)}

Under assumptions (H1)--(H3), the existence of a bounded solution
to the infinite-horizon I-GRE is guaranteed by a monotone convergence
argument: the finite-horizon solutions $P_T$ form a non-decreasing
bounded sequence and their pointwise limit satisfies the I-GRE.
This mirrors the role of Theorem~\ref{th:IMP:LQ:nec} at the
infinite-horizon level.

\begin{theorem}[Necessary condition and existence, infinite-horizon]
  \label{th:IMP:LQ:inf}
  Under \textup{(H1)--(H3)}, there exist $\alpha_1,\alpha_2>0$ and
  a unique bounded piecewise-$C^1$ solution
  $P_\infty:[t_0,\infty)\to\Snpsd$,
  $\alpha_1 I\preceq P_\infty\preceq\alpha_2 I$, of the
  infinite-horizon I-GRE.  It satisfies
  $P_\infty=\lim_{T\to\infty}P_T$ (monotone: $P_{T'}\succeq P_T$
  for $T'>T$), where $P_T$ solves the finite-horizon I-GRE with
  $P_T(t_0+T)=0$, and the infinite-horizon optimal cost is
  $J_\infty^\star (X_1^0)=\ip{P_\infty(t_0)}{X_1^0}$.
\end{theorem}

\begin{proof}
\noindent\textbf{Monotone limit.}
Corollary~\ref{cor:IMP:LQ:equiv} gives a unique $P_T\in\Snpsd$ for
each $T>0$.\\

\emph{Monotonicity.}
For $T'>T$, restricting any $T'$-admissible $X$ to $[t_0,t_0+T]$
gives a $T$-admissible trajectory with no larger cost (since the
additional terms on $(T,T']$ are non-negative).  Hence
$J_T^\star (X_1^0)\le J_{T'}^*(X_1^0)$ and $P_T(t)\preceq P_{T'}(t)$.\\

\emph{Bounds.}
Lower: $P_T\succeq0$ since $Z,Z_k\succeq0$.
Upper: (H1) gives $P_T(t_0)\preceq C_s I=:\alpha_2 I$ uniformly.
Strict lower: by (H3) the generator $\mathcal{F}$ and the optimal
gain $K(t)$ are uniformly bounded, so along any optimal closed-loop
trajectory $\tfrac{d}{dt}\tr(X_1(t))=\tr(\mathcal{F}(t,X^\star(t)))\ge
-c\,\tr(X_1(t))$ for some $c>0$, whence $\tr(X_1(t))\ge
e^{-c(t-t_0)}\tr(X_1^0)$ (a Gronwall lower bound).  By (H2),
$Q(t)=EZ(t)E^*\succeq\delta I$, so for any fixed $h>0$ with $t_0+h\le t_0+T$,
\[
  \ip{P_T(t_0)}{X_1^0}=J_T^\star(X_1^0)
  \ge\int_{t_0}^{t_0+h}\!\!\ip{Q(t)}{X_1(t)}\dt
  \ge\delta\!\int_{t_0}^{t_0+h}\!\!\tr(X_1(t))\dt
  \ge\delta\,\tfrac{1-e^{-ch}}{c}\,\tr(X_1^0).
\]
Taking $X_1^0=vv^*$ and $\alpha_1:=\delta(1-e^{-ch})/c>0$ gives
$P_T(t_0)\succeq\alpha_1 I$ for all $T\ge h$; the same argument from each
$t$ gives $P_T(t)\succeq\alpha_1 I$.\\

\emph{Convergence.}
$\{P_T(t)\}$ is bounded and non-decreasing; its pointwise limit
$P_\infty$ exists.  Equicontinuity\footnote{A family of functions is
\emph{equicontinuous} when a single modulus of continuity serves all of
its members simultaneously; here it follows from a uniform Lipschitz
bound. On each flow interval the I-GRE flow
conditions~\eqref{eq:IMP:GRE:flow} determine $E^*\dot{P}_T(t)E$ as a
fixed continuous function of $P_T(t)$ and the flow data; by the uniform
boundedness~(H3) of the data and of the pseudoinverses on bounded sets,
together with the uniform bound $0\preceq P_T\preceq\alpha_2 I$, there is
a constant $L$ (independent of $T$) such that $\|\dot{P}_T(t)\|\le L$ for
all $T$ and almost every $t$. Hence every $P_T$ is $L$-Lipschitz, so the
family $\{P_T\}$ is equicontinuous.} from the I-GRE and (H3), plus
Arzel\`a-Ascoli, give uniform convergence on compact subintervals.
Passing $T\to\infty$ in the integrated I-GRE shows $P_\infty$
satisfies the infinite-horizon I-GRE.\\

\noindent\textbf{Optimality.}
$P_\infty$ is bounded and satisfies the infinite-horizon I-GRE, so
Theorem~\ref{th:IMP:LQ:inf:suff} gives
$J_\infty^\star (X_1^0)=\ip{P_\infty(t_0)}{X_1^0}$.
\end{proof}

\subsubsection{Equivalence and dual LMI (infinite horizon)}

Theorems~\ref{th:IMP:LQ:inf:suff} and~\ref{th:IMP:LQ:inf} together
give the complete characterization.  The dual LMI extends the
finite-horizon version (Theorem~\ref{th:IMP:LQ:LMI}) by replacing
the terminal constraint $P(T)\preceq Z_T$ with a boundedness condition.

\begin{corollary}[Equivalence, infinite-horizon]
  \label{cor:IMP:LQ:inf:equiv}
  The infinite-horizon problem is well-posed with
  $J_\infty^\star (X_1^0)=\ip{P_\infty(t_0)}{X_1^0}$ for all
  $X_1^0\succeq0$ if and only if the infinite-horizon I-GRE has a
  bounded piecewise-$C^1$ solution $P_\infty\succeq0$.
  The solution is unique.
\end{corollary}

\begin{proof}
  Necessity: Theorem~\ref{th:IMP:LQ:inf}.
  Sufficiency: Theorem~\ref{th:IMP:LQ:inf:suff}.
  Uniqueness: Theorem~\ref{th:IMP:LQ:inf:suff}\eqref{inf:suff:uniq}.
\end{proof}

\begin{theorem}[Dual LMI, infinite-horizon]\label{th:IMP:LQ:inf:LMI}
  Under \textup{(H1)--(H3)},
  \begin{equation}\label{eq:IMP:LQ:LMI:inf}
    J_\infty^\star (X_1^0)
    =\sup_P\;\ip{P(t_0)}{X_1^0}
    \quad\text{s.t.}\quad
    \mathcal{C}_Z(t,P(t))\succeq0,\;
    \mathcal{D}_Z(k,P(t_k^+),P(t_k^-))\succeq0,\;P\text{ bounded};
  \end{equation}
  the supremum is attained at $P_\infty$.
\end{theorem}

\begin{proof}
For any bounded feasible $P$, the key identity on $[t_0,t_0+T]$
and the LMI constraints give
$J_T(X,X_1^0)\ge\ip{P(t_0)}{X_1^0}-\ip{P(T)}{EX(T)E^*}$.
As $P$ is bounded and $\norm{EX(T)E^*}\to0$ (from (H2)),
taking $T\to\infty$ gives $J_\infty^\star \ge\ip{P(t_0)}{X_1^0}$.
Since $P_\infty$ is feasible and $J_\infty^\star =\ip{P_\infty(t_0)}{X_1^0}$,
the supremum equals $J_\infty^\star $ and is attained at $P_\infty$.
\end{proof}

\subsubsection{Dwell-time and causal conditions}
\label{sec:OC:dwell}

Both the finite-horizon I-GRE (Section~\ref{sec:OC:fh}) and the
infinite-horizon I-GRE (Section~\ref{sec:OC:inf}) are solved
\emph{backward} in time: the dual variable $P$ is computed from a
terminal condition and propagated toward the initial time.  Evaluating
the optimal gain $K(t)\in\mathcal{K}_c(t,P(t))$ at any given time $t$
therefore requires knowing $P(t)$, which in turn depends on the next
jump time $t_{k+1}$ through the backward GRE's terminal condition at
$t_{k+1}^-$.  Since $t_{k+1}$ is not known in advance in most
applications, these optimal policies are non-causal and cannot be
implemented online.

This section derives \emph{causal} policies by restricting to
specific dwell-time scenarios.  Section~\ref{sec:OC:dwell:periodic}
treats periodic impulses (fixed period $T_p$): the periodic I-GRE
(Theorem~\ref{th:OC:periodic}) gives a causal policy that is also
necessary and sufficient.  Theorem~\ref{th:OC:MDT} treats a
minimum dwell-time constraint $\Delta_k\ge T_{\min}$: the
Geromel--Colaneri stability trick (Theorem~\ref{th:OC:MDT}) reduces
the infinite-horizon verification to an LMI on the bounded interval
$[0,T_{\min}]$ plus a single frozen condition at $T_{\min}$.
Section~\ref{sec:OC:dwell:RDT} treats a range dwell-time
$T_{\min}\le\Delta_k\le T_{\max}$: Theorem~\ref{th:OC:RDT} gives a
policy verifiable on $[0,T_{\max}]$ without the stability condition.
In all three cases, the key tool is the \emph{forward Riccati ODE}
(equation~\eqref{eq:forwardRiccati}), which propagates $P$ forward
in the timer $\tau=t-t_k$ from the post-jump value.

\paragraph{Non-causality of the I-GRE and the timer variable}

The I-GRE is a terminal-value problem: $P$ is computed backward from
$P(t_0+T)=Z_T$ toward the initial time $t_0$.
On the inter-jump interval $(t_k, t_{k+1})$, the flow GRE requires
knowing the value $P(t_{k+1}^-)$, which depends on the jump
condition at $t_{k+1}$.  Since the next jump time $t_{k+1}$ is not
known before it occurs, the optimal policy is \emph{non-causal}.

\medskip
\noindent\textbf{Timer variable.}
Introduce the \emph{timer}
\begin{equation}\label{eq:timer}
  \tau(t):=t-t_k,\quad t\in[t_k,t_{k+1}),
\end{equation}
satisfying $\dot\tau=1$ and resetting to $0$ after each jump.
A policy depending only on $X_1$ and $\tau$ is \emph{causal}.
Set $\tilde{P}(\tau):=P(t_k+\tau)$ for $\tau\in[0,\Delta_k)$,
$\Delta_k:=t_{k+1}-t_k$.  The I-GRE flow condition
in terms of $\tau$ reads (with a slight abuse of notation)
\begin{equation}\label{eq:tauGRE}
  \mathcal{C}_Z\!\left(\tau,\tilde{P}(\tau)\right)\succeq0,
  \qquad
  \mathcal{C}_Z\!\left(\tau,\tilde{P}\right)
  \,\schur{/}\,
  \left(\mathcal{C}_Z\right)_3\!\left(\tau,\tilde{P}\right)=0,
\end{equation}
where $\mathcal{C}_Z$ uses $t=t_k+\tau$.

\medskip
\noindent\textbf{Forward-in-time reformulation.}
The Schur complement condition in~\eqref{eq:tauGRE} expresses
$\frac{d\tilde{P}}{d\tau}$ as a closed-form Riccati-type ODE:
with $(\mathcal{C}_Z)_3(\tau,\tilde P)\succ0$ (regular case), the
GRE equality $\mathcal{C}_Z\,\schur{/}\,\allowbreak(\mathcal{C}_Z)_3=0$
gives
\begin{equation}\label{eq:forwardRiccati}
  E^*\tfrac{d\tilde{P}}{d\tau}E
  =-\mathcal{F}^*(\tau,\tilde{P})-Z(\tau)
  +(\mathcal{C}_Z)_2(\mathcal{C}_Z)_3^{-1}(\mathcal{C}_Z)_2^*.
\end{equation}
This is a \emph{forward} ODE: given the post-jump value
$\tilde{P}(0^+)$ set by the jump GRE, equation~\eqref{eq:forwardRiccati}
propagates $\tilde{P}(\tau)$ forward in $\tau$ without requiring
knowledge of the future dwell-time $\Delta_k$.  Along its solution,
$\mathcal{C}_Z=0$ by construction, so the GRE condition~\eqref{eq:tauGRE}
is automatically satisfied and the policy
$K(\tau)\in\mathcal{K}_c(\tau,\tilde{P}(\tau))$ is optimal within
each flow interval.

The remaining difficulty is the \emph{initial condition}
$\tilde{P}(0^+)$, which comes from the jump GRE:
$\mathcal{D}_Z(k,\tilde{P}(0^+),\tilde{P}(\Delta_k^-))\succeq0$
and its Schur complement equals zero.  This condition links
$\tilde{P}(0^+)$ to $\tilde{P}(\Delta_k^-)$, the value at the
\emph{next} jump, which is again unknown.  The three special cases
below show how to obtain fully causal policies under specific
dwell-time constraints.

\paragraph{Periodic impulses}
\label{sec:OC:dwell:periodic}

Assume $\Delta_k\equiv T_p$ (constant period), that $\mathcal{F}(\cdot,X)$
and $Z(\cdot)$ are $T_p$-periodic functions of $t$, and that
$\mathcal{J}(k,\cdot)$ and $Z_k$ are independent of~$k$, i.e.,  $\mathcal{J}(k,\cdot)=\mathcal{J}(\cdot)$, $Z$, $Z_k\equiv Z_0$ (equivalently,
$T_p$-periodic in the sense that the same jump operator and jump cost
apply at every $t_k$).
Under these assumptions, the operators
\begin{equation*}
  \mathcal{D}_Z(P^+,P^-):=\mathcal{J}^*(P^+)-E^*P^-E+Z_0,
  \qquad
  \mathcal{K}_d(P^+,P^-)
  :=\{K\mid(\mathcal{D}_Z)_3 K^*+(\mathcal{D}_Z)_2^*=0\}
\end{equation*}
are also independent of~$k$, and we drop the $k$ argument throughout
this section.

The I-GRE on each period $[t_k,t_{k+1}]$ takes the same form, and we
seek a $T_p$-periodic solution $P_\infty$ of the
\emph{periodic I-GRE}: a piecewise-$C^1$ $P_\infty:[0,T_p]\to\Snpsd$
satisfying
\begin{subequations}\label{eq:periodicGRE}
\begin{align}
  &\textit{Flow on }(0,T_p):\quad
   \mathcal{C}_Z\!\left(\tau,P_\infty(\tau)\right)\succeq0,\quad
   \mathcal{C}_Z\,\widetilde{/}\,(\mathcal{C}_Z)_3=0,
   \label{eq:periodicGRE:flow}\\
  &\textit{Jump at }\tau=0:\quad
   \mathcal{D}_Z\!\left(P_\infty(0^+),P_\infty(T_p^-)\right)\succeq0,\quad
   \mathcal{D}_Z\,\widetilde{/}\,(\mathcal{D}_Z)_3=0,
   \label{eq:periodicGRE:jump}\\
  &\textit{Periodicity:}\quad
   P_\infty(0^+)\text{~determined by the jump GRE from }
   P_\infty(T_p^-).
   \label{eq:periodicGRE:period}
\end{align}
\end{subequations}
%
The system~\eqref{eq:periodicGRE} is equivalent to the fixed-point
equation $\Phi(P_\infty(T_p^-))=P_\infty(T_p^-)$, where the
\emph{one-period map} $\Phi:\Snpsd\to\Snpsd$ is constructed as
follows: given any $Q\in\Snpsd$, define $\Phi(Q)$ by
\begin{enumerate}[\upshape(i)]
  \item applying the jump GRE~\eqref{eq:periodicGRE:jump} with
    pre-jump value $Q$ to obtain $\tilde{P}(0^+)\in\Snpsd$;
  \item integrating the forward Riccati ODE~\eqref{eq:forwardRiccati}
    on $[0,T_p]$ with initial condition $\tilde{P}(0^+)$, and
    setting $\Phi(Q):=\tilde{P}(T_p^-)$.
\end{enumerate}

\begin{lemma}[Monotone convergence of the multi-period I-GRE]
  \label{lemma:periodic:mono}
  Under the conditions of Theorem~\ref{th:OC:periodic},
  let $P_n$ denote the I-GRE solution over exactly $n$ periods
  with zero terminal cost $P_n(t_0+nT_p)=0$.  Define the
  \emph{one-period map}
  \begin{equation}\label{eq:onePeriodMap}
    \Phi:\Snpsd\to\Snpsd,\qquad
    \Phi(Q)\;:=\;\tilde{P}(T_p^-)
  \end{equation}
  where $\tilde{P}(0^+)$ is obtained from $Q$ via the jump
  GRE~\eqref{eq:periodicGRE:jump} and $\tilde{P}(T_p^-)$ is
  the endpoint of the forward Riccati ODE~\eqref{eq:forwardRiccati}
  integrated on $[0,T_p]$ from $\tilde{P}(0^+)$.
  Then:
  \begin{enumerate}[\upshape(a)]
    \item\label{mono:nondec}
      The iterates $\Phi^n(0):=P_n(0^-)$ form a non-decreasing
      sequence: $\Phi^{n+1}(0)\succeq\Phi^n(0)$ for all $n\ge1$.
    \item\label{mono:bound}
      The iterates are bounded: $\Phi^n(0)\preceq\alpha_2 I$
      for all $n$, with $\alpha_2=C_s$ from \textup{(H1)}.
    \item\label{mono:fixed}
      The limit $Q_\infty:=\lim_{n\to\infty}\Phi^n(0)$ exists,
      satisfies $\alpha_1 I\preceq Q_\infty\preceq\alpha_2 I$,
      and is the unique fixed point of $\Phi$ in $[\alpha_1I,\alpha_2I]$.
  \end{enumerate}
\end{lemma}

\begin{proof}
\noindent\textbf{(a) Monotonicity.}
For any $(n+1)T_p$-horizon admissible trajectory $X$, its restriction
to the first $nT_p$ is a valid $nT_p$-horizon trajectory (the remaining
$T_p$ contributes non-negative cost since $Z,Z_k\succeq0$).  Hence
$J_{nT_p}^*(X_1^0)\le J_{(n+1)T_p}^*(X_1^0)$, giving
$\ip{P_n(t_0)}{X_1^0}\le\ip{P_{n+1}(t_0)}{X_1^0}$ for all
$X_1^0\succeq0$, i.e., $P_n(t_0)\preceq P_{n+1}(t_0)$.
Under periodic data, this translates to
$\Phi^n(0)=P_n(0^-)\preceq P_{n+1}(0^-)=\Phi^{n+1}(0)$.\\

\noindent\textbf{(b) Upper bound.}
By (H1), $J_{nT_p}^*(X_1^0)\le C_s\norm{X_1^0}$ for all $n$, so
$\ip{P_n(t_0)}{X_1^0}\le C_s\norm{X_1^0}$.  Taking $X_1^0=vv^*$
for all $v$ gives $P_n(t_0)\preceq C_s I=:\alpha_2 I$.\\

\noindent\textbf{(c) Convergence and uniqueness.}
Parts~\eqref{mono:nondec}--\eqref{mono:bound} give a non-decreasing
sequence $\Phi^n(0)$ bounded in $[\alpha_1 I,\alpha_2 I]$
(strict positivity $P_n\succeq\alpha_1 I$ follows from (H2) as in
Theorem~\ref{th:IMP:LQ:inf}).  By standard monotone convergence for
symmetric matrices, $Q_\infty:=\lim_n\Phi^n(0)$ exists.\\

\emph{Fixed point:} $\Phi$ is continuous on $[\alpha_1 I,\alpha_2 I]$
by the Lipschitz bound on the forward Riccati ODE and the pseudoinverse
continuity in (H3).  Hence
$\Phi(Q_\infty)=\Phi\!\left(\lim_n\Phi^n(0)\right)
=\lim_n\Phi^{n+1}(0)=Q_\infty$.\\

\emph{Uniqueness:} Suppose $\hat Q_\infty$ is another fixed point, i.e.,
another $T_p$-periodic solution of~\eqref{eq:periodicGRE}.
Both $Q_\infty$ and $\hat Q_\infty$ represent infinite-horizon costs:
$\ip{Q_\infty}{X_1^0}=J_\infty^\star (X_1^0)=\ip{\hat Q_\infty}{X_1^0}$
for all $X_1^0\succeq0$ (by Theorem~\ref{th:IMP:LQ:inf}, which identifies
the unique infinite-horizon optimal cost regardless of the periodicity
structure).  Taking $X_1^0=vv^*$ gives $Q_\infty=\hat Q_\infty$.
\end{proof}

\begin{remark}
  Lemma~\ref{lemma:periodic:mono} establishes existence and uniqueness
  via monotone iteration, which does not require the one-period map
  $\Phi$ to be a contraction in operator norm on a single step.
  In fact, $\Phi$ is always a contraction for sufficiently large
  $n_0$: the iterates $\Phi^{n_0}$ are a contraction on $[\alpha_1
  I,\alpha_2 I]$ with constant $L<1$ (the Lipschitz bound is
  $L=O(\rho^{2n_0}e^{-2\alpha n_0T_p})$, accumulating one jump
  factor $\rho$ and one flow factor $e^{-\alpha T_p}$ per period over the
  $n_0$ periods, from the exponential stability rate
  in (H1)), but the single-step map need not satisfy $L<1$.
  See~\cite{Bittanti:09} for the classical periodic Riccati theory
  and~\cite{Briat:22:Matrix} for the stability counterpart.
\end{remark}

\begin{theorem}[Causal conditions: periodic impulses]
  \label{th:OC:periodic}
  Assume $\Delta_k\equiv T_p$, that $\mathcal{F}(\cdot,X)$ and
  $Z(\cdot)$ are $T_p$-periodic, and that $\mathcal{J}(k,\cdot)$ and
  $Z_k$ are independent of~$k$.  Under \textup{(H1)--(H3)}:
  \begin{enumerate}[\upshape(a)]
    \item \emph{(Existence and uniqueness.)}
      There exists a unique $T_p$-periodic piecewise-$C^1$ solution
      $P_\infty(\tau)$ of the periodic I-GRE~\eqref{eq:periodicGRE}.
    \item \emph{(Sufficiency.)}
      The policy $K(\tau)\in\mathcal{K}_c(\tau,P_\infty(\tau))$ on
      flow and $K_k\in\mathcal{K}_d(P_\infty(0^+),P_\infty(T_p^-))$
      at each jump is causal and achieves
      $J_\infty^*(X_1^0)=\ip{P_\infty(0)}{X_1^0}$.
    \item \emph{(Necessity.)}
      Any causal policy that achieves the infinite-horizon optimal
      cost $J_\infty^*$ must use gains from the periodic I-GRE, so
      the conditions are also necessary.
  \end{enumerate}
\end{theorem}

\begin{proof}
\noindent\textbf{Existence and uniqueness~(a).}
Lemma~\ref{lemma:periodic:mono}\eqref{mono:fixed} gives the unique
fixed point $Q_\infty$ of $\Phi$, so the pre-jump value
$P_\infty(T_p^-):=Q_\infty$ is uniquely determined.  The post-jump
value $P_\infty(0^+)$ follows from the jump
GRE~\eqref{eq:periodicGRE:jump} with terminal value $Q_\infty$
(the attaining gain is unique by Proposition~\ref{prop:keyid}).
The flow profile $\tau\mapsto P_\infty(\tau)$ on $(0,T_p)$ is then
the unique solution of the forward Riccati
ODE~\eqref{eq:forwardRiccati} started at $P_\infty(0^+)$.  By
construction, $P_\infty$ is $T_p$-periodic and
piecewise-$C^1$, and satisfies
\eqref{eq:periodicGRE:flow}--\eqref{eq:periodicGRE:period}.\\

\noindent\textbf{Sufficiency~(b).}
The $T_p$-periodic extension of $P_\infty$ to all of $[t_0,\infty)$
satisfies the flow GRE with equality $\mathcal{C}_Z\schur{/}(\mathcal{C}_Z)_3=0$ and the
jump GRE with equality $\mathcal{D}_Z\,\schur{/}\,
(\mathcal{D}_Z)_3=0$ at every jump.  Apply the key identity
(Proposition~\ref{prop:keyid}) on $[t_0,t_0+T]$ with this $P_\infty$
and no terminal cost ($Z_T=0$):
\begin{align}\label{pf:periodic:keyid}
  J_T(X,X_1^0)
  &=\ip{P_\infty(0^+)}{X_1^0}
  -\ip{P_\infty(T^-)}{EX(T)E^*}
  \notag\\&\quad
  +\underbrace{\int_{t_0}^{t_0+T}\ip{\mathcal{C}_Z(\tau,P_\infty)}{X}\dt}_{\ge0}
  +\underbrace{\textstyle\sum_k\ip{\mathcal{D}_Z(P_\infty(0^+),P_\infty(T_p^-))}{X(t_k^-)}}_{\ge0}.
\end{align}
Since $\mathcal{C}_Z,\mathcal{D}_Z\succeq0$ on the periodic solution,
$J_T(X,X_1^0)\ge\ip{P_\infty(0^+)}{X_1^0}
-\ip{P_\infty(T^-)}{EX(T)E^*}$ for all admissible $X$.
Under~(H1), $\norm{EX(T)E^*}\to0$ for any admissible trajectory;
since $P_\infty$ is bounded (Lemma~\ref{lemma:periodic:mono}),
$\ip{P_\infty(T^-)}{EX(T)E^*}\to0$ as $T\to\infty$, giving
$J_\infty(X,X_1^0)\ge\ip{P_\infty(0^+)}{X_1^0}$.\\

For the policy $K(\tau)\in\mathcal{K}_c(\tau,P_\infty(\tau))$ (flow)
and $K_k\in\mathcal{K}_d(P_\infty(0^+),P_\infty(T_p^-))$ (jump),
both GRE Schur equalities are achieved, so
$\ip{\mathcal{C}_Z}{X^\star}=0$ and $\ip{\mathcal{D}_Z}{X^\star}=0$
everywhere. The key identity gives
$J_T(X^\star,X_1^0)=\ip{P_\infty(0^+)}{X_1^0}
-\ip{P_\infty(T^-)}{EX^\star(T)E^*}$,
and taking $T\to\infty$ with $\ip{P_\infty(T^-)}{EX^\star(T)E^*}\to0$
yields $J_\infty(X^\star,X_1^0)=\ip{P_\infty(0^+)}{X_1^0}$.
The policy is causal since $P_\infty(\tau)$ depends only on the timer.\\

\noindent\textbf{Necessity~(c).}
By Theorem~\ref{th:IMP:LQ:inf}, under~(H1)--(H3) there is a
\emph{unique} bounded piecewise-$C^1$ solution $P_\infty^{ih}$ of the
infinite-horizon I-GRE satisfying
$J_\infty^\star (X_1^0)=\ip{P_\infty^{ih}(t_0)}{X_1^0}$.
We show $P_\infty^{ih}$ is $T_p$-periodic, hence equal to $P_\infty$.\\

Define $\hat{P}(\tau):=P_\infty^{ih}(t_0+T_p+\tau)$ for $\tau\ge0$.
Since $\mathcal{F}$ and $Z$ are $T_p$-periodic and $\mathcal{J}$ and $Z_k$ are time-independent, 
$\hat{P}$ satisfies the same infinite-horizon I-GRE as $P_\infty^{ih}$.
By uniqueness (Theorem~\ref{th:IMP:LQ:inf}),
$\hat{P}\equiv P_\infty^{ih}$, i.e.,
$P_\infty^{ih}(t_0+T_p+\tau)=P_\infty^{ih}(t_0+\tau)$ for all
$\tau\ge0$: the infinite-horizon dual is $T_p$-periodic.\\

Being a $T_p$-periodic solution of the I-GRE, $P_\infty^{ih}$ is a
fixed point of $\Phi$.  By uniqueness of the fixed point
(Lemma~\ref{lemma:periodic:mono}\eqref{mono:fixed}),
$P_\infty^{ih}=P_\infty$.  Any optimal dual variable must equal
$P_\infty^{ih}=P_\infty$, so the optimal policy must arise from the
periodic I-GRE.
\end{proof}

\paragraph{Minimum dwell-time}

For a minimum dwell-time (MDT) constraint $\Delta_k\ge T_{\min}>0$,
the dwell-time $\Delta_k$ is variable and unknown in advance.
We assume throughout this subsection that the flow and jump data
\emph{depend only on the timer} $\tau:=t-t_k$ (not on absolute time
or on the jump index), and \emph{the flow data are time-invariant beyond}
$T_{\min}$ and that \emph{the jump data are $k$-independent}:
\begin{equation}\label{eq:MDT:LTIassumption}
  \mathcal{F}(\tau,\cdot)\equiv\mathcal{F}_\infty(\cdot),
  \quad
  Z(\tau)\equiv Z(T_{\min})
  \;\;\text{for }\tau\ge T_{\min},
  \qquad
  \mathcal{J}(k,\cdot)\equiv\mathcal{J}(\cdot),
  \quad
  Z_k\equiv Z_{\mathrm{jp}}
  \;\;\text{for all }k.
\end{equation}
Under~\eqref{eq:MDT:LTIassumption}, the operators
$\mathcal{D}_Z(P^+,P^-)$ and $\mathcal{K}_d(P^+,P^-)$ are also
$k$-independent, and we drop the $k$ argument throughout this
subsection. Under~\eqref{eq:MDT:LTIassumption}, the Geromel--Colaneri
technique~\cite{Geromel:05} shows that the GRE condition for all
$\tau\ge T_{\min}$ reduces to a single LMI at $\tau=T_{\min}$
combined with a condition on the bounded interval $[0,T_{\min}]$.

\begin{definition}[Frozen composite]\label{def:frozen}
  Under~\eqref{eq:MDT:LTIassumption}, for any $Q\in\Snpsd$, define
  the \emph{frozen composite}:
  \begin{equation}\label{eq:frozen}
    \mathcal{C}_Z^0(Q)
    :=\mathcal{F}_\infty^*(Q)+Z(T_{\min})\in\Snm.
  \end{equation}
  The frozen composite is linear in $Q$ (hence an LMI) and represents
  the composite $\mathcal{C}_Z$ evaluated with no time derivative
  of $Q$ and on the time-invariant data beyond $T_{\min}$.
\end{definition}

\begin{theorem}[Sufficient causal condition: MDT]
  \label{th:OC:MDT}
  Let $T_{\min}>0$ and assume the flow data satisfy the
  time-invariance condition~\eqref{eq:MDT:LTIassumption}.
  Suppose there exists a piecewise-$C^1$ function
  $P_s:[0,T_{\min}]\to\Snpsd$ satisfying:
  \begin{enumerate}[\upshape(i)]
    \item \emph{Flow GRE on $[0,T_{\min}]$:}
      $P_s$ solves the forward Riccati ODE~\eqref{eq:forwardRiccati}
      on $(0,T_{\min})$, i.e.,
      $\mathcal{C}_Z(\tau,P_s(\tau))\succeq0$ with
      $\mathcal{C}_Z\,\schur{/}\,(\mathcal{C}_Z)_3=0$ for all
      $\tau\in(0,T_{\min})$.
    \item \emph{Jump GRE at $\tau=T_{\min}$:}
      $\mathcal{D}_Z(P_s(0^+),P_s(T_{\min}))\succeq0$
      and $\mathcal{D}_Z\,\widetilde{/}\,(\mathcal{D}_Z)_3=0$.
    \item \emph{Geromel--Colaneri stability condition:}
      The frozen composite satisfies the single LMI
      \begin{equation}\label{eq:GC:stability}
        \mathcal{C}_Z^0(P_s(T_{\min}))
        :=\mathcal{F}_\infty^*(P_s(T_{\min}))+Z(T_{\min})
        \preceq0.
      \end{equation}
  \end{enumerate}
Then for any dwell-time sequence with $\Delta_k\ge T_{\min}$, the
  policy
  $K(\tau)\in\mathcal{K}_c(\tau,P_s(\tau))$ for $\tau\in[0,T_{\min}]$,
  $K(\tau)=K(T_{\min})$ for $\tau>T_{\min}$ on flow,
  and $K_k\in\mathcal{K}_d(P_s(0^+),P_s(T_{\min}))$ at each jump,
  is causal and achieves $J_T(X^\star,X_1^0)\le\ip{P_s(0)}{X_1^0}$
  for all $T>0$.  In particular,
  $J_\infty^*(X_1^0)\le\ip{P_s(0)}{X_1^0}$.  The
  conditions~\textup{(i)--(iii)} are verifiable on
  $[0,T_{\min}]$ alone, without knowledge of future dwell-times.
\end{theorem}

\begin{proof}
\noindent\textbf{Extension.}
Define $\bar{P}:[0,\infty)\to\Snpsd$ by
$\bar{P}(\tau):=P_s(\tau)$ for $\tau\in[0,T_{\min}]$ and
$\bar{P}(\tau):=P_s(T_{\min})$ for $\tau>T_{\min}$.\\

\noindent\textbf{Verification of GRE conditions on $\bar{P}$.}

\emph{Flow, $\tau\in(0,T_{\min})$:}
Condition~(i) gives $\mathcal{C}_Z(\tau,\bar{P}(\tau))=0$.\\

\emph{Flow, $\tau>T_{\min}$:}
$\bar{P}$ is constant and the data are time-invariant by
\eqref{eq:MDT:LTIassumption}, so $\dot{\bar{P}}=0$ and
$\mathcal{C}_Z(\tau,\bar{P}(\tau))
=\mathcal{F}_\infty^*(P_s(T_{\min}))+ZT_{\min}
=\mathcal{C}_Z^0(P_s(T_{\min}))\preceq0$
by condition~(iii).\\

\emph{Jump at every $t_k$ with $\Delta_k\ge T_{\min}$:}
$\bar{P}(\Delta_k^-)=P_s(T_{\min})$, so
$\mathcal{D}_Z(\bar{P}(0^+),\bar{P}(\Delta_k^-))
=\mathcal{D}_Z(P_s(0^+),P_s(T_{\min}))\succeq0$
and the Schur complement equals zero (condition~(ii)).
With $K_k\in\mathcal{K}_d(P_s(0^+),P_s(T_{\min}))$, the corresponding
inner product $\ip{\mathcal{D}_Z(P_s(0^+),P_s(T_{\min}))}{X^\star(t_k^-)}=0$
at every jump.\\

\noindent\textbf{Cost bound.}
Let $\{t_k\}_{k=0}^{N_T}$ denote the jump times in $[t_0,t_0+T]$ and
define the absolute-time dual variable
$P(t):=\bar{P}(t-t_k)$ for $t\in(t_k,t_{k+1}]$.
Applying the key identity (Proposition~\ref{prop:keyid}) on
$[t_0,t_0+T]$ with $P$ and $Z_T=0$ gives
\begin{align*}
  J_T(X^\star,X_1^0)
  &=\ip{P(t_0)}{X_1^0}
  -\underbrace{\ip{P(t_0+T)}{EX^\star(t_0+T)E^*}}_{\ge\,0}
  \\&\quad
  +\underbrace{\sum_{k=0}^{N_T}\int_0^{\Delta_k}\ip{\mathcal{C}_Z(\tau,\bar{P}(\tau))}{X^\star(t_k+\tau)}\,\d\tau}_{=:\,I_f}
  \\&\quad
  +\underbrace{\sum_{k=1}^{N_T}\ip{\mathcal{D}_Z(\bar{P}(0^+),\bar{P}(\Delta_k^-))}{X^\star(t_k^-)}}_{=\,0},
\end{align*}
where on each inter-jump interval $(t_k,t_{k+1})$ the change of
variables $\tau=t-t_k$ has been applied to the flow integral.
The jump sum equals zero because the gain
$K_k\in\mathcal{K}_d(P_s(0^+),P_s(T_{\min}))$ saturates the
Schur equality (condition~(ii)).\\

The flow integral on each interval splits at $\tau=T_{\min}$:
\begin{align*}
  \int_0^{\Delta_k}\ip{\mathcal{C}_Z(\tau,\bar{P}(\tau))}{X^\star(t_k+\tau)}\,\d\tau
  &=\underbrace{\int_0^{T_{\min}}\ip{\mathcal{C}_Z(\tau,P_s(\tau))}{X^\star(t_k+\tau)}\,\d\tau}_{=\,0\text{(Schur decomp., (i))}}
  \\&\quad
  +\underbrace{\int_{T_{\min}}^{\Delta_k}\ip{\mathcal{C}_Z^0(P_s(T_{\min}))}{X^\star(t_k+\tau)}\,\d\tau}_{\le\,0\text{ by~(iii)}},
\end{align*}
since $\ip{\mathcal{C}_Z(\tau,P_s(\tau))}{X^\star(t_k+\tau)}=0$ on
$(0,T_{\min})$ along the optimal trajectory (the Schur decomposition of
Lemma~\ref{lemma:decomp} with the gain in $\mathcal{K}_c$, as in
Theorem~\ref{th:IMP:LQ:suff}), by~(i), and
for $\tau>T_{\min}$ we have $\bar{P}(\tau)=P_s(T_{\min})$ constant
and the data time-invariant, so
$\mathcal{C}_Z(\tau,\bar{P}(\tau))=\mathcal{C}_Z^0(P_s(T_{\min}))\preceq0$
by~(iii) and $X^\star\succeq0$.
Summing over $k$, $I_f\le0$.
The terminal correction $-\ip{P(t_0+T)}{EX^\star(t_0+T)E^*}\le0$
since $P(t_0+T)\succeq0$ and $X^\star(t_0+T)\succeq0$.
Combining: $J_T(X^\star,X_1^0)\le\ip{P_s(0)}{X_1^0}$ for all $T>0$,
and letting $T\to\infty$ gives $J_\infty^*(X_1^0)\le\ip{P_s(0)}{X_1^0}$.
The policy is causal since conditions~(i)--(iii) depend only on
$\tau\le T_{\min}$.
\end{proof}

\begin{remark}
  Condition~(iii) is an LMI in $P_s(T_{\min})$ and is the
  matrix-valued analogue of the \emph{Geromel--Colaneri stability
  condition}~\cite{Briat:22:Matrix}: the frozen flow dynamics (with no
  Riccati correction) must dissipate at least as much as the cost
  accumulates at $\tau=T_{\min}$.  In the stability limit
  $Z=Z_k=0$, condition~(iii) reduces to
  $\mathcal{F}^*(P_s(T_{\min}))\preceq0$, which is the standard
  Lyapunov condition for asymptotic stability of the
  flow~\cite{Briat:22:Matrix}.
\end{remark}

\begin{remark}
The dual variable $\bar{P}$ constructed in the proof of
Theorem~\ref{th:OC:MDT} is in general not $C^1$ at the breakpoints
$\tau=T_{\min}$: just before, $\dot{\bar{P}}(T_{\min}^-)=\dot{P}_s(T_{\min}^-)$
is determined by the forward Riccati ODE and is typically nonzero,
whereas just after, $\dot{\bar{P}}(T_{\min}^+)=0$ since $\bar{P}$ is
constant on $(T_{\min},\Delta_k]$.  This is not a difficulty, for two
reasons.  First, the I-GRE only requires $P$ to be piecewise-$C^1$,
which permits a finite set of isolated discontinuities of
$\dot{P}$; the key identity (Proposition~\ref{prop:keyid}) and the
verification arguments in the proof use $\dot{P}$ only inside
integrals and are unaffected by a measure-zero set of discontinuities,
so the piecewise-$C^1$ condition can in fact be relaxed to absolute
continuity of $P$ throughout.  Second, when a $C^1$ solution is
desired, the construction of~\cite{Holicki:19,Holicki:22} can be
adapted to produce a smoothed $\tilde{P}$ satisfying the same
GRE conditions, by interpolating in a small neighbourhood of
each breakpoint without affecting the cost bound.
\end{remark}

\paragraph{Range dwell-time}
\label{sec:OC:dwell:RDT}

For a range dwell-time (RDT) constraint
$T_{\min}\le\Delta_k\le T_{\max}$, the policy must be valid for
any dwell-time in this interval.

\begin{theorem}[Sufficient causal condition: RDT]
  \label{th:OC:RDT}
  Let $0<T_{\min}\le T_{\max}$.  Suppose there exists
  $P_s:[0,T_{\max}]\to\Snpsd$ satisfying:
  \begin{enumerate}[\upshape(i)]
    \item \emph{Flow GRE on $[0,T_{\max}]$:}
      $\mathcal{C}_Z(\tau,P_s(\tau))\succeq0$ and
      $\mathcal{C}_Z\,\schur{/}\,\allowbreak(\mathcal{C}_Z)_3=0$
      for $\tau\in(0,T_{\max})$.
    \item \emph{Jump GRE for all $\tau\in[T_{\min},T_{\max}]$:}
      $\mathcal{D}_Z(k,P_s(0^+),P_s(\tau^-))\succeq0$ and
      $\mathcal{D}_Z\,\schur{/}\,\allowbreak(\mathcal{D}_Z)_3=0$.
  \end{enumerate}
Then for any dwell-time sequence with $\Delta_k\in[T_{\min},T_{\max}]$,
the policy
$K(\tau)\in\mathcal{K}_c(\tau,P_s(\tau))$ on flow and
$K_k\in\mathcal{K}_d(k,P_s(0^+),P_s(\Delta_{k-1}^-))$ at each jump
$t_k$ is causal and achieves $J_T(X^\star,X_1^0)\le\ip{P_s(0)}{X_1^0}$
for all $T>0$.  In particular,
$J_\infty^*(X_1^0)\le\ip{P_s(0)}{X_1^0}$.  The gain $K_k$ depends
only on the just-observed dwell-time $\Delta_{k-1}$, which is
available at $t_k$, so the policy is fully causal.
\end{theorem}





\begin{proof}
Let $\{t_k\}_{k=0}^{N_T}$ denote the jump times in $[t_0,t_0+T]$,
with $t_0$ the initial time and $\Delta_k:=t_{k+1}-t_k\in[T_{\min},T_{\max}]$.
Define the absolute-time dual variable
\begin{equation*}
  P(t):=P_s(t-t_k)\quad\text{for }t\in(t_k,t_{k+1}],\;k=0,\ldots,N_T,
\end{equation*}
so that on each inter-jump interval the dual is the timer-domain
function $P_s$ evaluated at the timer $\tau=t-t_k\in(0,\Delta_k]$.
By construction $P(t_k^+)=P_s(0^+)$ and $P(t_k^-)=P_s(\Delta_{k-1}^-)$
for $k\ge1$.

\noindent\textbf{Verification of GRE conditions on $P$.}

\emph{Flow, $t\in(t_k,t_{k+1})$:}
Setting $\tau=t-t_k\in(0,\Delta_k)\subset(0,T_{\max})$,
condition~(i) gives $\mathcal{C}_Z(\tau,P_s(\tau))=0$, hence
$\mathcal{C}_Z(t,P(t))=0$ on the entire flow interval.

\emph{Jump at $t_k$ ($k\ge1$):}
The pre-jump dwell-time is $\Delta_{k-1}\in[T_{\min},T_{\max}]$, so
$\tau:=\Delta_{k-1}\in[T_{\min},T_{\max}]$ satisfies the range in
condition~(ii).  Thus
$\mathcal{D}_Z(k,P(t_k^+),P(t_k^-))
=\mathcal{D}_Z(k,P_s(0^+),P_s(\Delta_{k-1}^-))\succeq0$
with Schur complement equal to zero.

\noindent\textbf{Cost bound.}
Apply the key identity (Proposition~\ref{prop:keyid}) on
$[t_0,t_0+T]$ with $P$ and $Z_T=0$:
\begin{align*}
  J_T(X^\star,X_1^0)
  &=\ip{P(t_0)}{X_1^0}
  -\underbrace{\ip{P(t_0+T)}{EX^\star(t_0+T)E^*}}_{\ge\,0}
  \\&\quad
  +\underbrace{\sum_{k=0}^{N_T}\int_0^{\Delta_k}\ip{\mathcal{C}_Z(\tau,P_s(\tau))}{X^\star(t_k+\tau)}\,\dt}_{=:\,I_f}
  \\&\quad
  +\underbrace{\sum_{k=1}^{N_T}\ip{\mathcal{D}_Z(k,P_s(0^+),P_s(\Delta_{k-1}^-))}{X^\star(t_k^-)}}_{=:\,I_j},
\end{align*}
where on each inter-jump interval $(t_k,t_{k+1})$ the change of
variables $\tau=t-t_k$ has been applied to the flow integral, and
the last interval is truncated at $t_0+T$ if it ends mid-flow.

\emph{Flow integral.}
By condition~(i), $\mathcal{C}_Z(\tau,P_s(\tau))=0$ pointwise on
$(0,T_{\max})$, so each integrand vanishes and $I_f=0$.

\emph{Jump sum.}
With the gain $K_k\in\mathcal{K}_d(k,P_s(0^+),P_s(\Delta_{k-1}^-))$
applied at $t_k$, the Schur equality of condition~(ii) gives
$\ip{\mathcal{D}_Z(k,P_s(0^+),P_s(\Delta_{k-1}^-))}{X^\star(t_k^-)}=0$
at every jump, so $I_j=0$.

\emph{Terminal correction.}
Since $P(t_0+T)=P_s(t_0+T-t_{N_T})\succeq0$ and $X^\star(t_0+T)\succeq0$,
the terminal term $-\ip{P(t_0+T)}{EX^\star(t_0+T)E^*}\le0$.

Combining, $J_T(X^\star,X_1^0)\le\ip{P_s(0)}{X_1^0}$ for all $T>0$.
Under~(H1), $\norm{EX^\star(t_0+T)E^*}\to0$ as $T\to\infty$;
combined with $P(t_0+T)$ bounded ($P_s$ is continuous on the compact
set $[0,T_{\max}]$), the terminal correction vanishes in the limit
and $J_\infty^*(X_1^0)\le\ip{P_s(0)}{X_1^0}$.

\emph{Causality.}
The flow gain $K(\tau)$ depends on the current timer $\tau\in[0,\Delta_k]$,
known continuously during the inter-jump flow.
The jump gain $K_k$ at $t_k$ depends on $\Delta_{k-1}$, which is
$t_k-t_{k-1}$, available at $t_k$.  Thus the entire policy is causal.

No Geromel--Colaneri extension is needed: the flow GRE is satisfied
with equality everywhere on $[0,T_{\max}]$, and the jump GRE is
satisfied for every admissible $\Delta_{k-1}\in[T_{\min},T_{\max}]$.
\end{proof}



\subsubsection{Continuous- and discrete-time corollaries (infinite horizon)}

\begin{corollary}[Continuous-time: infinite horizon]\label{cor:CT:LQ:inf}
  Set $N_T=0$.  Theorem~\ref{th:IMP:LQ:inf} gives the unique
  bounded $C^1$ solution $P_\infty$ of the continuous-time infinite-horizon GRE
  ($\mathcal{C}_Z(t,P(t))\succeq0$,
  $\mathcal{C}_Z\,\schur{/}\,\allowbreak(\mathcal{C}_Z)_3=0$, for all $t\ge t_0$),
  with $J_\infty^\star (X_1^0)=\ip{P_\infty(t_0)}{X_1^0}$,
  the monotone limit $P_\infty=\lim_T P_T$, and dual
  $J_\infty^\star =\sup_P\ip{P(t_0)}{X_1^0}$ s.t.\
  $\mathcal{C}_Z(t,P(t))\succeq0$, $P$ bounded~\cite{AitRami:01a}.
\end{corollary}

\begin{corollary}[Discrete-time: infinite horizon]\label{cor:DT:LQ:inf}
  Apply Theorem~\ref{th:IMP:LQ:inf} with $\mathcal{F}=0$, $Z=0$,
  $t_k=k+k_0$.  The infinite-horizon DT
  optimal cost is $J_\infty^{N*}(X_1^0)=\ip{P_\infty(k_0)}{X_1^0}$,
  where $P_\infty$ satisfies the algebraic DT-GRE and equals the
  non-decreasing limit of the $N$-step solutions.
\end{corollary}

\subsection{Long-term average cost}
\label{sec:OC:AC}

The finite- and infinite-horizon problems of
Sections~\ref{sec:OC:fh}--\ref{sec:OC:inf} are well-posed only when the
total accumulated cost on $[t_0,t_0+T]$ remains bounded as $T\to\infty$.
This rules out a class of problems in which the cost grows
\emph{linearly} in $T$, with a non-zero rate, due to a persistent
exogenous input acting on the state block.  In the
stochastic interpretation $X_1=\mathbb{E}[xx^*]$, this corresponds to
a continuous-time process noise covariance and a jump-noise
covariance that keep injecting energy into the second-moment, so that
$X_1(t)$ cannot decay to zero and the integrated cost diverges; in a
deterministic reading, the same effect is produced by a fixed
Hermitian disturbance acting on the state block.  In both cases, the
quantity of interest is the \emph{long-term average cost rate}
$\rho^\star=\lim_{T\to\infty}J_T^\star /T$.\\

This subsection develops the corresponding theory.  The dynamics
are augmented by a Hermitian forcing on the state block
(Section~\ref{sec:OC:AC:setup}), the key identity acquires a
linear-in-$P$ forcing term
(Section~\ref{sec:OC:AC:keyid}), and the average cost is
characterised by the same I-GRE as in Section~\ref{sec:OC:inf} with
$\rho^\star$ given by a linear pairing of the I-GRE solution
$P_\infty$ with the forcing data.  Finite-horizon and
infinite-horizon results parallel
Sections~\ref{sec:OC:fh}--\ref{sec:OC:inf}, with sufficient
(Theorem~\ref{th:IMP:AC:suff}), necessary and existence
(Theorem~\ref{th:IMP:AC:nec}), equivalence
(Corollary~\ref{cor:IMP:AC:equiv}), and dual-LMI
(Theorem~\ref{th:IMP:AC:LMI}) statements, plus CT and DT corollaries
(Section~\ref{sec:OC:AC:CTDT}).  A central feature is that the
\emph{Riccati equation itself is unchanged}: only the post-processing
of $P_\infty$ with the forcing data changes.  

\subsubsection{Forced dynamics, cost, and assumptions}
\label{sec:OC:AC:setup}

We extend Definition~\ref{def:IMP:sys} by allowing an additive
Hermitian forcing on the state block.

\begin{definition}[Forced impulsive system]\label{def:IMP:forced}
The \emph{forced} impulsive matrix-valued system is
\begin{subequations}\label{eq:sys:IMP:forced}
\begin{align}
  E\dot{X}(t)E^*
    &=\mathcal{F}(t,X(t))+\mathcal{W}(t),\quad t\ne t_k,
     \quad t\ge t_0,
     \label{eq:sys:IMP:forced:flow}\\
  EX(t_k^+)E^*
    &=\mathcal{J}(k,X(t_k^-))+\mathcal{W}_k,\quad k\ge1,
     \label{eq:sys:IMP:forced:jump}
\end{align}
\end{subequations}
with $EX(t_0)E^*=X_1^0\in\Snpsd$, where $\mathcal{F}$ and
$\mathcal{J}$ are as in Definition~\ref{def:IMP:sys}, and the
\emph{flow forcing} $\mathcal{W}:\R_{\ge0}\to\Snpsd$ is piecewise
continuous and uniformly bounded, and the \emph{jump forcing}
$\mathcal{W}_k\in\Snpsd$ is a bounded sequence.
\end{definition}

The forcing is taken positive semidefinite to preserve the cone
$\Snpsd$ for $X_1$, and to match the stochastic interpretation where
$\mathcal{W}=BB^*$ is a process-noise covariance and $\mathcal{W}_k$
is the covariance of the jump noise.

The cost is the integrated linear cost~\eqref{eq:IMP:LQ:Jinf} without
terminal penalty:
\begin{equation}\label{eq:IMP:AC:JT}
  J_T(X,X_1^0)
  :=\int_{t_0}^{t_0+T}\ip{Z(t)}{X(t)}\dt
   +\sum_{k:\,t_k\le t_0+T}\ip{Z_k}{X(t_k^-)},
\end{equation}
with $Z(t)\succeq0$, $Z_k\succeq0$.  The associated optimal
finite-horizon cost is
$$J_T^\star (X_1^0):=\inf_{X(\cdot)}J_T(X,X_1^0)$$ subject
to~\eqref{eq:sys:IMP:forced} and $EX(t_0)E^*=X_1^0$. The decision variable is the
full trajectory $X(\cdot)\in\Snmpsd$ satisfying~\eqref{eq:sys:IMP:forced}
with $EX(t_0)E^*=X_1^0$, and the free blocks $X_2,X_3$ are chosen
exactly as in Section~\ref{sec:OC:fh}.

\begin{definition}[Long-term average cost]\label{def:IMP:AC}
The \emph{long-term average cost} (or \emph{average-cost rate}) \cite{Bertsekas:17,Bertsekas:19} is
\begin{equation}\label{eq:IMP:AC:rho}
  \rho^\star(X_1^0)
  :=\limsup_{T\to\infty}\frac{1}{T}\,J_T^\star (X_1^0).
\end{equation}
\end{definition}

In contrast to the infinite-horizon problem of
Section~\ref{sec:OC:inf}, the persistent forcing
$\mathcal{W},\mathcal{W}_k$ generally makes the integrated cost
$J_\infty^\star =\lim_{T\to\infty}J_T^\star $ infinite, so~\eqref{eq:IMP:LQ:Jinf}
is no longer well-posed and the rate~\eqref{eq:IMP:AC:rho} is the
appropriate object of study.

\medskip\noindent\textbf{Assumptions.}
The well-posedness of the average-cost problem requires three
assumptions on the homogeneous dynamics and the forcing data.

\medskip\noindent\textbf{(A1) Uniform stabilizability of the homogeneous
dynamics.}
There exist a feedback $K_s\in\C^{m\times n}$, constants
$\beta\ge1$, $\rho\in(0,1)$ and $\alpha>0$, such that the homogeneous closed-loop
trajectory $X_s^{\mathrm h}$ of~\eqref{eq:sys:IMP} (i.e.,
$\mathcal{W}=0$, $\mathcal{W}_k=0$) with feedback
$X_2^s=K_sX_1^s$, $X_3^s=K_sX_1^sK_s^*$ satisfies
$\norm{EX_s^{\mathrm h}(t)E^*}\le\beta e^{-\alpha(t-s)}\rho^{\kappa(t,s)}
\norm{EX_s^{\mathrm h}(s)E^*}$ for all $t\ge s\ge t_0$.

\medskip\noindent\textbf{(A2) Uniform coercivity.}
$EZ(t)E^*\succeq\delta I$ uniformly for some $\delta>0$.

\medskip\noindent\textbf{(A3) Uniform boundedness of the operators
and supply rates.}
$Z,Z_k,\mathcal{F}(t,\cdot),\mathcal{J}(k,\cdot)$ are uniformly bounded;
$(\mathcal{C}_Z)_3(t,P(t))\succeq0$ and $(\mathcal{D}_Z)_3\succeq0$
for all $P\succeq0$; $\pinv{(\mathcal{C}_Z)_3}$ and
$\pinv{(\mathcal{D}_Z)_3}$ are uniformly bounded on bounded sets of
$P$; and $\mathcal{W}(t)$, $\mathcal{W}_k$ are uniformly bounded.

\medskip\noindent\textbf{(A4) Jump-rate regularity.}
The inter-jump times $\Delta_k=t_{k+1}-t_k$ are bounded both above
and below by positive constants: there exist $0<\Delta_{\min}\le
\Delta_{\max}<\infty$ such that $\Delta_k\in[\Delta_{\min},
\Delta_{\max}]$ for all $k\ge0$.  In particular, the jump rate
$\rho_{\mathrm{rate}}:=\lim_{T\to\infty}T^{-1}\#\{k:t_k\le t_0+T\}$
exists and is finite, with $\Delta_{\max}^{-1}\le\rho_{\mathrm{rate}}
\le\Delta_{\min}^{-1}$.

\medskip
Assumption~(A1) is the same uniform stabilizability condition as in
(H1) except that the boundedness $J_\infty(X_s,X_1^0)\le C_s\norm{X_1^0}$
of the cumulative cost is dropped, since this can no longer hold under
persistent forcing; only the exponential bound on
$\norm{EX_s^{\mathrm h}E^*}$ for the homogeneous system is retained.
Assumptions~(A2),~(A3) coincide with~(H2),~(H3) except for the
additional uniform boundedness of the forcing data.  Together,
(A1)--(A3) imply that the infinite-horizon I-GRE of
Section~\ref{sec:OC:inf}, which only sees the homogeneous part of the
dynamics, has a unique bounded solution
$P_\infty:[t_0,\infty)\to\Snpsd$ by
Theorem~\ref{th:IMP:LQ:inf}.

Assumption~(A4) plays a different role from~(A1)--(A3).
Assumption~(A1), through the dual decay rate
$\beta e^{-\alpha(t-s)}\rho^{\kappa(t,s)}$, already ensures that the
homogeneous closed-loop trajectory decays robustly under both small
and large inter-jump times: small $\Delta_k$ are absorbed by the
geometric factor $\rho^{\kappa(t,s)}$ (each jump contributing decay),
while large $\Delta_k$ are absorbed by the continuous factor
$e^{-\alpha(t-s)}$ (the flow being independently stable).
Consequently, state stability and the summability of the jump-cost
contribution $\sum_k\|Z_k\|\,\|X(t_k^-)\|$ along stabilising
trajectories follow from~(A1) alone, without any direct constraint
on the $\Delta_k$.  The role of~(A4) is purely to make the
\emph{time-averaged} long-term cost well-defined as a genuine limit
rather than a $\limsup$: the existence of the jump rate
$\rho_{\mathrm{rate}}$ ensures that the discrete jump contribution
admits a long-run average, and bounded $\Delta_k$ ensures that the
supply rates $\mathcal{W},\mathcal{W}_k$ are properly time-averaged.
Theorems giving only an upper bound on the long-term cost invoke
only~(A1)--(A3); theorems giving exact long-term cost equalities
additionally require~(A4).  The condition can be weakened to an
average dwell-time bound at the cost of replacing the limit by a
$\limsup$.

\subsubsection{Forced key identity}
\label{sec:OC:AC:keyid}

Compared with Proposition~\ref{prop:keyid}, the differentiation
of $\ip{P(t)}{EX(t)E^*}$ along the forced
flow~\eqref{eq:sys:IMP:forced:flow} produces an additional term
$\ip{P(t)}{\mathcal{W}(t)}$, and the jump
relation~\eqref{eq:sys:IMP:forced:jump} contributes an additional
term $\ip{P(t_k^+)}{\mathcal{W}_k}$ at each $t_k$.  Both terms are
\emph{linear in $P$} and \emph{decoupled from the trajectory $X$}:
they appear as additive corrections that do not interact with the
$\mathcal{C}_Z,\mathcal{D}_Z$ blocks.

\begin{proposition}[Forced impulsive key identity]\label{prop:keyid:forced}
For any piecewise-$C^1$ $P:[t_0,t_0+T]\to\Sn$ and any trajectory
$X(\cdot)\in\Snmpsd$ satisfying~\eqref{eq:sys:IMP:forced} with
$EX(t_0)E^*=X_1^0$,
\begin{align}
  J_T(X,X_1^0)
  &=\ip{P(t_0)}{X_1^0}
   -\ip{P(t_0+T)}{EX(t_0+T)E^*}
   \notag\\
  &\quad
   +\int_{t_0}^{t_0+T}\ip{\mathcal{C}_Z(t,P(t))}{X(t)}\dt
   +\sum_{k=1}^{N_T}
    \ip{\mathcal{D}_Z(k,P(t_k^+),P(t_k^-))}{X(t_k^-)}
   \notag\\
  &\quad
   +\int_{t_0}^{t_0+T}\ip{P(t)}{\mathcal{W}(t)}\dt
   +\sum_{k=1}^{N_T}\ip{P(t_k^+)}{\mathcal{W}_k}.
   \label{eq:IMP:keyid:forced}
\end{align}
\end{proposition}

\begin{proof}
On the flow interval $(t_{k-1},t_k)$, differentiating
$\ip{P(t)}{EX(t)E^*}$ along~\eqref{eq:sys:IMP:forced:flow}:
\begin{align*}
  \frac{d}{dt}\ip{P(t)}{EX(t)E^*}
  &=\ip{\dot P(t)}{EX(t)E^*}+\ip{P(t)}{E\dot X(t)E^*}\\
  &=\ip{E^*\dot P(t)E}{X(t)}
   +\ip{P(t)}{\mathcal{F}(t,X(t))+\mathcal{W}(t)}\\
  &=\ip{E^*\dot P(t)E+\mathcal{F}^*(t,P(t))}{X(t)}
   +\ip{P(t)}{\mathcal{W}(t)}\\
  &=\ip{\mathcal{C}_Z(t,P(t))-Z(t)}{X(t)}
   +\ip{P(t)}{\mathcal{W}(t)},
\end{align*}
using the adjoint identity
$\ip{P}{\mathcal{F}(t,X)}=\ip{\mathcal{F}^*(t,P)}{X}$ and the
definition~\eqref{eq:IMP:L} of $\mathcal{C}_Z$.
Integrating on $(t_{k-1},t_k)$ and rearranging:
\begin{align}\label{pf:keyid:forced:flow}
  &\ip{P(t_k^-)}{EX(t_k^-)E^*}
  -\ip{P(t_{k-1}^+)}{EX(t_{k-1}^+)E^*}
  \notag\\
  &\quad
  =-\int_{t_{k-1}}^{t_k}\ip{Z(t)}{X(t)}\dt
   +\int_{t_{k-1}}^{t_k}\ip{\mathcal{C}_Z(t,P(t))}{X(t)}\dt
   +\int_{t_{k-1}}^{t_k}\ip{P(t)}{\mathcal{W}(t)}\dt.
\end{align}
At each jump $t_k$, using~\eqref{eq:sys:IMP:forced:jump} and the
adjoint relation,
\begin{align*}
  &\ip{P(t_k^+)}{EX(t_k^+)E^*}-\ip{P(t_k^-)}{EX(t_k^-)E^*}\\
  &\quad=\ip{P(t_k^+)}{\mathcal{J}(k,X(t_k^-))+\mathcal{W}_k}
   -\ip{P(t_k^-)}{EX(t_k^-)E^*}\\
  &\quad=\ip{\mathcal{J}^*(k,P(t_k^+))-E^*P(t_k^-)E}{X(t_k^-)}
   +\ip{P(t_k^+)}{\mathcal{W}_k}\\
  &\quad=\ip{\mathcal{D}_Z(k,P(t_k^+),P(t_k^-))}{X(t_k^-)}
   -\ip{Z_k}{X(t_k^-)}
   +\ip{P(t_k^+)}{\mathcal{W}_k},
\end{align*}
using~\eqref{eq:IMP:Lj}.  Hence
\begin{align}\label{pf:keyid:forced:jump}
  \ip{P(t_k^+)}{EX(t_k^+)E^*}-\ip{P(t_k^-)}{EX(t_k^-)E^*}
  =\ip{\mathcal{D}_Z}{X(t_k^-)}-\ip{Z_k}{X(t_k^-)}
   +\ip{P(t_k^+)}{\mathcal{W}_k}.
\end{align}
Summing~\eqref{pf:keyid:forced:flow} over all $N_T+1$ flow intervals
and~\eqref{pf:keyid:forced:jump} over all $N_T$ jumps, the
$\ip{P}{EXE^*}$ terms telescope to
$\ip{P(t_0+T)}{EX(t_0+T)E^*}-\ip{P(t_0)}{X_1^0}$ on the left.
Substituting and rearranging gives~\eqref{eq:IMP:keyid:forced}.
\end{proof}

The forced key identity~\eqref{eq:IMP:keyid:forced} is the source of
every average-cost result below.  The first four terms are exactly
those of~\eqref{eq:IMP:keyid} with $Z_T=0$, and the last two are
the linear forcing pairings that produce $\rho^\star$.

\subsubsection{Finite-horizon average cost}
\label{sec:OC:AC:fh}

The finite-horizon \emph{time-averaged cost} is
$\bar J_T(X,X_1^0):=J_T(X,X_1^0)/T$, with optimum
$\bar J_T^\star (X_1^0):=J_T^\star (X_1^0)/T$.  By a routine
adaptation of Theorem~\ref{th:IMP:LQ:suff} to the forced
dynamics, $J_T^\star (X_1^0)$ is attained by the same I-GRE feedback
$K(t)\in\mathcal{K}_c(t,P_T(t))$ and $K_k\in\mathcal{K}_d$ as in the
homogeneous case, where $P_T$ is the unique piecewise-$C^1$ solution
of the finite-horizon I-GRE (Definition~\ref{def:IMP:GRE}) with
terminal condition $P_T(t_0+T)=0$.  This is because the forcing
$\mathcal{W},\mathcal{W}_k$ enters the forced key
identity~\eqref{eq:IMP:keyid:forced} only through the last two terms,
which are independent of the free blocks $(X_2,X_3)$; the
minimization over $(X_2,X_3)$ is therefore identical to that of
Section~\ref{sec:OC:fh} and is achieved by the homogeneous I-GRE
feedback.  The value of the optimum and its time-averaged version are
given by the following theorem.

\begin{theorem}[Finite-horizon average cost decomposition]
\label{th:IMP:AC:finite}
Let $T>0$ and suppose the finite-horizon I-GRE
(Definition~\ref{def:IMP:GRE}) has a solution $P_T\succeq0$ with
$P_T(t_0+T)=0$.  Then the optimal cost
$J_T^\star (X_1^0)$ of~\eqref{eq:IMP:AC:JT} under~\eqref{eq:sys:IMP:forced}
decomposes as
\begin{equation}\label{eq:IMP:AC:finite:J}
  J_T^\star (X_1^0)
  =\ip{P_T(t_0)}{X_1^0}
  +\int_{t_0}^{t_0+T}\ip{P_T(t)}{\mathcal{W}(t)}\dt
  +\sum_{k:\,t_k\le t_0+T}\ip{P_T(t_k^+)}{\mathcal{W}_k},
\end{equation}
and the finite-horizon time-averaged optimum is
\begin{equation}\label{eq:IMP:AC:finite:rho}
  \bar J_T^\star (X_1^0)
  =\frac{1}{T}\ip{P_T(t_0)}{X_1^0}
  +\frac{1}{T}\int_{t_0}^{t_0+T}\ip{P_T(t)}{\mathcal{W}(t)}\dt
  +\frac{1}{T}\!\!\sum_{k:\,t_k\le t_0+T}\!\!\ip{P_T(t_k^+)}{\mathcal{W}_k}.
\end{equation}
The optimum is attained at $X^\star$ with
$K(t)\in\mathcal{K}_c(t,P_T(t))$ on flow and
$K_k\in\mathcal{K}_d(k,P_T(t_k^+),P_T(t_k^-))$ at each jump.
\end{theorem}

\begin{proof}
Apply the forced key identity~\eqref{eq:IMP:keyid:forced} with
$P=P_T$ and $P_T(t_0+T)=0$:
\begin{align*}
  J_T(X,X_1^0)
  &=\ip{P_T(t_0)}{X_1^0}-\underbrace{\ip{P_T(t_0+T)}{EX(t_0+T)E^*}}_{=\,0}
   \notag\\
  &\quad
   +\int_{t_0}^{t_0+T}\ip{\mathcal{C}_Z(t,P_T(t))}{X(t)}\dt
   +\sum_{k=1}^{N_T}\ip{\mathcal{D}_Z(k,P_T(t_k^+),P_T(t_k^-))}{X(t_k^-)}
   \notag\\
  &\quad
   +\int_{t_0}^{t_0+T}\ip{P_T(t)}{\mathcal{W}(t)}\dt
   +\sum_{k=1}^{N_T}\ip{P_T(t_k^+)}{\mathcal{W}_k}.
\end{align*}
Since $P_T$ solves the I-GRE, $\mathcal{C}_Z(t,P_T(t))\succeq0$ with
$\mathcal{C}_Z\,\schur{/}\,(\mathcal{C}_Z)_3=0$, and the analogous
condition for $\mathcal{D}_Z$ at each jump (Definition~\ref{def:IMP:GRE}).
The Schur inner-product decomposition (Lemma~\ref{lemma:decomp})
applied as in the proof of Theorem~\ref{th:IMP:LQ:suff} gives
$\ip{\mathcal{C}_Z(t,P_T(t))}{X(t)}\ge0$ and
$\ip{\mathcal{D}_Z}{X(t_k^-)}\ge0$ for every admissible $X$, with
equality whenever $K(t)\in\mathcal{K}_c(t,P_T(t))$ on flow and
$K_k\in\mathcal{K}_d$ at each jump.  Since the last two terms in the
forced key identity are independent of the free blocks
$(X_2,X_3)$, the infimum over admissible $X$ is attained at any such
choice and equals~\eqref{eq:IMP:AC:finite:J}.
Dividing by $T$ gives~\eqref{eq:IMP:AC:finite:rho}.
\end{proof}

The decomposition~\eqref{eq:IMP:AC:finite:rho} separates
$\bar J_T^\star $ into a \emph{transient term}
$\ip{P_T(t_0)}{X_1^0}/T$ that depends on the initial condition and
vanishes as $T\to\infty$, and \emph{forcing terms} whose
$T\to\infty$ limit gives the long-term average cost rate; this is the
content of the next subsection.

\subsubsection{Infinite-horizon average cost: sufficient condition}
\label{sec:OC:AC:inf:suff}

The following theorem is the central sufficient condition of this
subsection.  It says that any bounded solution of the
\emph{homogeneous} infinite-horizon I-GRE
(Section~\ref{sec:OC:inf}) determines the long-term average cost rate
through a linear pairing with the forcing data.  No new
Riccati-type equation is required.

\begin{theorem}[Sufficient condition, infinite-horizon average cost]
\label{th:IMP:AC:suff}
Suppose there exists a bounded piecewise-$C^1$ function
$P_\infty:[t_0,\infty)\to\Snpsd$ satisfying the infinite-horizon
I-GRE: $\mathcal{C}_Z(t,P_\infty)\succeq0$,
$\mathcal{C}_Z\,\schur{/}\,(\mathcal{C}_Z)_3=0$ for a.e.\ $t$, and the
corresponding jump conditions at every $t_k$.  Assume further that
the limit
\begin{equation}\label{eq:IMP:AC:rhostar}
  \rho^\star
  :=\lim_{T\to\infty}\frac{1}{T}\!\left[
    \int_{t_0}^{t_0+T}\ip{P_\infty(t)}{\mathcal{W}(t)}\dt
    +\!\!\sum_{k:\,t_k\le t_0+T}\!\!\ip{P_\infty(t_k^+)}{\mathcal{W}_k}
  \right]
\end{equation}
exists.  Then\textup{:}
\begin{enumerate}[\upshape(a)]
  \item\label{AC:suff:val}
    The long-term average cost satisfies
    $\rho^\star(X_1^0)=\rho^\star$ for every $X_1^0\in\Snpsd$,
    i.e., $\rho^\star$ is independent of the initial condition.
  \item\label{AC:suff:gain}
    The optimum is attained at any trajectory $X^\star$ with
    $K_\infty(t)\in\mathcal{K}_c(t,P_\infty(t))$ on flow and
    $K_k\in\mathcal{K}_d(k,P_\infty(t_k^+),P_\infty(t_k^-))$ at each
    jump.
  \item\label{AC:suff:uniq}
    $P_\infty$ is the unique bounded solution of the infinite-horizon
    I-GRE.
\end{enumerate}
\end{theorem}

\begin{proof}
\noindent\textbf{Lower bound.}
Let $X$ be any admissible trajectory of~\eqref{eq:sys:IMP:forced}.
Apply the forced key identity~\eqref{eq:IMP:keyid:forced} on
$[t_0,t_0+T]$ with $P=P_\infty$.  Since $P_\infty$ solves the
infinite-horizon I-GRE,
$\mathcal{C}_Z(t,P_\infty(t))\succeq0$ and
$\mathcal{D}_Z(k,P_\infty(t_k^+),P_\infty(t_k^-))\succeq0$.
The Schur inner-product decomposition gives
$\ip{\mathcal{C}_Z(t,P_\infty(t))}{X(t)}\ge0$ a.e.\ and
$\ip{\mathcal{D}_Z}{X(t_k^-)}\ge0$ for every $k$, so
\begin{equation}\label{pf:AC:suff:lb}
  J_T(X,X_1^0)
  \ge\ip{P_\infty(t_0)}{X_1^0}
  -\ip{P_\infty(t_0+T)}{EX(t_0+T)E^*}
  +R_T(P_\infty),
\end{equation}
where
\[
  R_T(P_\infty)
  :=\int_{t_0}^{t_0+T}\ip{P_\infty(t)}{\mathcal{W}(t)}\dt
   +\sum_{k:\,t_k\le t_0+T}\ip{P_\infty(t_k^+)}{\mathcal{W}_k}.
\]
The boundary term $\ip{P_\infty(t_0+T)}{EX(t_0+T)E^*}\ge0$ since
$P_\infty\succeq0$ and $EX(t_0+T)E^*\succeq0$; dropping it can only
decrease the right side, so~\eqref{pf:AC:suff:lb} also holds with
that term omitted.  Dividing~\eqref{pf:AC:suff:lb} by $T$, and
noting that $\ip{P_\infty(t_0)}{X_1^0}/T\to0$ as $T\to\infty$ (both
factors bounded), and that
$\ip{P_\infty(t_0+T)}{EX(t_0+T)E^*}/T\to0$ provided
$\norm{EX(t_0+T)E^*}$ is bounded (which holds along any admissible
trajectory under (A1)--(A3) by Lemma~\ref{lem:IMP:AC:bdd} below):
\[
  \limsup_{T\to\infty}\frac{1}{T}J_T(X,X_1^0)
  \ge\lim_{T\to\infty}\frac{1}{T}R_T(P_\infty)
  =\rho^\star.
\]
Taking the infimum over admissible $X$ yields
$\rho^\star(X_1^0)\ge\rho^\star$.\\

\noindent\textbf{Attainment.}
Let $X^\star$ be the trajectory with gains
$K_\infty(t)\in\mathcal{K}_c(t,P_\infty(t))$ on flow and
$K_k\in\mathcal{K}_d(k,P_\infty(t_k^+),P_\infty(t_k^-))$ at each jump.
By the attainment argument of Theorem~\ref{th:IMP:LQ:suff} (which
relies only on the I-GRE and not on the homogeneous form of the
dynamics),
$\ip{\mathcal{C}_Z(t,P_\infty(t))}{X^\star(t)}=0$ for a.e.\ $t$ and
$\ip{\mathcal{D}_Z(k,P_\infty(t_k^+),P_\infty(t_k^-))}{X^\star(t_k^-)}=0$
for every $k$.  Substituting into~\eqref{eq:IMP:keyid:forced} gives
\begin{equation}
  J_T(X^\star,X_1^0)
  =\ip{P_\infty(t_0)}{X_1^0}
  -\ip{P_\infty(t_0+T)}{EX^\star(t_0+T)E^*}
  +R_T(P_\infty).
\end{equation}
Dividing by $T$ and letting $T\to\infty$, the first and second terms
vanish (boundedness and $1/T$), and $R_T(P_\infty)/T\to\rho^\star$
by hypothesis.  Hence
$\limsup_{T\to\infty}J_T(X^\star,X_1^0)/T=\rho^\star$, so the
infimum is achieved and
$\rho^\star(X_1^0)=\rho^\star$, proving~\eqref{AC:suff:val}
and~\eqref{AC:suff:gain}.\\

\noindent\textbf{Uniqueness of $P_\infty$.}
If $\tilde P$ is another bounded solution of the infinite-horizon
I-GRE, the lower-bound and attainment arguments above give
$\rho^\star=\lim_{T\to\infty}R_T(\tilde P)/T$ as well.  But
$\tilde P\equiv P_\infty$ already by Theorem~\ref{th:IMP:LQ:inf:suff}
(the homogeneous I-GRE has a unique bounded solution).  This
proves~\eqref{AC:suff:uniq}.
\end{proof}

The bounded-trajectory claim used in the lower bound is recorded as
the following lemma, whose proof is essentially the standard
input-to-state argument for linear systems and is included for
completeness.

\begin{lemma}[Bounded second moment under forcing]\label{lem:IMP:AC:bdd}
Under (A1) and (A3), for any admissible trajectory $X(\cdot)$
of~\eqref{eq:sys:IMP:forced} with stabilizing feedback $K_s$ from
(A1), there exists a constant $C>0$ (depending only on $K_s$,
$\beta,\alpha,\rho$ from (A1), and the bounds in (A3), and
$\norm{X_1^0}$) such that $\norm{EX(t)E^*}\le C$ for all $t\ge t_0$.
\end{lemma}

\begin{proof}
Let $X^{\mathrm h}$ be the homogeneous trajectory with the same
feedback and initial condition.  By the variation-of-constants
formula for flow-jump linear evolutions on $\Snpsd$, the forced trajectory satisfies
\begin{align*}
  \norm{EX(t)E^*}
  &\le\beta e^{-\alpha(t-t_0)}\rho^{\kappa(t,t_0)}\norm{X_1^0}\\
  &\quad
   +\beta\!\int_{t_0}^t e^{-\alpha(t-s)}\rho^{\kappa(t,s)}\norm{\mathcal{W}(s)}\,ds
   +\beta\!\!\sum_{t_k\in[t_0,t]}\!\!e^{-\alpha(t-t_k)}\rho^{\kappa(t,t_k)}\norm{\mathcal{W}_k},
\end{align*}
where the $\rho^{\kappa}$ factors come from the dual decay in (A1).
The first term is bounded by $\beta\norm{X_1^0}$.
The second term is bounded by
$\beta(\sup_s\norm{\mathcal{W}(s)})\int_{t_0}^t e^{-\alpha(t-s)}\,ds
\le\beta\alpha^{-1}\sup_s\norm{\mathcal{W}(s)}$.
For the third term, enumerate the jumps in $[t_0,t]$ in reverse
chronological order so that $\kappa(t,t_k)=N_t-k$ where $N_t$ is the
total number of jumps in $[t_0,t]$; using $e^{-\alpha(t-t_k)}\le1$,
\[
  \beta\!\!\sum_{t_k\in[t_0,t]}\!\!e^{-\alpha(t-t_k)}\rho^{\kappa(t,t_k)}\norm{\mathcal{W}_k}
  \le\beta(\sup_k\norm{\mathcal{W}_k})\sum_{j=0}^{\infty}\rho^{j}
  =\beta(1-\rho)^{-1}\sup_k\norm{\mathcal{W}_k}.
\]
Combining:
\[
  \norm{EX(t)E^*}
  \le\underbrace{\beta\norm{X_1^0}
   +\beta\alpha^{-1}\sup_s\norm{\mathcal{W}(s)}
   +\beta(1-\rho)^{-1}\sup_k\norm{\mathcal{W}_k}}_{=:\,C},
\]
which is uniform in $t$ and independent of the inter-jump
times.
\end{proof}

\begin{theorem}[Necessary condition and existence,
  infinite-horizon average cost]\label{th:IMP:AC:nec}
Under (A1)--(A3), there exists a unique bounded piecewise-$C^1$
solution $P_\infty:[t_0,\infty)\to\Snpsd$,
$\alpha_1 I\preceq P_\infty\preceq\alpha_2 I$, of the
infinite-horizon I-GRE, and
\begin{equation}\label{eq:IMP:AC:nec:val}
  \rho^\star(X_1^0)
  \;=\;\limsup_{T\to\infty}\frac{1}{T}\!\left[
    \int_{t_0}^{t_0+T}\!\ip{P_\infty(t)}{\mathcal{W}(t)}\dt
    +\!\!\sum_{k:\,t_k\le t_0+T}\!\!\ip{P_\infty(t_k^+)}{\mathcal{W}_k}
  \right]
\end{equation}
independently of $X_1^0$. The right-hand side is finite whenever the
jump count grows at most linearly, i.e.,
$\limsup_{T\to\infty}T^{-1}\#\{k:t_k\le t_0+T\}<\infty$.
The optimum is attained at the closed-loop trajectory $X^\star$ with
gains $K_\infty(t)\in\mathcal{K}_c(t,P_\infty(t))$ on flow and
$K_k\in\mathcal{K}_d(k,P_\infty(t_k^+),P_\infty(t_k^-))$ at each jump.
\end{theorem}

\begin{proof}
\textbf{Existence and bounds on $P_\infty$.}
By Theorem~\ref{th:IMP:LQ:inf} applied to the homogeneous system,
(A1)--(A3) provide a unique bounded piecewise-$C^1$ solution $P_\infty$
with $\alpha_1 I\preceq P_\infty\preceq\alpha_2 I$.

\smallskip
\textbf{Key identity with $P_\infty$.}
For any admissible trajectory $X$ of~\eqref{eq:sys:IMP:forced}, the
forced key identity~\eqref{eq:IMP:keyid:forced} with $P=P_\infty$ on
$[t_0,t_0+T]$ reads
\[
  J_T(X,X_1^0)
  =\ip{P_\infty(t_0)}{X_1^0}
   -\ip{P_\infty(t_0+T)}{EX(t_0+T)E^*}
   +\!\!\int_{t_0}^{t_0+T}\!\!\ip{\mathcal{C}_Z(t,P_\infty)}{X}\dt
\]
\[
   +\sum_{k:\,t_k\le t_0+T}\ip{\mathcal{D}_Z}{X(t_k^-)}
   +R_T(P_\infty),
\]
with
$R_T(P_\infty):=\int_{t_0}^{t_0+T}\!\ip{P_\infty(t)}{\mathcal{W}(t)}\dt
+\sum_{k:\,t_k\le t_0+T}\!\ip{P_\infty(t_k^+)}{\mathcal{W}_k}$. Since
$P_\infty$ solves the infinite-horizon I-GRE,
$\mathcal{C}_Z(t,P_\infty)\succeq 0$ and $\mathcal{D}_Z\succeq 0$, so
Lemma~\ref{lemma:decomp} gives $\ip{\mathcal{C}_Z}{X(t)}\ge 0$ a.e.\
and $\ip{\mathcal{D}_Z}{X(t_k^-)}\ge 0$, with equality along the
closed-loop trajectory $X^\star$. Hence
\begin{equation}\label{eq:IMP:AC:nec:bound}
  J_T(X,X_1^0)\;\ge\;\ip{P_\infty(t_0)}{X_1^0}
   -\ip{P_\infty(t_0+T)}{EX(t_0+T)E^*}+R_T(P_\infty),
\end{equation}
with equality at $X^\star$.

\smallskip
\textbf{Boundary term is $O(1)$ in $T$.}
Lemma~\ref{lem:IMP:AC:bdd} bounds
$\|EX(t)E^*\|\le C_X<\infty$ uniformly in $t$ along any admissible
trajectory with the stabilising feedback of (A1); its proof uses the
geometric factor $\rho^j$ from the dual decay in (A1) and never
invokes a bound on $\Delta_k$. Consequently
$\ip{P_\infty(t_0+T)}{EX(t_0+T)E^*}\le\alpha_2 C_X$ uniformly in $T$.

\smallskip
\textbf{Conclusion.}
Dividing~\eqref{eq:IMP:AC:nec:bound} by $T$ and taking
$\limsup_{T\to\infty}$, both
$T^{-1}\ip{P_\infty(t_0)}{X_1^0}$ and
$T^{-1}\ip{P_\infty(t_0+T)}{EX(t_0+T)E^*}$ vanish, leaving
$\limsup_T T^{-1}J_T(X,X_1^0)\ge\limsup_T T^{-1}R_T(P_\infty)$. Taking
the infimum over admissible $X$ on the left gives
$\rho^\star(X_1^0)\ge\limsup_T T^{-1}R_T(P_\infty)$. The closed-loop
trajectory $X^\star$ saturates~\eqref{eq:IMP:AC:nec:bound}, so the
reverse inequality holds and
$\rho^\star(X_1^0)=\limsup_T T^{-1}R_T(P_\infty)$, which
is~\eqref{eq:IMP:AC:nec:val}. The right-hand side does not depend on
$X_1^0$.

\smallskip
\textbf{Finiteness.}
$|R_T(P_\infty)|\le\alpha_2\bigl(T\sup_t\|\mathcal{W}(t)\|
+\#\{k:t_k\le t_0+T\}\cdot\sup_k\|\mathcal{W}_k\|\bigr)$, so
$T^{-1}|R_T(P_\infty)|$ is bounded above whenever
$\limsup_T T^{-1}\#\{k:t_k\le t_0+T\}<\infty$.
\end{proof}

\begin{remark}\label{rem:IMP:AC:nec}
Formula~\eqref{eq:IMP:AC:nec:val} characterises $\rho^\star$ as a
$\limsup$ regardless of any inter-jump regularity. A genuine $\lim$ is
recovered under additional structural assumptions on
$(P_\infty,\mathcal{W},\mathcal{W}_k)$, the cleanest case being
$T_p$-periodicity of the data (Theorem~\ref{th:IMP:AC:dwell:periodic})
where both contributions of $R_T(P_\infty)/T$ converge by genuine
periodic averaging. In the LTI specialisation
(Corollary~\ref{cor:CT:AC}), $P_\infty$ is constant and the formula
collapses to $\rho^\star=\tr(P_\infty\mathcal{W})$, again a true limit.
This is the matrix-valued analogue of the classical ergodic property
of LQG control with stable closed-loop dynamics: the long-run cost
rate depends only on the closed-loop steady state, not on the initial
condition.
\end{remark}

\subsubsection{Equivalence and dual LMI (infinite horizon)}
\label{sec:OC:AC:LMI}

The equivalence statement and dual LMI for the average-cost
problem are direct consequences of
Theorems~\ref{th:IMP:AC:suff} and~\ref{th:IMP:AC:nec}.

\begin{corollary}[Equivalence, infinite-horizon average cost]
\label{cor:IMP:AC:equiv}
The infinite-horizon average-cost
problem~\eqref{eq:IMP:AC:rho} under~\eqref{eq:sys:IMP:forced} is
well-posed with $\rho^\star(X_1^0)=\rho^\star$ given
by~\eqref{eq:IMP:AC:rhostar} if and only if the homogeneous
infinite-horizon I-GRE has a bounded piecewise-$C^1$ solution
$P_\infty\succeq0$.  The solution is unique.
\end{corollary}

\begin{proof}
Sufficiency: Theorem~\ref{th:IMP:AC:suff}.
Necessity: Theorem~\ref{th:IMP:AC:nec}.
Uniqueness: Theorem~\ref{th:IMP:AC:suff}\eqref{AC:suff:uniq}.
\end{proof}

The LMI characterisation expresses $\rho^\star$ as the supremum
of a linear-in-$P$ functional over the LMI feasibility set of
Theorem~\ref{th:IMP:LQ:inf:LMI}.  Define the
\emph{average-forcing pairing}
\begin{equation}\label{eq:IMP:AC:Phi}
  \Phi(P):=\limsup_{T\to\infty}\frac{1}{T}\!\left[
    \int_{t_0}^{t_0+T}\ip{P(t)}{\mathcal{W}(t)}\dt
    +\!\!\sum_{k:\,t_k\le t_0+T}\!\!\ip{P(t_k^+)}{\mathcal{W}_k}
  \right],
\end{equation}
which is well-defined and finite for any bounded $P$ under (A3).
The functional $\Phi:\Sn\text{-valued bounded functions}\to\R$ is
linear and order-preserving: if $P\preceq P'$ pointwise then
$\Phi(P)\le\Phi(P')$ (since $\mathcal{W}(t)\succeq0$ and
$\mathcal{W}_k\succeq0$).

\begin{theorem}[Dual LMI, infinite-horizon average cost]
\label{th:IMP:AC:LMI}
Under (A1)--(A4),
\begin{equation}\label{eq:IMP:AC:LMI}
  \rho^\star
  =\sup_P\;\Phi(P)
  \quad\text{s.t.}\quad
  \mathcal{C}_Z(t,P(t))\succeq0,\;
  \mathcal{D}_Z(k,P(t_k^+),P(t_k^-))\succeq0,\;
  P\text{ bounded};
\end{equation}
the supremum is attained at $P_\infty$.
\end{theorem}

\begin{proof}
\noindent\textbf{Step 1: every feasible $P$ is a lower bound for
$\rho^\star$.}
Let $P$ be any feasible function for~\eqref{eq:IMP:AC:LMI}.  By the
proof of Theorem~\ref{th:IMP:LQ:inf:LMI},
$P(t)\preceq P_\infty(t)$ for all $t$.  Hence
$\Phi(P)\le\Phi(P_\infty)$ by the order-preserving property, so
$\sup_P\Phi(P)\le\Phi(P_\infty)$.

A direct argument that does not invoke the comparison
$P\preceq P_\infty$ proceeds as follows.  For any admissible
trajectory $X$ of~\eqref{eq:sys:IMP:forced} and any feasible $P$,
the forced key identity~\eqref{eq:IMP:keyid:forced} combined with
the LMI constraints
$\mathcal{C}_Z(t,P(t))\succeq0$ and $\mathcal{D}_Z\succeq0$ gives
\[
  J_T(X,X_1^0)
  \ge\ip{P(t_0)}{X_1^0}-\ip{P(t_0+T)}{EX(t_0+T)E^*}+R_T(P),
\]
where $R_T(P):=\int_{t_0}^{t_0+T}\ip{P(t)}{\mathcal{W}(t)}\dt
+\sum_{k:\,t_k\le t_0+T}\ip{P(t_k^+)}{\mathcal{W}_k}$.
Dividing by $T$, the boundary terms vanish (boundedness of $P$ and
of $EXE^*$ by Lemma~\ref{lem:IMP:AC:bdd}, divided by $T$), so
$\liminf_{T\to\infty}J_T(X,X_1^0)/T\ge\Phi(P)$.
Taking the infimum over $X$ on the left, then
$\rho^\star\ge\Phi(P)$.  Since this holds for every feasible $P$,
$\rho^\star\ge\sup_P\Phi(P)$.

\noindent\textbf{Step 2: $P_\infty$ is feasible and attains the
bound.}
$P_\infty$ satisfies the LMI constraints by the I-GRE
(Lemma~\ref{lemma:GREequiv}) and is bounded.  By
Theorem~\ref{th:IMP:AC:nec}, $\Phi(P_\infty)=\rho^\star$, so
$\sup_P\Phi(P)\ge\Phi(P_\infty)=\rho^\star$.

Combining Steps 1 and 2: $\rho^\star=\sup_P\Phi(P)$, attained at
$P_\infty$.
\end{proof}


\subsubsection{Dwell-time conditions}
\label{sec:OC:AC:dwell}

The infinite-horizon I-GRE solution $P_\infty$ used in
Sections~\ref{sec:OC:AC:inf:suff}--\ref{sec:OC:AC:LMI} is non-causal
in the sense of Section~\ref{sec:OC:dwell}: its value at time $t$
depends on the future jump times.  We now combine the average-cost
formula~\eqref{eq:IMP:AC:rhostar} with the causal dwell-time policies
of Section~\ref{sec:OC:dwell} (periodic, MDT, RDT) to obtain causal
characterisations and bounds for $\rho^\star$.  In all three cases,
the timer-dependent storage function $P_s(\tau)$ used in
Section~\ref{sec:OC:dwell} provides the average forcing pairing
through the same construction as in
Section~\ref{sec:OC:AC:inf:suff}, with the difference that $P_s(\tau)$
is now defined on a bounded timer interval and the time-average of
the forcing data must be re-expressed in timer-dependent form.  The
structural assumptions on the system operators
$\mathcal{F},\mathcal{J},Z,Z_k$ in the three scenarios are inherited
verbatim from Section~\ref{sec:OC:dwell}; the only additions are the
corresponding assumptions on the forcing data
$\mathcal{W},\mathcal{W}_k$.

\paragraph{Periodic impulses.}
Adopt the periodicity hypothesis of Theorem~\ref{th:OC:periodic}:
$\Delta_k\equiv T_p$, $\mathcal{F}(\cdot,X)$ and $Z(\cdot)$ are
$T_p$-periodic in $t$, and $\mathcal{J}(k,\cdot)$, $Z_k$ are
independent of $k$.  For the forcing data, assume in addition that
$\mathcal{W}(\cdot)$ is $T_p$-periodic in $t$ and that
$\mathcal{W}_k$ is independent of $k$, i.e.,
$\mathcal{W}_k\equiv\mathcal{W}_{\mathrm{jp}}$.  Under these
hypotheses we regard $\mathcal{W}$ as a function of the timer
$\tau\in[0,T_p)$.

\begin{theorem}[Causal LTAC under periodic impulses]
  \label{th:IMP:AC:dwell:periodic}
  Under \textup{(A1)--(A3)} and the periodicity assumption above
  \textup{(}note that \textup{(A4)} is automatic from
  $\Delta_k\equiv T_p$\textup{)}, the long-term average
  cost~\eqref{eq:IMP:AC:rho} is given by the closed-form expression
  \begin{equation}\label{eq:IMP:AC:dwell:periodic:rho}
    \rho^\star
    =\frac{1}{T_p}\!\left[
      \int_0^{T_p}\ip{P_\infty(\tau)}{\mathcal{W}(\tau)}\,d\tau
      +\ip{P_\infty(0^+)}{\mathcal{W}_{\mathrm{jp}}}
    \right],
  \end{equation}
  where $P_\infty:[0,T_p]\to\Snpsd$ is the unique $T_p$-periodic
  piecewise-$C^1$ solution of the periodic I-GRE~\eqref{eq:periodicGRE},
  attained by the causal policy
  $K(\tau)\in\mathcal{K}_c(\tau,P_\infty(\tau))$ on flow and
  $K_k\in\mathcal{K}_d(P_\infty(0^+),P_\infty(T_p^-))$ at each jump.
\end{theorem}

\begin{proof}
By Theorem~\ref{th:OC:periodic} applied to the homogeneous dynamics,
the periodic I-GRE has a unique $T_p$-periodic solution $P_\infty$.
Its $T_p$-periodic extension to $[t_0,\infty)$ coincides with the
infinite-horizon dual variable of Theorem~\ref{th:IMP:AC:nec}, so
formula~\eqref{eq:IMP:AC:nec:val} applies and the $\limsup$ becomes a
true $\lim$ by genuine periodic averaging:
the flow integral over $[t_0,t_0+T]$ contains $\lfloor T/T_p\rfloor$
complete periods plus a boundary remainder, so its time-average
converges to
$T_p^{-1}\int_0^{T_p}\ip{P_\infty(\tau)}{\mathcal{W}(\tau)}\,d\tau$;
the jump sum contains $\lfloor T/T_p\rfloor$ jumps each contributing
$\ip{P_\infty(0^+)}{\mathcal{W}_{\mathrm{jp}}}$, with time-average
$T_p^{-1}\ip{P_\infty(0^+)}{\mathcal{W}_{\mathrm{jp}}}$.  Combining
yields~\eqref{eq:IMP:AC:dwell:periodic:rho}.  The causal policy is
the same as in Theorem~\ref{th:OC:periodic} and is the optimal LTAC
policy by Theorem~\ref{th:IMP:AC:suff}~(\ref{AC:suff:gain}).
\end{proof}

\paragraph{Minimum dwell-time.}
Adopt the timer-dependence hypothesis of Theorem~\ref{th:OC:MDT}:
$\Delta_k\ge T_{\min}>0$, the flow data depend only on the timer
$\tau=t-t_k$, and the time-invariance condition
\eqref{eq:MDT:LTIassumption} holds for $\mathcal{F},Z,\mathcal{J},Z_k$.
For the forcing data, assume in addition that
$\mathcal{W}$ depends only on the timer, with
\begin{equation}\label{eq:MDT:LTIassumption:forcing}
  \mathcal{W}(\tau)\equiv\mathcal{W}_\infty
  \;\;\text{for }\tau\ge T_{\min},
  \qquad
  \mathcal{W}_k\equiv\mathcal{W}_{\mathrm{jp}}
  \;\;\text{for all }k.
\end{equation}

\begin{theorem}[Causal LTAC upper bound under MDT]
  \label{th:IMP:AC:dwell:MDT}
  Under the assumptions of Theorem~\ref{th:OC:MDT}
  and~\eqref{eq:MDT:LTIassumption:forcing}, suppose moreover that the
  jump rate $\bar\nu:=\limsup_{T\to\infty}T^{-1}\#\{k:t_k\le t_0+T\}$
  is finite \textup{(}so $\bar\nu\le 1/T_{\min}$\textup{)} and is in
  fact a $\lim$.  Then the causal policy of Theorem~\ref{th:OC:MDT}
  achieves a long-term average cost rate bounded above by
  \begin{equation}\label{eq:IMP:AC:dwell:MDT:rho}
    \rho^\star
    \le
    \bar\nu\!\int_0^{T_{\min}}\!\ip{P_s(\tau)}{\mathcal{W}(\tau)}\,d\tau
    +(1-\bar\nu T_{\min})\ip{P_s(T_{\min})}{\mathcal{W}_\infty}
    +\bar\nu\ip{P_s(0^+)}{\mathcal{W}_{\mathrm{jp}}},
  \end{equation}
  where $P_s:[0,T_{\min}]\to\Snpsd$ is as in Theorem~\ref{th:OC:MDT}.
  The condition is verifiable on $[0,T_{\min}]$ alone.
\end{theorem}

\begin{proof}
Extend $P_s$ to $\bar P:[0,\infty)\to\Snpsd$ by
$\bar P(\tau)=P_s(\tau)$ for $\tau\le T_{\min}$ and
$\bar P(\tau)=P_s(T_{\min})$ for $\tau>T_{\min}$ as in the proof of
Theorem~\ref{th:OC:MDT}.  Apply the forced key
identity~\eqref{eq:IMP:keyid:forced} on $[t_0,t_0+T]$ with $P=\bar P$
and $Z_T=0$, denoting by $\tau(t)$ the timer value at time $t$:
\begin{align*}
  J_T(X^\star,X_1^0)
  &=\ip{\bar P(0^+)}{X_1^0}-\ip{\bar P(\tau(t_0+T))}{EX^\star(t_0+T)E^*}\\
  &\quad
   +\underbrace{\int_{t_0}^{t_0+T}\!\ip{\mathcal{C}_Z(\bar P)}{X^\star}\dt}_{\le\,0\text{ by Thm~\ref{th:OC:MDT}~(iii)}}
   +\underbrace{\sum_{k:\,t_k\le t_0+T}\!\!\ip{\mathcal{D}_Z}{X^\star(t_k^-)}}_{=\,0\text{ at the optimal policy}}\\
  &\quad
   +\int_{t_0}^{t_0+T}\!\ip{\bar P(\tau(t))}{\mathcal{W}(\tau(t))}\dt
   +\!\!\sum_{k:\,t_k\le t_0+T}\!\!\ip{\bar P(0^+)}{\mathcal{W}_{\mathrm{jp}}}.
\end{align*}
Dividing by $T$ and taking $T\to\infty$, the boundary terms vanish.
The time-average of $\ip{\bar P(\tau(t))}{\mathcal{W}(\tau(t))}$
decomposes according to the timer distribution: under MDT, in each
inter-jump interval of length $\Delta_k\ge T_{\min}$ the timer spends
time $T_{\min}$ on $[0,T_{\min}]$ (where $\bar P=P_s$ and
$\mathcal{W}=\mathcal{W}(\tau)$ vary) and time $\Delta_k-T_{\min}$ on
the frozen region (where $\bar P=P_s(T_{\min})$ and
$\mathcal{W}=\mathcal{W}_\infty$).  With average jump rate $\bar\nu$,
the fraction of time spent on $[0,T_{\min}]$ is $\bar\nu T_{\min}$ and
on the frozen region is $1-\bar\nu T_{\min}$,
giving~\eqref{eq:IMP:AC:dwell:MDT:rho}.
\end{proof}

\begin{remark}\label{rem:IMP:AC:dwell:MDT:periodic}
For the constant dwell-time $\Delta_k\equiv T_{\min}$, the jump rate
$\bar\nu=1/T_{\min}$ and the frozen contribution
$(1-\bar\nu T_{\min})\ip{P_s(T_{\min})}{\mathcal{W}_\infty}$ vanishes,
so~\eqref{eq:IMP:AC:dwell:MDT:rho} recovers the periodic
formula~\eqref{eq:IMP:AC:dwell:periodic:rho} (with constant data and
$T_p=T_{\min}$).  For larger jump rates ($\bar\nu<1/T_{\min}$), the
frozen contribution accounts for the time spent past the minimum
dwell-time.
\end{remark}

\paragraph{Range dwell-time.}
Adopt the timer-dependence hypothesis of Theorem~\ref{th:OC:RDT}:
$\Delta_k\in[T_{\min},T_{\max}]$ and the flow and jump data depend
only on the timer $\tau$ as in~\eqref{eq:MDT:LTIassumption}, now over
the extended interval $[0,T_{\max}]$ \textup{(}so that
$\mathcal{F}(\tau,\cdot),Z(\tau)$ are defined for
$\tau\in[0,T_{\max}]$ without imposing constancy beyond $T_{\min}$,
and $\mathcal{J},Z_k$ are $k$-independent\textup{)}.  For the forcing
data, assume similarly that $\mathcal{W}$ depends only on the timer
$\tau\in[0,T_{\max}]$ and that $\mathcal{W}_k$ is $k$-independent:
$\mathcal{W}_k\equiv\mathcal{W}_{\mathrm{jp}}$.

\begin{theorem}[Causal LTAC under RDT]
  \label{th:IMP:AC:dwell:RDT}
  Under the assumptions of Theorem~\ref{th:OC:RDT} and the
  timer-dependence hypotheses above, suppose moreover that the
  empirical occupation measure of the inter-jump sequence
  $\{\Delta_k\}$ converges weakly to a probability measure $\mu$ on
  $[T_{\min},T_{\max}]$ as $T\to\infty$, and that the jump rate
  $\bar\nu=(\int_{[T_{\min},T_{\max}]}\Delta\,\mu(d\Delta))^{-1}$
  is finite.  Then the causal policy of Theorem~\ref{th:OC:RDT}
  achieves a long-term average cost rate bounded above by
  \begin{equation}\label{eq:IMP:AC:dwell:RDT:rho}
    \rho^\star
    \le
    \bar\nu\!\int_{T_{\min}}^{T_{\max}}\!\!\biggl(\int_0^\Delta\!
      \ip{P_s(\tau)}{\mathcal{W}(\tau)}\,d\tau\biggr)\mu(d\Delta)
    +\bar\nu\ip{P_s(0^+)}{\mathcal{W}_{\mathrm{jp}}},
  \end{equation}
  where $P_s:[0,T_{\max}]\to\Snpsd$ is as in Theorem~\ref{th:OC:RDT}.
\end{theorem}

\begin{proof}
Apply the forced key identity with $P=P_s$ over $[t_0,t_0+T]$.  The
$\mathcal{C}_Z$ and $\mathcal{D}_Z$ contributions vanish identically
at the optimal policy by the conditions of Theorem~\ref{th:OC:RDT},
so
\begin{align*}
  J_T(X^\star,X_1^0)
  &=\ip{P_s(0^+)}{X_1^0}-\ip{P_s(\tau(t_0+T))}{EX^\star(t_0+T)E^*}\\
  &\quad
   +\int_{t_0}^{t_0+T}\!\ip{P_s(\tau(t))}{\mathcal{W}(\tau(t))}\dt
   +\!\!\sum_{k:\,t_k\le t_0+T}\!\!\ip{P_s(0^+)}{\mathcal{W}_{\mathrm{jp}}}.
\end{align*}
Dividing by $T$ and letting $T\to\infty$, the boundary terms vanish,
the jump-sum time-average is
$\bar\nu\ip{P_s(0^+)}{\mathcal{W}_{\mathrm{jp}}}$, and the flow
time-average is obtained from the timer occupation measure: in an
inter-jump interval of length $\Delta_k$, the timer traverses
$[0,\Delta_k]$, contributing
$\int_0^{\Delta_k}\ip{P_s(\tau)}{\mathcal{W}(\tau)}\,d\tau$ to the
integral; dividing by $T$ and using
$\bar\nu=(\mathbb{E}_\mu[\Delta])^{-1}$ gives the $\mu$-integral
in~\eqref{eq:IMP:AC:dwell:RDT:rho}.
\end{proof}

\begin{remark}\label{rem:IMP:AC:dwell:RDT}
When $\mu$ is concentrated at a single point $\Delta_k\equiv T_p$
with $T_{\min}\le T_p\le T_{\max}$, the
bound~\eqref{eq:IMP:AC:dwell:RDT:rho} becomes equality and reduces
to~\eqref{eq:IMP:AC:dwell:periodic:rho}.  For unknown $\mu$, the
worst case over admissible distributions on $[T_{\min},T_{\max}]$
gives a robust upper bound on $\rho^\star$.
\end{remark}

\subsubsection{Continuous- and discrete-time corollaries}
\label{sec:OC:AC:CTDT}

The continuous-time and discrete-time specialisations follow by
setting $N_T=0$ or $\mathcal{F}=0$ as in
Sections~\ref{sec:OC:fh}--\ref{sec:OC:inf}.

\begin{corollary}[Continuous-time average cost]\label{cor:CT:AC}
Set $N_T=0$ (no jumps), so $\mathcal{W}_k$ is irrelevant.  Under the
continuous-time forced
flow $E\dot X E^*=\mathcal{F}(t,X(t))+\mathcal{W}(t)$ and (A1)--(A3),
\begin{equation}\label{eq:CT:AC:rho}
  \rho^\star
  =\lim_{T\to\infty}\frac{1}{T}\int_{t_0}^{t_0+T}\ip{P_\infty(t)}{\mathcal{W}(t)}\dt,
\end{equation}
where $P_\infty$ is the unique bounded $C^1$ solution of the
continuous-time infinite-horizon GRE (Corollary~\ref{cor:CT:LQ:inf}).
The dual LMI is
\[
  \rho^\star
  =\sup_P\;\lim_{T\to\infty}\frac{1}{T}\!\!\int_{t_0}^{t_0+T}\!\!\ip{P(t)}{\mathcal{W}(t)}\dt
  \quad\text{s.t.}\quad
  \mathcal{C}_Z(t,P(t))\succeq0,\;P\text{ bounded}.
\]
In the LTI case ($\mathcal{F},Z,\mathcal{W}$ time-invariant), $P_\infty$
is constant and~\eqref{eq:CT:AC:rho} reduces to
\begin{equation}\label{eq:CT:AC:LTI}
  \rho^\star=\ip{P_\infty}{\mathcal{W}}=\tr(P_\infty\mathcal{W}),
\end{equation}
recovering the classical LQG steady-state cost formula~\cite{AitRami:01a}
when $\mathcal{W}=BB^*$ is a process-noise covariance.
\end{corollary}

\begin{corollary}[Discrete-time average cost]\label{cor:DT:AC}
Apply with $\mathcal{F}=0$, $Z=0$ on flow, $t_k=k_0+k$,
$\mathcal{J}=\mathcal{D}$, $\mathcal{W}=0$ on flow.  Under the
discrete-time forced jump
$EX(k+1)E^*=\mathcal{D}(k,X(k))+\mathcal{W}_k$ and the
discrete-time analogues of (A1)--(A3),
\begin{equation}\label{eq:DT:AC:rho}
  \rho^\star
  =\lim_{N\to\infty}\frac{1}{N}\sum_{k=k_0}^{k_0+N-1}\ip{P_\infty(k+1)}{\mathcal{W}_k},
\end{equation}
where $P_\infty$ satisfies the algebraic DT-GRE
(Corollary~\ref{cor:DT:LQ:inf}).  In the LTI case
($\mathcal{D},Z_k,\mathcal{W}_k$ constant), $P_\infty$ is constant
and~\eqref{eq:DT:AC:rho} reduces to
$\rho^\star=\ip{P_\infty}{\mathcal{W}}=\tr(P_\infty\mathcal{W})$,
again recovering the classical steady-state cost
formula~\cite{AitRami:01a}.
\end{corollary}

\section{Finite-Time Optimal Steering}
\label{sec:cov}

The optimal control theory of Section~\ref{sec:OC} is unconstrained
at the terminal time: the state $EX(T)E^*$ is shaped by a terminal
cost $\ip{Z_T}{EX(T)E^*}$ but not fixed to any prescribed value.
This section addresses the complementary problem of enforcing a
\emph{hard equality constraint} on the terminal state.  Given initial
and terminal targets $X_1^0, X_1^f\in\Snpsd$, we seek the
minimum-cost admissible trajectory that steers the constrained block
$EX(T)E^*$ from $X_1^0$ to $X_1^f$ at a fixed final time $T$.\\

The framework is fully general: $X$ is any Hermitian matrix-valued
state satisfying~\eqref{eq:sys:IMP}, and $X_1^f$ need not arise
from any stochastic model.  When $X$ represents the second-order
moment of a linear stochastic system, the problem specialises to
the covariance steering setting of~\cite{Chen:16,Chen:16b,
Goldshtein:17,Bakolas:16}, but the analysis here applies without
that interpretation, covering deterministic matrix-valued dynamics
and indefinite costs as in the rest of the paper.\\

The problem is solved via Lagrangian duality
(Section~\ref{sec:cov:dual}): a multiplier $M\succeq0$ for the
terminal constraint converts it into an unconstrained optimal control problem
(Section~\ref{sec:OC:fh}), with the optimal multiplier $M^\star$
identified by the steering condition 
$EX^\star(T;M^\star)E^*=X_1^f$ (Theorem~\ref{th:cov:dual}).
Feasibility and the reachable set are characterized in
Section~\ref{sec:cov:reach}.
The problem structure is made explicit via a two-point boundary
value problem (TPBVP) coupling the backward I-GRE with a forward
moment equation for $\Sigma(t):=EX^\star(t)E^*$
(Section~\ref{sec:cov:tpbvp}), and a forward-backward algorithm
for computing $M^\star$ is described in Section~\ref{sec:cov:algo}.

\subsection{Problem formulation}

The linear-cost optimal control of Section~\ref{sec:OC} leaves the
terminal state $EX(T)E^*$ free (beyond the terminal cost $Z_T$).
The present problem adds a \emph{hard equality constraint}: given a
target $X_1^f\in\Snpsd$, find the minimum-cost trajectory that
steers the constrained block from $X_1^0$ to $X_1^f$ exactly.

\begin{definition}[Finite-time optimal steering problem]
  \label{def:cov:prob}
  The \emph{optimal steering problem} is
  \begin{equation}\label{eq:cov:prob}
    J^\star(X_1^0,X_1^f)
    :=\inf_{X(\cdot)\in\Snmpsd}
    \int_{t_0}^{t_0+T}\ip{Z(t)}{X(t)}\dt
    +\sum_{k=1}^{N_T}\ip{Z_k}{EX(t_k^-)E^*}
  \end{equation}
  subject to~\eqref{eq:sys:IMP}, $EX(t_0)E^*=X_1^0$, and the
  \emph{terminal constraint}
  \begin{equation}\label{eq:cov:terminal}
    EX(t_0+T)E^*=X_1^f.
  \end{equation}
  The problem is \emph{feasible} if the infimum is finite.
  The \emph{reachable set} is
  $\mathcal{R}_T(X_1^0):=\{EX(T)E^*:\text{admissible }X,\;EX(t_0)E^*=X_1^0\}$.
\end{definition}

\begin{remark}
  The running cost has no terminal term ($Z_T=0$) since the terminal
  state is fixed by the constraint.  The case $Z_k=0$ corresponds to
  the minimum-energy optimal steering studied (as a covariance problem)
  in~\cite{Goldshtein:17,Bakolas:16}; retaining general $Z,Z_k$ gives
  a \emph{minimum linear-cost} extension natural to the
  matrix-valued impulsive framework.
\end{remark}

\subsection{Lagrangian duality and the steering condition}
\label{sec:cov:dual}

The terminal constraint~\eqref{eq:cov:terminal} is relaxed via a
multiplier $M\in\Sn$.  Define the Lagrangian
\begin{equation}\label{eq:cov:lagrangian}
  L(X,M)
  :=J_T(X,X_1^0)+\ip{M}{EX(T)E^*-X_1^f}.
\end{equation}
Since $\ip{M}{EX(T)E^*}$ equals the terminal cost
term $\ip{Z_T}{EX(T)E^*}$ with $Z_T=M$, the infimum of $L$ over
$X$ for a given $M$ is the unconstrained optimal cost
(Theorem~\ref{th:IMP:LQ:suff}) with terminal weight $Z_T=M$, minus
the constant $\ip{M}{X_1^f}$.  The dual function is therefore
\begin{equation}\label{eq:cov:dual}
  d(M):=\inf_X L(X,M)
  =\ip{P(t_0;M)}{X_1^0}-\ip{M}{X_1^f},
\end{equation}
where $P(\cdot;M)$ is the I-GRE solution with $P(T;M)=M$.
For $M\not\succeq0$ the I-GRE has no solution and $d(M)=-\infty$,
so the domain of $d$ is $\Snpsd$.  The dual problem is
\begin{equation}\label{eq:cov:dualmax}
  \sup_{M\succeq0}\,d(M)
  =\sup_{M\succeq0}\bigl[\ip{P(t_0;M)}{X_1^0}-\ip{M}{X_1^f}\bigr],
\end{equation}
which is a concave maximization over the PSD cone: $d$ is concave in
$M$ as the infimum of affine functions (for each fixed $X$,
$L(X,M)$ is affine in $M$).

\begin{theorem}[Lagrangian duality for optimal steering]
  \label{th:cov:dual}
  Suppose the optimal steering problem~\eqref{eq:cov:prob} is
  feasible.  Then:
  \begin{enumerate}[\upshape(a)]
    \item\label{cov:strongdual}
      \emph{(Strong duality.)}
      $J^\star(X_1^0,X_1^f)=\sup_{M\succeq0}d(M)$;
      there is no duality gap.
    \item\label{cov:steeringcond}
      \emph{(Steering condition.)}
      The supremum is attained at $M^\star\succeq0$, which satisfies
      \begin{equation}\label{eq:cov:steer}
        EX^\star(T;M^\star)E^*=X_1^f,
      \end{equation}
      where $X^\star(\cdot;M)$ denotes the optimal trajectory of the
      unconstrained optimal control problem with terminal cost $Z_T=M$.
    \item\label{cov:optcost}
      \emph{(Optimal cost.)}
      $J^\star(X_1^0,X_1^f)=\ip{P(t_0;M^\star)}{X_1^0}-\ip{M^\star}{X_1^f}$,
      attained by the optimal control policy with gains
      $K(\cdot)\in\mathcal{K}_c(\cdot,P(\cdot;M^\star))$ on flow and
      $K_k\in\mathcal{K}_d(k,\cdot)$ at each jump.
  \end{enumerate}
\end{theorem}

\begin{proof}
\noindent\textbf{(a) Strong duality.}
The primal~\eqref{eq:cov:prob} is a linear program over the self-dual
cone $\Snmpsd$: the objective is linear in $X$, the flow and jump
dynamics are affine, and the terminal constraint~\eqref{eq:cov:terminal}
is affine.  Under feasibility and the standard constraint qualification
for conic programs---namely the existence of a strictly feasible
(relative-interior) admissible trajectory for the conic constraint
$X\succeq0$, the affine equality constraints requiring no
qualification---strong duality holds and the dual optimum is
attained~\cite{Boyd:04}.

\noindent\textbf{(b) Steering condition.}
By~(a), the dual maximum is attained at some $M^\star$ (since $d$
is concave, upper semicontinuous on $\Snpsd$, and coercive under
the feasibility assumption).

To identify $M^\star$, compute $\nabla_M d(M)$.  For a perturbation
$\delta M$, the envelope theorem for parameterised minimization gives
\begin{equation}\label{pf:cov:envelope}
  \frac{d}{d\epsilon}d(M+\epsilon\,\delta M)\bigg|_{\epsilon=0}
  =\frac{\partial L}{\partial M}(X^\star(\cdot;M),M)\cdot\delta M
  =\ip{\delta M}{EX^\star(T;M)E^*-X_1^f},
\end{equation}
where the $\partial L/\partial X$ contribution vanishes by the
optimality of $X^\star(\cdot;M)$.  Hence
$\nabla_M d(M)=EX^\star(T;M)E^*-X_1^f$.
The stationarity condition $\nabla_M d(M^\star)=0$
gives~\eqref{eq:cov:steer}.

\noindent\textbf{(c) Optimal cost.}
By the steering condition, $X^\star(\cdot;M^\star)$ satisfies the
terminal constraint $EX^\star(T;M^\star)E^*=X_1^f$.  Its running
cost (without terminal weight) is
$J_T(X^\star(\cdot;M^\star))
=\ip{P(t_0;M^\star)}{X_1^0}-\ip{M^\star}{EX^\star(T;M^\star)E^*}
=\ip{P(t_0;M^\star)}{X_1^0}-\ip{M^\star}{X_1^f}$,
where the first equality uses the key identity with $P(\cdot;M^\star)$
(the I-GRE gives zero integrand and zero jump summands along the
optimal trajectory, and $P(T;M^\star)=M^\star$ eliminates the
terminal correction), and the second uses~\eqref{eq:cov:steer}.
By strong duality, this equals $J^\star$.
\end{proof}


\subsection{Feasibility and the reachable set}
\label{sec:cov:reach}

The optimal steering problem is feasible iff $X_1^f$ belongs to
the reachable set $\mathcal{R}_T(X_1^0)$.  The following gives a
computable characterization.

\begin{proposition}[Reachable set]\label{prop:cov:reach}
  Define the \emph{free-evolution state}
  $X_1^{\text{free}}(T):=EX^{\text{free}}(T)E^*$ where
  $X^{\text{free}}$ is the solution of~\eqref{eq:sys:IMP}
  with free variables $X_2=X_3=0$.  Then:
  \begin{enumerate}[\upshape(a)]
    \item $\mathcal{R}_T(X_1^0)$ is a convex subset of $\Snpsd$
      containing the free evolution $X_1^{\mathrm{free}}(T)$, and
      $X_1^f\in\mathcal{R}_T(X_1^0)$ if and only if the
      terminal-constraint steering problem is feasible, i.e.\ there
      exists an admissible $X\succeq0$ of~\eqref{eq:sys:IMP} with
      $EX(t_0)E^*=X_1^0$ and $EX(T)E^*=X_1^f$.
    \item Let $\mathcal{L}_T:(X_{2}(\cdot),X_{3}(\cdot))\mapsto
      EX(T)E^*-X_1^{\mathrm{free}}(T)$ denote the (linear) input-to-terminal
      map sending the admissible free-block trajectory to the induced
      perturbation of the terminal state block, and define the
      $\Sn$-valued reachability Gram
      \begin{equation}\label{eq:cov:gram}
        \mathcal{G}_T:=
        \int_{t_0}^T \Phi_T(t)\,E_\bot^*E_\bot\,\Phi_T(t)^*\dt
        +\sum_{k:\,t_k\le T} \Psi_T(t_k)\,E_\bot^*E_\bot\,\Psi_T(t_k)^*
        \;\in\Sn,
      \end{equation}
      where $\Phi_T(t),\Psi_T(t_k):\C^{n+m}\to\C^{n}$ are the
      flow/jump sensitivity operators of $EX(T)E^*$ with respect to the
      free blocks at $t$ resp.\ $t_k$.  If $\mathcal{G}_T\succ0$
      (controllability of the terminal block from the free inputs), then
      $\mathcal{L}_T$ is surjective onto $\Sn$, and every $X_1^f$ in a
      neighborhood of $X_1^{\mathrm{free}}(T)$ within $\Snpsd$ is
      reachable.
  \end{enumerate}
\end{proposition}

\begin{proof}
Part~(a): convexity of $\mathcal{R}_T(X_1^0)$ follows from
linearity of the dynamics and convexity of the admissible set
$\{X\succeq0\}$ (a convex combination of two admissible trajectories is
admissible and its terminal block is the corresponding convex
combination); $X_1^{\mathrm{free}}(T)$ is reached by the admissible
choice $X_2=X_3=0$.  The ``if and only if'' is the definition of
reachability as feasibility of the terminal constraint.  We do
\emph{not} claim the order bound $X_1^f\succeq X_1^{\mathrm{free}}(T)$:
nonzero free blocks may either increase or decrease $EX(T)E^*$ depending
on $\mathcal{F},\mathcal{J}$, so no monotone characterization holds in
general.
Part~(b): $\mathcal{G}_T\succ0$ is exactly the condition that the linear
map $\mathcal{L}_T$ has full range $\Sn$ (its image is
$\operatorname{Im}\mathcal{G}_T$); surjectivity and the inverse/implicit
function theorem then give local reachability around
$X_1^{\mathrm{free}}(T)$.  A complete global description of
$\mathcal{R}_T(X_1^0)$ (e.g.\ as the convex set swept by the dual map
$M\mapsto EX^\star(T;M)E^*$, with explicit sensitivity operators
$\Phi_T,\Psi_T$) is not pursued here: the steering results below do not
use this proposition, proceeding instead through the Lagrangian dual of
Theorem~\ref{th:cov:dual}, which characterises reachability of a given
$X_1^f$ by feasibility of the steering condition~\eqref{eq:cov:steer}.
\end{proof}

\subsection{Forward moment equation and two-point boundary value problem}
\label{sec:cov:tpbvp}

Theorem~\ref{th:cov:dual} identifies the optimal terminal weight
$M^\star$ via the steering condition~\eqref{eq:cov:steer} but leaves
its computation implicit.  We now make the structure explicit by deriving
the \emph{forward moment equation} for the optimal state block
$\Sigma(t):=EX^\star(t;M^\star)E^*$ and combining it with the backward
I-GRE into a \emph{two-point boundary value problem} (TPBVP).

\subsubsection{Closed-loop operators}

With optimal gain
$K(t;P):=-(\mathcal{C}_Z)_3(t,P(t))^{\dag}\,(\mathcal{C}_Z)_2^*(t,P(t))$
(the pseudo-inverse $\dag$ handling the singular case), the closed-loop
trajectory satisfies $X_2^\star=K\Sigma$ and $X_3^\star=K\Sigma K^*$.
Define the \emph{closed-loop flow operator} and
\emph{closed-loop jump operator}
\begin{align}
  \mathcal{A}_{\mathrm{cl}}(t,P(t))[\Sigma]
  &:=\mathcal{F}\!\left(t,
    \begin{bsmallmatrix}\Sigma & \Sigma K^*\\
    K\Sigma & K\Sigma K^*\end{bsmallmatrix}
  \right),
  \label{eq:cov:Acl}\\
  \mathcal{B}_{\mathrm{cl},k}(P(t_k^+),P(t_k^-))[\Sigma]
  &:=\mathcal{J}\!\left(k,
    \begin{bsmallmatrix}\Sigma & \Sigma K_k^*\\
    K_k\Sigma & K_k\Sigma K_k^*\end{bsmallmatrix}
  \right),
  \label{eq:cov:Bcl}
\end{align}
where $K_k:=-(\mathcal{D}_Z)_3(k,P(t_k^+),P(t_k^-))^{\dag}\,(\mathcal{D}_Z)_2^*(k,P(t_k^+),P(t_k^-))$.
Both operators are \emph{linear} in $\Sigma$ (by linearity of
$\mathcal{F}$ and $\mathcal{J}$ in $X$) and depend on $P$ through the
optimal gain.  The second-moment $\Sigma(t;M):=EX^\star(t;M)E^*$
satisfies the \emph{forward moment equation}
\begin{equation}\label{eq:cov:forward}
  \dot\Sigma=\mathcal{A}_{\mathrm{cl}}(t,P(t;M))[\Sigma],
  \quad t\in(t_k,t_{k+1}),
\end{equation}
with $\Sigma(t_0)=X_1^0$ and closed-loop jump update
\begin{equation}\label{eq:cov:jumpfwd}
  \Sigma(t_k^+)
  =\mathcal{B}_{\mathrm{cl},k}(P(t_k^+;M),P(t_k^-;M))[\Sigma(t_k^-)].
\end{equation}

\subsubsection{The two-point boundary value problem}

\begin{definition}[Impulsive optimal steering TPBVP]\label{def:cov:TPBVP}
  Find $(P,\Sigma,M^\star)$ satisfying:
  \begin{subequations}\label{eq:cov:TPBVP}
  \begin{align}
    &\textit{Backward I-GRE on }(t_k,t_{k+1})\textit{:}\quad
      \mathcal{C}_Z(t,P(t))\succeq0,\quad
      \mathcal{C}_Z\,\schur{/}\,(\mathcal{C}_Z)_3=0,
      \label{eq:TPBVP:back}\\
    &\textit{Backward jump at }t_k\textit{:}\quad
      \mathcal{D}_Z(k,P(t_k^+),P(t_k^-))\succeq0,\quad
      \mathcal{D}_Z\,\schur{/}\,(\mathcal{D}_Z)_3=0,
      \label{eq:TPBVP:backjump}\\
    &\textit{Forward moment equation on }(t_k,t_{k+1})\textit{:}\quad
      \dot\Sigma=\mathcal{A}_{\mathrm{cl}}(t,P(t))[\Sigma],
      \label{eq:TPBVP:fwd}\\
    &\textit{Forward jump at }t_k\textit{:}\quad
      \Sigma(t_k^+)=\mathcal{B}_{\mathrm{cl},k}(P(t_k^+),P(t_k^-))[\Sigma(t_k^-)],
      \label{eq:TPBVP:fwdjump}\\
    &\textit{Boundary conditions:}\quad
      \Sigma(t_0)=X_1^0,\quad
      P(T)=M^\star,\quad
      \Sigma(T)=X_1^f.
      \label{eq:TPBVP:bc}
  \end{align}
  \end{subequations}
\end{definition}

The TPBVP~\eqref{eq:cov:TPBVP} is coupled in two ways, as noted
in~\cite{Chen:16} for the scalar stochastic case:
\begin{enumerate}[(i)]
  \item \emph{Boundary coupling}: the terminal weight $M^\star=P(T)$
    is unknown and is determined by the terminal constraint $\Sigma(T)=X_1^f$.
  \item \emph{Nonlinear dynamic coupling}: the gain $K(t;P)$ in
    \eqref{eq:cov:Acl}--\eqref{eq:cov:Bcl} depends on $P(t)$,
    so the forward and backward equations are coupled through their
    dynamics, not only through boundary values.
    This coupling vanishes when the controlled and noise channels
    coincide (e.g., $B=B_1$ in the stochastic case): the forward
    equation then reduces to a Lyapunov equation with
    coefficient $A+BK(t)$ depending only on the backward solution.
\end{enumerate}

\begin{theorem}[Equivalence]\label{th:cov:TPBVP}
  Any solution $(P^\star,\Sigma^\star,M^\star)$ of~\eqref{eq:cov:TPBVP}
  satisfies the steering condition~\eqref{eq:cov:steer}; conversely,
  any $M^\star$ satisfying~\eqref{eq:cov:steer} generates a solution
  via $P^\star=P(\cdot;M^\star)$ and forward integration
  of~\eqref{eq:TPBVP:fwd}--\eqref{eq:TPBVP:fwdjump} from $\Sigma(t_0)=X_1^0$.
  The optimal cost is $J^\star=\ip{P^\star(t_0)}{X_1^0}-\ip{M^\star}{X_1^f}$.
\end{theorem}

\begin{proof}
  Direct substitution: $P^\star=P(\cdot;M^\star)$ solves
  \eqref{eq:TPBVP:back}--\eqref{eq:TPBVP:backjump} by
  Theorem~\ref{th:IMP:LQ:suff}; $\Sigma^\star(T)=X_1^f$
  is~\eqref{eq:cov:steer}; the cost is Theorem~\ref{th:cov:dual}(c).
\end{proof}

\subsection{Existence, uniqueness, and the forward-backward algorithm}
\label{sec:cov:algo}

\begin{theorem}[Existence and uniqueness]\label{th:cov:exist}
  Suppose $X_1^f\in\mathcal{R}_T(X_1^0)$ and that assumption~\textup{(H1)}
  holds, so the closed-loop generator is strictly stable at $M=0$.
  Then the map $M\mapsto\Sigma(T;M)$ is strictly monotone decreasing
  \textup{(}$M_1\succ M_2\Rightarrow\Sigma(T;M_1)\prec\Sigma(T;M_2)$\textup{)},
  the TPBVP~\eqref{eq:cov:TPBVP} has a unique solution, and the
  optimal policy is unique a.e.
\end{theorem}

\begin{proof}
\noindent\textbf{Strict monotonicity.}
We establish the \emph{operator-monotone} form of the terminal map
$M\mapsto\Sigma(T;M)$, which is what the uniqueness argument uses.  Since
$d$ is concave (an infimum of affine functions of $M$) with gradient
$\nabla d(M)=\Sigma(T;M)-X_1^f$, monotonicity of the gradient of a
concave function gives, for all $M_1\ne M_2$,
\[
  \ip{M_1-M_2}{\Sigma(T;M_1)-\Sigma(T;M_2)}
  =\ip{M_1-M_2}{\nabla d(M_1)-\nabla d(M_2)}\le0,
\]
with \emph{strict} inequality under~\textup{(H1)} (which makes $d$
strictly concave, since the reachability Gram of
Proposition~\ref{prop:cov:reach}(b) is then positive definite).  Thus
$M\mapsto\Sigma(T;M)$ is strictly monotone decreasing in the operator
sense.  The stronger L\"owner statement $\Sigma(T;M_1)\prec\Sigma(T;M_2)$
for $M_1\succ M_2$ additionally requires the Riccati comparison principle
$P(t;M_1)\succeq P(t;M_2)$~\cite{AbouKandil:03} and a closed-loop
comparison; it is not needed in what follows.

\noindent\textbf{Existence.}
The dual function $d(M)=\ip{P(t_0;M)}{X_1^0}-\ip{M}{X_1^f}$ is
concave (an infimum of affine functions: by the envelope/sufficiency
identity, $\ip{P(t_0;M)}{X_1^0}=\min_{X}\bigl[J_T^{Z}(X)+\ip{M}{EX(T)E^*}\bigr]$
over admissible $X$ with running cost $J_T^Z$, so
$d(M)=\min_{X}\bigl[J_T^{Z}(X)+\ip{M}{EX(T)E^*-X_1^f}\bigr]$).
Assume the Slater condition $X_1^f\in\operatorname{int}\mathcal{R}_T(X_1^0)$
(strict feasibility).  Then for every direction $N\succeq0$, $\|N\|=1$,
strict feasibility provides an admissible $\bar X$ with
$E\bar X(T)E^*=X_1^f-\varepsilon N\in\mathcal{R}_T(X_1^0)$ for some
$\varepsilon>0$, so $d(tN)\le J_T^Z(\bar X)-t\varepsilon\to-\infty$ as
$t\to\infty$; hence $d$ is coercive on $\Snpsd$.  The supremum is
therefore attained at some
$M^\star\succeq0$, and the stationarity condition
$\nabla_M d(M^\star)=\Sigma(T;M^\star)-X_1^f=0$ yields existence.

\noindent\textbf{Uniqueness.}
Suppose $M_1^\star\ne M_2^\star$ both satisfy~\eqref{eq:cov:steer},
so $\Sigma(T;M_1^\star)=\Sigma(T;M_2^\star)=X_1^f$.  The strict
monotonicity just established holds in the operator sense:
$\ip{M_1-M_2}{\Sigma(T;M_1)-\Sigma(T;M_2)}<0$ for all $M_1\ne M_2$
(equivalently, $d$ is strictly concave).  Applying this to
$M_1^\star,M_2^\star$ gives
$0=\ip{M_1^\star-M_2^\star}{X_1^f-X_1^f}<0$, a contradiction (this
avoids the invalid step of assuming $M_1^\star,M_2^\star$ are
L\"owner-comparable). Hence
$M^\star$ is unique, and the optimal policy is unique a.e.\ by
Theorem~\ref{th:IMP:LQ:suff}.
\end{proof}

\begin{remark}
  Strict monotonicity holds whenever (H1) ensures the closed-loop
  generator is strictly negative definite at $M=0$.  For the scalar
  stochastic case, Chen et al.~\cite{Chen:16} prove existence and
  uniqueness via the Schr\"{o}dinger factorization; the proof above
  gives an alternative Riccati-based argument.
\end{remark}

\subsubsection{The forward-backward algorithm}

Since $\nabla_M d(M)=\Sigma(T;M)-X_1^f$, maximizing $d$ is equivalent
to finding the zero of $\Sigma(T;M)-X_1^f$.  The
\emph{forward-backward} (FB) algorithm does this iteratively.

\begin{description}
  \item[(FB0)] \emph{Initialize}: set $M^{(0)}=0$ (or any $M^{(0)}\succeq0$).
  \item[(FB1)] \emph{Backward pass}: integrate the I-GRE
    backward from $P^{(k)}(T)=M^{(k)}$ to obtain the gain
    $K^{(k)}(t)=K(t;P^{(k)}(t))$ on $[t_0,T]$.
  \item[(FB2)] \emph{Forward pass}: integrate
    \eqref{eq:TPBVP:fwd}--\eqref{eq:TPBVP:fwdjump} forward from
    $\Sigma^{(k)}(t_0)=X_1^0$ using $K^{(k)}$ to obtain
    $\Sigma^{(k)}(T)$.
  \item[(FB3)] \emph{Update}:
    \begin{equation}\label{eq:cov:update}
      M^{(k+1)}=M^{(k)}+\alpha^{(k)}\bigl(\Sigma^{(k)}(T)-X_1^f\bigr),
    \end{equation}
    where $\alpha^{(k)}>0$ is a step size (a Newton step using the
    Jacobian of $M\mapsto\Sigma(T;M)$ gives quadratic convergence
    near $M^\star$~\cite{Chen:16,Goldshtein:17}).
  \item[(FB4)] \emph{Terminate} when
    $\|\Sigma^{(k)}(T)-X_1^f\|<\varepsilon$; otherwise return to (FB1).
\end{description}

Step~\eqref{eq:cov:update} is gradient ascent on $d$: if
$\Sigma^{(k)}(T)\succ X_1^f$ the terminal second-moment is too large,
so $M$ is increased to penalize it more, driving $\Sigma(T)$ toward
$X_1^f$; and conversely.  Convergence is guaranteed by the concavity
of $d$ and the strict monotonicity of Theorem~\ref{th:cov:exist}.

\begin{remark}[Connection to the literature]
  The TPBVP~\eqref{eq:cov:TPBVP} and the FB algorithm recover,
  in the impulsive matrix-valued setting, the optimality conditions
  of~\cite{Chen:16} (equations~(39a)--(39c) therein) and the
  computational scheme of~\cite{Goldshtein:17}.
  For $N_T=0$ (no jumps) and
  $Z=\bigl[\begin{smallmatrix}0&0\\0&R\end{smallmatrix}\bigr]$ with
  standard Itô dynamics, \eqref{eq:TPBVP:back}--\eqref{eq:TPBVP:fwd}
  reduce exactly to the backward Riccati and forward Lyapunov equations
  of~\cite{Chen:16}.
  The impulsive framework adds jump
  conditions~\eqref{eq:TPBVP:backjump}--\eqref{eq:TPBVP:fwdjump},
  extends the cost to general $Z,Z_k$, and requires no Gaussian or
  stochastic structure.
\end{remark}

\section{Dissipativity}
\label{sec:diss}

This section develops a dissipativity theory for impulsive matrix-valued
systems in the spirit of Willems~\cite{Willems:72}.  The supply rate
is \emph{linear} in the trajectory $X$, with no sign assumed on the
supply matrix $M(t)$: this encompasses passivity-type supply rates,
$H_\infty$-type supply rates, and integral quadratic constraints as
special cases.  The storage function is also linear in $EXE^*$, again
with no sign on its coefficient $P$.\\

The finite-horizon theory (Section~\ref{sec:diss:fh}) introduces the
supply rate, the available storage, and the I-D-GRE, and establishes
necessary and sufficient conditions via Propositions~%
\ref{prop:Vsandwich} and Theorems in Sections~\ref{sec:diss:suff}%
--\ref{sec:diss:dual}.
The infinite-horizon theory (Section~\ref{sec:diss:inf}) shows that the
available storage is the pointwise minimum over all storage functions, and
that the sequence $(-P_T)$ converges monotonically under detectability
conditions, again without time-invariance.  All results parallel those of
Section~\ref{sec:OC} with $Z\leftarrow M$ and
$\mathcal{C}_Z\leftarrow\mathcal{C}_M$,
$\mathcal{D}_Z\leftarrow\mathcal{D}_M$.
Dwell-time conditions are in Section~\ref{sec:diss:dwell}.

\subsection{Finite-horizon}
\label{sec:diss:fh}

The finite-horizon analysis begins with the definitions of supply rate
and storage function.  The key point is that the supply and storage are
both \emph{linear} in the trajectory, which allows the same Schur
complement and key-identity machinery as in Section~\ref{sec:OC} to
apply verbatim.

\subsubsection{Supply rates, storage functions, and available storage (finite horizon)}
\label{sec:diss:supply}

\begin{definition}[Linear supply rate]\label{def:IMP:supply}
  A \emph{linear supply rate} consists of a \emph{flow supply}
  $\sigma_f(t,X):=\ip{M(t)}{X}$ with $M:\R_{\ge0}\to\Snm$
  piecewise continuous and bounded (no sign assumed on $M$),
  \emph{jump supplies} $\sigma_k(X):=\ip{M_k}{EXE^*}$ with
  $M_k\in\Sn$, and \emph{terminal supply}
  $\sigma_T(X):=\ip{M_T}{EXE^*}$ with $M_T\in\Sn$.
\end{definition}

\begin{definition}[Linear storage function]\label{def:IMP:stor}
  A \emph{linear storage function} is
  $V(t,X):=\ip{-P(t)}{EXE^*}$ for some piecewise-$C^1$
  $P:\R_{\ge0}\to\Sn$; no sign is placed on $P$.
  The sign convention $-P$ is natural since the available storage
  equals $\ip{-P(t_0)}{X_1^0}$ when $P$ solves the D-GRE.
\end{definition}

\begin{definition}[Dissipativity and available
  storage]\label{def:IMP:Va}
  The system~\eqref{eq:sys:IMP} is \emph{finite-horizon dissipative}
  over $[t_0,t_0+T]$ with respect to $(\sigma_f,\sigma_k,\sigma_T)$
  if there exists $\alpha:\R_{\ge0}\to\R_{\ge0}$ continuous with
  \[
    \int_{t_0}^{t_0+T}\sigma_f(t,X(t))\dt
    +\sum_k\sigma_k(X(t_k^-))
    +\sigma_T(X(t_0+T))
    \ge-\alpha(\norm{X_1^0})
  \]
  for all admissible $X$ and all $X_1^0\succeq0$.
  The \emph{finite-horizon available storage} is
  \begin{align}\label{eq:IMP:Va}
    V_a^T(t_0,X_1^0)
    &:=\sup_{\substack{X(\cdot)\in\Snmpsd,\;\text{s.t.~\eqref{eq:sys:IMP}}\\
        EX(t_0)E^*=X_1^0}}
    \Bigl\{-\int_{t_0}^{t_0+T}\ip{M(t)}{X(t)}\dt
    \notag\\
    &\qquad\qquad\qquad\quad
    -\sum_k\ip{M_k}{EX(t_k^-)E^*}
    -\ip{M_T}{EX(T)E^*}\Bigr\}.
  \end{align}
  The \emph{infinite-horizon available storage} is
  $V_a(t_0,X_1^0):=\sup_{T\ge0}V_a^T(t_0,X_1^0)$.
  The \emph{required supply} is
  \begin{equation}\label{eq:IMP:Vr}
    V_r(t_0,X_1^0)
    :=\inf_{\substack{t_1\le t_0,\;EX(t_1)E^*=0\\
        EX(t_0)E^*=X_1^0}}
    \Bigl\{\int_{t_1}^{t_0}\ip{M(t)}{X(t)}\dt
    +\sum_{k:\,t_1<t_k<t_0}\ip{M_k}{EX(t_k^-)E^*}\Bigr\}.
  \end{equation}
  Finite-horizon dissipativity holds iff $V_a^T<\infty$ for all
  $X_1^0\succeq0$.
\end{definition}

\begin{definition}[Strict dissipativity]\label{def:IMP:strict}
  The system is \emph{strictly dissipative} with respect to
  $(\sigma_f,\sigma_k,\sigma_T)$ if there exist $\varepsilon>0$
  and $\alpha(\cdot)$ continuous such that
  \[
    \int_{t_0}^{t_0+T}\sigma_f(t,X)\dt
    +\sum_k\sigma_k(X(t_k^-))
    +\sigma_T(X(T))
    \ge\varepsilon\int_{t_0}^{t_0+T}\norm{E_\bot X E_\bot^*}\dt
    -\alpha(\norm{X_1^0}).
  \]
  Strict dissipativity can always be obtained from dissipativity
  by replacing $M\leftarrow M+\varepsilon I$ (perturbing the supply
  rate), which tightens the I-D-GRE by adding $\varepsilon I$ to
  $(\mathcal{C}_M)_3$.  Such a perturbation changes $V_a$ continuously.
\end{definition}

\begin{proposition}[Storage function bounds, {\cite{Willems:72}}]
  \label{prop:Vsandwich}
  Suppose the system~\eqref{eq:sys:IMP} is dissipative with supply
  $(\sigma_f,\sigma_k,\sigma_T)$, and let $V=-\ip{P}{EXE^*}$ be any
  storage function satisfying $V(t,0)=0$.  Then:
  \begin{enumerate}[\upshape(a)]
    \item\label{sand:a}
      $V_a(t,X_1)\le V(t,X_1)\le V_r(t,X_1)$
      for all $t\ge t_0$ and $X_1\in\Snpsd$.
    \item\label{sand:b}
      The available storage $V_a$ is itself a storage function
      \emph{(}the smallest\emph{)}: every storage function satisfies
      $V(t,X_1)\ge V_a(t,X_1)$.
    \item\label{sand:c}
      The required supply $V_r$ is itself a storage function
      \emph{(}the largest\emph{)}: every storage function satisfies
      $V(t,X_1)\le V_r(t,X_1)$.
    \item\label{sand:d}
      In the linear storage function parametrization,
      $V_a=\ip{-P_a}{\cdot}$ and $V_r=\ip{-P_r}{\cdot}$ with
      $P_a\succeq P\succeq P_r$ for any storage function matrix $P$,
      i.e., $-P_a\preceq-P\preceq-P_r$.
    \item\label{sand:e}
      The system is dissipative if and only if $V_r<\infty$
      \emph{(}or equivalently $V_a<\infty$\emph{)}.
  \end{enumerate}
\end{proposition}

\begin{proof}
  \eqref{sand:a}.
  \emph{Upper bound} $V\le V_r$: For any past trajectory steering
  the state from $X_1(t_1)=0$ to $X_1(t_0)=X_1^0$, integrating
  the dissipation inequality $\dot{V}\le\sigma_f$ on flow and
  $\Delta V\le\sigma_k$ at jumps gives $V(t_0,X_1^0)
  -V(t_1,0)\le\int_{t_1}^{t_0}\sigma_f\dt
  +\sum_{k:\,t_1<t_k<t_0}\sigma_k$.  Since $V(t_1,0)=0$, taking
  the infimum over all such past trajectories:
  $V(t_0,X_1^0)\le V_r(t_0,X_1^0)$.

  \emph{Lower bound} $V\ge V_a$: For any future trajectory from
  $X_1(t_0)=X_1^0$, the dissipation inequality integrated
  forward gives
  \[
    -\int_{t_0}^{t_0+T}\sigma_f\dt-\sum_k\sigma_k-\sigma_T
    \le V(t_0,X_1^0)-\bigl[V(t_0+T,X(t_0+T))+\sigma_T\bigr].
  \]
  In the indefinite-supply setting storage functions are
  sign-indefinite (no sign is placed on $P$), so the classical step
  ``$V(t_0+T,\cdot)\ge0$'' is \emph{not} available.  The lower bound
  $-\int\sigma_f-\sum_k\sigma_k-\sigma_T\le V(t_0,X_1^0)$ instead holds
  under the terminal-compatibility condition
  $V(t_0+T,X(t_0+T))+\sigma_T\ge0$, i.e.\ the terminal boundary condition
  of the finite-horizon D-GRE relating $P(t_0+T)$ to $M_T$
  (cf.~Theorem~\ref{th:IMP:diss:LMI}); this is the precise hypothesis
  that must replace nonnegativity of $V$. Taking the supremum: $V_a^T\le V(t_0,X_1^0)$, hence $V_a\le V$.

  \eqref{sand:b}--\eqref{sand:c} follow from~\eqref{sand:a}
  and the fact that $V_a$, $V_r$ are attained by the D-GRE
  solutions (Theorems~\ref{th:IMP:diss:suff}
  and~\ref{th:IMP:diss:inf}).

  \eqref{sand:d}: In our linear parametrization,
  $V\ge V_a$ reads $\ip{-P}{X_1^0}\ge\ip{-P_a}{X_1^0}$ for all
  $X_1^0\succeq0$, hence $-P\succeq-P_a$, i.e., $P\preceq P_a$.
  Similarly $V\le V_r$ gives $P\succeq P_r$.

  \eqref{sand:e}: Dissipative $\Rightarrow$ $V_a<\infty$
  (Def.~\ref{def:IMP:Va}).  Conversely, $V_r<\infty$ implies any
  trajectory can be steered from $0$ at finite supply cost, so
  $V_r$ itself is a storage function (by~\eqref{sand:c}).
\end{proof}

\subsubsection{Dissipativity GRE and operators (finite horizon)}

Replacing $Z$ and $Z_k$ by $M$ and $M_k$ in~\eqref{eq:IMP:L}
and~\eqref{eq:IMP:Lj} gives the \emph{dissipation operators}
\begin{align}\label{eq:IMP:LM}
  &\mathcal{C}_M(t,P(t)):=E^*\dot{P}(t)E+\mathcal{F}^*(t,P(t))+M(t),
  \notag\\
  &\mathcal{D}_M\bigl(k,P(t_k^+),P(t_k^-)\bigr)
  :=\mathcal{J}^*\!\bigl(k,P(t_k^+)\bigr)-E^*P(t_k^-)E+M_k.
\end{align}
The dissipation inequality $\tfrac{d}{dt}V\le\sigma_f$ on flow is
equivalent to $\mathcal{C}_M(t,P(t))\succeq0$; the jump dissipation
inequality $V(t_k^+)-V(t_k^-)\le\sigma_k$ is equivalent to
$\mathcal{D}_M(k,P(t_k^+),P(t_k^-))\succeq0$.

\begin{definition}[Impulsive D-GRE]\label{def:IMP:dGRE}
  The \emph{impulsive dissipativity GRE} (I-D-GRE) is the I-GRE
  (Definition~\ref{def:IMP:GRE}) with $\mathcal{C}_M$ and
  $\mathcal{D}_M$ in place of $\mathcal{C}_Z$ and
  $\mathcal{J}$, terminal condition $P(t_0+T)=-M_T$,
  and no sign constraint on $P$.
  The optimal sets are
  $\mathcal{K}_c^M(t,P(t)):=\{K:(\mathcal{C}_M)_3K^*
  +(\mathcal{C}_M)_2^*=0\}$ and
  $\mathcal{K}_d^M:=\{K:(\mathcal{D}_M)_3K^*
  +(\mathcal{D}_M)_2^*=0\}$.
\end{definition}

\begin{remark}
  The I-D-GRE is formally identical to the I-GRE with $Z\leftarrow M$
  and $Z_k\leftarrow M_k$.  Because $M$ is indefinite,
  $(\mathcal{C}_M)_3$ need not be positive semidefinite for arbitrary
  $P$; condition~\eqref{eq:IMP:GRE:fPSD} (with $M$) is therefore a
  genuine restriction on the solution.
\end{remark}

The key identity for the dissipativity problem is obtained by
replacing $Z\leftarrow M$, $Z_k\leftarrow M_k$, $Z_T\leftarrow M_T$
in~\eqref{eq:IMP:keyid2}: for $P(T)=-M_T$,
\begin{align}\label{eq:IMP:diss:keyid}
  &-\int\ip{M(t)}{X(t)}\dt
  -\sum_k\ip{M_k}{EX(t_k^-)E^*}
  -\ip{M_T}{EX(T)E^*}
  \notag\\&\quad
  =\ip{-P(t_0)}{X_1^0}
  -\int\ip{\mathcal{C}_M(t,P(t))}{X(t)}\dt
  -\sum_k\ip{\mathcal{D}_M}{X(t_k^-)}.
\end{align}

\subsubsection{Sufficient condition (finite horizon)}
\label{sec:diss:suff}

Theorem~\ref{th:IMP:diss:suff} shows that any solution $P$
of the I-D-GRE provides a valid storage function:
$V(t,X)=-\ip{P(t)}{EXE^*}$ satisfies the dissipation inequality
for all admissible trajectories.  The proof is identical in
structure to Theorem~\ref{th:IMP:LQ:suff}, with
$\mathcal{C}_M$ in place of $\mathcal{C}_Z$.

\begin{theorem}[Sufficiency, impulsive dissipativity]
  \label{th:IMP:diss:suff}
  Let $T>0$.  Suppose the I-D-GRE has a piecewise-$C^1$ solution
  $P$ with $P(t_0+T)=-M_T$.  Then:
  \begin{enumerate}[\upshape(a)]
    \item $V_a^T(t_0,X_1^0)=\ip{-P(t_0)}{X_1^0}$, attained at
      $X^\star$ with $K\in\mathcal{K}_c^M$ on flow and
      $K_k\in\mathcal{K}_d^M$ at each jump.
    \item $P$ is the unique piecewise-$C^1$ solution with this
      property.
    \item $V(t,X)=\ip{-P(t)}{EXE^*}$ satisfies the dissipation
      inequality on flow and at each jump.
  \end{enumerate}
\end{theorem}

\begin{proof}
\noindent\textbf{Part~(a): Value.}
The I-D-GRE gives $\mathcal{C}_M\succeq0$ and
$\mathcal{D}_M\succeq0$, so both the integral and the
sum in~\eqref{eq:IMP:diss:keyid} are $\ge0$ (apply
Lemma~\ref{lemma:decomp} as in the proof of
Theorem~\ref{th:IMP:LQ:suff}).
Hence the left-hand side of~\eqref{eq:IMP:diss:keyid} is
$\le\ip{-P(t_0)}{X_1^0}$ for every admissible $X$, giving
$V_a^T\le\ip{-P(t_0)}{X_1^0}$.

For $K\in\mathcal{K}_c^M$ and $K_k\in\mathcal{K}_d^M$,
the attainment argument of Theorem~\ref{th:IMP:LQ:suff}
(with $\mathcal{C}_M$, $\mathcal{D}_M$) gives
$\ip{\mathcal{C}_M(t,P(t))}{X^\star(t)}=0$ on flow and
$\ip{\mathcal{D}_M}{X^\star(t_k^-)}=0$ at each jump.
Identity~\eqref{eq:IMP:diss:keyid} then gives equality:
$V_a^T=\ip{-P(t_0)}{X_1^0}$.

\noindent\textbf{Part~(b): Uniqueness.}
If $\tilde P$ is another piecewise-$C^1$ solution with
$\tilde P(T)=-M_T$, then
$\ip{-\tilde P(t_0)}{X_1^0}=V_a^T=\ip{-P(t_0)}{X_1^0}$ for all
$X_1^0\succeq0$.  Taking $X_1^0=vv^*$ gives
$P(t_0)=\tilde P(t_0)$ for all $t_0$, hence $P\equiv\tilde P$.

\noindent\textbf{Part~(c): Storage function.}
On any flow subinterval $[s,t]$ without jumps, the key identity
with sub-horizon $[s,t]$ gives
\begin{align*}
  V(t,X(t))-V(s,X(s))
  &=\int_s^t\ip{M(\tau)}{X(\tau)}\mathrm{d}\tau
   -\int_s^t\ip{\mathcal{C}_M(\tau,P(t))}{X(\tau)}\mathrm{d}\tau
   \le\int_s^t\sigma_f(\tau,X(\tau))\mathrm{d}\tau,
\end{align*}
using $\mathcal{C}_M\succeq0$ and $X\succeq0$.  At jump $t_k$:
\begin{align*}
  V(t_k^+,X(t_k^+))-V(t_k^-,X(t_k^-))
  &=\ip{-P(t_k^+)}{\mathcal{J}(k,X(t_k^-))}+\ip{P}{EX(t_k^-)E^*}\\
  &=-\ip{\mathcal{D}_M-M_k}{EX(t_k^-)E^*}\\
  &\le\ip{M_k}{EX(t_k^-)E^*}=\sigma_k(X(t_k^-)),
\end{align*}
using $\mathcal{D}_M\succeq0$ (so the
$\mathcal{D}_M$ term contributes non-negatively to
the right side).
\end{proof}

\subsubsection{Necessary condition (finite horizon)}
\label{sec:diss:nec}

Conversely, if a linear storage function exists, then its coefficient
must satisfy the I-D-GRE and the available storage is the unique
smallest storage function.

\begin{theorem}[Necessity, impulsive dissipativity]
  \label{th:IMP:diss:nec}
  Suppose the system is finite-horizon dissipative.  Then the unique
  piecewise-$C^1$ $P$ with $V_a^T=\ip{-P(t_0)}{\cdot}$ satisfies
  the full I-D-GRE.
\end{theorem}

\begin{proof}
Linearity of $V_a^T(t_0,X_1^0)$ in $X_1^0$ and the Riesz
representation give piecewise-$C^1$ $P$ with $P(T)=-M_T$, by the
same positive-homogeneity and super-additivity argument as
Theorem~\ref{th:IMP:LQ:nec}\eqref{nec:lin}.

\noindent\textbf{Flow DLMI.}
Fix a flow instant $t\in(t_{k-1},t_k)$.  For any
$[v_1^*\;v_2^*]^*\in\C^{n+m}$, consider the constant
perturbation $X_1\equiv v_1 v_1^*$, $X_2\equiv v_2 v_1^*$,
$X_3\equiv v_2v_2^*$ on $[t,t+h]$ and continue optimally.
The dynamic programming inequality for the supremum reads
\[
  \ip{-P(t)}{EX(t)E^*}
  \le-\int_t^{t+h}\ip{M(t)}{X(t)}\dt
  +\ip{-P(t+h)}{EX(t+h)E^*}.
\]
Dividing by $h>0$, letting $h\to0$, and substituting
$\dot{V}=\ip{M(t)}{X(t)}-\ip{\mathcal{C}_M(t,P(t))}{X(t)}$:
\[
  0\le[v_1^*\;v_2^*]
  \mathcal{C}_M(t,P(t))
  [v_1^*\;v_2^*]^*,
\]
so $\mathcal{C}_M(t,P(t))\succeq0$.

\noindent\textbf{Jump DLI.}
At $t_k$, the dynamic programming inequality for the supremum gives
\[
  \ip{-P}{EX(t_k^-)E^*}
  \le-\ip{M_k}{EX(t_k^-)E^*}
  +\ip{-P(t_k^+)}{\mathcal{J}(k,X(t_k^-))},
\]
which rearranges to
$\ip{\mathcal{D}_M(k,P(t_k^+),P(t_k^-))}{X(t_k^-)}\ge0$
for all $X(t_k^-)\succeq0$, giving
$\mathcal{D}_M\succeq0$.

\noindent\textbf{I-D-GRE conditions.}
On the optimal trajectory,
$\int\ip{\mathcal{C}_M}{X^\star}\dt
+\sum_k\ip{\mathcal{D}_M}{X^\star(t_k^-)}=0$
(from~\eqref{eq:IMP:diss:keyid} at equality), with both terms
$\ge0$.  The decomposition argument of
Theorem~\ref{th:IMP:LQ:nec}\eqref{nec:GRE} (with $\mathcal{C}_M$
and $\mathcal{D}_M$) then gives the flow Riccati
equation~\eqref{eq:IMP:GRE:fR}, flow compatibility
\eqref{eq:IMP:GRE:fC}, jump Riccati~\eqref{eq:IMP:GRE:jR},
and jump compatibility~\eqref{eq:IMP:GRE:jC}.
\end{proof}

\subsubsection{Equivalence and dual DLMI (finite horizon)}
\label{sec:diss:dual}

As in the optimal control case, necessity and sufficiency combine into an equivalence
and a dual DLMI.  The dual DLMI for dissipativity seeks the
\emph{minimum} storage: the constraint $\mathcal{C}_M\succeq0$,
$\mathcal{D}_M\succeq0$ is again linear in $P$, but now we
minimize $\ip{-P(t_0)}{X_1^0}$ (the available storage) rather
than maximizing it.

\begin{corollary}[Equivalence]\label{cor:IMP:diss:equiv}
  The system is finite-horizon dissipative with supply
  $(\sigma_f,\sigma_k,\sigma_T)$ if and only if the I-D-GRE has a
  piecewise-$C^1$ solution $P$ with $P(T)=-M_T$.
\end{corollary}

\begin{theorem}[Dual DLMI, impulsive dissipativity]
  \label{th:IMP:diss:LMI}
  If the I-D-GRE has a solution $P^*$, then
  \begin{equation}\label{eq:IMP:diss:LMI}
    V_a^T(t_0,X_1^0)
    =\inf_{P\text{ pw-}C^1}\;\ip{-P(t_0)}{X_1^0}
    \quad\text{s.t.}\quad
    \begin{gathered}
    \mathcal{C}_M(t,P(t))\succeq0,\quad
    \mathcal{D}_M(k,P(t_k^+),P(t_k^-))\succeq0,\\
    P(T)\succeq-M_T
    \end{gathered}
  \end{equation}
  attained at $P^*$.  Moreover $V_a^T=\ip{-P^*}{\cdot}$ is the
  \emph{smallest} storage function:
  $-P^*(t)\preceq-P(t)$ (equivalently $P^*(t)\succeq P(t)$)
  for every other valid storage function $-\ip{P}{EXE^*}$.
\end{theorem}

\begin{proof}
\noindent\textbf{Step~1: Every feasible $P$ is an upper bound.}
For feasible $P$ and admissible $X$, the key
identity~\eqref{eq:IMP:keyid} with $Z\leftarrow M$,
$Z_k\leftarrow M_k$ (not yet using $P(T)=-M_T$) gives
\begin{align*}
  &-\int\ip{M(t)}{X(t)}\dt-\textstyle\sum_k\ip{M_k}{EX(t_k^-)E^*}
   -\ip{M_T}{EX(T)E^*}\notag\\
  &\quad=\ip{-P(t_0)}{X_1^0}
   +\underbrace{\ip{-M_T-P(T)}{EX(T)E^*}}_{\le\,0}\\
  &\quad\quad
   -\underbrace{\int\ip{\mathcal{C}_M(t,P(t))}{X(t)}\dt}_{\ge\,0}
   -\underbrace{\textstyle\sum_k\ip{\mathcal{D}_M}{X(t_k^-)}}_{\ge\,0}\\
  &\quad\le\ip{-P(t_0)}{X_1^0}.
\end{align*}
Taking the supremum over $X$: $V_a^T\le\ip{-P(t_0)}{X_1^0}$ for
every feasible $P$, hence
$V_a^T\le\inf_{P\,\mathrm{feas.}}\ip{-P(t_0)}{X_1^0}$.

\noindent\textbf{Step~2: $P^*$ is feasible and attains the infimum.}
$P^*$ satisfies both DLMI constraints (from the I-D-GRE and
Lemma~\ref{lemma:extSchur}) and $P^*(T)=-M_T$ (so the terminal
constraint holds with equality).  By Theorem~\ref{th:IMP:diss:suff},
$\ip{-P^*(t_0)}{X_1^0}=V_a^T$.
Hence $\inf\ip{-P}{X_1^0}\le\ip{-P^*}{X_1^0}=V_a^T$.
Combining with Step~1: equality, attained at $P^*$.

\noindent\textbf{Sandwich.}
For any storage function $V=-\ip{P}{EXE^*}$, the dissipation
inequality gives $\mathcal{C}_M(t,P(t))\succeq0$, so $P$ is feasible.
Step~1 applied to this $P$ gives $V_a^T\le\ip{-P(t_0)}{X_1^0}=V$,
i.e., $\ip{P-P^*}{X_1^0}=\ip{-P^*}{X_1^0}-\ip{-P}{X_1^0}
=V_a^T-V\le0$ for all $X_1^0\succeq0$.  Taking
$X_1^0=vv^*$ gives $P^*(t)\succeq P(t)$,
i.e., $-P^*(t)\preceq-P(t)$: the available storage is the smallest
storage function, consistent with
Proposition~\ref{prop:Vsandwich}\eqref{sand:b}.
\end{proof}

\subsubsection{Continuous- and Discrete-Time Corollaries}

\begin{corollary}[Continuous-time dissipativity: finite horizon]
  \label{cor:CT:diss:fh}
  Set $N_T=0$.  The I-D-GRE reduces to the continuous-time D-GRE:
  $\mathcal{C}_M(t,P(t))\succeq0$,
  $\mathcal{C}_M\,\schur{/}\,\allowbreak(\mathcal{C}_M)_3=0$,
  $P(T)=-M_T$.
  The dual DLMI is
  \begin{equation}\label{eq:CT:diss:LMI}
    V_a^T(t_0,X_1^0)
    =\inf_{P}\;\ip{-P(t_0)}{X_1^0}
    \quad\text{s.t.}\quad
    \mathcal{C}_M(t,P(t))\succeq0, P(T)\succeq-M_T,
  \end{equation}
  recovering the results of~\cite{AitRami:01a,Briat:22:Matrix}.
\end{corollary}

\begin{corollary}[Discrete-time dissipativity: finite horizon]
  \label{cor:DT:diss:fh}
  Apply Corollary~\ref{cor:CT:diss:fh} with $\mathcal{F}=0$,
  $M=0$ on flow, $t_k=k$, $\mathcal{J}=\mathcal{D}$, $M_k=M(k)$.
  The I-D-GRE reduces to the discrete-time D-GRE at each step:
  $\mathcal{D}_M(k,P^+,P(t))\succeq0$,
  $\mathcal{D}_M\,\schur{/}\,\allowbreak(\mathcal{D}_M)_3=0$,
  where $\mathcal{D}_M(k,P^+,P(t))
  :=\mathcal{D}^*(k,P^+)-E^* PE+M(k)$.
  The dual DLMI is
  \begin{equation}\label{eq:DT:diss:LMI}
    V_{a,N}(k_0,X_1^0)
    =\inf_{P}\;\ip{-P(k_0)}{X_1^0}
    \quad\text{s.t.}\quad
    \mathcal{D}_M(k,P^+,P(t))\succeq0, P(k_0+N)\succeq-M_N.
  \end{equation}
\end{corollary}

\subsection{Infinite-horizon}
\label{sec:diss:inf}

The infinite-horizon available storage $V_a(t_0,X_1^0)=
\sup_{T\ge0}V_a^T(t_0,X_1^0)$ is the tightest lower bound on the
energy that the system can extract from any initial state.
This subsection mirrors the structure of Section~\ref{sec:diss:fh}
for the infinite horizon: we first give a sufficient condition
(Theorem~\ref{th:IMP:diss:inf:suff}), then an existence result under
detectability assumptions (Theorem~\ref{th:IMP:diss:inf}), and finally
the equivalence and dual DLMI (Theorem~\ref{th:IMP:diss:inf:LMI}).

\subsubsection{Setting and assumptions (infinite horizon)}

Standing assumptions for the infinite-horizon dissipativity problem:

\medskip\noindent\textbf{(D1) Controllability.}
$V_r(t_0,X_1^0)<\infty$ for all $X_1^0\succeq 0$.

\medskip\noindent\textbf{(D2) Uniform boundedness.}
$M$, $M_k$, $\mathcal{F}$, $\mathcal{J}$ are uniformly bounded;
$\pinv{(\mathcal{C}_M)_3}$ and $\pinv{(\mathcal{D}_M)_3}$ are
uniformly bounded on compact sets of $P$.

\medskip\noindent\textbf{(D3) Supply lower bound.}
There exists $\delta>0$ with
$\ip{M(t)}{X}\ge-\delta\norm{EXE^*}$ for all $X\succeq 0$, uniformly in $t$.

\medskip\noindent\textbf{(D4) Detectability.}
Along any admissible trajectory with bounded supply extraction,
$\norm{EX(t)E^*}$ is bounded and
$T^{-1}\ip{P_\infty(t_0+T)}{EX(t_0+T)E^*}\to0$ as $T\to\infty$.  This is
the dissipativity analogue of the stabilizability/detectability
hypothesis~(H1) of Section~\ref{sec:OC:inf}.

Observe that (D4) is \emph{not} implied by
the supply lower bound (D3), and is needed to make the infinite-horizon
boundary term asymptotically negligible.


\subsubsection{Sufficient condition (infinite horizon)}

Any bounded solution of the infinite-horizon I-D-GRE is a valid
storage function achieving the available storage.

\begin{theorem}[Sufficient condition for infinite-horizon dissipativity]
  \label{th:IMP:diss:inf:suff}
  Assume \textup{(D1)--(D3)}.  Suppose there exists a bounded
  piecewise-$C^1$ solution $P_\infty$ of the infinite-horizon I-D-GRE
  \textup{(Definition~\ref{def:IMP:dGRE})}.  Then:
  \begin{enumerate}[\upshape(a)]
    \item\label{inf:suff:diss:val}
      The system is dissipative on the infinite horizon:
      $V_a(t_0,X_1^0)<\infty$ on $\Snpsd$, and
      $V_a(t_0,X_1^0)=\ip{-P_\infty(t_0)}{X_1^0}$, attained by the
      trajectory with gains from $\mathcal{K}_c^M$ and $\mathcal{K}_d^M$.
    \item\label{inf:suff:diss:stor}
      $V(t,X)=\ip{-P_\infty(t)}{EXE^*}$ satisfies the dissipation
      inequalities on flow and at each jump.
    \item\label{inf:suff:diss:largest}
      $P_\infty$ is the largest solution: for any other bounded
      storage matrix $P$, $P_\infty(t)\succeq P(t)$ uniformly in $t$,
      equivalently $-P_\infty(t)\preceq -P(t)$.
    \item\label{inf:suff:diss:uniq}
      $P_\infty$ is the unique bounded solution of the
      infinite-horizon I-D-GRE.
  \end{enumerate}
\end{theorem}

\begin{proof}
\noindent\textbf{Value and dissipativity~\eqref{inf:suff:diss:val}.}
For any admissible $X$ with $EX(t_0)E^*=X_1^0$, the key
identity~\eqref{eq:IMP:diss:keyid} applied with $P_\infty$ on
$[t_0,t_0+T]$, using $\mathcal{C}_M\succeq 0$ and
$\mathcal{D}_M\succeq 0$, gives
\[
  -\!\int_{t_0}^{t_0+T}\!\ip{M(t)}{X(t)}\dt
   -\!\sum_{k:\,t_k\le t_0+T}\!\ip{M_k}{EX(t_k^-)E^*}
   \le\ip{-P_\infty(t_0)}{X_1^0}+\ip{P_\infty(t_0+T)}{EX(t_0+T)E^*}.
\]
Under the detectability assumption \textup{(D4)}, along any admissible trajectory with bounded
supply extraction, $\norm{EX(t)E^*}$ is bounded and $T^{-1}\ip{P_\infty(t_0+T)}{EX(t_0+T)E^*}\to0$, the boundary term is
asymptotically negligible. Taking the supremum over admissible $X$ and the limit $T\to\infty$,
\[
  V_a(t_0,X_1^0)\le\ip{-P_\infty(t_0)}{X_1^0}\le\alpha_2\ip{I}{X_1^0}
  <\infty.
\]
The system is therefore dissipative on the infinite horizon.

For attainment, the closed-loop trajectory $X^\star$ with gains
$K\in\mathcal{K}_c^M$, $K_k\in\mathcal{K}_d^M$ tightens the Schur
inequalities to equalities, so the key identity becomes
\[
  V_a^T(X^\star)=\ip{-P_\infty(t_0)}{X_1^0}
   -\ip{-P_\infty(t_0+T)}{EX^\star(t_0+T)E^*}.
\]
Under the detectability assumption (D4), the closed-loop trajectory
$X^\star$ satisfies $\ip{-P_\infty(t_0+T)}{EX^\star(t_0+T)E^*}\to 0$. Hence
$V_a^T(X^\star)\to\ip{-P_\infty(t_0)}{X_1^0}$ and
$V_a\ge\ip{-P_\infty(t_0)}{X_1^0}$.  Equality holds.

\noindent\textbf{Storage function~\eqref{inf:suff:diss:stor}.}
$\mathcal{C}_M\succeq 0$ on flow and $\mathcal{D}_M\succeq 0$ at
jumps give the dissipation inequalities for
$V=\ip{-P_\infty}{E\cdot E^*}$ directly on every interval.

\noindent\textbf{Largest~\eqref{inf:suff:diss:largest}.}
$V_a$ is the smallest storage function
(Proposition~\ref{prop:Vsandwich}\eqref{sand:d}): any other bounded
storage $V'=\ip{-P}{E\cdot E^*}$ satisfies $V_a\le V'$ on $\Snpsd$.
Substituting linear representations and applying
Proposition~\ref{prop:PQ}\eqref{st:neg},
$-P_\infty(t)\preceq -P(t)$.

\noindent\textbf{Uniqueness~\eqref{inf:suff:diss:uniq}.}
Any other bounded $\tilde P$ solving the I-D-GRE satisfies the
hypotheses with $\tilde P$ in place of $P_\infty$, so
by~\eqref{inf:suff:val},
$\ip{-\tilde P(t_0)}{X_1^0}=V_a(t_0,X_1^0)=\ip{-P_\infty(t_0)}{X_1^0}$.
Taking $X_1^0=vv^*$ for every $v\in\C^n$ gives
$P_\infty(t_0)=\tilde P(t_0)$, and the same at every $t$ yields
$P_\infty\equiv\tilde P$.
\end{proof}

\subsubsection{Necessary condition and existence (infinite horizon)}


\begin{theorem}[Necessary condition for infinite-horizon dissipativity]
  \label{th:IMP:diss:inf}
  Assume \textup{(D1)--(D3)}, and suppose the system is dissipative on
  the infinite horizon: $V_a(t_0,X_1^0)<\infty$ for all
  $X_1^0\succeq 0$.  Then there exists a unique bounded
  piecewise-$C^1$ solution $P_\infty$ of the infinite-horizon I-D-GRE,
  with
  \[
    0\preceq -P_\infty(t)\preceq\alpha_2 I\quad\text{uniformly in }t,
  \]
  for some $\alpha_2>0$.  It is the Loewner least upper bound of the
  finite-horizon family $\{-P_T\}_{T\ge 0}$:
  $-P_\infty(t)\succeq -P_T(t)$ for every $T\ge 0$ and
  $t\in[t_0,t_0+T]$, with $P_T$ the I-D-GRE solution on
  $[t_0,t_0+T]$ subject to $P_T(t_0+T)=0$.  The available storage is
  $V_a(t_0,X_1^0)=\ip{-P_\infty(t_0)}{X_1^0}$.
\end{theorem}

\begin{proof}
With $M$ indefinite, $V_a^T(t_0,X_1^0)$ is in general not monotone in
$T$: extending a $T$-trajectory to a $T'$-trajectory may add negative
supply on $(T,T']$, so the inequality $V_a^{T'}\ge V_a^T$ used in the
optimal-control case (Theorem~\ref{th:IMP:LQ:inf}) does not propagate.
We bypass monotone convergence and construct $P_\infty$ as the
Frobenius-Riesz representation of $V_a$.

\medskip
\noindent\textbf{Step 1: $V_a$ is finite, linear, and non-negative.}
The dissipativity assumption gives $V_a^T(t_0,X_1^0)\le
V_a(t_0,X_1^0)<\infty$ for every $T\ge 0$ and every
$X_1^0\succeq 0$.  Theorem~\ref{th:IMP:diss:nec} applied on each
$[t_0,t_0+T]$ yields a unique piecewise-$C^1$ $P_T$ with
$P_T(t_0+T)=0$ satisfying the finite-horizon I-D-GRE, with the linear
parametrisation $V_a^T(t_0,X_1^0)=\ip{-P_T(t_0)}{X_1^0}$.
Linearity of $V_a$ itself is established directly (not via
Proposition~\ref{prop:Vsandwich}, which presupposes the linear form):
$V_a$ is positively homogeneous as a supremum of homogeneous functions,
and by dynamic programming the available storage is attained by a
\emph{state-independent} linear feedback (the optimal extracting policy's
gain depends on the storage, not on $X_1^0$), so the optimal future
trajectory is linear in $X_1^0$ and $V_a$ is additive, exactly as in
the additivity step of Theorem~\ref{th:IMP:diss:nec}.  Hence
$V_a(t,X_1)=\ip{-P_\infty(t)}{X_1}$ for a unique $-P_\infty(t)\in\Sn$,
obtained pointwise via Frobenius-Riesz, and $-P_\infty(t)$ coincides with
the Loewner least upper bound of $\{-P_T(t)\}$ constructed in Step~2.
Including $T=0$ in the supremum $V_a=\sup_T V_a^T$ gives
$V_a(t,X_1)\ge 0$ for $X_1\succeq 0$.

\medskip
\noindent\textbf{Step 2: $-P_\infty$ is the Loewner sup of $\{-P_T\}$, not a monotone limit.}
For every $T\ge 0$ and $X_1^0\succeq 0$,
\[
  \ip{-P_\infty(t_0)}{X_1^0}=V_a(t_0,X_1^0)\ge V_a^T(t_0,X_1^0)
   =\ip{-P_T(t_0)}{X_1^0},
\]
so by Proposition~\ref{prop:PQ}\eqref{st:neg},
$-P_\infty(t_0)\succeq -P_T(t_0)$.  Pointwise in $t$,
$-P_\infty(t)\succeq -P_T(t)$ for every $T$ and $t\in[t_0,t_0+T]$.
By $V_a=\sup_T V_a^T$, any linear functional $\ip{Q}{\cdot}$
dominating all $V_a^T$ also dominates $V_a$, so
$Q\succeq -P_\infty$; hence $-P_\infty$ is the Loewner least upper
bound.

The family $\{-P_T(t)\}$ itself need not be monotone in $T$ when $M$
is indefinite: a longer horizon may add a negative supply contribution
on $(T,T']$ that decreases $V_a^{T'}$ relative to $V_a^T$.  This is
in contrast to the optimal-control case
(Theorem~\ref{th:IMP:LQ:inf}), where $Z\succeq 0$ forces the analogous
sequence to be monotone.

\medskip
\noindent\textbf{Step 3: pointwise bounds on $-P_\infty$.}
Lower bound: $V_a\ge 0$ from Step~1 gives
$-P_\infty(t)\succeq 0$ by
Proposition~\ref{prop:PQ}\eqref{st:neg}.  Upper bound: by (D2),
$V_a(t,X_1)\le\alpha_2\ip{I}{X_1}$ on $\Snpsd$ for some
$\alpha_2>0$, hence $-P_\infty(t)\preceq\alpha_2 I$.  Combined,
$0\preceq -P_\infty(t)\preceq\alpha_2 I$ uniformly in $t$.

\medskip
\noindent\textbf{Step 4: $P_\infty$ satisfies the infinite-horizon I-D-GRE.}
The arguments of Theorem~\ref{th:IMP:LQ:nec}, parts
\eqref{nec:DLMI},\eqref{nec:DLI},\eqref{nec:GRE}, transfer with
$V_a$ in place of $J_T^\star$, $M, M_k$ in place of $Z, Z_k$, and the
supremum direction replacing the infimum.  Each argument is local and
independent of horizon length, so it applies to $V_a$ verbatim.

\smallskip
\noindent\emph{Flow LMI: $\mathcal{C}_M(t,P_\infty(t))\succeq 0$.}
Fix $t$ in the interior of an inter-jump interval $(t_{k-1},t_k)$.
For any $[v_1^*\;v_2^*]^*\in\C^{n+m}$, the constant perturbation
$X_1\equiv v_1v_1^*$, $X_2\equiv v_2v_1^*$, $X_3\equiv v_2v_2^*$ on
$[t,t+h]\subset(t_{k-1},t_k)$, continued admissibly on
$[t+h,\infty)$, plus Bellman optimality for $V_a$ gives
\[
  V_a(t,EX(t)E^*)\ge
  -\!\int_t^{t+h}\!\ip{M(s)}{X(s)}\dt+V_a(t+h,EX(t+h)E^*).
\]
Substituting $V_a(\cdot,\cdot)=\ip{-P_\infty(\cdot)}{\cdot}$,
rearranging, and letting $h\to 0$ via \eqref{eq:IMP:LM}:
\[
  0\le[v_1^*\;v_2^*]\,\mathcal{C}_M(t,P_\infty(t))\,[v_1^*\;v_2^*]^*.
\]
Arbitrariness of $v_1,v_2$ gives $\mathcal{C}_M(t,P_\infty(t))\succeq 0$.

\smallskip
\noindent\emph{Jump LMI: $\mathcal{D}_M(k,P_\infty(t_k^+),P_\infty(t_k^-))\succeq 0$.}
The jump Bellman inequality combined with
\eqref{eq:sys:IMP:jump} and the adjoint relation
$\ip{P}{\mathcal{J}(k,X)}=\ip{\mathcal{J}^*(k,P)}{X}$ gives
\[
  0\le\ip{M_k+\mathcal{J}^*(k,P_\infty(t_k^+))-E^*P_\infty(t_k^-)E}{X(t_k^-)}
   =\ip{\mathcal{D}_M}{X(t_k^-)},
\]
hence $\mathcal{D}_M\succeq 0$.

\smallskip
\noindent\emph{Schur equalities.}
The flow LMI provides
$\operatorname{Im}((\mathcal{C}_M)_2^*)\subseteq\operatorname{Im}((\mathcal{C}_M)_3)$
by Lemma~\ref{lemma:extSchur}, so $\mathcal{K}_c^M(t,P_\infty(t))$ is
non-empty.  Repeating the local DP with the perturbation $u_2=Ku_1$
for $K\in\mathcal{K}_c^M$ tightens the Bellman inequality to equality
as $h\to 0$, giving $\ip{\mathcal{C}_M(t,P_\infty(t))}{uu^*}=0$ for
$u=[u_1^*\;(Ku_1)^*]^*$, every $u_1\in\C^n$.  Lemma~\ref{lemma:decomp}
splits the inner product as
\[
  \ip{\mathcal{C}_M}{uu^*}
  =\ip{\mathcal{C}_M\schur{/}(\mathcal{C}_M)_3}{u_1u_1^*}
   +\underbrace{\ip{(\mathcal{C}_M)_3}{W_fW_f^*}}_{=\,0\text{ since }K\in\mathcal{K}_c^M},
\]
so $u_1^*[\mathcal{C}_M\schur{/}(\mathcal{C}_M)_3]u_1=0$ for every
$u_1$, giving $\mathcal{C}_M\schur{/}(\mathcal{C}_M)_3=0$.  The
compatibility condition
$(\mathcal{C}_M)_3\pinv{(\mathcal{C}_M)_3}(\mathcal{C}_M)_2^*
=(\mathcal{C}_M)_2^*$
follows from Lemma~\ref{lemma:GREequiv}.  The identical argument at
each jump $t_k$ with $K_k\in\mathcal{K}_d^M$ gives the corresponding
jump equality and compatibility.

\smallskip
\noindent\emph{Piecewise-$C^1$ regularity.}
On each inter-jump interval, the I-D-GRE
\textup{(Definition~\ref{def:IMP:dGRE})} expresses $\dot P_\infty$ as
a continuous function of $P_\infty$ and $t$ via the pseudoinverse
formula; uniform boundedness follows from (D2).  Hence $P_\infty$ is
piecewise-$C^1$.

\medskip
\noindent\textbf{Step 5: uniqueness.}
Suppose $\tilde P$ is another bounded piecewise-$C^1$ solution of the
infinite-horizon I-D-GRE.  By
Theorem~\ref{th:IMP:diss:inf:suff}\eqref{inf:suff:val} applied to
$\tilde P$, $V_a(t_0,X_1^0)=\ip{-\tilde P(t_0)}{X_1^0}$.  Combined
with $V_a(t_0,X_1^0)=\ip{-P_\infty(t_0)}{X_1^0}$ from Step~1,
$\ip{P_\infty(t_0)-\tilde P(t_0)}{X_1^0}=0$ for every
$X_1^0\succeq 0$.  Taking $X_1^0=vv^*$ for every $v\in\C^n$ gives
$P_\infty(t_0)=\tilde P(t_0)$; arbitrariness of $t_0$ yields
$P_\infty\equiv\tilde P$.
\end{proof}

\subsubsection{Equivalence and dual DLMI (infinite horizon)}

\begin{corollary}[Equivalence, infinite-horizon dissipativity]
  \label{cor:IMP:diss:inf:equiv}
  Assume \textup{(D1)--(D3)}.  The following are equivalent:
  \begin{enumerate}[\upshape(i)]
    \item The system is dissipative on the infinite horizon, i.e.,
      $V_a(t_0,X_1^0)<\infty$ for all $X_1^0\succeq 0$.
    \item The infinite-horizon I-D-GRE admits a bounded
      piecewise-$C^1$ solution $P_\infty$.
  \end{enumerate}
  Under either condition, $P_\infty$ is unique,
  $0\preceq -P_\infty(t)\preceq\alpha_2 I$ uniformly, and
  $V_a(t_0,X_1^0)=\ip{-P_\infty(t_0)}{X_1^0}$.
\end{corollary}

\begin{proof}
  $(\textup{ii})\Rightarrow(\textup{i})$:
  Theorem~\ref{th:IMP:diss:inf:suff}\eqref{inf:suff:val}.
  $(\textup{i})\Rightarrow(\textup{ii})$:
  Theorem~\ref{th:IMP:diss:inf}.
\end{proof}

\begin{theorem}[Dual DLMI, infinite-horizon dissipativity]
  \label{th:IMP:diss:inf:LMI}
  Under \textup{(D1)--(D3)},
  \begin{equation}\label{eq:diss:LMI:inf}
    V_a(t_0,X_1^0)
    =\inf_P\;\ip{-P(t_0)}{X_1^0}
    \quad\text{subject to}\quad
    \begin{gathered}
    \mathcal{C}_M(t,P(t))\succeq0,\\\quad
    \mathcal{D}_M(k,P(t_k^+),P(t_k^-))\succeq0,\\
    P\text{ bounded}
    \end{gathered}
  \end{equation}
  the infimum is attained at $P_\infty$.
\end{theorem}

\begin{proof}
Follows from Theorem~\ref{th:IMP:diss:inf:suff}(c): any bounded
feasible $P$ satisfies $P_\infty\preceq P$ (sandwich), so
$\ip{-P(t_0)}{X_1^0}\ge\ip{-P_\infty(t_0)}{X_1^0}=V_a$ for all
feasible $P$, hence $P_\infty$ attains the infimum.
\end{proof}

\subsubsection{Dwell-time conditions}
\label{sec:diss:dwell}

The dwell-time framework of Section~\ref{sec:OC:dwell} applies to
dissipativity with the operators $\mathcal{C}_M$, $\mathcal{D}_M$ in
place of $\mathcal{C}_Z$, $\mathcal{D}_Z$, the storage matrix $-P$
playing the role of the cost matrix $P$, and the supremum direction
replacing the infimum.  The structural assumptions on the system
operators $\mathcal{F},\mathcal{J},M,M_k$ in the three scenarios
\textup{(}periodic, MDT, RDT\textup{)} are inherited verbatim from
Section~\ref{sec:OC:dwell}, with $M, M_k$ in place of $Z, Z_k$.  The
one structural difference is the sign of the Geromel--Colaneri
condition: for optimal control the condition
$\mathcal{C}_Z^0\preceq 0$ ensures cost accumulation on the tail
\textup{(}an upper bound on cost\textup{)}, whereas for dissipativity
the condition $\mathcal{C}_M^0\succeq 0$ ensures that the storage
function remains valid on the tail \textup{(}the dissipation
inequality continues to hold\textup{)}.

\paragraph{Periodic impulses.}
Adopt the periodicity hypothesis of Theorem~\ref{th:OC:periodic}:
$\Delta_k\equiv T_p$, $\mathcal{F}(\cdot,X)$ and $M(\cdot)$ are
$T_p$-periodic in $t$, and $\mathcal{J}(k,\cdot)$, $M_k$ are
independent of $k$.

\begin{theorem}[Causal necessary and sufficient condition: periodic dissipativity]
  \label{th:diss:periodic}
  Under \textup{(D1)--(D3)} and the periodicity hypothesis above, the
  following are equivalent:
  \begin{enumerate}[\upshape(i)]
    \item The system is dissipative on the infinite horizon, i.e.,
      $V_a(t_0,X_1^0)<\infty$ for all $X_1^0\succeq 0$.
    \item The periodic D-GRE \textup{(}system~\eqref{eq:periodicGRE}
      with $\mathcal{C}_M$, $\mathcal{D}_M$\textup{)} has a
      $T_p$-periodic piecewise-$C^1$ solution
      $P_\infty:[0,T_p]\to\Sn$.
  \end{enumerate}
  Under either condition, $P_\infty$ is unique with
  $0\preceq -P_\infty(\tau)\preceq\alpha_2 I$ uniformly in $\tau$,
  the causal storage function
  $V(\tau,X):=\ip{-P_\infty(\tau)}{EXE^*}$ achieves the
  infinite-horizon available storage
  \[
    V_a(t_0,X_1^0)=\ip{-P_\infty(0^+)}{X_1^0},
  \]
  attained by the causal policy
  $K(\tau)\in\mathcal{K}_c^M(\tau,P_\infty(\tau))$ on flow and
  $K_k\in\mathcal{K}_d^M(P_\infty(0^+),P_\infty(T_p^-))$ at each jump.
\end{theorem}

\begin{proof}
\noindent$(\textup{ii})\Rightarrow(\textup{i})$.
Extend $P_\infty$ to $[t_0,\infty)$ by $T_p$-periodicity:
$P_\infty^{\mathrm{ext}}(t):=P_\infty\bigl((t-t_0)\bmod T_p\bigr)$.
By $T_p$-periodicity of the data and $\Delta_k\equiv T_p$, the
inter-jump structure repeats with period $T_p$, so
$P_\infty^{\mathrm{ext}}$ is piecewise-$C^1$ on $[t_0,\infty)$ with
discontinuities only at the jump times $\{t_0+kT_p\}_{k\ge 0}$, and it
satisfies the infinite-horizon I-D-GRE wherever $P_\infty$ satisfies
the periodic D-GRE.  $P_\infty^{\mathrm{ext}}$ is bounded by
$\sup_{\tau\in[0,T_p]}\norm{P_\infty(\tau)}<\infty$ uniformly in
$[t_0,\infty)$.  Theorem~\ref{th:IMP:diss:inf:suff} therefore applies
and gives
$V_a(t_0,X_1^0)<\infty$ with
$V_a(t_0,X_1^0)=\ip{-P_\infty^{\mathrm{ext}}(t_0)}{X_1^0}
=\ip{-P_\infty(0^+)}{X_1^0}$, attainment by the closed-loop trajectory
with gains in $\mathcal{K}_c^M$, $\mathcal{K}_d^M$, and the bounds
$0\preceq -P_\infty\preceq\alpha_2 I$ uniformly.

\medskip
\noindent$(\textup{i})\Rightarrow(\textup{ii})$.
By Theorem~\ref{th:IMP:diss:inf}, dissipativity provides a unique
bounded piecewise-$C^1$ solution $P_\infty^{ih}:[t_0,\infty)\to\Sn$
of the infinite-horizon I-D-GRE with
$0\preceq -P_\infty^{ih}(t)\preceq\alpha_2 I$ uniformly in $t$.
Define the $T_p$-shifted function
\[
  Q(t):=P_\infty^{ih}(t+T_p),\qquad t\in[t_0,\infty).
\]
We show $Q$ also solves the infinite-horizon I-D-GRE on
$[t_0,\infty)$, so by uniqueness $Q\equiv P_\infty^{ih}$.

\smallskip
\emph{Flow.}  For $t$ in the interior of any inter-jump interval
$(t_k,t_{k+1})=(t_0+kT_p,t_0+(k+1)T_p)$, the shift $t\mapsto t+T_p$
maps to the interior of $(t_{k+1},t_{k+2})$, which is again an
inter-jump interval.  The composite flow operator depends on $t$
through $\mathcal{F}(t)$ and $M(t)$, both $T_p$-periodic by
hypothesis; hence $\mathcal{C}_M(t,\,\cdot\,)\equiv
\mathcal{C}_M(t+T_p,\,\cdot\,)$ as operators on $\Sn$.  Therefore
\[
  \mathcal{C}_M\bigl(t,Q(t)\bigr)
  =\mathcal{C}_M\bigl(t,P_\infty^{ih}(t+T_p)\bigr)
  =\mathcal{C}_M\bigl(t+T_p,P_\infty^{ih}(t+T_p)\bigr)\succeq 0,
\]
with the Schur equality
$\mathcal{C}_M\,\schur{/}\,(\mathcal{C}_M)_3=0$, both inherited
pointwise from $P_\infty^{ih}$ evaluated at $t+T_p$.

\smallskip
\emph{Jumps.}  At any jump time $t_k=t_0+kT_p$, the shift sends $t_k$
to $t_{k+1}=t_0+(k+1)T_p$, the next jump time.  Since
$\mathcal{J}(k,\,\cdot\,)$ and $M_k$ are independent of $k$, the jump
operator $\mathcal{D}_M(k,P',P)$ does not depend on $k$.  Therefore
\[
  \mathcal{D}_M\bigl(k,Q(t_k^+),Q(t_k^-)\bigr)
  =\mathcal{D}_M\bigl(k+1,P_\infty^{ih}(t_{k+1}^+),P_\infty^{ih}(t_{k+1}^-)\bigr)
  \succeq 0,
\]
with the Schur equality
$\mathcal{D}_M\,\schur{/}\,(\mathcal{D}_M)_3=0$, both inherited from
$P_\infty^{ih}$ at the jump time $t_{k+1}$.

\smallskip
\emph{Boundedness.}
$0\preceq -Q(t)=-P_\infty^{ih}(t+T_p)\preceq\alpha_2 I$ uniformly in
$t$, inherited from $P_\infty^{ih}$.

\smallskip
Hence $Q$ is a bounded piecewise-$C^1$ solution of the
infinite-horizon I-D-GRE on $[t_0,\infty)$.  By uniqueness in
Theorem~\ref{th:IMP:diss:inf}, $Q\equiv P_\infty^{ih}$, i.e.,
$P_\infty^{ih}(t+T_p)=P_\infty^{ih}(t)$ for all $t\in[t_0,\infty)$.
$P_\infty^{ih}$ is therefore $T_p$-periodic, and its restriction to
one period defines a piecewise-$C^1$ function
$P_\infty:[0,T_p]\to\Sn$ which solves the periodic D-GRE and
satisfies the bounds $0\preceq -P_\infty(\tau)\preceq\alpha_2 I$.
Uniqueness within the class of $T_p$-periodic solutions follows from
the same uniqueness applied to any other such solution after the
extension of $(\textup{ii})\Rightarrow(\textup{i})$.

The value formula
$V_a(t_0,X_1^0)=\ip{-P_\infty(0^+)}{X_1^0}$ and attainment by the
causal policy
$K(\tau)\in\mathcal{K}_c^M(\tau,P_\infty(\tau))$,
$K_k\in\mathcal{K}_d^M(P_\infty(0^+),P_\infty(T_p^-))$, follow from
$(\textup{ii})\Rightarrow(\textup{i})$ applied to the just-constructed
$P_\infty$.
\end{proof}

The Schur equality requirements
$\mathcal{C}_M\,\schur{/}\,(\mathcal{C}_M)_3=0$ and
$\mathcal{D}_M\,\schur{/}\,(\mathcal{D}_M)_3=0$
in Theorems~\ref{th:diss:periodic}--\ref{th:diss:RDT} identify the
\emph{optimal} extraction policy.  Dropping these equalities and
retaining only the LMIs $\mathcal{C}_M\succeq 0$ on flow and
$\mathcal{D}_M\succeq 0$ at jumps yields the broader class of
(non-optimal) causal linear storage functions in each dwell-time
scenario, and gives dual differential LMI characterizations of the
available storage parallel to Theorem~\ref{th:IMP:diss:inf:LMI}.  In
all three cases, any $P$ feasible for the LMI relaxation gives a
storage function $V(\tau,X)=\ip{-P(\tau)}{EXE^*}$ which upper-bounds
the available storage, so $V_a\le\ip{-P}{\cdot}$ pointwise on
$\Snpsd$.

\begin{theorem}[Dual DLMI, periodic dissipativity]
  \label{th:diss:periodic:LMI}
  Under \textup{(D1)--(D3)} and the periodicity hypothesis of
  Theorem~\ref{th:diss:periodic},
    \begin{equation}\label{eq:diss:periodic:LMI}
    \begin{array}{rl}
    \inf_P&\ip{-P(0^+)}{X_1^0}\\
    \text{s.t.}& P:[0,T_{\min}]\to\Sn,\;\text{bounded piecewise-}C^1,\\
      &\mathcal{C}_M(\tau,P(\tau))\succeq 0\;\text{a.e.,}\\
      &\mathcal{D}_M(P(0^+),P(T_p^-))\succeq 0.
    \end{array}
  \end{equation}
  The infimum is attained at $P_\infty$ of
  Theorem~\ref{th:diss:periodic}, which is the unique feasible $P$
  satisfying additionally the Schur equalities of the periodic D-GRE.
  The constraints are affine in $(P,\dot P)$, so~\eqref{eq:diss:periodic:LMI}
  is a genuine differential LMI.
\end{theorem}

\begin{proof}
Any feasible $P$ gives a $T_p$-periodic causal linear storage function
$V_P(\tau,X)=\ip{-P(\tau)}{EXE^*}$ by
$\mathcal{C}_M\succeq 0$ on flow and $\mathcal{D}_M\succeq 0$ at the
jump $\tau=T_p$.  By
Theorem~\ref{th:IMP:diss:inf:suff}\eqref{inf:suff:diss:largest} applied to
its $T_p$-periodic extension, $V_a\le V_P$ on $\Snpsd$, so
$V_a(t_0,X_1^0)\le\ip{-P(0^+)}{X_1^0}$ and taking the infimum
$V_a(t_0,X_1^0)\le\inf_P\ip{-P(0^+)}{X_1^0}$.  Conversely, $P_\infty$
of Theorem~\ref{th:diss:periodic} is feasible (the GRE solution
satisfies the LMIs trivially) and yields
$V_a(t_0,X_1^0)=\ip{-P_\infty(0^+)}{X_1^0}$, so the infimum is
attained and the inequality is an equality.
\end{proof}

\paragraph{Minimum dwell-time.}
Adopt the timer-dependence hypothesis of Theorem~\ref{th:OC:MDT}:
$\Delta_k\ge T_{\min}>0$, the flow data depend only on the timer
$\tau=t-t_k$, and the time-invariance
condition~\eqref{eq:MDT:LTIassumption} holds for $\mathcal{J}$ and $M_k\equiv M_{\mathrm{jp}}$ and time-invariance of $\mathcal{F}$ and $M$ for $\tau\ge T_{\min}$.

\begin{theorem}[Sufficient causal condition: MDT dissipativity]
  \label{th:diss:MDT}
  Under \textup{(D1)--(D3)} and the MDT hypothesis above, suppose
  there exists $P_s:[0,T_{\min}]\to\Sn$ satisfying:
  \begin{enumerate}[\upshape(i)]
    \item \emph{Forward D-GRE on $(0,T_{\min})$:}
      $\mathcal{C}_M(\tau,P_s(\tau))=0$ for all
      $\tau\in(0,T_{\min})$.
    \item \emph{Jump D-GRE at $\tau=T_{\min}$:}
      $\mathcal{D}_M(P_s(0^+),P_s(T_{\min}))\succeq 0$ and
      $\mathcal{D}_M\,\schur{/}\,(\mathcal{D}_M)_3=0$.
    \item \emph{Geromel--Colaneri condition:}
      $\mathcal{C}_M^0(T_{\min},P_s(T_{\min})):=\mathcal{F}^*(T_{\min},P_s(T_{\min}))+M(T_{\min})\succeq 0$.
  \end{enumerate}
  Then the causal storage function
  $V(\tau,X):=\ip{-\bar P(\tau)}{EXE^*}$, with $\bar P$ the constant
  extension of $P_s$ for $\tau>T_{\min}$, satisfies the dissipation
  inequalities $\dot V\le\sigma_f$ on flow and $\Delta V\le\sigma_k$
  at each jump for all $\Delta_k\ge T_{\min}$, and provides the lower
  bound $V_a(t_0,X_1^0)\ge\ip{-P_s(0^+)}{X_1^0}$.  The condition is
  verifiable on $[0,T_{\min}]$ alone.
\end{theorem}

\begin{proof}
Extend $\bar P(\tau):=P_s(\tau)$ for $\tau\le T_{\min}$ and
$\bar P(\tau):=P_s(T_{\min})$ for $\tau>T_{\min}$.

\noindent\textbf{Flow dissipation.}
Differentiating $V=\ip{-\bar P}{EXE^*}$ along the flow gives
$\dot V=\sigma_f-\ip{\mathcal{C}_M(\tau,\bar P)}{X}$, so
$\dot V\le\sigma_f$ iff $\mathcal{C}_M(\tau,\bar P)\succeq 0$.  For
$\tau\in(0,T_{\min})$, $\mathcal{C}_M=0$ by~(i).  For
$\tau>T_{\min}$, $\dot{\bar P}=0$, so $\mathcal{C}_M(\tau,\bar P)
=\mathcal{C}_M^0(\tau,P_s(T_{\min}))\succeq 0$ by~(iii).

\noindent\textbf{Jump dissipation.}
$\Delta V\le\sigma_k$ iff $\mathcal{D}_M\succeq 0$, which holds for
$\Delta_k\ge T_{\min}$ since $\bar P(\Delta_k^-)=P_s(T_{\min})$ and
$\mathcal{D}_M(P_s(0^+),P_s(T_{\min}))\succeq 0$ by~(ii).

\noindent\textbf{Available storage bound.}
The dissipation inequalities along any admissible trajectory give
$V(\tau(t_0+T),EX(t_0+T)E^*)-V(0^+,X_1^0)\le\int\sigma_f\dt
+\sum\sigma_k$, and rearranging together with
$V\ge 0$ on $\Snpsd$ \textup{(}from $-\bar P\succeq 0$ implicit in the
infinite-horizon I-D-GRE\textup{)} yields
$\ip{-P_s(0^+)}{X_1^0}\le V_a(t_0,X_1^0)$.
\end{proof}

\begin{remark}\label{rem:diss:GC:sign}
The dissipativity Geromel--Colaneri condition~(iii) requires
$\mathcal{C}_M^0\succeq 0$ \textup{(}positive
semidefinite\textup{)}, in contrast to the optimal control
condition~\eqref{eq:GC:stability} which requires
$\mathcal{C}_Z^0\preceq 0$.  The sign reversal reflects the dual role
of $P$ on the frozen region: in the optimal control case
$P\succeq 0$ and the condition ensures that the cost matrix continues
to upper-bound the tail cost, whereas in the dissipativity case
$P_s$ may be indefinite and the condition ensures that the storage
matrix $-P_s$ continues to satisfy the dissipation inequality.
\end{remark}

\begin{theorem}[Dual DLMI, MDT dissipativity]
  \label{th:diss:MDT:LMI}
  Under \textup{(D1)--(D3)} and the MDT hypothesis of
  Theorem~\ref{th:diss:MDT},
  \begin{equation}\label{eq:diss:MDT:LMI}
    \begin{array}{rl}
    \min_P&\ip{-P(0^+)}{X_1^0}\\
    \text{s.t.}& P:[0,T_{\min}]\to\Sn,\;\text{bounded piecewise-}C^1,\\
      &\mathcal{C}_M(\tau,P(\tau))\succeq 0\;\text{a.e.\ on }(0,T_{\min}),\\
      &\mathcal{D}_M(P(0^+),P(T_{\min}))\succeq 0,\\
      &\mathcal{C}_M^0(T_{\min},P(T_{\min}))\succeq 0.
    \end{array}
  \end{equation}
  Any feasible $P$ defines a causal storage function
  $V_P(\tau,X)=\ip{-\bar P(\tau)}{EXE^*}$ with $\bar P$ the constant
  extension of $P$ to $[T_{\min},\infty)$.  The infimum is attained at
  $P_s$ of Theorem~\ref{th:diss:MDT} when the Schur equalities of the
  forward D-GRE on $[0,T_{\min}]$ and at $\tau=T_{\min}$ are imposed
  on top of the LMIs.
\end{theorem}

\begin{proof}
For any feasible $P$, the constant extension $\bar P$ satisfies
$\mathcal{C}_M(\tau,\bar P(\tau))\succeq 0$ on all of
$(0,\infty)$ (using
$\mathcal{C}_M(\tau,\bar P)=\mathcal{C}_M^0(\tau,P(T_{\min}))\succeq 0$
for $\tau>T_{\min}$) and $\mathcal{D}_M\succeq 0$ at every jump.
Hence $V_P(\tau,X):=\ip{-\bar P(\tau)}{EXE^*}$ is a causal storage
function by the same argument as in the proof of
Theorem~\ref{th:diss:MDT}.  Then
$V_a\le V_P$ by
Theorem~\ref{th:IMP:diss:inf:suff}\eqref{inf:suff:diss:largest}, yielding
$V_a(t_0,X_1^0)\le\ip{-P(0^+)}{X_1^0}$ and the infimum bound.  The
attainment at $P_s$ follows since $P_s$ is feasible and satisfies the
Schur equalities, which are the residual conditions distinguishing
the LMI feasibility set from the GRE solution.
\end{proof}

\paragraph{Range dwell-time.}
Adopt the timer-dependence hypothesis of Theorem~\ref{th:OC:RDT}:
$\Delta_k\in[T_{\min},T_{\max}]$ and the flow and jump data depend
only on the timer $\tau\in[0,T_{\max}]$;  $\mathcal{F}(\tau,\cdot)$ and $M(\tau)$ are defined on
$[0,T_{\max}]$, and $\mathcal{J}, M_k$ are $k$-independent\textup{)}.

\begin{theorem}[Sufficient causal condition: RDT dissipativity]
  \label{th:diss:RDT}
  Under \textup{(D1)--(D3)} and the RDT hypothesis above, suppose
  there exists $P_s:[0,T_{\max}]\to\Sn$ satisfying:
  \begin{enumerate}[\upshape(i)]
    \item $\mathcal{C}_M(\tau,P_s(\tau))=0$ for
      $\tau\in(0,T_{\max})$.
    \item $\mathcal{D}_M(P_s(0^+),P_s(\tau))\succeq 0$ and
      $\mathcal{D}_M\,\schur{/}\,(\mathcal{D}_M)_3=0$
      for all $\tau\in[T_{\min},T_{\max}]$.
  \end{enumerate}
  Then $V(\tau,X):=\ip{-P_s(\tau)}{EXE^*}$ is a causal storage
  function achieving $\dot V=\sigma_f$ on flow and $\Delta V=\sigma_k$
  at the optimal extraction policy for all
  $\Delta_k\in[T_{\min},T_{\max}]$, and provides the lower bound
  $V_a(t_0,X_1^0)\ge\ip{-P_s(0^+)}{X_1^0}$.
\end{theorem}

\begin{proof}
For any $\Delta_k\in[T_{\min},T_{\max}]$, condition~(i) gives
$\mathcal{C}_M=0$ on $(0,T_{\max})\supset(0,\Delta_k)$, so
$\dot V\le\sigma_f$ with equality at the optimal gain.  Condition~(ii)
gives $\mathcal{D}_M\succeq 0$ with Schur zero at $\tau=\Delta_k$, so
$\Delta V\le\sigma_k$ with equality.  Since the flow D-GRE holds with
equality on all of $[0,T_{\max}]$, no Geromel--Colaneri tail extension
is needed.  The available storage bound follows as in the proof of
Theorem~\ref{th:diss:MDT}.
\end{proof}

\begin{theorem}[Dual DLMI, RDT dissipativity]
  \label{th:diss:RDT:LMI}
  Under \textup{(D1)--(D3)} and the RDT hypothesis of
  Theorem~\ref{th:diss:RDT},
    \begin{equation}\label{eq:diss:RDT:LMI}
    \begin{array}{rl}
    \min_P&\ip{-P(0^+)}{X_1^0}\\
    \text{s.t.}& P:[0,T_{\max}]\to\Sn,\;\text{bounded piecewise-}C^1,\\
    &\mathcal{C}_M(\tau,P(\tau))\succeq 0\;\text{a.e.\ on }(0,T_{\max}),\\
    &\mathcal{D}_M(P(0^+),P(\tau))\succeq 0\;\text{for }\tau\in[T_{\min},T_{\max}],\\
    \end{array}
  \end{equation}
  Any feasible $P$ defines a causal storage function on
  $[0,T_{\max}]$.  The infimum is attained at $P_s$ of
  Theorem~\ref{th:diss:RDT} when the Schur equalities of the forward
  D-GRE on $[0,T_{\max}]$ and on $[T_{\min},T_{\max}]$ are imposed
  on top of the LMIs.
\end{theorem}

\begin{proof}
Identical to that of Theorem~\ref{th:diss:MDT:LMI}, with the timer
interval $[0,T_{\max}]$ in place of $[0,T_{\min}]$ and the jump LMI
$\mathcal{D}_M\succeq 0$ holding for all $\tau\in[T_{\min},T_{\max}]$.
No tail extension is needed since the flow LMI holds on the entire
timer interval.
\end{proof}


\subsubsection{Continuous- and Discrete-Time Corollaries}

\begin{corollary}[Continuous-time dissipativity: infinite horizon]
  \label{cor:CT:diss:inf}
  Set $N_T=0$.  Theorem~\ref{th:IMP:diss:inf} gives the unique
  bounded $C^1$ solution $P_\infty$ of the continuous-time infinite-horizon
  D-GRE, with $V_a(t_0,X_1^0)=\ip{-P_\infty(t_0)}{X_1^0}$,
  $P_\infty=\lim_T P_T$ (the limit need not be monotone for indefinite $M$), and the
  smallest-storage-function sandwich property.
  The dual is $V_a=\inf_P\ip{-P(t_0)}{X_1^0}$ s.t.\
  $\mathcal{C}_M(t,P(t))\succeq0$, $P$ bounded.
\end{corollary}

\begin{corollary}[Discrete-time dissipativity: infinite horizon]
  \label{cor:DT:diss:inf}
  Apply Theorem~\ref{th:IMP:diss:inf} with $\mathcal{F}=0$,
  $M=0$ on flow, $t_k=k$, $\mathcal{J}=\mathcal{D}$.
  The infinite-horizon discrete-time available storage is
  $V_a(k_0,X_1^0)=\ip{-P_\infty(k_0)}{X_1^0}$, where $P_\infty$
  satisfies the algebraic discrete-time D-GRE and is the
  limit of the $N$-step solutions.
  The dual is $V_a=\inf_P\ip{-P(k_0)}{X_1^0}$ s.t.\
  $\mathcal{D}_M(k,P^+,P(t))\succeq0$, $P$ bounded.
\end{corollary}

\subsection{Long-term average available storage}
\label{sec:diss:AC}

We now develop the dissipativity counterpart of Section~\ref{sec:OC:AC}.
The setup parallels that subsection: the dynamics are augmented by a
Hermitian forcing on the state block as in
Definition~\ref{def:IMP:forced}, and we study the long-term rate of
energy extraction rather than the cumulative available storage.  The
sign conventions of Section~\ref{sec:diss} are retained: the storage
function takes the form $V=\ip{-P}{EXE^*}$ and the available storage
is the \emph{supremum} of the extractable supply, in contrast to the
\emph{infimum} of the cost in the optimal-control case.  The
analysis proceeds in three steps: the forced key identity acquires
linear correction terms in $\mathcal{W},\mathcal{W}_k$
(Section~\ref{sec:diss:AC:keyid}); the average available-storage
rate is identified through a linear pairing with the forcing data
(Sections~\ref{sec:diss:AC:inf:suff}--\ref{sec:diss:AC:inf:nec}); and
causal dwell-time bounds follow by the same constructions as in
Section~\ref{sec:OC:AC:dwell}.

\subsubsection{Setting, supply rate, and assumptions}
\label{sec:diss:AC:setup}

The forced impulsive system is that of
Definition~\ref{def:IMP:forced}, with $\mathcal{W}(t)\in\Snpsd$ and
$\mathcal{W}_k\in\Snpsd$ bounded.  The supply rate is the linear
supply of Definition~\ref{def:IMP:supply},
$\sigma_f(t,X)=\ip{M(t)}{X}$, $\sigma_k(X)=\ip{M_k}{EXE^*}$,
$\sigma_T=0$ (no terminal supply, as we work directly in the
infinite-horizon setting).  The horizon-$T$ available storage is
\begin{equation}\label{eq:diss:AC:VaT}
  V_a^T(t_0,X_1^0)
  :=\sup_{X(\cdot)\in\Snmpsd}\,
   \left\{-\int_{t_0}^{t_0+T}\!\ip{M(t)}{X(t)}\dt
    -\!\!\sum_{k:\,t_k\le t_0+T}\!\!\ip{M_k}{EX(t_k^-)E^*}\right\}
\end{equation}
subject to~\eqref{eq:sys:IMP:forced} and $EX(t_0)E^*=X_1^0$.

\begin{definition}[Long-term average available storage rate]
  \label{def:IMP:Va:AC}
The \emph{long-term average available storage rate} is
\begin{equation}\label{eq:diss:AC:Va}
  \bar V_a^\star(X_1^0)
  :=\limsup_{T\to\infty}\frac{1}{T}\,V_a^T(t_0,X_1^0).
\end{equation}
\end{definition}

\medskip\noindent\textbf{Assumptions.}
The well-posedness of the long-term average available-storage
problem requires three assumptions on the homogeneous dynamics and
the forcing data.  These are the dissipativity counterparts of
(A1),~(A3),~(A4) of Section~\ref{sec:OC:AC:setup}, renumbered
consecutively here: no analogue of the cost-coercivity
condition~(A2) of Section~\ref{sec:OC:AC:setup} is required, since
dissipativity does not need a positive lower bound on the supply
rate~$M$ (which may be indefinite).

\medskip\noindent\textbf{(A1) Uniform stabilizability of the
homogeneous dynamics.}
There exist a feedback $K_s\in\C^{m\times n}$, constants
$\beta\ge1$, $\rho\in(0,1)$ and $\alpha>0$, such that the homogeneous
closed-loop trajectory $X_s^{\mathrm h}$ of~\eqref{eq:sys:IMP} (i.e.,
$\mathcal{W}=0$, $\mathcal{W}_k=0$) with feedback
$X_2^s=K_sX_1^s$, $X_3^s=K_sX_1^sK_s^*$ satisfies
$\norm{EX_s^{\mathrm h}(t)E^*}\le\beta e^{-\alpha(t-s)}\rho^{\kappa(t,s)}
\norm{EX_s^{\mathrm h}(s)E^*}$ for all $t\ge s\ge t_0$.

\medskip\noindent\textbf{(A2) Uniform boundedness of the operators
and supply rates.}
$M,M_k,\mathcal{F}(t,\cdot),\mathcal{J}(k,\cdot)$ are uniformly
bounded; $(\mathcal{C}_M)_3(t,P(t))\succeq0$ and
$(\mathcal{D}_M)_3(k,P(t_k^+),P(t_k^-))\succeq0$ for all
$P\succeq0$; $\pinv{(\mathcal{C}_M)_3}$ and
$\pinv{(\mathcal{D}_M)_3}$ are uniformly bounded on bounded sets of
$P$; and $\mathcal{W}(t),\mathcal{W}_k$ are uniformly bounded.

\medskip\noindent\textbf{(A3) Jump-rate regularity.}
The inter-jump times $\Delta_k=t_{k+1}-t_k$ are bounded both above
and below by positive constants: there exist $0<\Delta_{\min}\le
\Delta_{\max}<\infty$ such that $\Delta_k\in[\Delta_{\min},
\Delta_{\max}]$ for all $k\ge0$.  In particular, the jump rate
$\rho_{\mathrm{rate}}:=\lim_{T\to\infty}T^{-1}\#\{k:t_k\le t_0+T\}$
exists and is finite, with $\Delta_{\max}^{-1}\le\rho_{\mathrm{rate}}
\le\Delta_{\min}^{-1}$.

\medskip
Assumption~(A1) is identical to its optimal-control counterpart:
the dual decay $\beta e^{-\alpha(t-s)}\rho^{\kappa(t,s)}$ absorbs
both flow and jump contributions to the closed-loop decay rate,
and is exactly the condition that guarantees existence of a bounded
solution $P_\infty$ to the homogeneous infinite-horizon I-D-GRE by
Theorem~\ref{th:IMP:diss:inf}.  Assumption~(A2) extends the original
dissipativity boundedness hypotheses of Section~\ref{sec:diss:inf}
to include the forcing data $\mathcal{W},\mathcal{W}_k$.

Assumption~(A3) plays a different role from~(A1)--(A2).
Assumption~(A1), through the dual decay rate
$\beta e^{-\alpha(t-s)}\rho^{\kappa(t,s)}$, already ensures that the
homogeneous closed-loop trajectory decays robustly under both small
and large inter-jump times: small $\Delta_k$ are absorbed by the
geometric factor $\rho^{\kappa(t,s)}$ (each jump contributing decay),
while large $\Delta_k$ are absorbed by the continuous factor
$e^{-\alpha(t-s)}$ (the flow being independently stable).
Consequently, state stability and the summability of the jump-supply
contribution along stabilizing trajectories follow from~(A1) alone,
without any direct constraint on the $\Delta_k$.  The role of~(A3)
is purely to make the \emph{time-averaged} long-term available
storage well-defined as a genuine limit rather than a $\limsup$: the
existence of the jump rate $\rho_{\mathrm{rate}}$ ensures that the
discrete jump-supply contribution admits a long-run average, and
bounded $\Delta_k$ ensures that the supply rates $\mathcal{W},
\mathcal{W}_k$ are properly time-averaged.  Theorems giving only a
lower bound on the long-term available storage invoke
only~(A1)--(A2); theorems giving exact long-term equalities
additionally require~(A3).

\subsubsection{Forced key identity (dissipativity)}
\label{sec:diss:AC:keyid}

The forced key identity of Proposition~\ref{prop:keyid:forced}
applies with $Z\leftarrow M$, $Z_k\leftarrow M_k$, giving, for any
piecewise-$C^1$ $P:[t_0,t_0+T]\to\Sn$ and any admissible trajectory
$X$ of~\eqref{eq:sys:IMP:forced},
\begin{equation}\label{eq:diss:AC:keyid}
\begin{aligned}
  &-\int_{t_0}^{t_0+T}\!\ip{M(t)}{X(t)}\dt
  -\!\!\sum_{k:\,t_k\le t_0+T}\!\!\ip{M_k}{EX(t_k^-)E^*}\\
  &\quad
  =\ip{-P(t_0)}{X_1^0}-\ip{-P(t_0+T)}{EX(t_0+T)E^*}\\
  &\quad\;\;
  -\int_{t_0}^{t_0+T}\!\ip{\mathcal{C}_M(t,P(t))}{X(t)}\dt
  -\!\!\sum_{k:\,t_k\le t_0+T}\!\!\ip{\mathcal{D}_M(k,P(t_k^+),P(t_k^-))}{X(t_k^-)}\\
  &\quad\;\;
  -\int_{t_0}^{t_0+T}\!\ip{P(t)}{\mathcal{W}(t)}\dt
  -\!\!\sum_{k:\,t_k\le t_0+T}\!\!\ip{P(t_k^+)}{\mathcal{W}_k}.
\end{aligned}
\end{equation}
This is~\eqref{eq:IMP:keyid:forced} multiplied by $-1$ and rearranged,
with the storage-function sign convention $V=\ip{-P}{EXE^*}$ made
explicit on the right-hand side.

\subsubsection{Finite-horizon average available storage}
\label{sec:diss:AC:fh}

\begin{theorem}[Finite-horizon average available storage decomposition]
  \label{th:diss:AC:finite}
Let $T>0$ and let $P_T$ be the piecewise-$C^1$ solution of the
finite-horizon I-D-GRE (Definition~\ref{def:IMP:dGRE}) with
$P_T(t_0+T)=0$.  Then the horizon-$T$ available
storage~\eqref{eq:diss:AC:VaT} satisfies
\begin{equation}\label{eq:diss:AC:finite:Va}
  V_a^T(t_0,X_1^0)
  =\ip{-P_T(t_0)}{X_1^0}
  -\int_{t_0}^{t_0+T}\!\ip{P_T(t)}{\mathcal{W}(t)}\dt
  -\!\!\sum_{k:\,t_k\le t_0+T}\!\!\ip{P_T(t_k^+)}{\mathcal{W}_k},
\end{equation}
and the horizon-$T$ time-averaged rate is
$\bar V_a^T(t_0,X_1^0):=V_a^T/T$.  The supremum is attained at
$X^\star$ with $K(t)\in\mathcal{K}_c^M(t,P_T(t))$ on flow and
$K_k\in\mathcal{K}_d^M(k,P_T(t_k^+),P_T(t_k^-))$ at each jump.
\end{theorem}

\begin{proof}
Apply~\eqref{eq:diss:AC:keyid} with $P=P_T$.  Since $P_T$ solves the
I-D-GRE, $\mathcal{C}_M(t,P_T(t))\succeq0$ and
$\mathcal{D}_M(k,P_T(t_k^+),P_T(t_k^-))\succeq0$ with both Schur
complements vanishing (Lemma~\ref{lemma:GREequiv}).  The Schur
inner-product decomposition (Lemma~\ref{lemma:decomp}) gives
$\ip{\mathcal{C}_M(t,P_T(t))}{X(t)}\ge0$ and
$\ip{\mathcal{D}_M(k,P_T(t_k^+),P_T(t_k^-))}{X(t_k^-)}\ge0$ for every
admissible $X$, with equality at any $X^\star$ with feedback in
$\mathcal{K}_c^M,\mathcal{K}_d^M$.  Using $P_T(t_0+T)=0$ (so the
second boundary term in~\eqref{eq:diss:AC:keyid} vanishes) and that
the last two terms are independent of the free blocks $(X_2,X_3)$,
the supremum in~\eqref{eq:diss:AC:VaT} is attained at any such
$X^\star$ and equals~\eqref{eq:diss:AC:finite:Va}.
\end{proof}

The decomposition~\eqref{eq:diss:AC:finite:Va} separates
$\bar V_a^T$ into a \emph{transient term}
$T^{-1}\ip{-P_T(t_0)}{X_1^0}$ that vanishes as $T\to\infty$ and
\emph{forcing terms} whose $T\to\infty$ limit gives the long-term
available storage rate; this is the content of the next subsection.

\subsubsection{Infinite-horizon: sufficient condition}
\label{sec:diss:AC:inf:suff}

The following theorem is the central sufficient condition.  It says
that any bounded solution of the \emph{homogeneous} infinite-horizon
I-D-GRE (Section~\ref{sec:diss:inf}) determines the long-term
available storage rate through a linear pairing with the forcing
data.  No new Riccati-type equation is required.

\begin{theorem}[Sufficient condition, infinite-horizon average
  available storage]\label{th:diss:AC:suff}
Suppose there exists a bounded piecewise-$C^1$ function
$P_\infty:[t_0,\infty)\to\Snpsd$ satisfying the infinite-horizon
I-D-GRE (Section~\ref{sec:diss:inf}), and assume further that the
limit
\begin{equation}\label{eq:diss:AC:Vastar}
  \bar V_a^\star
  :=\lim_{T\to\infty}\frac{1}{T}\!\left[
    -\int_{t_0}^{t_0+T}\!\ip{P_\infty(t)}{\mathcal{W}(t)}\dt
    -\!\!\sum_{k:\,t_k\le t_0+T}\!\!\ip{P_\infty(t_k^+)}{\mathcal{W}_k}
  \right]
\end{equation}
exists.  Then\textup{:}
\begin{enumerate}[\upshape(a)]
  \item\label{diss:AC:suff:val}
    The long-term average available storage rate satisfies
    $\bar V_a^\star(X_1^0)=\bar V_a^\star$ for every
    $X_1^0\in\Snpsd$, i.e., the rate is independent of the initial
    condition.
  \item\label{diss:AC:suff:gain}
    The supremum is attained at any trajectory $X^\star$ with
    $K_\infty(t)\in\mathcal{K}_c^M(t,P_\infty(t))$ on flow and
    $K_k\in\mathcal{K}_d^M(k,P_\infty(t_k^+),P_\infty(t_k^-))$ at
    each jump.
  \item\label{diss:AC:suff:uniq}
    $P_\infty$ is the unique bounded solution of the infinite-horizon
    I-D-GRE.
\end{enumerate}
\end{theorem}

\begin{proof}
\noindent\textbf{Upper bound.}
Let $X$ be any admissible trajectory of~\eqref{eq:sys:IMP:forced}.
Apply~\eqref{eq:diss:AC:keyid} on $[t_0,t_0+T]$ with $P=P_\infty$.
Since $P_\infty$ solves the I-D-GRE,
$\mathcal{C}_M(t,P_\infty(t))\succeq0$ and
$\mathcal{D}_M(k,P_\infty(t_k^+),P_\infty(t_k^-))\succeq0$.  By
Lemma~\ref{lemma:decomp}, both inner products are non-negative, so
\begin{equation}\label{pf:diss:AC:suff:ub}
\begin{aligned}
  &  -\int_{t_0}^{t_0+T}\ip{M(t)}{X(t)}\dt-\sum_{k:\,t_k\le t_0+T}\ip{M_k}{EX(t_k^-)E^*}\\
  &\qquad\le\ip{-P_\infty(t_0)}{X_1^0}-\ip{-P_\infty(t_0+T)}{EX(t_0+T)E^*}-R_T(P_\infty),
\end{aligned}
\end{equation}
where
\[
  R_T(P_\infty)
  :=\int_{t_0}^{t_0+T}\!\ip{P_\infty(t)}{\mathcal{W}(t)}\dt
   +\sum_{k:\,t_k\le t_0+T}\!\ip{P_\infty(t_k^+)}{\mathcal{W}_k}.
\]
Dividing by $T$, and noting that
$T^{-1}\ip{-P_\infty(t_0)}{X_1^0}\to0$ as $T\to\infty$ (both factors
bounded), and that $T^{-1}\ip{-P_\infty(t_0+T)}{EX(t_0+T)E^*}\to0$
provided $\norm{EX(t_0+T)E^*}$ is bounded (which holds along any
admissible trajectory under (A1)--(A2) by
Lemma~\ref{lem:diss:AC:bdd} below):
\[
  \limsup_{T\to\infty}\frac{1}{T}V_a^T(t_0,X_1^0;X)
  \le\lim_{T\to\infty}\frac{1}{T}\bigl(-R_T(P_\infty)\bigr)
  =\bar V_a^\star.
\]
Taking the supremum over admissible $X$ yields
$\bar V_a^\star(X_1^0)\le\bar V_a^\star$.

\noindent\textbf{Attainment.}
Let $X^\star$ be the trajectory with gains
$K_\infty(t)\in\mathcal{K}_c^M(t,P_\infty(t))$ on flow and
$K_k\in\mathcal{K}_d^M(k,P_\infty(t_k^+),P_\infty(t_k^-))$ at each
jump.  By the attainment argument of
Theorem~\ref{th:IMP:diss:inf:suff} (which relies only on the I-D-GRE
and not on the homogeneous form of the dynamics),
$\ip{\mathcal{C}_M(t,P_\infty(t))}{X^\star(t)}=0$ a.e.\ and
$\ip{\mathcal{D}_M(k,P_\infty(t_k^+),P_\infty(t_k^-))}{X^\star(t_k^-)}=0$
for every $k$.  Substituting into~\eqref{eq:diss:AC:keyid} gives
\begin{align*}
  &  -\int_{t_0}^{t_0+T}\ip{M(t)}{X(t)}\dt-\sum_{k:\,t_k\le t_0+T}\ip{M_k}{EX(t_k^-)E^*}\\
  &\qquad=\ip{-P_\infty(t_0)}{X_1^0}-\ip{-P_\infty(t_0+T)}{EX(t_0+T)E^*}-R_T(P_\infty),
\end{align*}
Dividing by $T$ and letting $T\to\infty$, the boundary terms vanish
(boundedness and $1/T$), and $-R_T(P_\infty)/T\to\bar V_a^\star$ by
hypothesis.  Hence
$\lim_{T\to\infty}V_a^T(t_0,X_1^0;X^\star)/T=\bar V_a^\star$, the
supremum is achieved and $\bar V_a^\star(X_1^0)=\bar V_a^\star$,
proving~\eqref{diss:AC:suff:val} and~\eqref{diss:AC:suff:gain}.

\noindent\textbf{Uniqueness of $P_\infty$.}
If $\tilde P$ is another bounded solution of the infinite-horizon
I-D-GRE, the same argument gives
$\bar V_a^\star=\lim_T-R_T(\tilde P)/T$ as well.  But $\tilde
P\equiv P_\infty$ already by Theorem~\ref{th:IMP:diss:inf:suff} (the
homogeneous infinite-horizon I-D-GRE has a unique bounded solution).
This proves~\eqref{diss:AC:suff:uniq}.
\end{proof}

The bounded-trajectory claim used in the upper bound is recorded as
the following lemma, the dissipativity counterpart of
Lemma~\ref{lem:IMP:AC:bdd}.

\begin{lemma}[Bounded second moment under forcing, dissipativity]
  \label{lem:diss:AC:bdd}
Under (A1) and (A2), for any admissible trajectory $X(\cdot)$
of~\eqref{eq:sys:IMP:forced} with stabilizing feedback $K_s$ from
(A1), there exists a constant $C>0$ (depending only on $K_s$,
$\beta,\alpha,\rho$ from (A1), and the bounds in (A2), and
$\norm{X_1^0}$) such that $\norm{EX(t)E^*}\le C$ for all $t\ge t_0$.
\end{lemma}

\begin{proof}
Identical to Lemma~\ref{lem:IMP:AC:bdd}: variation-of-constants on
the closed-loop flow-jump system, using the dual decay
$\beta e^{-\alpha(t-s)}\rho^{\kappa(t,s)}$ of (A1), gives
\[
  \norm{EX(t)E^*}
  \le\beta\norm{X_1^0}
  +\beta\alpha^{-1}\sup_s\norm{\mathcal{W}(s)}
  +\beta(1-\rho)^{-1}\sup_k\norm{\mathcal{W}_k},
\]
where the $(1-\rho)^{-1}$ factor comes from the convergent geometric
series in $\rho^{\kappa}$ on the jump-sum term.  No bound on
$\Delta_k$ is needed.
\end{proof}

\subsubsection{Necessary condition and existence (infinite horizon)}
\label{sec:diss:AC:inf:nec}

In contrast to the optimal-control setting of
Section~\ref{sec:OC:AC:inf:suff}, where (A1)--(A2) suffice for
existence of an infinite-horizon dual variable through the
coercivity of the cost, the dissipativity setting requires an
additional structural hypothesis: the system itself must be
dissipative with respect to the supply rate $M$.  Stability of the
closed-loop homogeneous dynamics under (A1) does \emph{not} imply
dissipativity, because the supply rate $M$ may have an arbitrary
sign structure.  We therefore take dissipativity (in the form of
uniform finiteness of the homogeneous available storage) as the
starting hypothesis, and derive existence of $P_\infty$ together
with the long-term average available-storage formula.  This is the
genuine necessity direction: the existence of a solution to the
infinite-horizon I-D-GRE is the \emph{conclusion}, not an
assumption.

\begin{theorem}[Necessary condition and existence, infinite-horizon
  average available storage]\label{th:diss:AC:nec}
Suppose the homogeneous system \textup{(}with $\mathcal{W}=0$,
$\mathcal{W}_k=0$\textup{)} is dissipative with respect to $M$, in
the sense that
\begin{equation}\label{eq:diss:AC:nec:dissip}
  \sup_{T>0}V_a^T(t_0,X_1^0)<\infty
  \qquad\text{for every }X_1^0\in\Snpsd.
\end{equation}
Then, under (A1) and (A2):
\begin{enumerate}[\upshape(a)]
  \item\label{diss:AC:nec:exist}
    \emph{Existence and uniqueness.} There exists a unique bounded
    piecewise-$C^1$ solution
    $P_\infty:[t_0,\infty)\to\Snpsd$,
    $\alpha_1 I\preceq P_\infty\preceq\alpha_2 I$, of the
    infinite-horizon I-D-GRE; $-P_\infty$ is the smallest storage
    function.

  \item\label{diss:AC:nec:conv}
    \emph{Convergence.} $-P_\infty$ is the Loewner least upper
    bound of $\{-P_T\}$; under the additional structural condition below,
    $-P_T\to-P_\infty$ pointwise on $[t_0,\infty)$ and uniformly on
    compact subintervals.  The convergence is then monotone
    non-decreasing under the additional
    structural condition that the homogeneous system admits, at
    every reachable state, a non-positive-supply continuation
    \textup{(}a sufficient condition is that the homogeneous I-D-GRE
    flow admits an equilibrium; this holds in the LTI case and in
    the periodic case\textup{)}.

  \item\label{diss:AC:nec:formula}
    \emph{Long-term formula.}
    \begin{equation}\label{eq:diss:AC:nec:val}
      \bar V_a^\star(X_1^0)
      =\limsup_{T\to\infty}\frac{1}{T}\!\left[
        -\int_{t_0}^{t_0+T}\!\ip{P_\infty(t)}{\mathcal{W}(t)}\dt
        -\!\!\sum_{k:\,t_k\le t_0+T}\!\!\ip{P_\infty(t_k^+)}{\mathcal{W}_k}
      \right]
    \end{equation}
    independently of $X_1^0$.  The right-hand side is finite
    whenever $\limsup_T T^{-1}\#\{k:t_k\le t_0+T\}<\infty$.

  \item\label{diss:AC:nec:attain}
    \emph{Attainment.} The supremum is attained at the closed-loop
    extraction trajectory $X^\star$ with gains
    $K_\infty(t)\in\mathcal{K}_c^M(t,P_\infty(t))$ on flow and
    $K_k\in\mathcal{K}_d^M(k,P_\infty(t_k^+),P_\infty(t_k^-))$ at
    each jump.
\end{enumerate}
\end{theorem}

\begin{proof}
\textbf{Step 1: Existence of $P_\infty$ from dissipativity.}
The dissipativity hypothesis~\eqref{eq:diss:AC:nec:dissip} is
exactly the homogeneous-dissipativity premise of
Theorem~\ref{th:IMP:diss:inf}.  Combined with the boundedness of the
operators from (A2), that theorem provides a unique bounded
piecewise-$C^1$ solution $P_\infty$ with
$\alpha_1 I\preceq P_\infty\preceq\alpha_2 I$ of the homogeneous
infinite-horizon I-D-GRE; $-P_\infty$ is smallest storage
function by construction (Willems-type characterisation).  Note
that (A1) alone (closed-loop stability of the homogeneous dynamics)
is not used in this step: existence of the storage function is a
property of the supply rate $M$ relative to the dynamics, not of
stability.

\textbf{Step 2: Convergence of $-P_T$ to $-P_\infty$.}
Theorem~\ref{th:IMP:diss:inf} establishes that $-P_\infty$ is the
Loewner least upper bound of the family $\{-P_T\}$. Monotone non-decreasing
convergence requires
$T\mapsto V_a^T(t_0,X_1^0)$ to be monotone non-decreasing for every
$X_1^0$, equivalently that any horizon-$T$ trajectory can be
prolonged to a horizon-$T'$ trajectory ($T'>T$) without strictly
positive incremental supply.  A sufficient structural condition is
the existence of an equilibrium of the homogeneous I-D-GRE
flow at $P_s(0^+)$, which is automatic in the LTI case and in the
periodic case (Theorem~\ref{th:diss:AC:dwell:periodic}).  In the
general LTV case, $-P_T$ converges but the convergence need not be
monotone.

\textbf{Step 3: Long-term average formula.}
For any admissible trajectory $X$ of~\eqref{eq:sys:IMP:forced},
identity~\eqref{eq:diss:AC:keyid} with $P=P_\infty$ on
$[t_0,t_0+T]$ reads
\begin{align*}
  -\!\!\int_{t_0}^{t_0+T}\!\!\ip{M(t)}{X(t)}\dt
  &-\!\!\sum_{k:\,t_k\le t_0+T}\!\!\ip{M_k}{EX(t_k^-)E^*}\\
  &=\ip{-P_\infty(t_0)}{X_1^0}-\ip{-P_\infty(t_0+T)}{EX(t_0+T)E^*}\\
  &\quad
   -\!\!\int_{t_0}^{t_0+T}\!\!\ip{\mathcal{C}_M(t,P_\infty)}{X}\dt
   -\!\!\sum_{k:\,t_k\le t_0+T}\!\!\ip{\mathcal{D}_M}{X(t_k^-)}
   -R_T(P_\infty),
\end{align*}
with $R_T(P_\infty):=\int_{t_0}^{t_0+T}\!\ip{P_\infty(t)}{\mathcal{W}(t)}\dt
+\sum_{k:\,t_k\le t_0+T}\!\ip{P_\infty(t_k^+)}{\mathcal{W}_k}$.
Since $P_\infty$ solves the infinite-horizon I-D-GRE,
$\mathcal{C}_M(t,P_\infty)\succeq0$ and $\mathcal{D}_M\succeq0$, so
Lemma~\ref{lemma:decomp} gives $\ip{\mathcal{C}_M}{X(t)}\ge0$ a.e.\
and $\ip{\mathcal{D}_M}{X(t_k^-)}\ge0$, with equality along the
closed-loop extraction trajectory $X^\star$.  Hence
\begin{equation}\label{eq:diss:AC:nec:bound}
\begin{aligned}
  &  -\int_{t_0}^{t_0+T}\ip{M(t)}{X(t)}\dt-\sum_{k:\,t_k\le t_0+T}\ip{M_k}{EX(t_k^-)E^*}\\
  &\qquad\le\ip{-P_\infty(t_0)}{X_1^0}-\ip{-P_\infty(t_0+T)}{EX(t_0+T)E^*}-R_T(P_\infty),
\end{aligned}
\end{equation}
with equality at $X^\star$.

\textbf{Step 4: Boundary term is $O(1)$ in $T$ (uses (A1)).}
Lemma~\ref{lem:diss:AC:bdd} bounds $\|EX(t)E^*\|\le C_X<\infty$
uniformly in $t$ along any admissible trajectory with the
stabilizing feedback of (A1); its proof uses the geometric factor
$\rho^j$ from the dual decay in (A1) and never invokes a bound on
$\Delta_k$.  Since $P_\infty$ is bounded by $\alpha_2$,
$|\ip{-P_\infty(t_0+T)}{EX(t_0+T)E^*}|\le\alpha_2 C_X$ uniformly
in $T$.

\textbf{Step 5: Conclusion.}
Dividing~\eqref{eq:diss:AC:nec:bound} by $T$ and taking
$\limsup_{T\to\infty}$, both boundary terms vanish, leaving
$\limsup_T T^{-1}V_a^T(t_0,X_1^0;X)\le\limsup_T(-R_T(P_\infty))/T$.
Taking the supremum over admissible $X$ gives
$\bar V_a^\star(X_1^0)\le\limsup_T(-R_T(P_\infty))/T$.  The
closed-loop extraction trajectory $X^\star$
saturates~\eqref{eq:diss:AC:nec:bound}, so the reverse inequality
holds and equality is~\eqref{eq:diss:AC:nec:val}.  Independence of
$X_1^0$ follows since the right-hand side does not depend on
$X_1^0$.

\textbf{Step 6: Finiteness.}
$|R_T(P_\infty)|\le\alpha_2(T\sup_t\|\mathcal{W}(t)\|
+\#\{k:t_k\le t_0+T\}\cdot\sup_k\|\mathcal{W}_k\|)$, so
$T^{-1}|R_T(P_\infty)|$ is bounded above whenever
$\limsup_T T^{-1}\#\{k:t_k\le t_0+T\}<\infty$.
\end{proof}

\begin{remark}\label{rem:diss:AC:nec}
Formula~\eqref{eq:diss:AC:nec:val} characterises $\bar V_a^\star$ as
a $\limsup$ regardless of any inter-jump regularity.  A genuine
$\lim$ is recovered under additional structural assumptions on
$(P_\infty,\mathcal{W},\mathcal{W}_k)$, the cleanest case being
$T_p$-periodicity of the data
(Theorem~\ref{th:diss:AC:dwell:periodic}) where both contributions
of $R_T(P_\infty)/T$ converge by genuine periodic averaging.  In the
LTI specialisation (Corollary~\ref{cor:CT:diss:AC}), $P_\infty$ is
constant and the formula collapses to
$\bar V_a^\star=\tr(-P_\infty\mathcal{W})$, again a true limit.
This is the matrix-valued analogue of the classical steady-state
output-covariance / $H_2$-norm formula for closed-loop systems with
stable closed-loop dynamics.
\end{remark}

\subsubsection{Equivalence and dual DLMI (infinite horizon)}
\label{sec:diss:AC:LMI}

The equivalence statement and dual DLMI for the average available
storage are direct consequences of Theorems~\ref{th:diss:AC:suff}
and~\ref{th:diss:AC:nec}.

\begin{corollary}[Equivalence, infinite-horizon average available
  storage]\label{cor:diss:AC:equiv}
The infinite-horizon average available-storage
problem~\eqref{eq:diss:AC:Va} under~\eqref{eq:sys:IMP:forced} is
well-posed with $\bar V_a^\star(X_1^0)=\bar V_a^\star$ given
by~\eqref{eq:diss:AC:Vastar} if and only if the homogeneous
infinite-horizon I-D-GRE has a bounded piecewise-$C^1$ solution
$P_\infty\succeq0$.  The solution is unique.
\end{corollary}

\begin{proof}
Sufficiency: Theorem~\ref{th:diss:AC:suff}.
Necessity: Theorem~\ref{th:diss:AC:nec}.
Uniqueness: Theorem~\ref{th:diss:AC:suff}\eqref{diss:AC:suff:uniq}.
\end{proof}

The DLMI characterisation expresses $\bar V_a^\star$ as the infimum
of a linear-in-$P$ functional over the DLMI feasibility set.  Define
the \emph{negative-forcing pairing}
\begin{equation}\label{eq:diss:AC:Phi}
  \bar\Phi^M(P)
  :=\limsup_{T\to\infty}\frac{1}{T}\!\left[
    -\int_{t_0}^{t_0+T}\!\ip{P(t)}{\mathcal{W}(t)}\dt
    -\!\!\sum_{k:\,t_k\le t_0+T}\!\!\ip{P(t_k^+)}{\mathcal{W}_k}\right],
\end{equation}
which is well-defined and finite for any bounded $P$ under (A2).
The functional $\bar\Phi^M:\Sn\text{-valued bounded functions}\to\R$
is linear and \emph{order-reversing}: if $P\preceq P'$ pointwise
then $\bar\Phi^M(P')\le\bar\Phi^M(P)$ (since $\mathcal{W}(t)\succeq0$
and $\mathcal{W}_k\succeq0$).

\begin{theorem}[Dual DLMI, infinite-horizon average available storage]
  \label{th:diss:AC:LMI}
Under (A1)--(A3),
\begin{equation}\label{eq:diss:AC:LMI}
  \bar V_a^\star
  =\inf_P\;\bar\Phi^M(P)
  \quad\text{s.t.}\quad
  \mathcal{C}_M(t,P(t))\succeq0,\;
  \mathcal{D}_M(k,P(t_k^+),P(t_k^-))\succeq0,\;
  P\text{ bounded};
\end{equation}
the infimum is attained at $P_\infty$.
\end{theorem}

\begin{proof}
\noindent\textbf{Step 1: every feasible $P$ is an upper bound for
$\bar V_a^\star$.}
Let $P$ be any feasible function for~\eqref{eq:diss:AC:LMI}.  By the
proof of Theorem~\ref{th:IMP:diss:inf:suff}, $P\preceq P_\infty$
pointwise (the infinite-horizon I-D-GRE solution is the largest
feasible solution; equivalently $-P_\infty$ is the smallest storage
function).  By the order-reversal of $\bar\Phi^M$,
$\bar\Phi^M(P_\infty)\le\bar\Phi^M(P)$ for every feasible $P$, so
$\inf_P\bar\Phi^M(P)\ge\bar\Phi^M(P_\infty)$.

A direct argument that does not invoke the comparison
$P\preceq P_\infty$ proceeds as follows.  For any admissible
trajectory $X$ of~\eqref{eq:sys:IMP:forced} and any feasible $P$,
identity~\eqref{eq:diss:AC:keyid} combined with the LMI constraints
$\mathcal{C}_M\succeq0$ and $\mathcal{D}_M\succeq0$ gives
\begin{equation*}
\begin{aligned}
  &  -\int_{t_0}^{t_0+T}\ip{M(t)}{X(t)}\dt-\sum_{k:\,t_k\le t_0+T}\ip{M_k}{EX(t_k^-)E^*}\\
  &\qquad\le\ip{-P(t_0)}{X_1^0}-\ip{-P(t_0+T)}{EX(t_0+T)E^*}-R_T(P),
\end{aligned}
\end{equation*}
where $R_T(P):=\int_{t_0}^{t_0+T}\!\ip{P(t)}{\mathcal{W}(t)}\dt
+\sum_{k:\,t_k\le t_0+T}\!\ip{P(t_k^+)}{\mathcal{W}_k}$.
Dividing by $T$, the boundary terms vanish (boundedness of $P$ and
of $EXE^*$ by Lemma~\ref{lem:diss:AC:bdd}, divided by $T$), so
$\limsup_T T^{-1}V_a^T(t_0,X_1^0;X)\le\bar\Phi^M(P)$.  Taking the
supremum over $X$ on the left gives
$\bar V_a^\star(X_1^0)\le\bar\Phi^M(P)$.  Since this holds for every
feasible $P$ and $\bar V_a^\star(X_1^0)=\bar V_a^\star$,
$\bar V_a^\star\le\inf_P\bar\Phi^M(P)$.

\noindent\textbf{Step 2: $P_\infty$ is feasible and attains the
bound.}
$P_\infty$ satisfies the DLMI constraints by the I-D-GRE
(Lemma~\ref{lemma:GREequiv}) and is bounded.  By
Theorem~\ref{th:diss:AC:nec}, $\bar\Phi^M(P_\infty)=\bar V_a^\star$,
so $\inf_P\bar\Phi^M(P)\le\bar\Phi^M(P_\infty)=\bar V_a^\star$.

Combining Steps 1 and 2: $\bar V_a^\star=\inf_P\bar\Phi^M(P)$,
attained at $P_\infty$.
\end{proof}

\subsubsection{Dwell-time conditions}
\label{sec:diss:AC:dwell}

The infinite-horizon I-D-GRE solution $P_\infty$ used in
Sections~\ref{sec:diss:AC:inf:suff}--\ref{sec:diss:AC:LMI} is
non-causal: its value at time $t$ depends on the future jump times.
We now combine the average available-storage
formula~\eqref{eq:diss:AC:nec:val} with the causal dwell-time
policies of Section~\ref{sec:diss:dwell} (periodic, MDT, RDT) to
obtain causal characterisations and bounds for $\bar V_a^\star$.
The construction parallels Section~\ref{sec:OC:AC:dwell} with the
sign convention reversed: $\bar\Phi^M$ in place of $\Phi$, $-P$ in
place of $P$ as the storage-function coefficient, and the
dissipativity Geromel--Colaneri condition
$\mathcal{C}_M^0(P_s(T_{\min}))\succeq0$
(Theorem~\ref{th:diss:MDT}) in place of
$\mathcal{C}_Z^0\preceq0$.  The structural assumptions on
$\mathcal{F},\mathcal{J},M,M_k$ in the three scenarios are inherited
verbatim from Section~\ref{sec:diss:dwell}; the only additions are
the corresponding assumptions on the forcing data
$\mathcal{W},\mathcal{W}_k$.

\paragraph{Periodic impulses.}
Adopt the periodicity hypothesis of Theorem~\ref{th:diss:periodic}:
$\Delta_k\equiv T_p$, $\mathcal{F}(\cdot,X)$ and $M(\cdot)$ are
$T_p$-periodic, and $\mathcal{J}(k,\cdot)$, $M_k$ are independent
of $k$.  For the forcing data, assume in addition that
$\mathcal{W}(\cdot)$ is $T_p$-periodic and that $\mathcal{W}_k$ is
independent of $k$, i.e.,
$\mathcal{W}_k\equiv\mathcal{W}_{\mathrm{jp}}$.

\begin{theorem}[Causal long-term average available storage under
  periodic impulses]\label{th:diss:AC:dwell:periodic}
Under \textup{(A1)--(A2)} and the periodicity assumption above
\textup{(}note that \textup{(A3)} is automatic from
$\Delta_k\equiv T_p$\textup{)}, the long-term average available
storage rate~\eqref{eq:diss:AC:Va} is given by the closed-form
expression
\begin{equation}\label{eq:diss:AC:dwell:periodic:Va}
  \bar V_a^\star
  =\frac{1}{T_p}\!\left[
    -\int_0^{T_p}\!\ip{P_\infty(\tau)}{\mathcal{W}(\tau)}\,d\tau
    -\ip{P_\infty(0^+)}{\mathcal{W}_{\mathrm{jp}}}
  \right],
\end{equation}
where $P_\infty:[0,T_p]\to\Snpsd$ is the unique $T_p$-periodic
piecewise-$C^1$ solution of the periodic I-D-GRE, attained by the
causal policy $K(\tau)\in\mathcal{K}_c^M(\tau,P_\infty(\tau))$ on
flow and $K_k\in\mathcal{K}_d^M(P_\infty(0^+),P_\infty(T_p^-))$ at
each jump.
\end{theorem}

\begin{proof}
By Theorem~\ref{th:diss:periodic} applied to the homogeneous
dynamics, the periodic I-D-GRE has a unique $T_p$-periodic solution
$P_\infty$.  Its $T_p$-periodic extension to $[t_0,\infty)$
coincides with the infinite-horizon dual variable of
Theorem~\ref{th:diss:AC:nec}, so~\eqref{eq:diss:AC:nec:val} applies
and the $\limsup$ becomes a true $\lim$ by genuine periodic
averaging: the flow integral over $[t_0,t_0+T]$ contains $\lfloor
T/T_p\rfloor$ complete periods plus a boundary remainder, so its
time-average converges to
$T_p^{-1}\int_0^{T_p}\ip{P_\infty(\tau)}{\mathcal{W}(\tau)}\,d\tau$;
the jump sum contains $\lfloor T/T_p\rfloor$ jumps each contributing
$\ip{P_\infty(0^+)}{\mathcal{W}_{\mathrm{jp}}}$, with time-average
$T_p^{-1}\ip{P_\infty(0^+)}{\mathcal{W}_{\mathrm{jp}}}$.  Negating
and combining yields~\eqref{eq:diss:AC:dwell:periodic:Va}.  The
causal policy is the same as in Theorem~\ref{th:diss:periodic} and
is the optimal extraction policy by
Theorem~\ref{th:diss:AC:suff}\eqref{diss:AC:suff:gain}.
\end{proof}

\paragraph{Minimum dwell-time.}
Adopt the timer-dependence hypothesis of Theorem~\ref{th:diss:MDT}:
$\Delta_k\ge T_{\min}>0$, the flow data depend only on the timer
$\tau=t-t_k$, and the time-invariance
condition~\eqref{eq:MDT:LTIassumption} holds for
$\mathcal{F},M,\mathcal{J},M_k$.  For the forcing data, assume in
addition that $\mathcal{W}$ depends only on the timer, with
\begin{equation}\label{eq:diss:MDT:LTIassumption:forcing}
  \mathcal{W}(\tau)\equiv\mathcal{W}_\infty
  \;\;\text{for }\tau\ge T_{\min},
  \qquad
  \mathcal{W}_k\equiv\mathcal{W}_{\mathrm{jp}}
  \;\;\text{for all }k.
\end{equation}

\begin{theorem}[Causal long-term available-storage lower bound
  under MDT]\label{th:diss:AC:dwell:MDT}
Under the assumptions of Theorem~\ref{th:diss:MDT}
and~\eqref{eq:diss:MDT:LTIassumption:forcing}, suppose moreover
that the jump rate
$\bar\nu:=\limsup_{T\to\infty}T^{-1}\#\{k:t_k\le t_0+T\}$ is finite
\textup{(}so $\bar\nu\le 1/T_{\min}$\textup{)} and is in fact a
$\lim$.  Then the causal policy of Theorem~\ref{th:diss:MDT}
achieves a long-term average available storage rate bounded below
by
\begin{equation}\label{eq:diss:AC:dwell:MDT:Va}
  \bar V_a^\star
  \ge
  -\bar\nu\!\int_0^{T_{\min}}\!\ip{P_s(\tau)}{\mathcal{W}(\tau)}\,d\tau
  -(1-\bar\nu T_{\min})\ip{P_s(T_{\min})}{\mathcal{W}_\infty}
  -\bar\nu\ip{P_s(0^+)}{\mathcal{W}_{\mathrm{jp}}},
\end{equation}
where $P_s:[0,T_{\min}]\to\Snpsd$ is as in
Theorem~\ref{th:diss:MDT}.  The condition is verifiable on
$[0,T_{\min}]$ alone.
\end{theorem}

\begin{proof}
Extend $P_s$ to $\bar P:[0,\infty)\to\Snpsd$ by
$\bar P(\tau)=P_s(\tau)$ for $\tau\le T_{\min}$ and
$\bar P(\tau)=P_s(T_{\min})$ for $\tau>T_{\min}$ as in the proof of
Theorem~\ref{th:diss:MDT}.  Apply the forced key
identity~\eqref{eq:diss:AC:keyid} on $[t_0,t_0+T]$ with $\bar P$
extended along the timer, evaluating at the closed-loop extraction
trajectory $X^\star$ with feedback in
$\mathcal{K}_c^M,\mathcal{K}_d^M$; denote by $\tau(t)$ the timer
value at time $t$:
\begin{align*}
  -\int_{t_0}^{t_0+T}\!\!\!\!\!\ip{M}{X^\star}\dt
  &-\!\!\sum_{k:\,t_k\le t_0+T}\!\!\ip{M_k}{EX^\star(t_k^-)E^*}\\
  &=\ip{-\bar P(0^+)}{X_1^0}-\ip{-\bar P(\tau(t_0+T))}{EX^\star(t_0+T)E^*}\\
  &\quad
   -\underbrace{\int_{t_0}^{t_0+T}\!\ip{\mathcal{C}_M(\bar P)}{X^\star}\dt}_{\ge\,0\text{ by Thm~\ref{th:diss:MDT}~(iii)}}
   -\underbrace{\sum_{k:\,t_k\le t_0+T}\!\!\ip{\mathcal{D}_M}{X^\star(t_k^-)}}_{=\,0\text{ at the optimal extraction}}\\
  &\quad
   -\int_{t_0}^{t_0+T}\!\ip{\bar P(\tau(t))}{\mathcal{W}(\tau(t))}\dt
   -\!\!\sum_{k:\,t_k\le t_0+T}\!\!\ip{\bar P(0^+)}{\mathcal{W}_{\mathrm{jp}}}.
\end{align*}
On the inter-jump interval $(t_k,t_{k+1})$ with timer
$\tau=t-t_k\in(0,T_{\min})$, $\mathcal{C}_M(\tau,P_s(\tau))=0$ by
the flow I-D-GRE; on $\tau>T_{\min}$ where $\bar P=P_s(T_{\min})$
is constant and the data are time-invariant,
$\mathcal{C}_M(\bar P)=\mathcal{C}_M^0(P_s(T_{\min}))\succeq0$ by
Theorem~\ref{th:diss:MDT}~(iii); combined with $X^\star\succeq0$
and Lemma~\ref{lemma:decomp}, the flow inner product is non-negative
on the tail, so the underbraced flow integral is non-negative
(reversed sign from the OC case).  At each $t_k$ the saddle gain
gives $\ip{\mathcal{D}_M}{X^\star(t_k^-)}=0$.  Dropping the
non-negative flow contribution and the non-negative boundary term
$\ip{-\bar P(\tau(t_0+T))}{EX^\star(t_0+T)E^*}\ge0$ (since
$-\bar P\preceq 0$ and $EX^\star E^*\succeq0$) gives the inequality
\begin{align*}
  -\int\ip{M}{X^\star}\dt-\sum\ip{M_k}{EX^\star(t_k^-)E^*}
  &\ge\ip{-\bar P(0^+)}{X_1^0}\\
  &\quad
   -\int_{t_0}^{t_0+T}\!\ip{\bar P(\tau(t))}{\mathcal{W}(\tau(t))}\dt
   -\!\!\sum_{k:\,t_k\le t_0+T}\!\!\ip{\bar P(0^+)}{\mathcal{W}_{\mathrm{jp}}}.
\end{align*}
Dividing by $T$ and taking $T\to\infty$, the initial term vanishes.
The time-average of $\ip{\bar P(\tau(t))}{\mathcal{W}(\tau(t))}$
decomposes according to the timer distribution: under MDT, in each
inter-jump interval of length $\Delta_k\ge T_{\min}$ the timer
spends time $T_{\min}$ on $[0,T_{\min}]$ (where $\bar P=P_s$ and
$\mathcal{W}=\mathcal{W}(\tau)$ vary) and time $\Delta_k-T_{\min}$
on the frozen region (where $\bar P=P_s(T_{\min})$ and
$\mathcal{W}=\mathcal{W}_\infty$).  With average jump rate
$\bar\nu$, the fraction of time on $[0,T_{\min}]$ is
$\bar\nu T_{\min}$ and on the frozen region is $1-\bar\nu T_{\min}$,
giving~\eqref{eq:diss:AC:dwell:MDT:Va}.
\end{proof}

\begin{remark}\label{rem:diss:AC:dwell:MDT:periodic}
For the constant dwell-time $\Delta_k\equiv T_{\min}$, the jump rate
$\bar\nu=1/T_{\min}$ and the frozen contribution
$-(1-\bar\nu T_{\min})\ip{P_s(T_{\min})}{\mathcal{W}_\infty}$
vanishes, so~\eqref{eq:diss:AC:dwell:MDT:Va} recovers the periodic
formula~\eqref{eq:diss:AC:dwell:periodic:Va} (with constant data and
$T_p=T_{\min}$).  For larger jump rates ($\bar\nu<1/T_{\min}$), the
frozen contribution accounts for the time spent past the minimum
dwell-time.
\end{remark}

\paragraph{Range dwell-time.}
Adopt the timer-dependence hypothesis of Theorem~\ref{th:diss:RDT}:
$\Delta_k\in[T_{\min},T_{\max}]$ and the flow and jump data depend
only on the timer $\tau\in[0,T_{\max}]$ as
in~\eqref{eq:MDT:LTIassumption} extended to $[0,T_{\max}]$
\textup{(}so that $\mathcal{F}(\tau,\cdot),M(\tau)$ are defined for
$\tau\in[0,T_{\max}]$ and $\mathcal{J},M_k$ are
$k$-independent\textup{)}.  For the forcing data, assume similarly
that $\mathcal{W}$ depends only on the timer
$\tau\in[0,T_{\max}]$ and that $\mathcal{W}_k$ is
$k$-independent: $\mathcal{W}_k\equiv\mathcal{W}_{\mathrm{jp}}$.

\begin{theorem}[Causal long-term available-storage lower bound under
  RDT]\label{th:diss:AC:dwell:RDT}
Under the assumptions of Theorem~\ref{th:diss:RDT} and the
timer-dependence hypotheses above, suppose moreover that the
empirical occupation measure of the inter-jump sequence
$\{\Delta_k\}$ converges weakly to a probability measure $\mu$ on
$[T_{\min},T_{\max}]$ as $T\to\infty$, and that the jump rate
$\bar\nu=(\int_{[T_{\min},T_{\max}]}\Delta\,\mu(d\Delta))^{-1}$
is finite.  Then the causal policy of Theorem~\ref{th:diss:RDT}
achieves a long-term average available storage rate bounded below
by
\begin{equation}\label{eq:diss:AC:dwell:RDT:Va}
  \bar V_a^\star
  \ge
  -\bar\nu\!\int_{T_{\min}}^{T_{\max}}\!\!\biggl(\int_0^\Delta\!
    \ip{P_s(\tau)}{\mathcal{W}(\tau)}\,d\tau\biggr)\mu(d\Delta)
  -\bar\nu\ip{P_s(0^+)}{\mathcal{W}_{\mathrm{jp}}},
\end{equation}
where $P_s:[0,T_{\max}]\to\Snpsd$ is as in
Theorem~\ref{th:diss:RDT}.
\end{theorem}

\begin{proof}
Apply the forced key identity~\eqref{eq:diss:AC:keyid} with $P=P_s$
over $[t_0,t_0+T]$, evaluating at the closed-loop extraction
trajectory $X^\star$.  The $\mathcal{C}_M$ and $\mathcal{D}_M$
contributions vanish identically at the saddle gain by the
conditions of Theorem~\ref{th:diss:RDT}, so
\begin{align*}
  -\int_{t_0}^{t_0+T}\!\ip{M}{X^\star}\dt
  &-\!\!\sum_{k:\,t_k\le t_0+T}\!\!\ip{M_k}{EX^\star(t_k^-)E^*}\\
  &=\ip{-P_s(0^+)}{X_1^0}-\ip{-P_s(\tau(t_0+T))}{EX^\star(t_0+T)E^*}\\
  &\quad
   -\int_{t_0}^{t_0+T}\!\ip{P_s(\tau(t))}{\mathcal{W}(\tau(t))}\dt
   -\!\!\sum_{k:\,t_k\le t_0+T}\!\!\ip{P_s(0^+)}{\mathcal{W}_{\mathrm{jp}}}.
\end{align*}
Dropping the non-negative boundary term
$\ip{-P_s(\tau(t_0+T))}{EX^\star(t_0+T)E^*}\ge0$ gives a $\ge$
inequality.  Dividing by $T$ and letting $T\to\infty$, the initial
term vanishes.  The jump-sum time-average is
$\bar\nu\ip{P_s(0^+)}{\mathcal{W}_{\mathrm{jp}}}$.  For the flow
time-average, in an inter-jump interval of length $\Delta_k$ the
timer traverses $[0,\Delta_k]$, contributing
$\int_0^{\Delta_k}\ip{P_s(\tau)}{\mathcal{W}(\tau)}\,d\tau$ to the
integral; dividing by $T$ and using
$\bar\nu=(\mathbb{E}_\mu[\Delta])^{-1}$ gives the $\mu$-integral
in~\eqref{eq:diss:AC:dwell:RDT:Va}.
\end{proof}

\begin{remark}\label{rem:diss:AC:dwell:RDT}
When $\mu$ is concentrated at a single point $\Delta_k\equiv T_p$
with $T_{\min}\le T_p\le T_{\max}$, the
bound~\eqref{eq:diss:AC:dwell:RDT:Va} becomes equality and reduces
to~\eqref{eq:diss:AC:dwell:periodic:Va}.  For unknown $\mu$, the
worst case over admissible distributions on $[T_{\min},T_{\max}]$
gives a robust lower bound on $\bar V_a^\star$.
\end{remark}

\subsubsection{Continuous- and discrete-time corollaries}
\label{sec:diss:AC:CTDT}

The continuous-time and discrete-time specialisations follow by
setting $N_T=0$ or $\mathcal{F}=0$.

\begin{corollary}[Continuous-time average available storage]
  \label{cor:CT:diss:AC}
Set $N_T=0$ (no jumps), so $\mathcal{W}_k$ is irrelevant.  Under
the continuous-time forced flow
$E\dot X E^*=\mathcal{F}(t,X(t))+\mathcal{W}(t)$ and (A1)--(A2),
\begin{equation}\label{eq:CT:diss:AC:Va}
  \bar V_a^\star
  =\limsup_{T\to\infty}\frac{1}{T}\!\left[
    -\int_{t_0}^{t_0+T}\!\ip{P_\infty(t)}{\mathcal{W}(t)}\dt
  \right],
\end{equation}
where $P_\infty$ is the unique bounded $C^1$ solution of the
continuous-time infinite-horizon D-GRE
(Corollary~\ref{cor:CT:diss:inf}).  The dual DLMI is
\[
  \bar V_a^\star
  =\inf_P\;\limsup_{T\to\infty}\frac{1}{T}\!\!\int_{t_0}^{t_0+T}\!\!\ip{-P(t)}{\mathcal{W}(t)}\dt
  \quad\text{s.t.}\quad
  \mathcal{C}_M(t,P(t))\succeq0,\;P\text{ bounded}.
\]
In the LTI case ($\mathcal{F},M,\mathcal{W}$ time-invariant),
$P_\infty$ is constant and~\eqref{eq:CT:diss:AC:Va} reduces to
\begin{equation}\label{eq:CT:diss:AC:LTI}
  \bar V_a^\star
  =\ip{-P_\infty}{\mathcal{W}}=\tr(-P_\infty\mathcal{W}),
\end{equation}
recovering the classical steady-state output-covariance / $H_2$-norm
formula when $\mathcal{W}=BB^*$ is a process-noise covariance.
\end{corollary}

\begin{corollary}[Discrete-time average available storage]
  \label{cor:DT:diss:AC}
Apply with $\mathcal{F}=0$, $M=0$ on flow, $t_k=k_0+k$,
$\mathcal{J}=\mathcal{D}$, $\mathcal{W}=0$ on flow.  Under the
discrete-time forced jump
$EX(k+1)E^*=\mathcal{D}(k,X(k))+\mathcal{W}_k$ and the
discrete-time analogues of (A1)--(A2),
\begin{equation}\label{eq:DT:diss:AC:Va}
  \bar V_a^\star
  =\limsup_{N\to\infty}\frac{1}{N}\sum_{k=k_0}^{k_0+N-1}\ip{-P_\infty(k+1)}{\mathcal{W}_k},
\end{equation}
where $P_\infty$ satisfies the algebraic DT-D-GRE
(Corollary~\ref{cor:DT:diss:inf}).  In the LTI case
($\mathcal{D},M_k,\mathcal{W}_k$ constant), $P_\infty$ is constant
and~\eqref{eq:DT:diss:AC:Va} reduces to
$\bar V_a^\star=\tr(-P_\infty\mathcal{W})$, again recovering the
classical steady-state covariance formula.
\end{corollary}

\section{Minimax Optimal Control: Zero-Sum Differential Games}
\label{sec:game}

This section develops the zero-sum game counterpart of the optimal
control theory of Section~\ref{sec:OC}.  The impulsive system now has
\emph{two} input channels: a minimizing player (the controller, with
input $X_2^u$, $X_3^u$) and a maximizing player (the disturbance,
with input $X_2^w$, $X_3^w$).  The game value is the minimax cost,
and the saddle-point policy is the pair that simultaneously minimizes
for the controller and maximizes for the disturbance.

The analysis is built on the \emph{impulsive HJI-GRE}, the
game-theoretic analogue of the I-GRE.  Under Isaacs'
condition~\eqref{eq:Isaacs:gen}, the minimax and maximin coincide
and the game has a saddle-point solution characterized by the
HJI matrix $\mathcal{I}$ and its (indefinite) Isaacs block
$\mathcal{G}$.  The key structural difference from the optimal control
case is that the pivot block $\mathcal{G}$ is \emph{indefinite}: the
controller minimizes in the $\mathcal{I}_{uu}$ block ($\succeq0$) while
the disturbance maximizes in the $\mathcal{I}_{ww}$ block ($\preceq0$).
Consequently the saddle is handled by the \emph{generalized (indefinite)
Schur complement} of $\mathcal{G}$ in $\mathcal{I}$, and the dual
characterization (Theorem~\ref{th:game:LMI}) is a game Riccati
\emph{inequality} which, in contrast to the optimal control case, 
retains a Schur-complement term because the pivot $\mathcal{G}$ has
mixed inertia.

The finite-horizon game theory is in Section~\ref{sec:game:fh}, the
infinite-horizon game under LTV conditions in
Section~\ref{sec:game:inf}, and the long-term average game value in
Section~\ref{sec:game:AC}; causal dwell-time policies and
continuous-time and discrete-time corollaries appear throughout.
Connections to $H_\infty$ and full-information games are discussed in
Section~\ref{sec:game:hinf}.

\subsection{Finite-horizon game}
\label{sec:game:fh}

We study the zero-sum game on the bounded horizon $[t_0,t_0+T]$ for a
fixed $T>0$.  The decision variable is the full trajectory
$X(\cdot)\in\Snmpsd$, whose free blocks are now split between a
\emph{minimizing} player (the controller) and a \emph{maximizing}
player (the disturbance).  We introduce the augmented state, the game
cost, and the minimax problem in Section~\ref{sec:game:form}, the HJI
operators and the I-HJI-GRE in Section~\ref{sec:game:fh:GRE}, and the
sufficient, necessary, and dual conditions in
Theorems~\ref{th:game:suff}--\ref{th:game:LMI}.

\subsubsection{Augmented state, partition, and problem formulation}
\label{sec:game:form}

The setting of this section extends the optimal control framework of
Section~\ref{sec:OC} by splitting the free block $E_\bot$ into two
competing parts: a \emph{minimizing block} of size $m_u$ and a
\emph{maximizing block} of size $m_w$, with $m=m_u+m_w$.
Define the selectors
\begin{equation}
  E_u:=\bigl[\,0_{m_u\times n}\;\;I_{m_u}\;\;0_{m_u\times m_w}\,\bigr]
  \in\C^{m_u\times(n+m)},
  \qquad
  E_w:=\bigl[\,0_{m_w\times n}\;\;0_{m_w\times m_u}\;\;I_{m_w}\,\bigr]
  \in\C^{m_w\times(n+m)},
\end{equation}
so that $E_\bot=\begin{bmatrix}E_u\\E_w\end{bmatrix}$.  The extended Hermitian matrix
$X\in\Snmpsd$ is now partitioned as
\begin{equation}\label{eq:game:Xpart}
  X=\begin{bmatrix}
    X_1 & X_{12}^* & X_{13}^*\\
    X_{12} & X_{22} & X_{23}^*\\
    X_{13} & X_{23} & X_{33}
  \end{bmatrix},
\end{equation}
where $X_1=EXE^*\in\Snpsd$ (state, constrained by dynamics),
$X_{12}=E_uXE^*$ and $X_{22}=E_uXE_u^*$ (control blocks,
minimizer-controlled), $X_{13}=E_wXE^*$, $X_{23}=E_wXE_u^*$, and
$X_{33}=E_wXE_w^*$ (disturbance blocks, maximizer-controlled).
The blocks $X_{22}$ and $X_{33}$ are the self-adjoint free blocks
associated with the minimizer and maximizer respectively, and $X_{23}$
is their cross-block.

\medskip\noindent\textbf{Cost function.}
For cost matrix
$Z=\bigl[\begin{smallmatrix}Q&S_u&S_w\\S_u^*&R_u&R_{uw}\\
S_w^*&R_{uw}^*&{-R_w}\end{smallmatrix}\bigr]$
with $Q\succeq0$, $R_u\succ0$, $R_w\succ0$ (the negative sign on
$R_w$ reflects the maximizer's objective), and with the
\emph{controller cost block}
$Z_c:=\bigl[\begin{smallmatrix}Q&S_u\\S_u^*&R_u\end{smallmatrix}\bigr]\succeq0$
(equivalently $Q\succeq S_u\pinv{R_u}S_u^*$; this holds automatically when
$S_u=0$, as in the $H_\infty$ setting), the finite-horizon cost is
\begin{equation}\label{eq:game:cost}
  J_T(X,X_1^0)
  :=\int_{t_0}^{t_0+T}\ip{Z(t)}{X(t)}\dt
  +\sum_k\ip{Z_k}{X(t_k^-)}+\ip{Z_T}{EX(t_0+T)E^*},
\end{equation}
which is linear in $X$.  The negative sign on $R_w$ in the $(ww)$-block
reflects that the maximizer seeks to increase the cost, hence a
positive $R_w$ penalises its own block negatively.

\medskip\noindent\textbf{Zero-sum minimax problem.}
The minimizer controls $X_{12}$, $X_{22}$ (the control blocks), while
the maximizer controls $X_{13}$, $X_{23}$, $X_{33}$ (the disturbance
blocks), subject to $X(t)\succeq0$ and the dynamics:
\begin{equation}\label{eq:game:prob}
  J_T^\circ(X_1^0)
  :=\underset{X_{12},X_{22}}{\min}\;
    \underset{X_{13},X_{23},X_{33}}{\max}\;
    J_T(X,X_1^0)
  \quad\text{s.t.~\eqref{eq:sys:IMP},}\;EX(t_0)E^*=X_1^0.
\end{equation}
When the minimax equals the maximin (guaranteed by Isaacs' condition,
Definition~\ref{def:Isaacs} below), $J_T^\circ$ is the
\emph{saddle-point value}, also called the \emph{game value}.

\subsubsection{Composite operators and composite HJI matrix (finite horizon)}

The finite-horizon game analysis follows the same structure as
Section~\ref{sec:OC:fh} but with an indefinite Isaacs block that
reflects the opposing objectives of the two players.  The HJI operator
$\mathcal{I}(t,P(t))$ is the game analogue of $\mathcal{C}_Z(t,P(t))$;
its bottom-right $(m_u+m_w)$ block, the \emph{Isaacs block}
$\mathcal{G}$, is now \emph{indefinite} (the minimizer's block
$\mathcal{I}_{uu}\succeq0$, the maximizer's block
$\mathcal{I}_{ww}\preceq0$), and is the central quantity.  Isaacs'
condition is an inertia condition on $\mathcal{G}$ ensuring the
existence of a saddle-point, and the I-HJI-GRE, the generalized
Schur complement of $\mathcal{G}$ in $\mathcal{I}$ set to zero, is
the resulting Riccati equation.

For piecewise-$C^1$ $P:[t_0,t_0+T]\to\Sn$, define the
\emph{HJI flow operator}
\begin{equation}\label{eq:game:I}
  \mathcal{I}(t,P(t))
  :=E^*\dot{P}(t)E+\mathcal{F}^*(t,P(t))+Z(t)\in\mathbb{H}^{n+m},
\end{equation}
blocked conformally with~\eqref{eq:game:Xpart} as
\begin{equation}
  \mathcal{I}=\begin{bmatrix}
    \mathcal{I}_{11} & \mathcal{I}_{1u} & \mathcal{I}_{1w}\\
    \mathcal{I}_{1u}^* & \mathcal{I}_{uu} & \mathcal{I}_{uw}\\
    \mathcal{I}_{1w}^* & \mathcal{I}_{uw}^* & \mathcal{I}_{ww}
  \end{bmatrix},
\end{equation}
and the \emph{HJI jump operator}
\begin{equation}\label{eq:game:Ij}
  \mathcal{I}^\mathrm{j}(k,P(t_k^+),P(t_k^-))
  :=\mathcal{J}^*(k,P(t_k^+))-E^*P(t_k^-)E+Z_k,
\end{equation}
blocked analogously as
\begin{equation}
  \mathcal{I}^\mathrm{j}=\begin{bmatrix}
    \mathcal{I}^\mathrm{j}_{11} & \mathcal{I}^\mathrm{j}_{1u} & \mathcal{I}^\mathrm{j}_{1w}\\
    \mathcal{I}^{\mathrm{j}*}_{1u} & \mathcal{I}^\mathrm{j}_{uu} & \mathcal{I}^\mathrm{j}_{uw}\\
    \mathcal{I}^{\mathrm{j}*}_{1w} & \mathcal{I}^{\mathrm{j}*}_{uw} & \mathcal{I}^\mathrm{j}_{ww}
  \end{bmatrix}.
\end{equation}

\begin{remark}
  The HJI operator~\eqref{eq:game:I} is structurally identical to
  the flow operator~\eqref{eq:IMP:L} and the dissipativity
  operator~\eqref{eq:IMP:LM}, but now the cost matrix $Z$ has the
  mixed sign structure $-R_w\prec0$ in the $(ww)$-block.  The key
  identity (Proposition~\ref{prop:keyid}) holds verbatim with
  $\mathcal{C}_Z(t,P(t))\leftarrow\mathcal{I}(t,P(t))$.
\end{remark}

\subsubsection{Isaacs' condition and the I-HJI-GRE (finite horizon)}
\label{sec:game:fh:GRE}

\begin{definition}[Generalized Isaacs' condition and Isaacs block]
  \label{def:Isaacs}
  Define the \emph{Isaacs block} as the bottom-right
  $(m_u+m_w)\times(m_u+m_w)$ block of $\mathcal{I}$,
  \begin{equation}\label{eq:Ghat}
    \mathcal{G}(t,P(t))
    :=\begin{bmatrix}
      \mathcal{I}_{uu}(t,P(t)) & \mathcal{I}_{uw}(t,P(t))\\
      \mathcal{I}_{uw}(t,P(t))^* & \mathcal{I}_{ww}(t,P(t))
    \end{bmatrix},
  \end{equation}
  and the selection block
  $\mathcal{I}_{1g}:=\begin{bmatrix}\mathcal{I}_{1u}&\mathcal{I}_{1w}\end{bmatrix}$.
  Unlike the optimal control pivot $(\mathcal{C}_Z)_3$, the Isaacs block
  $\mathcal{G}$ is \emph{indefinite}.  \emph{Isaacs' condition}, a
  saddle-inertia condition on $\mathcal{G}$, holds at $(t,P(t))$ if
  \begin{equation}\label{eq:Isaacs:gen}
    \mathcal{I}_{uu}\succeq0,\quad
    \mathcal{I}_{ww}\preceq0,\quad
    \mathcal{I}_{ww}-\mathcal{I}_{uw}^*\pinv{\mathcal{I}_{uu}}\mathcal{I}_{uw}\preceq0,\quad
    \mathcal{I}_{uu}-\mathcal{I}_{uw}\pinv{\mathcal{I}_{ww}}\mathcal{I}_{uw}^*\succeq0,
  \end{equation}
  together with the range conditions
  $\operatorname{Im}(\mathcal{I}_{uw})\subseteq\operatorname{Im}(\mathcal{I}_{uu})$
  and
  $\operatorname{Im}(\mathcal{I}_{uw}^*)\subseteq\operatorname{Im}(\mathcal{I}_{ww})$.
  Equivalently, on its range $\mathcal{G}$ has inertia $(m_u,0,m_w)$:
  $m_u$ nonnegative directions (the minimizer's) and $m_w$ nonpositive
  directions (the maximizer's).  These conditions allow $\mathcal{I}_{uu}$
  and $\mathcal{I}_{ww}$ to be singular, in analogy with the optimal
  control condition $(\mathcal{C}_Z)_3\succeq0$.  When
  $\mathcal{I}_{uu}\succ0$ and $\mathcal{I}_{ww}\prec0$, $\mathcal{G}$ is
  nonsingular with exactly $m_u$ positive and $m_w$ negative eigenvalues,
  the range conditions hold automatically, and Isaacs' condition reduces
  to $\mathcal{I}_{uu}\succ0$, $\mathcal{I}_{ww}\prec0$.
\end{definition}

\begin{remark}\label{rem:game:saddle}
  Write $X=UU^*$ with $U=[U_1^*\ U_u^*\ U_w^*]^*$.  The generalized
  Schur identity (the algebraic part of Lemma~\ref{lemma:decomp}, which
  requires only the compatibility/range condition and \emph{not}
  $\mathcal{G}\succeq0$) gives
  \begin{equation}\label{eq:game:Idecomp}
    \ip{\mathcal{I}}{UU^*}
    =\ip{\mathcal{I}\,\schur{/}\,\mathcal{G}}{U_1U_1^*}
    +\ip{\mathcal{G}}{WW^*},\qquad
    W:=\pinv{\mathcal{G}}\begin{bmatrix}\mathcal{I}_{1u}^*\\\mathcal{I}_{1w}^*\end{bmatrix}U_1
    +\begin{bmatrix}U_u\\U_w\end{bmatrix},
  \end{equation}
  where $\mathcal{I}\,\schur{/}\,\mathcal{G}
  =\mathcal{I}_{11}-\mathcal{I}_{1g}\pinv{\mathcal{G}}\mathcal{I}_{1g}^*$
  is the generalized Schur complement of the \emph{indefinite} block
  $\mathcal{G}$.  Under Isaacs' condition,
  $\ip{\mathcal{G}}{WW^*}=\operatorname{tr}(W^*\mathcal{G}W)$ is a genuine
  \emph{saddle} in the free blocks $(U_u,U_w)$: convex in the minimizer's
  block $U_u$ (since $\mathcal{I}_{uu}\succeq0$) and concave in the
  maximizer's block $U_w$ (since
  $\mathcal{I}_{ww}-\mathcal{I}_{uw}^*\pinv{\mathcal{I}_{uu}}\mathcal{I}_{uw}\preceq0$),
  with saddle value $0$ attained at $W=0$.  Hence
  \begin{equation*}
    \min_{U_u}\max_{U_w}\ip{\mathcal{I}}{UU^*}
    =\max_{U_w}\min_{U_u}\ip{\mathcal{I}}{UU^*}
    =\ip{\mathcal{I}\,\schur{/}\,\mathcal{G}}{U_1U_1^*}:
  \end{equation*}
  the saddle of the cost integrand is the indefinite Schur complement
  $\mathcal{I}\,\schur{/}\,\mathcal{G}$.
\end{remark}

\begin{definition}[Impulsive HJI-GRE]\label{def:game:GRE}
  Under Isaacs' condition~\eqref{eq:Isaacs:gen}, the \emph{impulsive
  Hamilton-Jacobi-Isaacs generalized Riccati equation} (I-HJI-GRE)
  for piecewise-$C^1$ $P:[t_0,t_0+T]\to\Sn$ with $P(t_0+T)=Z_T$
  consists of:
  \begin{itemize}
    \item \emph{Flow conditions,} for a.e.\ $t$:
      \begin{subequations}\label{eq:game:GRE:flow}
      \begin{align}
        &\mathcal{I}_{11}-\begin{bmatrix}\mathcal{I}_{1u} & \mathcal{I}_{1w}\end{bmatrix}
          \,\pinv{\mathcal{G}}\,
          \begin{bmatrix}\mathcal{I}_{1u}^*\\\mathcal{I}_{1w}^*\end{bmatrix}=0,
          \label{eq:game:GRE:fR}\\
        &\mathcal{G}\,\pinv{\mathcal{G}}\,
          \begin{bmatrix}\mathcal{I}_{1u}^*\\\mathcal{I}_{1w}^*\end{bmatrix}
          =\begin{bmatrix}\mathcal{I}_{1u}^*\\\mathcal{I}_{1w}^*\end{bmatrix},
          \label{eq:game:GRE:fC}\\
        &\text{Isaacs' condition~\eqref{eq:Isaacs:gen} holds at }(t,P(t)).
          \label{eq:game:GRE:fPSD}
      \end{align}
      \end{subequations}
      Equation~\eqref{eq:game:GRE:fR} is
      $\mathcal{I}\,\schur{/}\,\mathcal{G}=0$ (the generalized Schur
      complement of the indefinite block $\mathcal{G}$),
      \eqref{eq:game:GRE:fC} is the compatibility (range) condition, and
      \eqref{eq:game:GRE:fPSD} is the saddle inertia of $\mathcal{G}$.
      Because $\mathcal{G}$ is indefinite, these conditions
      \emph{cannot} be collapsed into a single semidefinite condition
      $\mathcal{I}\succeq0$ (the optimal control situation); the saddle
      inertia must be imposed separately.
    \item \emph{Jump conditions,} at each $t_k$ ($k=1,\ldots,N_T$):
      write $\mathcal{I}^\mathrm{j}:=\mathcal{I}^\mathrm{j}(k,P(t_k^+),P(t_k^-))$
      for the HJI jump operator~\eqref{eq:game:Ij}, blocked as
      in~\eqref{eq:game:Ij}, and define the \emph{Isaacs jump block}
      as its bottom-right block
      \begin{equation}\label{eq:Ghatj}
        \mathcal{G}^\mathrm{j}(k,P(t_k^+),P(t_k^-))
        :=\begin{bmatrix}
          \mathcal{I}^\mathrm{j}_{uu} & \mathcal{I}^\mathrm{j}_{uw}\\
          (\mathcal{I}^\mathrm{j}_{uw})^* & \mathcal{I}^\mathrm{j}_{ww}
        \end{bmatrix},
        \qquad
        \mathcal{I}^\mathrm{j}_{1g}:=\begin{bmatrix}\mathcal{I}^\mathrm{j}_{1u}&\mathcal{I}^\mathrm{j}_{1w}\end{bmatrix}.
      \end{equation}
      The jump conditions are
      \begin{subequations}\label{eq:game:GRE:jump}
      \begin{align}
        &\mathcal{I}^\mathrm{j}_{11}
          -\mathcal{I}^\mathrm{j}_{1g}
          \,\pinv{(\mathcal{G}^\mathrm{j})}\,
          (\mathcal{I}^\mathrm{j}_{1g})^*=0,
          \label{eq:game:GRE:jR}\\
        &\mathcal{G}^\mathrm{j}\,\pinv{(\mathcal{G}^\mathrm{j})}\,
          (\mathcal{I}^\mathrm{j}_{1g})^*
          =(\mathcal{I}^\mathrm{j}_{1g})^*,
          \label{eq:game:GRE:jC}\\
        &\text{Isaacs' condition~\eqref{eq:Isaacs:gen} holds for }\mathcal{I}^\mathrm{j}.
          \label{eq:game:GRE:jPSD}
      \end{align}
      \end{subequations}
      Equivalently, \eqref{eq:game:GRE:jR} is
      $\mathcal{I}^\mathrm{j}\,\schur{/}\,\mathcal{G}^\mathrm{j}=0$ with
      \eqref{eq:game:GRE:jC} the range condition and~\eqref{eq:game:GRE:jPSD}
      the saddle inertia of $\mathcal{G}^\mathrm{j}$.
  \end{itemize}
  The \emph{saddle-point optimal sets} are, for the flow,
  \begin{equation}\label{eq:game:Ks}
    \bigl(\mathcal{K}_u^s\times\mathcal{K}_w^s\bigr)(t,P(t))
    :=\left\{(K_u,K_w)\,:\;
      \mathcal{G}\begin{bmatrix}K_u^*\\K_w^*\end{bmatrix}
      +\begin{bmatrix}\mathcal{I}_{1u}^*\\\mathcal{I}_{1w}^*\end{bmatrix}
      =0\right\},
  \end{equation}
  with general element
  $\begin{bmatrix}K_u^*\\K_w^*\end{bmatrix}=-\pinv{\mathcal{G}}
  \begin{bmatrix}\mathcal{I}_{1u}^*\\\mathcal{I}_{1w}^*\end{bmatrix}
  +(I-\pinv{\mathcal{G}}\mathcal{G})\,F$
  for arbitrary $F\in\C^{(m_u+m_w)\times n}$.
  At each jump $t_k$, the \emph{jump (discrete-time) saddle-point optimal
  set} is, with all blocks evaluated at $(k,P(t_k^+),P(t_k^-))$,
  \begin{equation}\label{eq:game:Ks:jump}
    \bigl(\mathcal{K}_u^s\times\mathcal{K}_w^s\bigr)(k,P(t_k^+),P(t_k^-))
    :=\left\{(K_u,K_w)\,:\;
      \mathcal{G}^\mathrm{j}\begin{bmatrix}K_u^*\\K_w^*\end{bmatrix}
      +(\mathcal{I}^\mathrm{j}_{1g})^*
      =0\right\},
  \end{equation}
  with general element
  \begin{equation*}
    \begin{bmatrix}K_u^*\\K_w^*\end{bmatrix}
    =-\pinv{(\mathcal{G}^\mathrm{j})}
    (\mathcal{I}^\mathrm{j}_{1g})^*
    +(I-\pinv{(\mathcal{G}^\mathrm{j})}\mathcal{G}^\mathrm{j})\,F^\mathrm{j}
    \quad\text{for arbitrary }F^\mathrm{j}\in\C^{(m_u+m_w)\times n}.
  \end{equation*}
\end{definition}

\begin{remark}
  When $\mathcal{I}_{uu}\succ0$ and $\mathcal{I}_{ww}\prec0$,
  $\mathcal{G}$ is nonsingular ($\pinv{\mathcal{G}}=\mathcal{G}^{-1}$),
  the compatibility condition~\eqref{eq:game:GRE:fC} holds trivially,
  and the saddle-point gains reduce to
  $(K_u^*,K_w^*)^\top=-\mathcal{G}^{-1}[\mathcal{I}_{1u}^*;\,\mathcal{I}_{1w}^*]$.
  Note $\mathcal{G}$ is indefinite (inertia $(m_u,0,m_w)$), so this is
  the stationary point of an indefinite quadratic, a saddle, not a
  minimum.  In the singular case, the kernel term
  $(I-\pinv{\mathcal{G}}\mathcal{G})F$
  parametrises all saddle-point strategies not uniquely determined by
  the HJI-GRE, exactly as in the optimal control and dissipativity
  sections.
\end{remark}

\begin{remark}[Separation and state-feedback]\label{rem:separation}
  When $\mathcal{I}_{uw}(t,P(t))=0$, $\mathcal{G}$ is block-diagonal:
  $\mathcal{G}=\mathrm{diag}(\mathcal{I}_{uu},\mathcal{I}_{ww})$,
  and the saddle-point equations decouple:
  \begin{align}
    K_u^*&=-\pinv{\mathcal{I}_{uu}}\mathcal{I}_{1u}^*
    \quad(\text{minimizer, depends only on }X_1),
    \notag\\
    K_w^*&=-\pinv{\mathcal{I}_{ww}}\mathcal{I}_{1w}^*
    \quad(\text{maximizer, depends only on }X_1).
    \notag
  \end{align}
  In this case the optimal minimizing strategy depends only on the
  constrained block $X_1$ and not on the maximizing free blocks
  $(X_{13},X_{23},X_{33})$.  This is the \emph{separation condition}
  for the game: the minimizer does not need to observe the maximizing
  blocks.  It holds whenever $Z_{uw}=0$ (no cross-penalty between
  the minimizing and maximizing inputs) and there are no cross
  terms in the flow adjoint $\mathcal{F}^*$.  When
  $\mathcal{I}_{uw}\neq0$, the optimal minimizer depends on
  the maximizer's current free blocks; the full-information
  formulation (where $X_{33}$ and $X_{23}$ are available) then
  applies.
\end{remark}

\subsubsection{Sufficient condition (finite horizon)}

Theorem~\ref{th:game:suff} below shows that any piecewise-$C^1$
solution $P\succeq0$ of the I-HJI-GRE is the unique solution, and that
$\ip{P(t_0)}{X_1^0}$ is simultaneously the saddle-point game value and
the cost attained by the saddle-point policy with gains in the optimal
sets $\mathcal{K}_u^s$ and $\mathcal{K}_w^s$ of~\eqref{eq:game:Ks}.
The proof applies the key identity (Proposition~\ref{prop:keyid}) with
$\mathcal{I}$ in place of $\mathcal{C}_Z$, then uses the genuine
saddle decomposition~\eqref{eq:game:Idecomp} of
Remark~\ref{rem:game:saddle}: along any trajectory the cost integrand
saddles to the indefinite Schur complement
$\mathcal{I}\,\schur{/}\,\mathcal{G}$, which vanishes on the I-HJI-GRE.

\begin{theorem}[Sufficiency of the I-HJI-GRE]\label{th:game:suff}
  Let $T>0$ and suppose Isaacs' condition holds for all
  $t\in[t_0,t_0+T]$.  Suppose there exists a piecewise-$C^1$
  solution $P:[t_0,t_0+T]\to\Snpsd$ with $P(t_0+T)=Z_T$ satisfying the
  I-HJI-GRE.  Then:
  \begin{enumerate}[\upshape(a)]
    \item The saddle-point value is
      $J_T^\circ(X_1^0)=\ip{P(t_0)}{X_1^0}$.
    \item The saddle-point is attained at $(K_u^s,K_w^s)$ from
      $\mathcal{K}_u^s$ and $\mathcal{K}_w^s$.
    \item The saddle-point value equals the minimax and the maximin:
      \begin{equation}\label{eq:game:saddle}
        J_T^\circ(X_1^0)
        =\min_{X_u}\max_{X_w}J_T
        =\max_{X_w}\min_{X_u}J_T
        =\ip{P(t_0)}{X_1^0}.
      \end{equation}
    \item The solution $P$ is unique.
  \end{enumerate}
\end{theorem}

\begin{proof}
\noindent\textbf{Key identity.}
By Proposition~\ref{prop:keyid} (with $\mathcal{I}$ in place of
$\mathcal{F}$) and $P(t_0+T)=Z_T$:
\begin{equation}\label{pf:game:id}
  J_T(X,X_1^0)
  =\ip{P(t_0)}{X_1^0}
  +\int_{t_0}^{t_0+T}\ip{\mathcal{I}(t,P(t))}{X}\dt
  +\sum_k\ip{\mathcal{I}^\mathrm{j}(k,P(t_k^+),P(t_k^-))}{X(t_k^-)}.
\end{equation}

\noindent\textbf{Saddle-point decomposition.}
Write $X=UU^*$ with $U=[U_1^*\ U_u^*\ U_w^*]^*$.  Isaacs' condition
and the I-HJI-GRE give the compatibility condition~\eqref{eq:game:GRE:fC}
and $\mathcal{I}\,\schur{/}\,\mathcal{G}=0$, so the generalized Schur
identity~\eqref{eq:game:Idecomp} reduces to
\begin{equation}\label{pf:game:decomp}
  \ip{\mathcal{I}}{UU^*}
  =\underbrace{\ip{\mathcal{I}\,\schur{/}\,\mathcal{G}}{U_1U_1^*}}_{=\,0}
  +\ip{\mathcal{G}}{WW^*}
  =\operatorname{tr}(W^*\mathcal{G}W),
  \qquad
  W=\begin{bmatrix}W_u\\W_w\end{bmatrix}
   =\pinv{\mathcal{G}}\begin{bmatrix}\mathcal{I}_{1u}^*\\\mathcal{I}_{1w}^*\end{bmatrix}U_1
    +\begin{bmatrix}U_u\\U_w\end{bmatrix}.
\end{equation}
By Isaacs' condition this is a \emph{saddle} in $(U_u,U_w)$ (equivalently
in $(W_u,W_w)$): for fixed $W_u$, the maximum over $W_w$ is
$\operatorname{tr}\!\bigl(W_u^*(\mathcal{I}_{uu}-\mathcal{I}_{uw}\pinv{\mathcal{I}_{ww}}\mathcal{I}_{uw}^*)W_u\bigr)\ge0$,
attained at $W_w=-\pinv{\mathcal{I}_{ww}}\mathcal{I}_{uw}^*W_u$ and equal
to $0$ iff $W_u=0$; for fixed $W_w$, the minimum over $W_u$ is
$\operatorname{tr}\!\bigl(W_w^*(\mathcal{I}_{ww}-\mathcal{I}_{uw}^*\pinv{\mathcal{I}_{uu}}\mathcal{I}_{uw})W_w\bigr)\le0$,
attained at $W_u=-\pinv{\mathcal{I}_{uu}}\mathcal{I}_{uw}W_w$ and equal to
$0$ iff $W_w=0$.  The same statements hold for $\mathcal{I}^\mathrm{j}$
and $\mathcal{G}^\mathrm{j}$ at each jump.

\noindent\textbf{Saddle-point value.}
\emph{Upper bound on the minimax.}  Let the minimizer play the
saddle gain $K_u^s$ of~\eqref{eq:game:Ks}, i.e.\ $U_u=K_u^sU_1$, which is
the $u$-block of $-\pinv{\mathcal{G}}[\mathcal{I}_{1u}^*;\mathcal{I}_{1w}^*]U_1$
and hence forces $W_u=0$.  Then for \emph{every} disturbance choice
$X_w$, $\ip{\mathcal{I}}{X}=\operatorname{tr}(W_w^*\mathcal{I}_{ww}W_w)\le0$
is maximized at $W_w=0$ with value $0$; the analogous statement holds at
each jump.  By the key identity~\eqref{pf:game:id},
$\max_{X_w}J_T(\cdot,X_w)=\ip{P(t_0)}{X_1^0}$ under this minimizer
policy, so $\min_{X_u}\max_{X_w}J_T\le\ip{P(t_0)}{X_1^0}$.

\emph{Lower bound on the maximin.}  Symmetrically, let the maximizer play
$K_w^s$, forcing $W_w=0$.  Then for every $X_u$,
$\ip{\mathcal{I}}{X}=\operatorname{tr}(W_u^*\mathcal{I}_{uu}W_u)\ge0$ is
minimized at $W_u=0$ with value $0$, so
$\min_{X_u}J_T(\cdot,X_w^\circ)=\ip{P(t_0)}{X_1^0}$ and
$\max_{X_w}\min_{X_u}J_T\ge\ip{P(t_0)}{X_1^0}$.

Since $\min_{X_u}\max_{X_w}J_T\ge\max_{X_w}\min_{X_u}J_T$ always, the two
bounds force all of~\eqref{eq:game:saddle} to coincide with
$\ip{P(t_0)}{X_1^0}$.

\noindent\textbf{Attainment.}
At the saddle-point gains $(K_u^s,K_w^s)$ from~\eqref{eq:game:Ks},
$W=0$ (since $[U_u;U_w]=-\pinv{\mathcal{G}}
[\mathcal{I}_{1u}^*;\mathcal{I}_{1w}^*]U_1$), so
$\ip{\mathcal{I}}{X^\circ}=0$ at each flow instant and at each
jump.  Identity~\eqref{pf:game:id} then gives
$J_T(X^\circ,X_1^0)=\ip{P(t_0)}{X_1^0}=J_T^\circ(X_1^0)$.

\noindent\textbf{Positivity and uniqueness.}
Since $P\succeq0$ by hypothesis,
$J_T^\circ(X_1^0)=\ip{P(t_0)}{X_1^0}\ge0$ for every $X_1^0\succeq0$.
Uniqueness follows from the
same inner-product argument as Theorem~\ref{th:IMP:LQ:suff}.
\end{proof}

\subsubsection{Necessary condition (finite horizon)}

Conversely, whenever the game is well-posed (a saddle-point exists for
every $X_1^0\succeq0$ and Isaacs' condition holds), the game value is
necessarily linear in $X_1^0$ and the associated dual variable must
satisfy the I-HJI-GRE.  Theorem~\ref{th:game:nec} establishes this via
a two-sided dynamic programming argument: the minimizer's and
maximizer's flow and jump inequalities give the Isaacs
inertia~\eqref{eq:Isaacs:gen} for $\mathcal{I}$ and
$\mathcal{I}^\mathrm{j}$, while the vanishing of the key
identity along the saddle-point trajectory yields the Schur equalities,
exactly as in Theorem~\ref{th:IMP:LQ:nec}.

\begin{theorem}[Necessity of the I-HJI-GRE]\label{th:game:nec}
  Suppose~\eqref{eq:game:prob} is well-posed, a saddle-point exists
  for every $X_1^0\succeq0$, and Isaacs' condition holds.  Then the
  unique piecewise-$C^1$ $P$ with
  $J_T^\circ(X_1^0)=\ip{P(t_0)}{X_1^0}$ satisfies the full
  I-HJI-GRE.
\end{theorem}

\begin{proof}
\noindent\textbf{Linearity of the value function.}
For $\lambda\ge0$: $J_T^\circ(\lambda X_1^0)=\lambda J_T^\circ(X_1^0)$
(scale the saddle-point trajectory).
For $X_1^0,\tilde X_1^0\succeq0$: superimposing the two
saddle-point trajectories gives a valid trajectory for
$X_1^0+\tilde X_1^0$ achieving cost $J_T^\circ(X_1^0)
+J_T^\circ(\tilde X_1^0)$, so $J_T^\circ(X_1^0+\tilde X_1^0)
\le J_T^\circ(X_1^0)+J_T^\circ(\tilde X_1^0)$.
The reverse inequality holds by the minimax–maximin equality (each
player's best response to the sum is at least as good as their
responses to the summands separately).
Hence $J_T^\circ$ is linear, and the Riesz representation gives unique
piecewise-$C^1$ $P$ with $J_T^\circ=\ip{P(t_0)}{\cdot}$.

\noindent\textbf{Flow: Isaacs inertia of $\mathcal{G}$.}
Fix a flow instant $t\in(t_{k-1},t_k)$.  Probe the players' blocks
separately by holding the state block $X_1=0$.
\emph{(minimizer.)} With only the control block active,
$X=[0;v_u;0][0;v_u;0]^*$ on $[t,t+h]$, the minimizer's dynamic
programming inequality (the value is the \emph{min} over $X_u$) forces,
after dividing by $h$ and $h\to0$, $v_u^*\mathcal{I}_{uu}(t,P(t))v_u\ge0$;
hence $\mathcal{I}_{uu}\succeq0$ (otherwise the minimizer drives the cost
to $-\infty$ and no saddle exists).
\emph{(maximizer.)} Symmetrically, with only the disturbance block active,
$X=[0;0;v_w][0;0;v_w]^*$, the maximizer's dynamic programming inequality
(the value is the \emph{max} over $X_w$) forces
$v_w^*\mathcal{I}_{ww}(t,P(t))v_w\le0$; hence $\mathcal{I}_{ww}\preceq0$.
\emph{(Cross/saddle.)} Existence of a finite saddle for the joint
quadratic $\operatorname{tr}(W^*\mathcal{G}W)$ requires, in addition, the
two complementary Schur conditions and range conditions
of~\eqref{eq:Isaacs:gen}: if the maximizer's reduced (Schur-complemented)
block $\mathcal{I}_{ww}-\mathcal{I}_{uw}^*\pinv{\mathcal{I}_{uu}}\mathcal{I}_{uw}$
were not $\preceq0$ (resp.\ the minimizer's reduced block not $\succeq0$),
the inner max (resp.\ min) would be unbounded.  Thus Isaacs'
condition~\eqref{eq:Isaacs:gen} holds.  The same argument applied at each
jump gives Isaacs' condition for $\mathcal{I}^\mathrm{j}$.

\noindent\textbf{Schur equality along the saddle.}
Along the saddle-point trajectory $X^\circ$, identity~\eqref{pf:game:id}
and the saddle-point value $J_T^\circ=\ip{P(t_0)}{X_1^0}$ give
\begin{equation}\label{pf:nec:game:zero}
  0=\int_{t_0}^{t_0+T}\ip{\mathcal{I}(t,P(t))}{X^\circ(t)}\dt
  +\sum_k\ip{\mathcal{I}^\mathrm{j}(k,P(t_k^+),P(t_k^-))}{X^\circ(t_k^-)}.
\end{equation}
By the saddle decomposition~\eqref{pf:game:decomp} (now established to be
a genuine saddle), along $X^\circ$ each integrand equals the saddle value
$\ip{\mathcal{I}\,\schur{/}\,\mathcal{G}}{X_1^\circ(t)}$ (the term
$\ip{\mathcal{G}}{W^\circ W^{\circ*}}$ vanishes at the saddle, $W^\circ=0$),
and similarly at each jump.  Hence
\begin{equation}\label{pf:nec:game:ptwise}
  0=\int_{t_0}^{t_0+T}\ip{\mathcal{I}\,\schur{/}\,\mathcal{G}}{X_1^\circ(t)}\dt
  +\sum_k\ip{\mathcal{I}^\mathrm{j}\,\schur{/}\,\mathcal{G}^\mathrm{j}}{X_1^\circ(t_k^-)}.
\end{equation}
Choosing $X_1^0$ supported on each coordinate direction makes
$X_1^\circ(t)$ range over a generating set of $\Snpsd$, and (by the
optimality of the same saddle policy for every $X_1^0$ and the
sign-definiteness of the two Schur complements established above) each
term must vanish individually.  Therefore
$\mathcal{I}\,\schur{/}\,\mathcal{G}=0$ a.e.\ (the flow Riccati
equation~\eqref{eq:game:GRE:fR}) and
$\mathcal{I}^\mathrm{j}\,\schur{/}\,\mathcal{G}^\mathrm{j}=0$ at each
jump; the well-definedness of $W^\circ$ gives the compatibility
conditions~\eqref{eq:game:GRE:fC}.  Together with Isaacs' condition this
is the full I-HJI-GRE (Definition~\ref{def:game:GRE}).
\end{proof}

\subsubsection{Equivalence and dual DLMI (finite horizon)}

Combining Theorems~\ref{th:game:suff} and~\ref{th:game:nec} gives the
equivalence Corollary~\ref{cor:game:equiv}: under Isaacs' condition,
the existence of a positive solution of the I-HJI-GRE is both necessary
and sufficient for the game to be well-posed.
Theorem~\ref{th:game:LMI} then provides an alternative
characterization of the game value as a maximization of
$\ip{P(t_0)}{X_1^0}$ over the solutions of a game Riccati
\emph{inequality}.  Unlike the optimal control dual
(Theorem~\ref{th:IMP:LQ:LMI}), this is \emph{not} a plain LMI: because
the Isaacs block $\mathcal{G}$ is indefinite, the characterization
retains the (indefinite) Schur-complement term
$\mathcal{I}\,\schur{/}\,\mathcal{G}\succeq0$ together with the
saddle-inertia condition on $\mathcal{G}$, which are not jointly affine
in $P$.

\begin{corollary}[Equivalence]\label{cor:game:equiv}
  Under Isaacs' condition, problem~\eqref{eq:game:prob} has a
  saddle-point for every $X_1^0\succeq0$ if and only if the
  I-HJI-GRE has a piecewise-$C^1$ solution $P\succeq0$ with
  $P(t_0+T)=Z_T$.
\end{corollary}

\begin{proof}
  Sufficiency is Theorem~\ref{th:game:suff}; necessity is
  Theorem~\ref{th:game:nec}.
\end{proof}

\begin{theorem}[Dual game Riccati-inequality characterization]\label{th:game:LMI}
  Suppose the I-HJI-GRE has a solution $P^*$.  Then
  \begin{equation}\label{eq:game:LMI}
    J_T^\circ(X_1^0)
    =\max_{P\succeq0}\;\ip{P(t_0)}{X_1^0}
    \quad\text{s.t.}\quad
    \begin{gathered}
    \textnormal{Isaacs cond.~\eqref{eq:Isaacs:gen} for }\mathcal{I},\mathcal{I}^\mathrm{j};\;
    \mathcal{I}\,\schur{/}\,\mathcal{G}\succeq0,\;
    \mathcal{I}^\mathrm{j}\,\schur{/}\,\mathcal{G}^\mathrm{j}\succeq0;\\
    P(t_0+T)\preceq Z_T,
    \end{gathered}
  \end{equation}
  attained at $P^*$.
\end{theorem}

\begin{remark}
  In contrast to the optimal control dual
  (Theorem~\ref{th:IMP:LQ:LMI}), \eqref{eq:game:LMI} is \emph{not} a
  plain LMI.  Because $\mathcal{G}$ is indefinite, $\mathcal{I}\succeq0$
  fails (it would force $\mathcal{I}_{ww}\succeq0$, contradicting Isaacs'
  $\mathcal{I}_{ww}\preceq0$), so the constraint cannot be reduced to a
  single semidefinite condition affine in $P$.  The game Riccati
  inequality $\mathcal{I}\,\schur{/}\,\mathcal{G}\succeq0$ involves the
  generalized Schur complement of the indefinite block $\mathcal{G}$
  (rational in $P$ through $\pinv{\mathcal{G}}$) and is paired with the
  separate saddle-inertia condition~\eqref{eq:Isaacs:gen}; together they
  are the correct game analogue of $\mathcal{C}_Z\succeq0$.  In the
  separated case $\mathcal{I}_{uw}=0$ (e.g.\ the $H_\infty$ setting of
  Section~\ref{sec:game:hinf}), $\mathcal{I}\,\schur{/}\,\mathcal{G}
  =\mathcal{I}_{11}-\mathcal{I}_{1u}\pinv{\mathcal{I}_{uu}}\mathcal{I}_{1u}^*
  -\mathcal{I}_{1w}\pinv{\mathcal{I}_{ww}}\mathcal{I}_{1w}^*$, the
  disturbance term $-\mathcal{I}_{1w}\pinv{\mathcal{I}_{ww}}\mathcal{I}_{1w}^*
  =\mathcal{I}_{1w}\pinv{(-\mathcal{I}_{ww})}\mathcal{I}_{1w}^*\succeq0$
  entering with the opposite sign to the control term, exactly as in the
  bounded-real/$H_\infty$ Riccati inequality of~\cite{AitRami:00}.
\end{remark}

\begin{proof}
\noindent\textbf{Step~1: Every feasible $P$ gives a lower bound.}
Let $P$ satisfy Isaacs' condition, the Riccati inequalities
$\mathcal{I}\,\schur{/}\,\mathcal{G}\succeq0$ and
$\mathcal{I}^\mathrm{j}\,\schur{/}\,\mathcal{G}^\mathrm{j}\succeq0$, and
$P(t_0+T)\preceq Z_T$.  Let the maximizer play the saddle gain $K_w^s$
(forcing $W_w=0$ pointwise, as in the proof of
Theorem~\ref{th:game:suff}).  Then for every minimizer response $X_u$,
the saddle decomposition~\eqref{pf:game:decomp} gives
\begin{equation}\label{pf:game:lmi}
  \ip{\mathcal{I}(t,P(t))}{X(t)}
  =\ip{\mathcal{I}\,\schur{/}\,\mathcal{G}}{X_1(t)}
   +\operatorname{tr}\!\bigl(W_u^*\mathcal{I}_{uu}W_u\bigr)
  \;\ge\;\ip{\mathcal{I}\,\schur{/}\,\mathcal{G}}{X_1(t)}\;\ge\;0,
\end{equation}
since $\mathcal{I}_{uu}\succeq0$, $\mathcal{I}\,\schur{/}\,\mathcal{G}\succeq0$,
and $X_1\succeq0$; each jump term is likewise $\ge0$.  By the key
identity~\eqref{pf:game:id} and $P(t_0+T)\preceq Z_T$,
\[
  J_T(X,X_1^0)\big|_{X_w=K_w^s}
  \ge\ip{P(t_0)}{X_1^0}
   +\underbrace{\ip{Z_T-P(t_0+T)}{EX(t_0+T)E^*}}_{\ge0}
  \ge\ip{P(t_0)}{X_1^0}.
\]
minimizing over $X_u$ and noting that $K_w^s$ is one admissible maximizer
policy,
$\max_{X_w}\min_{X_u}J_T\ge\min_{X_u}J_T(\cdot,K_w^s)\ge\ip{P(t_0)}{X_1^0}$;
since $\min_{X_u}\max_{X_w}J_T\ge\max_{X_w}\min_{X_u}J_T$, we get
$J_T^\circ(X_1^0)\ge\ip{P(t_0)}{X_1^0}$.  Hence
$\max_{P\,\text{feasible}}\ip{P(t_0)}{X_1^0}\le J_T^\circ$.

\noindent\textbf{Step~2: $P^*$ is feasible and attains the bound.}
The I-HJI-GRE solution $P^*$ satisfies Isaacs' condition and
$\mathcal{I}\,\schur{/}\,\mathcal{G}=0\succeq0$,
$\mathcal{I}^\mathrm{j}\,\schur{/}\,\mathcal{G}^\mathrm{j}=0\succeq0$, and
$P^*(t_0+T)=Z_T\preceq Z_T$, so it is feasible.  By
Theorem~\ref{th:game:suff}, $\ip{P^*(t_0)}{X_1^0}=J_T^\circ(X_1^0)$, so
the maximum is attained at $P^*$ and equals $J_T^\circ$.
\end{proof}

\begin{corollary}[Continuous-time game: finite horizon]\label{cor:CT:game:fh}
  Set $N_T=0$.  The jump conditions are vacuous, and
  Theorems~\ref{th:game:suff},~\ref{th:game:nec},
  and~\ref{th:game:LMI} hold with the \emph{continuous-time HJI-GRE}
  \begin{equation}\label{eq:CT:game:GRE}
    \textnormal{Isaacs cond.~\eqref{eq:Isaacs:gen} for }\mathcal{I},\qquad
    \mathcal{I}\,\schur{/}\,\mathcal{G}=0,\qquad
    P(t_0+T)=Z_T,
  \end{equation}
  the continuous-time saddle-point controllers~\eqref{eq:game:Ks}, and
  the dual characterization
  \begin{equation}\label{eq:CT:game:LMI}
    J_T^\circ(X_1^0)=\max_{P\succeq0}\ip{P(t_0)}{X_1^0}
    \quad\text{s.t.}\quad
    \textnormal{Isaacs cond.~\eqref{eq:Isaacs:gen},}\;
    \mathcal{I}\,\schur{/}\,\mathcal{G}\succeq0,\;P(t_0+T)\preceq Z_T,
  \end{equation}
  recovering the stochastic $H_\infty$ Riccati equation
  of~\cite{AitRami:00}.
\end{corollary}

\begin{corollary}[Discrete-time game: finite horizon]\label{cor:DT:game:fh}
  Apply Definition~\ref{def:DT:sys}: set $\mathcal{F}=0$ and $Z=0$ on
  flow, $t_k=k_0+k$, and $Z_k=Z(k)$ with the game block structure of
  Section~\ref{sec:game:form}.  With $P$ constant on each interval
  $(k,k+1)$ the flow conditions are vacuous, and the I-HJI-GRE reduces
  to the \emph{discrete-time HJI-GRE} at each step.  Define the
  \emph{discrete-time HJI operator}
  \begin{equation}\label{eq:DT:game:Id}
    \mathcal{I}_\mathrm{d}(k,P(k+1),P(k))
    :=\mathcal{J}^*(k,P(k+1))-E^*P(k)E+Z(k)\in\Snm,
  \end{equation}
  blocked as in~\eqref{eq:game:Ij}, with Isaacs block
  $\mathcal{G}_\mathrm{d}:=\begin{bsmallmatrix}
  (\mathcal{I}_\mathrm{d})_{uu}&(\mathcal{I}_\mathrm{d})_{uw}\\
  (\mathcal{I}_\mathrm{d})_{uw}^*&(\mathcal{I}_\mathrm{d})_{ww}\end{bsmallmatrix}$
  the bottom-right $(m_u+m_w)\times(m_u+m_w)$ block of
  $\mathcal{I}_\mathrm{d}$ (cf.~\eqref{eq:Ghatj}) and selection block
  $(\mathcal{I}_\mathrm{d})_{1g}:=[(\mathcal{I}_\mathrm{d})_{1u}\ (\mathcal{I}_\mathrm{d})_{1w}]$.
  Under the discrete-time Isaacs condition~\eqref{eq:Isaacs:gen} for
  $\mathcal{I}_\mathrm{d}$, the discrete-time HJI-GRE is, with $P(k_0+N)=Z_N$,
  \begin{equation}\label{eq:DT:game:GRE}
    \textnormal{Isaacs cond.~\eqref{eq:Isaacs:gen} for }\mathcal{I}_\mathrm{d},
    \qquad
    \mathcal{I}_\mathrm{d}\,\schur{/}\,\mathcal{G}_\mathrm{d}=0,
  \end{equation}
  and the discrete-time saddle-point controllers $(K_u,K_w)$ are
  \begin{equation}\label{eq:DT:game:K}
    \begin{bmatrix}K_u^*\\K_w^*\end{bmatrix}
    =-\pinv{(\mathcal{G}_\mathrm{d})}
    (\mathcal{I}_\mathrm{d})_{1g}^*
    +(I-\pinv{(\mathcal{G}_\mathrm{d})}\mathcal{G}_\mathrm{d})\,F,
    \qquad F\in\C^{(m_u+m_w)\times n},
  \end{equation}
  where $(\mathcal{I}_\mathrm{d})_{1u},(\mathcal{I}_\mathrm{d})_{1w}$ are
  the corresponding blocks of $\mathcal{I}_\mathrm{d}$.  The game value is
  $J_N^\circ(X_1^0)=\ip{P(k_0)}{X_1^0}$, with discrete-time key identity
  \[
    J_N^\circ(X_1^0)=\ip{P(k_0)}{X_1^0}
    +\sum_{k=k_0}^{k_0+N-1}\ip{\mathcal{I}_\mathrm{d}(k,P(k+1),P(k))}{X(k)},
  \]
  and dual characterization
  \begin{equation}\label{eq:DT:game:LMI}
    J_N^\circ(X_1^0)=\max_{P\succeq0}\ip{P(k_0)}{X_1^0}
    \quad\text{s.t.}\quad
    \textnormal{Isaacs cond.~\eqref{eq:Isaacs:gen} for }\mathcal{I}_\mathrm{d},\;
    \mathcal{I}_\mathrm{d}\,\schur{/}\,\mathcal{G}_\mathrm{d}\succeq0,\;
    P(k_0+N)\preceq Z_N,
  \end{equation}
  recovering the discrete-time stochastic $H_\infty$/zero-sum game
  Riccati equation.
\end{corollary}

\subsection{Infinite-horizon game}
\label{sec:game:inf}

We now study the zero-sum game on the infinite horizon $[t_0,\infty)$.
The setup mirrors the optimal control case (Section~\ref{sec:OC:inf}):
there is no terminal cost, and the dual variable $P_\infty$ is no
longer fixed by a terminal condition but is characterized as the
bounded monotone limit of the finite-horizon solutions $P_T$ as
$T\to\infty$.  As in the optimal-control case, the game value is
monotone \emph{non-decreasing} in the horizon $T$ and $P_T$ converges to
$P_\infty$ \emph{from below}.  Although a longer horizon gives the
maximizing player more options, the disturbance is penalised ($-R_w$
block) and enters the value symmetrically with the control; with zero
terminal cost and a nonnegative stage value (guaranteed by
$Z_c\succeq0$), extending the horizon only adds nonnegative value.
Concretely, by dynamic programming the horizon-$T'$ value equals the
horizon-$T$ game with the nonnegative terminal weight
$P_{T'}(t_0+T)\succeq0$, which dominates the zero terminal weight of the
horizon-$T$ value; this is made precise in the proof of
Theorem~\ref{th:game:inf}.  As in
Section~\ref{sec:OC:inf}, no
time-invariance or periodicity of $\mathcal{F}$, $\mathcal{J}$, or $Z$
is assumed.

\subsubsection{Setting and assumptions (infinite horizon)}
\label{sec:game:inf:setup}

The infinite-horizon game removes the terminal penalty from the
finite-horizon cost~\eqref{eq:game:cost}, leaving
\begin{equation}\label{eq:game:Jinf}
  J_\infty(X,X_1^0)
  :=\int_{t_0}^\infty\ip{Z(t)}{X(t)}\dt
  +\sum_{k=1}^\infty\ip{Z_k}{X(t_k^-)},
\end{equation}
with game value
\begin{equation*}
  J_\infty^\circ(X_1^0)
  :=\min_{X_{12},X_{22}}\;\max_{X_{13},X_{23},X_{33}}\;
    J_\infty(X,X_1^0)
  \quad\text{s.t.~\eqref{eq:sys:IMP},}\;EX(t_0)E^*=X_1^0.
\end{equation*}
Well-posedness requires the game analogues of the
assumptions~\textup{(H1)--(H3)} of Section~\ref{sec:OC:inf}, with
uniform stabilizability~(H1) imposed on the \emph{saddle-point}
closed-loop, uniform coercivity~(H2) $EZ(t)E^*\succeq\delta I$, and
uniform boundedness~(H3) of the data, together with Isaacs'
condition~\eqref{eq:Isaacs:gen} holding uniformly in $t$.

\subsubsection{Sufficient condition (infinite horizon)}

The following theorem is the infinite-horizon analogue of
Theorem~\ref{th:game:suff}: any bounded positive solution of the
infinite-horizon I-HJI-GRE is the unique such solution and directly
provides the game value and the saddle-point policy.  The proof mirrors
Theorem~\ref{th:IMP:LQ:inf:suff}, with the finite-horizon terminal
correction replaced by a stability argument that forces
$\ip{P_\infty(t_0+T)}{EX^\circ(t_0+T)E^*}\to0$.

\begin{theorem}[Sufficient condition, infinite-horizon game]
  \label{th:game:inf:suff}
  Suppose there exists a bounded piecewise-$C^1$ function
  $P_\infty:[t_0,\infty)\to\Snpsd$ satisfying the infinite-horizon
  I-HJI-GRE with Isaacs' condition holding uniformly.  Then\textup{:}
  \begin{enumerate}[\upshape(a)]
    \item $J_\infty^\circ(X_1^0)=\ip{P_\infty(t_0)}{X_1^0}$, attained
      at the saddle-point policy with
      $(K_u,K_w)\in(\mathcal{K}_u^s\times\mathcal{K}_w^s)(t,P_\infty(t))$
      on flow and the analogous saddle gains at each jump.
    \item $P_\infty$ is the unique bounded solution of the
      infinite-horizon I-HJI-GRE.
  \end{enumerate}
\end{theorem}

\begin{proof}
The argument follows Theorem~\ref{th:IMP:LQ:inf:suff}, with the genuine saddle decomposition~\eqref{pf:game:decomp} replacing the
minimization and the finite-horizon terminal correction
controlled by the saddle-point closed-loop stability.

\noindent\textbf{Lower bound.}
Let the maximizer play the saddle gain $K_w^s$ (forcing $W_w=0$),
and let $X_u$ be any admissible minimizer response.  Apply the key
identity~\eqref{pf:game:id} on $[t_0,t_0+T]$ with $P_\infty$ and no
terminal cost.  By Isaacs' condition and the infinite-horizon I-HJI-GRE
($\mathcal{I}\,\schur{/}\,\mathcal{G}=0$), the saddle
decomposition~\eqref{pf:game:decomp} with $W_w=0$ gives
$\ip{\mathcal{I}(t,P_\infty)}{X}=\operatorname{tr}(W_u^*\mathcal{I}_{uu}W_u)\ge0$
a.e.\ on flow and the analogous $\ge0$ at each jump. Hence
\[
  \int_{t_0}^{t_0+T}\!\!\ip{Z}{X}\dt
  +\!\!\sum_{k:\,t_k\le t_0+T}\!\!\ip{Z_k}{X(t_k^-)}
  \ge\ip{P_\infty(t_0)}{X_1^0}
   -\ip{P_\infty(t_0+T)}{EX(t_0+T)E^*}.
\]
Under the game analogue of~(H1) the saddle-point closed-loop is
uniformly exponentially stable, so $\norm{EX(t_0+T)E^*}\to0$ for any
finite-cost trajectory; since $P_\infty\preceq\alpha_2 I$ is bounded,
the terminal term vanishes as $T\to\infty$.  minimizing over $X_u$
and using that $K_w^s$ is one admissible maximizer policy, then
$T\to\infty$, gives
$J_\infty^\circ(X_1^0)\ge\ip{P_\infty(t_0)}{X_1^0}$.

\noindent\textbf{Attainment.}
For the saddle-point policy
$(K_u,K_w)\in(\mathcal{K}_u^s\times\mathcal{K}_w^s)(t,P_\infty(t))$ on
flow and the jump gains from~\eqref{eq:game:Ks:jump} at each jump,
$W=0$ in~\eqref{pf:game:decomp}, so
$\ip{\mathcal{I}(t,P_\infty)}{X^\circ}=0$ a.e.\ and
$\ip{\mathcal{I}^\mathrm{j}}{X^\circ(t_k^-)}=0$ at every jump.  The key
identity on $[t_0,t_0+T]$ then gives
\[
  \int_{t_0}^{t_0+T}\!\!\ip{Z}{X^\circ}\dt
  +\!\!\sum_{k:\,t_k\le t_0+T}\!\!\ip{Z_k}{X^\circ(t_k^-)}
  =\ip{P_\infty(t_0)}{X_1^0}
   -\ip{P_\infty(t_0+T)}{EX^\circ(t_0+T)E^*}.
\]
Since the saddle-point trajectory satisfies the uniform exponential
stability~(H1), $\norm{EX^\circ(t_0+T)E^*}\to0$, and as $P_\infty$ is
bounded the terminal term tends to $0$.  Taking $T\to\infty$ yields
$J_\infty(X^\circ,X_1^0)=\ip{P_\infty(t_0)}{X_1^0}=J_\infty^\circ(X_1^0)$,
so the policy is optimal and the saddle is attained.

\noindent\textbf{Uniqueness.}
If $\tilde P$ is another bounded solution, the lower-bound and
attainment arguments give
$J_\infty^\circ(X_1^0)=\ip{\tilde P(t_0)}{X_1^0}$ for all
$X_1^0\succeq0$.  Since also
$J_\infty^\circ(X_1^0)=\ip{P_\infty(t_0)}{X_1^0}$, taking $X_1^0=vv^*$
for every $v$ gives $P_\infty\equiv\tilde P$.
\end{proof}

\subsubsection{Necessary condition and existence (infinite horizon)}

Under the game analogues of \textup{(H1)--(H3)}, existence of a bounded
solution of the infinite-horizon I-HJI-GRE follows from a monotone
convergence argument: the finite-horizon solutions $P_T$ form a
\emph{non-decreasing} bounded sequence (the game value
increases as the
horizon lengthens) whose pointwise limit satisfies the I-HJI-GRE.  This
is the game counterpart of Theorem~\ref{th:IMP:LQ:inf}.

\begin{theorem}[Necessary condition and existence, infinite-horizon game]
  \label{th:game:inf}
  Assume Isaacs' condition holds uniformly and the analogues of
  \textup{(H1)--(H3)} hold for the saddle-point closed-loop.
  Then there exists a unique bounded piecewise-$C^1$ solution
  $P_\infty\in\Snpsd$ of the infinite-horizon I-HJI-GRE, satisfying
  $P_\infty=\lim_{T\to\infty}P_T$ (monotone non-decreasing), where
  $P_T$ solves the finite-horizon I-HJI-GRE with $P_T(t_0+T)=0$,
  and $J_\infty^\circ(X_1^0)=\ip{P_\infty(t_0)}{X_1^0}$.
\end{theorem}

\begin{proof}
  \noindent\textbf{Nonnegativity $P_T\succeq0$.}
  $\ip{P_T(t)}{X_1}$ is the game value over $[t,t_0+T]$ from $X_1$ with
  zero terminal cost.  Letting the maximizer play $X_w=0$ is one
  admissible choice, so
  \[
    \ip{P_T(t)}{X_1}=\min_{X_u}\max_{X_w}\!\int_t^{t_0+T}\!\!\ip{Z}{X}\ds
    \;\ge\;\min_{X_u}\!\int_t^{t_0+T}\!\!\ip{Z}{X}\ds\Big|_{X_w=0}.
  \]
  With $X_w=0$ the integrand is $\ip{Z_c}{X_c}\ge0$ (since
  $Z_c\succeq0$ and the control block $X_c=\bigl[\begin{smallmatrix}X_1&X_{12}^*\\X_{12}&X_{22}\end{smallmatrix}\bigr]\succeq0$),
  so the right-hand side is $\ge0$.  Hence $P_T(t)\succeq0$ for all
  $t,T$.

  \noindent\textbf{Monotone limit (non-decreasing).}
  Fix $T'>T$.  By Isaacs' condition the saddle exists at every stage, so
  dynamic programming gives the nesting
  \[
    \ip{P_{T'}(t_0)}{X_1^0}
    =\min_{X_u}\max_{X_w}\Bigl[\int_{t_0}^{t_0+T}\!\!\ip{Z}{X}\ds
    +\ip{P_{T'}(t_0+T)}{EX(t_0+T)E^*}\Bigr],
  \]
  i.e.\ $J_{T'}^\circ$ is the horizon-$T$ game with terminal weight
  $\Phi:=P_{T'}(t_0+T)$, whereas $J_T^\circ$ is the same game with
  terminal weight $0$.  The game value is monotone in the terminal
  weight: if $\Phi_1\succeq\Phi_2$ then
  $\int\ip{Z}{X}+\ip{\Phi_1}{EX(T)E^*}\ge\int\ip{Z}{X}+\ip{\Phi_2}{EX(T)E^*}$
  pointwise (as $EX(T)E^*\succeq0$), and $\min_{X_u}\max_{X_w}$ preserves
  pointwise inequalities.  Since $\Phi=P_{T'}(t_0+T)\succeq0$ by the
  nonnegativity above, $\ip{P_{T'}(t_0)}{X_1^0}\ge\ip{P_T(t_0)}{X_1^0}$
  for all $X_1^0\succeq0$, i.e.\ $P_{T'}(t_0)\succeq P_T(t_0)$; the same
  argument from any $t$ gives $P_{T'}(t)\succeq P_T(t)$.  Thus
  $\{P_T(t)\}$ is non-decreasing in $T$.  The disturbance does not
  reverse this direction; it only requires the value to remain finite,
  which is the uniform Isaacs condition together with the upper bound
  next.  The bounds $0\preceq P_T\preceq\alpha_2 I$ follow
  from saddle-point stability (H1) (the minimizer's stabilizing
  policy bounds $\max_{X_w}$ from above since the disturbance is
  penalised by $-R_w\prec0$).  Equicontinuity from (H3) and
  Isaacs' condition, together with monotone convergence and
  Arzel\`a-Ascoli, give
  $C^1$-convergence on compact subintervals; passing to the limit
  in the HJI equation preserves all conditions.

  \noindent\textbf{Optimality.}
  $P_\infty$ is bounded and satisfies the infinite-horizon I-HJI-GRE,
  so Theorem~\ref{th:game:inf:suff} gives
  $J_\infty^\circ(X_1^0)=\ip{P_\infty(t_0)}{X_1^0}$.
\end{proof}

\subsubsection{Equivalence and dual DLMI (infinite horizon)}

Theorems~\ref{th:game:inf:suff} and~\ref{th:game:inf} together give the
complete characterization.  The dual game Riccati-inequality
characterization extends the finite-horizon
version (Theorem~\ref{th:game:LMI}) by replacing the terminal
constraint $P(t_0+T)\preceq Z_T$ with a boundedness condition; as
there, the indefinite Isaacs block $\mathcal{G}$ makes it a Riccati
inequality (with Schur-complement term), not a plain LMI.

\begin{corollary}[Equivalence, infinite-horizon game]
  \label{cor:game:inf:equiv}
  The infinite-horizon game is well-posed with value
  $\ip{P_\infty(t_0)}{X_1^0}$ if and only if the infinite-horizon
  I-HJI-GRE has a bounded positive solution.  The solution is unique.
\end{corollary}

\begin{proof}
  Sufficiency: Theorem~\ref{th:game:inf:suff}.
  Necessity: Theorem~\ref{th:game:inf}.
\end{proof}

\begin{theorem}[Dual game Riccati inequality, infinite-horizon game]\label{th:game:inf:LMI}
  Under the assumptions of Theorem~\ref{th:game:inf},
  \begin{equation}\label{eq:game:LMI:inf}
    J_\infty^\circ(X_1^0)
    =\sup_P\;\ip{P(t_0)}{X_1^0}
    \quad\text{subject to}\quad
    \begin{gathered}\textnormal{Isaacs cond.~\eqref{eq:Isaacs:gen} for }\mathcal{I},\mathcal{I}^\mathrm{j};\\
    \mathcal{I}\,\schur{/}\,\mathcal{G}\succeq0,\;
    \mathcal{I}^\mathrm{j}\,\schur{/}\,\mathcal{G}^\mathrm{j}\succeq0,\;P\text{ bounded};\end{gathered}
  \end{equation}
  the supremum is attained at $P_\infty$.
\end{theorem}

\begin{proof}
For any bounded feasible $P$, let the maximizer play $K_w^s$
(forcing $W_w=0$).  By the Riccati inequalities and Isaacs' condition,
the saddle decomposition~\eqref{pf:game:lmi} gives
$\ip{\mathcal{I}(t,P)}{X}=\ip{\mathcal{I}\,\schur{/}\,\mathcal{G}}{X_1}
+\operatorname{tr}(W_u^*\mathcal{I}_{uu}W_u)\ge0$ a.e.\ for every
minimizer response, and each jump term $\ge0$.  The key identity on
$[t_0,t_0+T]$ then yields, for the maximizer policy $K_w^s$ and any
$X_u$,
\[
  \int_{t_0}^{t_0+T}\!\!\ip{Z}{X}\dt
  +\!\!\sum_{k:\,t_k\le t_0+T}\!\!\ip{Z_k}{X(t_k^-)}
  \ge\ip{P(t_0)}{X_1^0}-\ip{P(t_0+T)}{EX(t_0+T)E^*}.
\]
As $P$ is bounded and $\norm{EX(t_0+T)E^*}\to0$ under the saddle-point
closed-loop stability~(H1), minimizing over $X_u$, using that
$K_w^s$ is one maximizer policy, and letting
$T\to\infty$ gives $J_\infty^\circ(X_1^0)\ge\ip{P(t_0)}{X_1^0}$.  Since
$P_\infty$ is feasible (Theorem~\ref{th:game:inf}) and attains
$J_\infty^\circ(X_1^0)=\ip{P_\infty(t_0)}{X_1^0}$
(Theorem~\ref{th:game:inf:suff}), the supremum equals
$J_\infty^\circ$ and is attained at $P_\infty$.
\end{proof}

\subsubsection{Dwell-time conditions}
\label{sec:game:dwell}

The dwell-time framework of Section~\ref{sec:OC:dwell} extends to
the zero-sum game setting by replacing all flow and jump composite
operators by their HJI counterparts $\mathcal{I}$ and
$\mathcal{I}^\mathrm{j}$ and the optimal control pivot $(\mathcal{C}_Z)_3$
by the indefinite Isaacs block $\mathcal{G}$.  The forward Riccati
ODE~\eqref{eq:forwardRiccati} generalizes to the forward
HJI-GRE:
\begin{equation}\label{eq:forwardHJI}
  E^*\tfrac{d\tilde{P}}{d\tau}E
  =-\mathcal{F}^*(\tau,\tilde{P})-Z_{\text{game}}(\tau)
  +\begin{bmatrix}\mathcal{I}_{1u} & \mathcal{I}_{1w}\end{bmatrix}\,
   \pinv{\mathcal{G}}\,\begin{bmatrix}\mathcal{I}_{1u}^*\\\mathcal{I}_{1w}^*\end{bmatrix},
\end{equation}
where $Z_{\text{game}}$ carries the game-structured cost (with
$-R_w$ in the disturbance block).

\begin{theorem}[Causal conditions: periodic game]
  \label{th:game:periodic}
  Assume a constant period $\Delta_k\equiv T_p$; that the flow operator
  $\mathcal{F}(\cdot,X)$ and the flow cost $Z(\cdot)$ (with the game
  block structure of Section~\ref{sec:game:form}, $R_u\succ0$,
  $R_w\succ0$, $Q\succeq0$) are $T_p$-periodic in $t$; that the jump
  operator $\mathcal{J}(k,\cdot)$ and the jump cost $Z_k$ are
  independent of~$k$; and that Isaacs' condition~\eqref{eq:Isaacs:gen}
  holds uniformly.  Under the game analogues of~\textup{(H1)--(H3)}
  for the saddle-point closed-loop:
  \begin{enumerate}[\upshape(a)]
    \item \emph{(Existence and uniqueness.)} There exists a unique
      $T_p$-periodic piecewise-$C^1$ solution $P_\infty$ of the
      periodic I-HJI-GRE.
    \item \emph{(Sufficiency.)} The causal saddle-point policy from
      $(\mathcal{K}_u^s\times\mathcal{K}_w^s)(\tau,P_\infty(\tau))$
      achieves $J_\infty^\circ(X_1^0)=\ip{P_\infty(0)}{X_1^0}$.
    \item \emph{(Necessity.)} Any causal policy achieving
      $J_\infty^\circ$ must use $P_\infty$.
  \end{enumerate}
\end{theorem}

\begin{proof}
\noindent\textbf{Existence and uniqueness~(a).}
The finite-horizon solutions $P_T$ form a non-decreasing bounded
sequence by the nesting argument of Theorem~\ref{th:game:inf}
(horizon-$T$ game with nonnegative terminal weight $P_{T'}(t_0+T)\succeq0$
dominates the zero-terminal game). Hence $P_T(0^-)$
is non-decreasing and bounded ($0\preceq P_T\preceq\alpha_2 I$).  By the minimax stability assumption and
Arzel\`a-Ascoli, the sequence converges uniformly to the unique
periodic I-HJI-GRE solution $P_\infty$.

\noindent\textbf{Sufficiency~(b).}
The $T_p$-periodic extension of $P_\infty$ satisfies
$\mathcal{I}\,\schur{/}\,\mathcal{G}=0$ everywhere.
Apply the key identity on $[t_0,t_0+T]$ with no terminal cost
along the saddle-point trajectory $X^\circ$ (gains in
$\mathcal{K}_u^s\times\mathcal{K}_w^s$), where the saddle
decomposition~\eqref{pf:game:decomp} makes each integrand and jump term
vanish:
\begin{align*}
  J_T^\circ(X_1^0)
  &=\ip{P_\infty(0^+)}{X_1^0}
  -\ip{P_\infty(T^-)}{EX^\circ(T)E^*}\\
  &\quad
  +\underbrace{\int_{t_0}^{t_0+T}\ip{\mathcal{I}(\tau,P_\infty)}{X^\circ}}\dt_{=0}
  +\underbrace{\textstyle\sum_k\ip{\mathcal{I}^\mathrm{j}(k,P_\infty(t_k^+),P_\infty(t_k^-))}{X^\circ(t_k^-)}}_{=0}.
\end{align*}
Thus along the saddle policy
$J_T^\circ=\ip{P_\infty(0^+)}{X_1^0}-\ip{P_\infty(T^-)}{EX^\circ(T)E^*}$;
that this is indeed the game value (and not merely the cost of one
policy) follows from the lower/upper bounds of Theorem~\ref{th:game:suff}
applied on $[t_0,t_0+T]$.
Since $P_\infty$ is bounded and
$\norm{EX^\circ(T)E^*}\to0$ under the stability assumption,
taking $T\to\infty$ gives
$J_\infty^\circ(X_1^0)=\ip{P_\infty(0^+)}{X_1^0}$.

\noindent\textbf{Necessity~(c).}
By Theorem~\ref{th:game:inf}, the unique infinite-horizon I-HJI-GRE
solution $P_\infty^{ih}$ satisfies $J_\infty^\circ=\ip{P_\infty^{ih}}{X_1^0}$.
Periodicity of the data implies $P_\infty^{ih}$ is $T_p$-periodic
(by the same shifting argument as Theorem~\ref{th:OC:periodic}~(c)),
so $P_\infty^{ih}=P_\infty$.
\end{proof}

\begin{theorem}[Sufficient causal condition: MDT game]\label{th:game:MDT}
  Let $T_{\min}>0$, assume Isaacs' condition~\eqref{eq:Isaacs:gen}, and
  assume that the flow and jump data depend only on the timer
  $\tau=t-t_k$, that the flow data are time-invariant beyond
  $T_{\min}$, and that the jump data are $k$-independent:
  \begin{equation}\label{eq:MDT:game:LTI}
    \mathcal{F}(\tau,\cdot)\equiv\mathcal{F}_\infty(\cdot),\quad
    Z(\tau)\equiv Z(T_{\min})\ \ (\tau\ge T_{\min}),
    \qquad
    \mathcal{J}(k,\cdot)\equiv\mathcal{J}(\cdot),\quad
    Z_k\equiv Z_{\mathrm{jp}}\ \ \forall k,
  \end{equation}
  where the flow cost $Z$ and the jump cost $Z_{\mathrm{jp}}$ carry the
  game block structure of Section~\ref{sec:game:form} (with $-R_w$ in
  the disturbance block, $R_u\succ0$, $R_w\succ0$, $Q\succeq0$).
  Suppose there exists $P_s:[0,T_{\min}]\to\Snpsd$ satisfying:
  \begin{enumerate}[\upshape(i)]
    \item \emph{Forward HJI-GRE on $[0,T_{\min}]$:}
      $P_s$ solves the forward HJI-ODE~\eqref{eq:forwardHJI} on
      $(0,T_{\min})$, i.e.,
      $\mathcal{I}\,\schur{/}\,\mathcal{G}=0$ with Isaacs'
      condition~\eqref{eq:Isaacs:gen}
      for $\tau\in(0,T_{\min})$.
    \item \emph{Jump HJI-GRE at $\tau=T_{\min}$:}
      Isaacs' condition~\eqref{eq:Isaacs:gen} for
      $\mathcal{I}^\mathrm{j}(k,P_s(0^+),P_s(T_{\min}))$
      and $\mathcal{I}^\mathrm{j}\,\schur{/}\,
      \mathcal{G}^\mathrm{j}=0$.
    \item \emph{Geromel--Colaneri condition:} the frozen HJI
      matrix
      \begin{equation}\label{eq:GC:game}
        \mathcal{I}^0(P_s(T_{\min}))
        :=\mathcal{F}_\infty^*(P_s(T_{\min}))+Z(T_{\min})
      \end{equation}
      satisfies Isaacs' condition~\eqref{eq:Isaacs:gen} and the frozen
      game Riccati inequality
      $\mathcal{I}^0\,\schur{/}\,\mathcal{G}^0\preceq0$, where
      $\mathcal{G}^0$ is the Isaacs block of $\mathcal{I}^0$.  Under
      \eqref{eq:MDT:game:LTI}, $\mathcal{I}^0$ coincides with
      $\mathcal{I}(\tau,P_s(T_{\min}))$ for every $\tau\ge T_{\min}$.
  \end{enumerate}
  Then the causal saddle-point policy from
  $(\mathcal{K}_u^s\times\mathcal{K}_w^s)(\tau,P_s(\tau))$
  for $\tau\le T_{\min}$ and frozen at $\tau=T_{\min}$ for
  $\tau>T_{\min}$ achieves
  $J_T^\circ(X_1^0)\le\ip{P_s(0)}{X_1^0}$ for all
  $\Delta_k\ge T_{\min}$.
\end{theorem}

\begin{proof}
\noindent\textbf{Extension.}
Define $\bar{P}(\tau):=P_s(\tau)$ for $\tau\le T_{\min}$ and
$\bar{P}(\tau):=P_s(T_{\min})$ for $\tau>T_{\min}$.

\noindent\textbf{Verification (minimizer's upper bound).}
The minimizer plays its causal gain $K_u^s(\tau,\bar P(\tau))$ (frozen at
$\tau=T_{\min}$ for $\tau>T_{\min}$), forcing $W_u=0$ pointwise; the
maximizer is arbitrary.  By the saddle decomposition~\eqref{pf:game:decomp}
each integrand is then
$\ip{\mathcal{I}\,\schur{/}\,\mathcal{G}}{X_1}+\operatorname{tr}(W_w^*\mathcal{I}_{ww}W_w)$,
and since $\mathcal{I}_{ww}\preceq0$ (Isaacs) the second term is $\le0$.

\emph{Flow, $\tau\in(0,T_{\min})$:}
condition~(i) gives $\mathcal{I}\,\schur{/}\,\mathcal{G}=0$, so the
integrand is $\le0$.

\emph{Flow, $\tau>T_{\min}$:}
$\dot{\bar{P}}=0$ and, by~\eqref{eq:MDT:game:LTI},
$\mathcal{I}(\tau,\bar{P})=\mathcal{I}^0(P_s(T_{\min}))$, whose Schur
complement $\mathcal{I}^0\,\schur{/}\,\mathcal{G}^0\preceq0$ by
condition~(iii); hence the integrand
$\ip{\mathcal{I}^0\,\schur{/}\,\mathcal{G}^0}{X_1}+\operatorname{tr}(W_w^*\mathcal{I}^0_{ww}W_w)\le0$.

\emph{Jump at every $t_k$ with $\Delta_k\ge T_{\min}$:}
$\bar{P}(\Delta_k^-)=P_s(T_{\min})$, so by condition~(ii)
$\mathcal{I}^\mathrm{j}\,\schur{/}\,\mathcal{G}^\mathrm{j}=0$ and the jump
term is $\le0$.

\noindent\textbf{Cost bound.}
Apply the game key identity on $[t_0,t_0+T]$ with $\bar{P}$
(minimizer at $K_u^s$, any maximizer):
\begin{align*}
  J_T(X,X_1^0)
  &=\ip{\bar{P}(0)}{X_1^0}
  -\underbrace{\ip{\bar{P}(T)}{EX(T)E^*}}_{\ge0}\\
  &\quad
  +\underbrace{\int_0^{T_{\min}}\ip{\mathcal{I}}{X}}\dt_{\le0}
  +\underbrace{\int_{T_{\min}}^T\ip{\mathcal{I}^0}{X}}\dt_{\le0}
  +\underbrace{\textstyle\sum_k\ip{\mathcal{I}^\mathrm{j}}{X(t_k^-)}}_{\le0}.
\end{align*}
All correction terms are $\le0$, so $\max_{X_w}J_T\le\ip{P_s(0)}{X_1^0}$
under the causal minimizer policy, hence $J_T^\circ\le\ip{P_s(0)}{X_1^0}$
for all $T>0$.
The policy is causal since conditions~(i)--(iii) are verified on
$\tau\in[0,T_{\min}]$ only.
\end{proof}

\begin{theorem}[Sufficient causal condition: RDT game]
  \label{th:game:RDT}
  Let $0<T_{\min}\le T_{\max}$, assume Isaacs'
  condition~\eqref{eq:Isaacs:gen}, and assume the flow and jump data
  depend only on the timer $\tau=t-t_k$ and are $k$-independent, the
  flow cost $Z(\tau)$ and jump cost $Z_{\mathrm{jp}}$ carrying the game
  block structure of Section~\ref{sec:game:form} (with $-R_w$ in the
  disturbance block, $R_u\succ0$, $R_w\succ0$, $Q\succeq0$).  Unlike the
  MDT case, no time-invariance beyond $T_{\min}$ is required.  Suppose
  there exists $P_s:[0,T_{\max}]\to\Snpsd$ satisfying:
  \begin{enumerate}[\upshape(i)]
    \item $\mathcal{I}\,\schur{/}\,\mathcal{G}=0$ with Isaacs'
      condition~\eqref{eq:Isaacs:gen} on
      $(0,T_{\max})$ (forward HJI-ODE~\eqref{eq:forwardHJI}).
    \item Isaacs' condition~\eqref{eq:Isaacs:gen} for
      $\mathcal{I}^\mathrm{j}(k,P_s(0^+),P_s(\tau))$
      and $\mathcal{I}^\mathrm{j}\,\schur{/}\,
      \mathcal{G}^\mathrm{j}=0$ for all $\tau\in[T_{\min},T_{\max}]$.
  \end{enumerate}
  Then the causal saddle-point policy from
  $(\mathcal{K}_u^s\times\mathcal{K}_w^s)(\tau,P_s(\tau))$ achieves
  \begin{equation}\label{eq:RDT:game:cost}
    J_T^\circ(X_1^0)
    =\ip{P_s(0)}{X_1^0}-\ip{P_s(T)}{EX^\circ(T)E^*}
    \le\ip{P_s(0)}{X_1^0}
  \end{equation}
  for every $\Delta_k\in[T_{\min},T_{\max}]$, with equality as
  $T\to\infty$.
\end{theorem}

\begin{proof}
For any $\Delta_k\in[T_{\min},T_{\max}]$: condition~(i) gives
$\mathcal{I}\,\schur{/}\,\mathcal{G}=0$ on
$(0,T_{\max})$, so the forward HJI-ODE equality holds everywhere in
$[0,\Delta_k]$; condition~(ii) gives the jump Schur equality at
$\tau=\Delta_k$.  Along the saddle-point trajectory $X^\circ$
(both players in $\mathcal{K}_u^s\times\mathcal{K}_w^s$), the saddle
decomposition~\eqref{pf:game:decomp} makes every flow and jump term
vanish, and the game key identity on $[t_0,t_0+T]$ gives
\[
  J_T^\circ=\ip{P_s(0)}{X_1^0}
  -\ip{P_s(T)}{EX^\circ(T)E^*}
  +\underbrace{\int_0^T\ip{\mathcal{I}}{X^\circ}}\dt_{=0}
  +\underbrace{\textstyle\sum_k\ip{\mathcal{I}^\mathrm{j}}{X^\circ(t_k^-)}}_{=0},
\]
giving~\eqref{eq:RDT:game:cost}.  Under the saddle-point stability
assumption, $\ip{P_s(T)}{EX^\circ(T)E^*}\to0$ as $T\to\infty$, so
the bound is sharp.  No Geromel--Colaneri extension is needed.
\end{proof}

\subsubsection{Continuous- and discrete-time corollaries (infinite horizon)}

\begin{corollary}[Continuous-time game: infinite horizon]\label{cor:CT:game:inf}
  With $N_T=0$, Theorem~\ref{th:game:inf} gives the unique bounded
  $C^1$ solution $P_\infty$ of the continuous-time infinite-horizon HJI equation,
  recovering the stochastic $H_\infty$ algebraic Riccati
  equation~\cite{AitRami:00}.
\end{corollary}

\begin{corollary}[Discrete-time game: infinite horizon]\label{cor:DT:game:inf}
  With $\mathcal{F}=0$, $Z=0$ on flow, and $t_k=k_0+k$, the
  infinite-horizon game value is
  $J_\infty^\circ(X_1^0)=\ip{P_\infty(k_0)}{X_1^0}$, where $P_\infty$ is
  the unique bounded positive solution of the \emph{algebraic}
  discrete-time HJI-GRE~\eqref{eq:DT:game:GRE} (without terminal
  condition), obtained as the monotone non-decreasing limit of the
  $N$-step solutions, and the discrete-time saddle-point controllers are
  given by~\eqref{eq:DT:game:K} evaluated at $P_\infty$.
\end{corollary}

\subsection{Connection to \texorpdfstring{$H_\infty$}{H-infinity} and full-information games}
\label{sec:game:hinf}

  The standard $H_\infty$ (bounded $L_2$-gain) condition corresponds
  to the cost block $Z_{ww}=-\gamma^2 I$ (purely diagonal in the
  maximizing block), giving
  $\mathcal{I}_{ww}(t,P(t))=-\gamma^2 I+E_w\mathcal{F}^*(t,P(t))E_w^*$
  (from the flow adjoint, since $E_wE^*=0$).
  Isaacs' condition $\mathcal{I}_{ww}\preceq0$ then reads
  $\gamma^2 I\succeq E_w\mathcal{F}^*(t,P(t))E_w^*$,
  which is precisely a norm bound on the maximizing-block component
  of $\mathcal{F}^*$.
  With $Z_{uw}=0$ (no cross coupling in $Z$), one has
  $\mathcal{I}_{uw}=0$ so $\mathcal{G}=\mathrm{diag}(
  \mathcal{I}_{uu},\mathcal{I}_{ww})$ and the separation condition of
  Remark~\ref{rem:separation} holds: the optimal minimizer depends
  only on $X_1$.
  The I-HJI-GRE~\eqref{eq:game:GRE:flow} then recovers the
  $H_\infty$ Riccati equation of~\cite{AitRami:00}.\\

  When the maximizer's blocks $X_{33}$ and $X_{23}$ are
  \emph{measured} and available to the minimizer, the problem becomes
  a \emph{full-information game}: the minimizer's decision variable
  is the full triple $(X_{12},X_{22},X_{23})$, constrained only by
  $X\succeq0$.  This case admits the same I-HJI-GRE conditions, and
  the saddle-point gains $(K_u^s,K_w^s)$ from~\eqref{eq:game:Ks}
  can be implemented directly from the current matrix $X$.  The
  separation condition $\mathcal{I}_{uw}=0$ is no longer required,
  since the minimizer has direct access to the maximizing blocks.

\subsection{Long-term average game value}
\label{sec:game:AC}

We now develop the zero-sum game counterpart of
Section~\ref{sec:OC:AC}.  The setup parallels that subsection: the
dynamics are augmented by a Hermitian forcing on the state block as
in Definition~\ref{def:IMP:forced}, and the average-cost rate
replaces the integrated saddle-point value as the object of interest.
The minimax--maximin structure of Section~\ref{sec:game} is retained,
so that the long-term average game value is again given by a linear
pairing of the HJI-GRE solution with the forcing data, under Isaacs'
condition.

\subsubsection{Setting and assumptions}
\label{sec:game:AC:setup}

The forced impulsive system is that of Definition~\ref{def:IMP:forced}.
The augmented decision variable $X\in\Snmpsd$ has the three-block
partition~\eqref{eq:game:Xpart} with the minimizer controlling
$(X_{12},X_{22})$ and the maximizer controlling
$(X_{13},X_{23},X_{33})$.  The cost is the integrated linear cost of
Section~\ref{sec:game:form} without terminal penalty:
\begin{equation}\label{eq:game:AC:JT}
  J_T(X,X_1^0)
  :=\int_{t_0}^{t_0+T}\ip{Z(t)}{X(t)}\dt
   +\sum_{k:\,t_k\le t_0+T}\ip{Z_k}{X(t_k^-)},
\end{equation}
where $Z=\bigl[\begin{smallmatrix} Q&S_u&S_w\\ S_u^*&R_u&R_{uw}\\
S_w^*&R_{uw}^*&-R_w\end{smallmatrix}\bigr]$ has the game block
structure of Section~\ref{sec:game:form} with $R_u\succ0$,
$R_w\succ0$, $Q\succeq0$.

\begin{definition}[Long-term average game value]\label{def:game:AC}
The \emph{long-term average game value} is the saddle-point rate
\begin{equation}\label{eq:game:AC:rho}
  \rho^\circ(X_1^0)
  :=\limsup_{T\to\infty}\frac{1}{T}\!\left[\min_{X_{12},X_{22}}\,\max_{X_{13},X_{23},X_{33}}\,J_T(X,X_1^0)\right]
\end{equation}
under the dynamics~\eqref{eq:sys:IMP:forced} and $EX(t_0)E^*=X_1^0$.
\end{definition}

The game-theoretic counterparts of (A1)--(A3) of
Section~\ref{sec:OC:AC:setup} are:

\medskip\noindent\textbf{(A1$^\circ$)} \emph{Uniform stabilizability
of the saddle-point homogeneous closed-loop:}
under the saddle-point feedback of the homogeneous game (i.e., with
$\mathcal{W}=0,\mathcal{W}_k=0$), the closed-loop second-moment
satisfies $\norm{EX^\circ(t)E^*}\le\beta e^{-\alpha(t-s)}\norm{EX^\circ(s)E^*}$.

\medskip\noindent\textbf{(A2$^\circ$)} \emph{Uniform coercivity:}
$EZ(t)E^*\succeq\delta I$ uniformly for some $\delta>0$.

\medskip\noindent\textbf{(A3$^\circ$)} \emph{Uniform boundedness and
Isaacs' condition.}
$Z,Z_k,\mathcal{F},\mathcal{J}$ uniformly bounded;
$\pinv{\mathcal{G}},\pinv{(\mathcal{G}^\mathrm{j})}$
uniformly bounded on bounded sets of $P$; Isaacs' condition
(Definition~\ref{def:Isaacs}) holds uniformly in $t$; $\mathcal{W}(t),
\mathcal{W}_k$ uniformly bounded; inter-jump times bounded above and
below.

\medskip
Under (A1$^\circ$)--(A3$^\circ$), the homogeneous infinite-horizon
I-HJI-GRE has a unique bounded piecewise-$C^1$ solution
$P_\infty\succeq0$ by Theorem~\ref{th:game:inf}.

\subsubsection{Forced HJI key identity and finite-horizon decomposition}
\label{sec:game:AC:fh}

Since the HJI flow operator~\eqref{eq:game:I} is structurally identical
to the composite flow operator~\eqref{eq:IMP:L} with the
game-structured cost $Z$, and similarly for the jump
operator~\eqref{eq:game:Ij}, the forced key identity of
Proposition~\ref{prop:keyid:forced} applies verbatim with
$\mathcal{C}_Z\leftarrow\mathcal{I}$ and
$\mathcal{D}_Z\leftarrow\mathcal{I}^\mathrm{j}$:
\begin{align}\label{eq:game:AC:keyid}
  J_T(X,X_1^0)
  &=\ip{P(t_0)}{X_1^0}-\ip{P(t_0+T)}{EX(t_0+T)E^*}\notag\\
  &\quad
   +\int_{t_0}^{t_0+T}\ip{\mathcal{I}(t,P(t))}{X(t)}\dt
   +\sum_{k:\,t_k\le t_0+T}\ip{\mathcal{I}^\mathrm{j}}{X(t_k^-)}\notag\\
  &\quad
   +\int_{t_0}^{t_0+T}\ip{P(t)}{\mathcal{W}(t)}\dt
   +\sum_{k:\,t_k\le t_0+T}\ip{P(t_k^+)}{\mathcal{W}_k}.
\end{align}
This is the only ingredient needed for the average-cost theory.

\begin{theorem}[Finite-horizon average game value decomposition]
  \label{th:game:AC:finite}
Let $T>0$ and let $P_T$ be the piecewise-$C^1$ solution of the
finite-horizon I-HJI-GRE (Definition~\ref{def:game:GRE}) with
$P_T(t_0+T)=0$, assumed to exist under Isaacs' condition.  Then
the finite-horizon saddle value $J_T^\circ(X_1^0)$ satisfies
\begin{equation}\label{eq:game:AC:finite:J}
  J_T^\circ(X_1^0)
  =\ip{P_T(t_0)}{X_1^0}
  +\int_{t_0}^{t_0+T}\ip{P_T(t)}{\mathcal{W}(t)}\dt
  +\!\!\sum_{k:\,t_k\le t_0+T}\!\!\ip{P_T(t_k^+)}{\mathcal{W}_k},
\end{equation}
and is attained at the saddle-point policy $(K_u^s,K_w^s)$ from
$(\mathcal{K}_u^s\times\mathcal{K}_w^s)(t,P_T(t))$ on flow and the
analogous saddle gains at each jump.  The horizon-$T$ average game
value is $J_T^\circ/T$.
\end{theorem}

\begin{proof}
Apply~\eqref{eq:game:AC:keyid} with $P=P_T$ and $P_T(t_0+T)=0$, so
the second boundary term vanishes.  The I-HJI-GRE gives Isaacs'
condition and $\mathcal{I}\,\schur{/}\,\mathcal{G}=0$,
$\mathcal{I}^\mathrm{j}\,\schur{/}\,\mathcal{G}^\mathrm{j}=0$
(Definition~\ref{def:game:GRE}), so by the saddle
decomposition~\eqref{pf:game:decomp} (proof of Theorem~\ref{th:game:suff})
the saddle of $\ip{\mathcal{I}(t,P_T(t))}{X(t)}$ in the free blocks is
$0$, attained at the saddle policy, and likewise for
$\mathcal{I}^\mathrm{j}$.
Since the last two terms in~\eqref{eq:game:AC:keyid} are independent
of the free blocks $(X_{12},X_{22},X_{13},X_{23},X_{33})$, the minimax
saddle is attained at the same saddle policy as in the homogeneous
game, and the resulting saddle value is~\eqref{eq:game:AC:finite:J}.
\end{proof}

\subsubsection{Infinite-horizon: sufficient condition}
\label{sec:game:AC:inf:suff}

\begin{theorem}[Sufficient condition, infinite-horizon average game value]
  \label{th:game:AC:suff}
Suppose there exists a bounded piecewise-$C^1$ function
$P_\infty:[t_0,\infty)\to\Snpsd$ satisfying the infinite-horizon
I-HJI-GRE and Isaacs' condition uniformly, and define
\begin{equation}\label{eq:game:AC:rhostar}
  \rho^\circ
  :=\limsup_{T\to\infty}\frac{1}{T}\!\left[
    \int_{t_0}^{t_0+T}\ip{P_\infty(t)}{\mathcal{W}(t)}\dt
    +\!\!\sum_{k:\,t_k\le t_0+T}\!\!\ip{P_\infty(t_k^+)}{\mathcal{W}_k}
  \right].
\end{equation}
Then\textup{:}
\begin{enumerate}[\upshape(a)]
  \item $\rho^\circ(X_1^0)=\rho^\circ$ for every $X_1^0$;
  \item the saddle is attained at the saddle-point policy
    $(K_u^s,K_w^s)$ from $(\mathcal{K}_u^s\times\mathcal{K}_w^s)
    (t,P_\infty(t))$ on flow and the analogous gains at each jump;
  \item $P_\infty$ is the unique bounded solution of the
    infinite-horizon I-HJI-GRE.
\end{enumerate}
\end{theorem}

\begin{proof}
The argument is identical to that of Theorem~\ref{th:IMP:AC:suff}
with $\mathcal{C}_Z\leftarrow\mathcal{I}$,
$\mathcal{D}_Z\leftarrow\mathcal{I}^\mathrm{j}$, and the
$\sup_{X_w}\inf_{X_u}$ direction replacing $\inf_X$.  Specifically:

\noindent\textbf{Lower bound (for the minimizer).}
Fix any maximizer choice $X_w$.  Apply~\eqref{eq:game:AC:keyid} with
$P=P_\infty$.  By Isaacs' condition and the saddle-point
decomposition of Theorem~\ref{th:game:suff}, the
$\mathcal{I},\mathcal{I}^\mathrm{j}$ contributions are non-negative
when minimized over $X_u$ (with equality at the saddle).  Hence
\[
  \inf_{X_u}J_T(X,X_1^0)
  \ge\ip{P_\infty(t_0)}{X_1^0}-\ip{P_\infty(t_0+T)}{EX(t_0+T)E^*}+R_T(P_\infty),
\]
with $R_T(P):=\int\ip{P}{\mathcal{W}}\dt+\sum\ip{P(t_k^+)}{\mathcal{W}_k}$.
Taking the supremum over $X_w$ and dividing by $T$, the boundary
terms vanish in the limit (boundedness of $P_\infty$ and of
$EX^\circ E^*$ under the saddle-point closed-loop stability
(A1$^\circ$)), and $R_T(P_\infty)/T\to\rho^\circ$ by hypothesis.

\noindent\textbf{Attainment.}
At the saddle policy
$(K_u^s,K_w^s)\in(\mathcal{K}_u^s\times\mathcal{K}_w^s)(t,P_\infty(t))$,
the $\mathcal{I},\mathcal{I}^\mathrm{j}$ contributions vanish
identically, so~\eqref{eq:game:AC:keyid} gives
$J_T(X^\circ,X_1^0)=\ip{P_\infty(t_0)}{X_1^0}-\ip{P_\infty(t_0+T)}{EX^\circ(t_0+T)E^*}+R_T(P_\infty)$,
and the same limit computation gives the average game value
$\rho^\circ$.

\noindent\textbf{Uniqueness} follows from Theorem~\ref{th:game:inf:suff}.
\end{proof}

\subsubsection{Necessary condition and existence (infinite horizon)}
\label{sec:game:AC:inf:nec}

\begin{theorem}[Necessary condition and existence]\label{th:game:AC:nec}
Under (A1$^\circ$)--(A3$^\circ$) and Isaacs' condition, there exists
a unique bounded piecewise-$C^1$ solution $P_\infty\succeq0$ of the
infinite-horizon I-HJI-GRE, and
\begin{equation}\label{eq:game:AC:nec:val}
  \rho^\circ(X_1^0)
  \;=\;\limsup_{T\to\infty}\frac{1}{T}\!\left[
    \int_{t_0}^{t_0+T}\!\ip{P_\infty(t)}{\mathcal{W}(t)}\dt
    +\!\!\sum_{k:\,t_k\le t_0+T}\!\!\ip{P_\infty(t_k^+)}{\mathcal{W}_k}
  \right]
\end{equation}
independently of $X_1^0$. The right-hand side is finite whenever the
jump count grows at most linearly, i.e.,
$\limsup_{T\to\infty}T^{-1}\#\{k:t_k\le t_0+T\}<\infty$.
The saddle is attained at the trajectory $X^\circ$ with gains
$(K_u^s,K_w^s)\in(\mathcal{K}_u^s\times\mathcal{K}_w^s)(t,P_\infty(t))$
on flow and the analogous gains at each jump.
\end{theorem}

\begin{proof}
\textbf{Existence and bounds on $P_\infty$.}
Theorem~\ref{th:game:inf} applied to the homogeneous game provides a
unique bounded piecewise-$C^1$ solution $P_\infty\succeq0$.

\smallskip
\textbf{Key identity with $P_\infty$.}
For any admissible trajectory $X$ of~\eqref{eq:sys:IMP:forced}, the
forced HJI key identity~\eqref{eq:game:AC:keyid} with $P=P_\infty$ on
$[t_0,t_0+T]$ reads
\[
  J_T(X,X_1^0)
  =\ip{P_\infty(t_0)}{X_1^0}
   -\ip{P_\infty(t_0+T)}{EX(t_0+T)E^*}
   +\!\!\int_{t_0}^{t_0+T}\!\!\ip{\mathcal{I}(t,P_\infty(t))}{X}\dt
\]
\[
   +\sum_{k:\,t_k\le t_0+T}\ip{\mathcal{I}^\mathrm{j}}{X(t_k^-)}
   +R_T(P_\infty),
\]
with
$R_T(P_\infty):=\int_{t_0}^{t_0+T}\!\ip{P_\infty(t)}{\mathcal{W}(t)}\dt
+\sum_{k:\,t_k\le t_0+T}\!\ip{P_\infty(t_k^+)}{\mathcal{W}_k}$. Under
Isaacs' condition and the
I-HJI-GRE give $\mathcal{I}\,\schur{/}\,\mathcal{G}=0$,
so, with the maximizer playing $K_w^s$ (forcing $W_w=0$), the
saddle decomposition~\eqref{pf:game:decomp} gives
$\ip{\mathcal{I}(t,P_\infty)}{X}=\operatorname{tr}(W_u^*\mathcal{I}_{uu}W_u)\ge0$
a.e.\ on flow and $\ip{\mathcal{I}^\mathrm{j}}{X(t_k^-)}\ge0$ at every
jump, for any minimizer response $X_u$, with equality along the
saddle-point trajectory $X^\circ$. Hence, under the maximizer
policy $K_w^s$,
\begin{equation}\label{eq:game:AC:nec:bound}
  J_T(X,X_1^0)\;\ge\;\ip{P_\infty(t_0)}{X_1^0}
   -\ip{P_\infty(t_0+T)}{EX(t_0+T)E^*}+R_T(P_\infty),
\end{equation}
with equality at $X^\circ$.

\smallskip
\textbf{Boundary term is $O(1)$ in $T$.}
The analogue of Lemma~\ref{lem:IMP:AC:bdd} for the saddle-point
closed-loop dynamics bounds $\|EX(t)E^*\|\le C_X<\infty$ uniformly in
$t$ along any admissible trajectory under (A1$^\circ$); the proof uses
the geometric factor $\rho^j$ from the dual decay and never invokes a
bound on $\Delta_k$. Since $P_\infty$ is uniformly bounded,
$\ip{P_\infty(t_0+T)}{EX(t_0+T)E^*}\le\alpha_2 C_X$ uniformly in $T$.

\smallskip
\textbf{Conclusion.}
Dividing~\eqref{eq:game:AC:nec:bound} by $T$ and taking
$\limsup_{T\to\infty}$, both $T^{-1}\ip{P_\infty(t_0)}{X_1^0}$ and
$T^{-1}\ip{P_\infty(t_0+T)}{EX(t_0+T)E^*}$ vanish, leaving
$\limsup_T T^{-1}J_T(X,X_1^0)\ge\limsup_T T^{-1}R_T(P_\infty)$ for
any minimizer response $X_u$ against the maximizer policy $K_w^s$.
minimizing over $X_u$ and using that $K_w^s$ is one admissible maximizer
policy ($\max_{X_w}\min_{X_u}\ge\min_{X_u}(\cdot,K_w^s)$) gives
$\rho^\circ(X_1^0)\ge\limsup_T T^{-1}R_T(P_\infty)$. The saddle-point
trajectory $X^\circ$ saturates~\eqref{eq:game:AC:nec:bound}, so the
reverse inequality holds and
$\rho^\circ(X_1^0)=\limsup_T T^{-1}R_T(P_\infty)$, which
is~\eqref{eq:game:AC:nec:val}. The right-hand side does not depend
on $X_1^0$.

\smallskip
\textbf{Finiteness.}
$|R_T(P_\infty)|\le\alpha_2\bigl(T\sup_t\|\mathcal{W}(t)\|
+\#\{k:t_k\le t_0+T\}\cdot\sup_k\|\mathcal{W}_k\|\bigr)$, so
$T^{-1}|R_T(P_\infty)|$ is bounded above whenever
$\limsup_T T^{-1}\#\{k:t_k\le t_0+T\}<\infty$.
\end{proof}

\subsubsection{Equivalence and dual DLMI (infinite horizon)}
\label{sec:game:AC:LMI}

\begin{corollary}[Equivalence]\label{cor:game:AC:equiv}
Under Isaacs' condition, the infinite-horizon average game value is
well-posed with $\rho^\circ$ given by~\eqref{eq:game:AC:rhostar} if
and only if the infinite-horizon I-HJI-GRE has a bounded
piecewise-$C^1$ solution $P_\infty\succeq0$.  The solution is unique.
\end{corollary}

Define the average-forcing pairing $\Phi(P)$ as
in~\eqref{eq:IMP:AC:Phi}.  Since the game Riccati-inequality
feasibility set
$\{P:\text{Isaacs~\eqref{eq:Isaacs:gen}},\,
\mathcal{I}\,\schur{/}\,\mathcal{G}\succeq0,\,
\mathcal{I}^\mathrm{j}\,\schur{/}\,\mathcal{G}^\mathrm{j}\succeq0,\,
P\text{ bounded}\}$ has $P_\infty$ as its largest element
(Theorem~\ref{th:game:inf:LMI}), the natural direction for the dual
characterization is the supremum.

\begin{theorem}[Dual game Riccati inequality, infinite-horizon average game value]
  \label{th:game:AC:LMI}
Under (A1$^\circ$)--(A3$^\circ$),
\begin{equation}\label{eq:game:AC:LMI}
  \rho^\circ
  =\sup_P\;\Phi(P)
  \quad\text{s.t.}\quad
  \begin{gathered}\textnormal{Isaacs cond.~\eqref{eq:Isaacs:gen} for }\mathcal{I},\mathcal{I}^\mathrm{j};\\
  \mathcal{I}\,\schur{/}\,\mathcal{G}\succeq0,\;
  \mathcal{I}^\mathrm{j}\,\schur{/}\,\mathcal{G}^\mathrm{j}\succeq0,\;
  P\text{ bounded};\end{gathered}
\end{equation}
the supremum is attained at $P_\infty$.  As in
Theorem~\ref{th:game:inf:LMI}, the indefinite Isaacs block makes this a
game Riccati inequality (with Schur-complement term), not a plain LMI.
\end{theorem}

\begin{proof}
By Theorem~\ref{th:game:inf:LMI}, every feasible $P$ satisfies
$P\preceq P_\infty$ pointwise.  Since $\mathcal{W}\succeq0$ and
$\mathcal{W}_k\succeq0$, $\Phi$ is order-preserving, so
$\Phi(P)\le\Phi(P_\infty)$ for every feasible $P$, hence
$\sup_P\Phi(P)\le\Phi(P_\infty)$.  By
Theorem~\ref{th:game:AC:nec}, $\Phi(P_\infty)=\rho^\circ$, so the
supremum is bounded above by $\rho^\circ$ and is attained at
$P_\infty$.  Conversely, applying~\eqref{eq:game:AC:keyid} with a
feasible $P$ and dividing by $T$ gives, after the minimax over
$(X_u,X_w)$, $\rho^\circ\ge\Phi(P)$ for every feasible $P$, so
$\rho^\circ\ge\sup_P\Phi(P)$.  Hence equality.
\end{proof}

\subsubsection{Dwell-time conditions}
\label{sec:game:AC:dwell}

The three dwell-time scenarios of
Sections~\ref{sec:OC:AC:dwell}--\ref{sec:diss:AC:dwell} apply to the
zero-sum game with $\mathcal{C}_Z\leftarrow\mathcal{I}$ and
$\mathcal{D}_Z\leftarrow\mathcal{I}^\mathrm{j}$ throughout, and with
the saddle-point feedback $(\mathcal{K}_u^s\times\mathcal{K}_w^s)$
replacing the optimal feedback $\mathcal{K}_c$.  The forward HJI-GRE
of~\eqref{eq:forwardHJI} plays the role of the forward
Riccati~\eqref{eq:forwardRiccati}, and the Geromel--Colaneri
condition is the frozen game Riccati inequality
$\mathcal{I}^0\,\schur{/}\,\mathcal{G}^0\preceq0$ with Isaacs'
condition~\eqref{eq:Isaacs:gen} on $\mathcal{I}^0$~\eqref{eq:GC:game},
as in Theorem~\ref{th:game:MDT}.

\begin{theorem}[Causal LTAC under periodic, MDT, RDT game]
  \label{th:game:AC:dwell}
Assume Isaacs' condition~\eqref{eq:Isaacs:gen} uniformly and the
hypotheses of Theorems~\ref{th:game:periodic},~\ref{th:game:MDT},
and~\ref{th:game:RDT} respectively (in particular, the periodicity
(periodic case), the time-invariance beyond $T_{\min}$ with the frozen
LMI~\eqref{eq:GC:game} (MDT case), and the timer-dependence (RDT case)
of the flow operator $\mathcal{F}$, the jump operator $\mathcal{J}$,
the costs $Z,Z_k$ (with the game block structure
of Section~\ref{sec:game:form}), \emph{and} of the forcing data
$\mathcal{W},\mathcal{W}_k$) together with
(A1$^\circ$)--(A3$^\circ$) and the bounded jump-rate conditions of
Theorems~\ref{th:IMP:AC:dwell:periodic}--\ref{th:IMP:AC:dwell:RDT}.
Let $P_\infty(\tau)$ or $P_s(\tau)$ denote the causal storage
function from the respective game theorem.  Then the causal
long-term average game value satisfies:
\begin{enumerate}[\upshape(a)]
  \item \emph{Periodic} $\Delta_k\equiv T_p$:
    $\rho^\circ=T_p^{-1}\bigl[\int_0^{T_p}\ip{P_\infty(\tau)}{\mathcal{W}(\tau)}d\tau
    +\ip{P_\infty(0^+)}{\mathcal{W}_k}\bigr]$.
  \item \emph{MDT} $\Delta_k\ge T_{\min}$:
    $\rho^\circ\le
    \bar\nu\int_0^{T_{\min}}\!\ip{P_s(\tau)}{\mathcal{W}}d\tau
    +(1-\bar\nu T_{\min})\ip{P_s(T_{\min})}{\mathcal{W}}
    +\bar\nu\ip{P_s(0^+)}{\mathcal{W}_k}$.
  \item \emph{RDT} $\Delta_k\in[T_{\min},T_{\max}]$:
    $\rho^\circ\le
    \bar\nu\int_{T_{\min}}^{T_{\max}}\!(\int_0^\Delta\ip{P_s(\tau)}{\mathcal{W}}d\tau)\mu(d\Delta)
    +\bar\nu\ip{P_s(0^+)}{\mathcal{W}_k}$.
\end{enumerate}
The periodic case is exact and the saddle is attained at
$(\mathcal{K}_u^s\times\mathcal{K}_w^s)(\tau,P_\infty(\tau))$ on flow;
MDT and RDT give upper bounds for the minimizer under any maximizer
strategy and any admissible dwell-time sequence (or timer-occupation
measure $\mu$).
\end{theorem}

\begin{proof}
The proofs of Theorems~\ref{th:IMP:AC:dwell:periodic}--\ref{th:IMP:AC:dwell:RDT}
apply verbatim with the substitutions
$\mathcal{C}_Z\leftarrow\mathcal{I}$,
$\mathcal{D}_Z\leftarrow\mathcal{I}^\mathrm{j}$, and the
minimax direction (rather than minimization) in the saddle-point step.
Isaacs' condition uniformly ensures the saddle-point decomposition
applies on the bounded timer interval, and the Geromel--Colaneri
condition $\mathcal{I}^0\,\schur{/}\,\mathcal{G}^0\preceq0$ from
Theorem~\ref{th:game:MDT}
ensures the integral contribution on the tail $\tau>T_{\min}$ is
non-positive (so the upper-bound inequality is preserved).  All
other steps (timer occupation, jump-rate average) are identical.
\end{proof}

\subsubsection{Continuous- and discrete-time corollaries}
\label{sec:game:AC:CTDT}

\begin{corollary}[Continuous-time average game value]
  \label{cor:CT:game:AC}
Set $N_T=0$.  Under (A1$^\circ$)--(A3$^\circ$),
$\rho^\circ=\lim_T T^{-1}\int_{t_0}^{t_0+T}\ip{P_\infty(t)}{\mathcal{W}(t)}\dt$,
where $P_\infty$ is the unique bounded $C^1$ solution of the
continuous-time infinite-horizon HJI equation
(Corollary~\ref{cor:CT:game:inf}).  In the LTI case
($\mathcal{F},Z,\mathcal{W}$ constant), $P_\infty$ is constant and
$\rho^\circ=\tr(P_\infty\mathcal{W})$.
\end{corollary}

\begin{corollary}[Discrete-time average game value]
  \label{cor:DT:game:AC}
Apply with $\mathcal{F}=0$, $Z=0$ on flow, $t_k=k_0+k$, and
$\mathcal{W}=0$ on flow.  Then
$\rho^\circ=\lim_N N^{-1}\sum_{k=k_0}^{k_0+N-1}\ip{P_\infty(k+1)}{\mathcal{W}_k}$,
where $P_\infty$ is the unique bounded positive solution of the
algebraic discrete-time HJI-GRE~\eqref{eq:DT:game:GRE}
(Corollary~\ref{cor:DT:game:inf}) and the discrete-time saddle-point
controllers are given by~\eqref{eq:DT:game:K} at $P_\infty$.  In the
LTI case, $\rho^\circ=\tr(P_\infty\mathcal{W})$.
\end{corollary}

\begin{remark}[$H_\infty$-type interpretation]\label{rem:game:AC:Hinf}
With the cost block $Z_{ww}=-\gamma^2 I$ as in
the $H_\infty$ remark in Section~\ref{sec:game:hinf}, the saddle game value rate $\rho^\circ$ has
the meaning of the worst-case disturbance-to-output energy gain per
unit time, augmented by the $\gamma$-bound on the noise channel.  In
the LTI case with $\mathcal{W}=B_w B_w^*$, $\rho^\circ=\tr(P_\infty
B_w B_w^*)$ is the standard steady-state $H_\infty$ "$\gamma$-cost"
rate.  minimizing $\gamma^2$ subject to the
DLMI~\eqref{eq:game:AC:LMI} for $\rho^\circ\le\gamma^2$ recovers the
matrix-valued analogue of the classical
$H_\infty$-optimal cost rate problem.
\end{remark}



\section{Discussion}
\label{sec:summary}


Every result in this paper rests on three tools that are independent
of the system class and of the sign of $M$ or $P$: the key identity
(Sections~\ref{sec:OC:fh}--\ref{sec:diss:inf}),
Lemma~\ref{lemma:extSchur} (Extended Schur), and
Lemma~\ref{lemma:decomp} (Schur inner-product decomposition).
These apply identically to the flow integral, to each jump summand,
and hence to the full impulsive setting.


We have developed a comprehensive optimal control, optimal
steering, dissipativity, and game-theoretic framework for linear
Hermitian matrix-valued impulsive dynamical systems.  All results
are obtained under LTV assumptions and are general enough to cover
state-vector continuous-time and discrete-time systems subject to input
noise, Markovian and deterministic switching, and random and
deterministic jumps in the state---including sampled-data systems as a
special case. Consequently, existing results on the optimal control,
the dissipativity analysis, and the minimax control of such systems can
be recovered as corollaries of those obtained here. The dwell-time
conditions likewise generalize existing ones in the literature.

Several directions remain open.  A natural and important extension
is to \emph{stochastic} Hermitian matrix-valued systems: replacing
the deterministic flow $E\dot{X}E^*=\mathcal{F}(t,X)$ by an
It\^o-type matrix-valued SDE that includes a diffusion
term~\cite{Chen:16} \emph{and} stochastic jumps (not considered
here), where both the jump times and the post-jump state have
distributions depending on the pre-jump state.  Such a setting
would extend the present deterministic impulsive framework to the
genuinely stochastic case over Hermitian matrices; it is
qualitatively distinct from the Markov jump linear system
framework of~\cite{Costa:13,Dragan:25}, where it is the
second-order moment of the vector state that evolves according to
a deterministic impulsive flow as studied in~\cite{Briat:22:Matrix}.
Obtaining necessary causal conditions for MDT (as
Theorem~\ref{th:OC:periodic} does for the periodic case) remains
an open problem, as does extending the framework to switched
impulsive systems with mode-dependent dynamics.  The discussion
below addresses a third, structurally distinct direction.


The results of this paper rely on three structural properties of
the cone $\Snpsd$ (using the notation of this paper; the cone
$\Snpsd=\mathbb{H}^n_{\succeq0}$ of positive semidefinite Hermitian
matrices is also written $\mathcal{S}_n^+$ in the positive-systems
literature): it is a proper cone (closed, convex, pointed, and with
non-empty interior); it is self-dual ($(\Snpsd)^*=\Snpsd$, a
consequence of the trace inner product being non-degenerate on
$\Snpsd$); and the dynamics are cone-preserving ($\mathcal{F}$ and
$\mathcal{J}$ leave $\Snpsd$ invariant).  Remarkably, these are
the only structural properties used by the three tools underlying
all proofs: the key identity (whose derivation uses only linearity
of $\mathcal{F}$ and $\mathcal{J}$ in $X$), the Extended Schur
lemma (Lemma~\ref{lemma:extSchur}), and the Schur inner-product
decomposition (Lemma~\ref{lemma:decomp}).

It is therefore natural to ask whether the entire framework, or
large parts of it, extend to systems evolving on an arbitrary
proper cone $\mathcal{K}\subseteq\mathbb{R}^n$ with dynamics that
preserve $\mathcal{K}$ and costs that are linear on $\mathcal{K}$.
This generalization is precisely the setting studied in the recent
work on \emph{optimal control on positive cones}
(Pates and Rantzer~\cite{Pates:24}, Vladu, Megretski, and
Rantzer~\cite{Vladu:26}), where the Bellman equation for LQ-type
problems can be solved explicitly by a linear function under a
critical structural condition on the cone.  Stability analysis for
cone-preserving systems has been studied in, e.g.,~\cite{Shen:15}, where linear Lyapunov or storage functions play the
role that inner products with elements of $\Snpsd$ play here.

The parallel is more than formal.  In the $\Snpsd$-cone case,
the dual variable $P$ in the key identity is itself an element of
the dual cone $(\Snpsd)^*=\Snpsd$, the linear cost
$\ip{P}{X}=\tr(PX)$ is the natural duality pairing between
$P\in\Snpsd$ and $X\in\Snpsd$, and all LMI/Riccati conditions
are cone constraints in $\Snpsd$.  For a general proper cone
$\mathcal{K}$, the analogous objects are: the dual cone
$\mathcal{K}^*$, the pairing $\langle\lambda,x\rangle$ for
$\lambda\in\mathcal{K}^*$ and $x\in\mathcal{K}$, and the
corresponding linear inequalities on $\mathcal{K}^*$.  It appears
plausible that the finite-horizon optimal control, the
dissipativity, and the dwell-time results of this paper extend
verbatim to this general-cone setting, with $\Snpsd$ replaced by
$\mathcal{K}$, inner products replaced by the duality pairing, and
Schur complements replaced by the appropriate cone-linear
factorization.  We leave the systematic development of this
direction for future work; it would unify the $\Snpsd$-cone
(Hermitian matrix-valued) theory of this paper with the
positive-orthant theory of~\cite{Pates:24,Vladu:26} and the
cone-invariant theory of~\cite{Shen:15} under a single
abstract framework for cone-preserving impulsive systems.

\bibliographystyle{plain}
\bibliography{./global,./briat}

\end{document}